\documentclass[11pt]{article}
\UseRawInputEncoding
\usepackage{bbm}
\usepackage{amssymb, amsthm, amsmath, amscd}
\usepackage{tikz-cd}
\newtheorem{thm}{Theorem}[section]
\newtheorem{lem}[thm]{Lemma}
\newtheorem{prop}[thm]{Proposition}
\newtheorem{cor}[thm]{Corollary}
\theoremstyle{definition}

\theoremstyle{definition}
\newtheorem{df}[thm]{Definition}
\theoremstyle{definition}
\newtheorem{rem}[thm]{Remark}

\theoremstyle{definition}

\renewcommand{\phi}{\varphi}

\newcommand{\N}{\mathbb{N}}

\newcommand{\Q}{\mathbb{Q}}
\newcommand{\R}{\mathbb{R}}
\newcommand{\C}{\mathbb{C}}

\numberwithin{equation}{section}

\newcommand{\Aff}{\operatorname{Aff}}

\newcommand{\id}{\operatorname{id}}

\newcommand{\cpc}{c.p.c.~map}

\newcommand{\hm}{homomorphism}
\newcommand{\dt}{\delta}
\newcommand{\ep}{\varepsilon}

\newcommand{\1}{\mbox{\large \bf 1}}

\newcommand{\la}{\langle}
\newcommand{\ra}{\rangle}
\newcommand{\andeqn}{\,\,\,{\rm and}\,\,\,}
\newcommand{\rforal}{\,\,\,{\rm for\,\,\,all}\,\,\,}
\newcommand{\CA}{$C^*$-algebra}
\newcommand{\SCA}{$C^*$-subalgebra}

\newcommand{\af}{{\alpha}}
\newcommand{\bt}{{\beta}}
\newcommand{\dist}{{\rm dist}}

\newcommand{\diag}{{\rm diag}}

\newcommand{\wilog}{without loss of generality}
\newcommand{\Wlog}{Without loss of generality}

\newcommand{\beq}{\begin{eqnarray}}
\newcommand{\eneq}{\end{eqnarray}}
\newcommand{\tforal}{\,\,\,\text{for\,\,\,all}\,\,\,}
\newcommand{\tand}{\,\,\,\text{and}\,\,\,}

\usepackage{amsfonts}
\usepackage{mathrsfs}
\usepackage{textcomp}
\usepackage[all]{xy}

\newcommand{\Her}{\mathrm{Her}}
\newcommand{\Cu}{\mathrm{Cu}}

\newcommand{\Zjs}{\mathcal Z}
\newcommand{\F}{\mathcal F}
\newcommand{\supp}{\operatorname{supp}}

\title{\CA s of Essentially Tracially Rank One}
\author{Huaxin Lin}
\date{}

\begin{document}

\maketitle

\begin{abstract}
We introduce the notion of essential tracial rank at most one for
simple \CA s. We show that separable simple \CA s with essential
tracial rank at most one have several regularity properties, including
stable rank one, strict comparison, and tracial approximate
divisibility. A finite amenable simple \CA\ with this property is
$\mathcal Z$-stable.  We also prove the converse. Thus this property
gives a new characterization of separable simple finite amenable
$\mathcal Z$-stable \CA s and shows that these algebras admit a
one-dimensional tracial approximation. In the unital real-rank-zero
case, this approximation can be strengthened to show that the algebra
has tracial rank zero. This completes a long-standing project. As a
byproduct, we prove that every tracial state on a unital separable
simple finite amenable \CA\ is quasidiagonal and that the algebra itself
is quasidiagonal. Consequently, every separable simple finite amenable
\CA\ of finite nuclear dimension has finite decomposition rank. None
of these results requires the UCT.

\end{abstract}

\section{Introduction}

The last two decades have witnessed rapid progress in the classification of
separable amenable simple \CA s. For large classes of such algebras,
the Elliott invariant, consisting of ordered $K$-theory together with
tracial data and the natural pairing between them, has been shown to
be a complete invariant. In the stably finite setting, the combination
of amenability, $\mathcal Z$-stability, and the Universal
Coefficient Theorem has led to classification results covering both
unital and stably projectionless  simple \CA s; see, for example,  \cite{GLNicbs}, \cite{GLNI}, \cite{GLNII}, \cite{EGLNII}, \cite{TWW}, \cite{CETWW},
\cite{GL2024}, and
the references therein (see also more recent article \cite{CGSTW}). These results confirm much of the original
Elliott classification program within the UCT class.

An important ingredient in the existing classification theory is
tracial approximation. 
In its earlier forms, tracial approximation approximately decomposes
the algebra into two orthogonal parts: one part is tracially small,
while the other is approximately contained in a subhomogeneous
$C^*$-algebra whose base space has dimension at most one (see \cite{LinTAF}, \cite{Lintr2001}, \cite{EGLNs}). In more general
formulations, such a decomposition is obtained only after tensoring
with a UHF algebra (see \cite{Wlocal} and \cite{Lininv}).
This type of approximation is extremely
useful: it provides access to stable uniqueness theorems, allows one
to calculate invariants, and plays an essential role in inductive
constructions.

There is, however, a limitation inherent in this form of approximation.
The subhomogeneous algebra occupies the tracially large part of the
decomposition and consequently carries most of the structural
information used in the approximation. In particular, its $K$-theory determines the pairing between
$K$-theory and traces.

 Thus, although tracial approximation by low-dimensional
subhomogeneous algebras is a powerful classification tool, it is not
$K$-theoretically neutral. Its use is closely tied to the construction
and control of $K$-theoretic data, and this is one of the places where
the UCT naturally enters the existing classification arguments.
Moreover, rationalized tracial approximation and tracial
approximation itself need not coincide (in the unital case, for example). 

The purpose of this paper is to study a different form of tracial
approximation. We seek a genuinely one-dimensional approximation that
captures the tracial and local analytic structure of the algebra
without prescribing its $K$-theory or, in particular, the corresponding  pairing.  We therefore relax two requirements
in the usual tracial approximation. First, we do not require the two
approximating parts to be orthogonal. Consequently, the pairing between
$K$-theory and traces is no longer dictated by the tracially large
$C^*$-subalgebra. Second, we require the good part to be  the image of 
a standard one-dimensional cone, which is a  contractible $C^*$-algebra with trivial $K$-theory.

This raises a basic question: is such a one-dimensional approximation
an additional regularity hypothesis, or is it an abstract approximation
property already possessed by $\mathcal Z$-stable algebras? One of the
main results of this paper shows that, under amenability, this property
is automatic. More precisely:

\begin{thm}\label{TA}
Let $A$ be a separable simple finite amenable \CA.
Then  $A$ is ${\cal Z}$-stable if and only if $ETR(A)\le 1.$
\end{thm}

\vspace{-0.1in}

\noindent
(See Definition \ref{D1} and Definition \ref{DD1}). 

Theorem \ref{TA} is a combination of  Theorem \ref{MMCC} and Theorem \ref{TmainZ}.

We now describe the approximation property more precisely. Its local
form is called property ETR1. Let $A$ be a simple $C^*$-algebra and
fix a strictly positive element $e_A$. Given a finite subset
$\mathcal F\subset A$, a number $\ep>0,$
a non-zero positive element $a\in A_+$, property ETR1 provides a positive
element  $e\in A$ with $0\le e\le 1,$ a finite-dimensional \CA\,  $D$, and a
fully normalized order-zero \cpc\,
\[
 \theta\colon D\longrightarrow\overline{(1-e)A(1-e)}.
\]
Such a map is induced by a
homomorphism
\[
 C_0((0,1])\otimes D\longrightarrow A.
\]
The algebra $C_0((0,1])\otimes D$ is one-dimensional and contractible.
The element $e$ approximately commutes with $\mathcal F$, the
compressions
\[
 (1-e)^{1/2}x(1-e)^{1/2}\rforal
 x\in\mathcal F,
\]
are uniformly close to $\theta(D)$, and $e$ is Cuntz subequivalent to
$a$. Thus, the part which may not be approximated  by the one-dimensional model can be
made arbitrarily small in  trace. In this sense, images of one-dimensional
cones are dense in the uniform tracial $L^2$-norm. The significance of
Theorem~1.1 is that this tracial approximation property is automatic
for finite, separable, simple, amenable, $\mathcal Z$-stable
$C^*$-algebras. Establishing a result of this form has been a
long-standing goal of the author.

There is, however, one additional nondegeneracy requirement though. The compression
\[
 (1-e)^{1/2}e_A(1-e)^{1/2}
\]
is required to be close to $\theta(d_{e_A})$ for a positive contraction
$d_{e_A}\in D$ such that
\[
 t(d_{e_A})\geq\kappa_{e_A}
 \qquad\text{for all }t\in T(D),
\]
where $\kappa_{e_A}>0$ is independent of $\mathcal F$, $\varepsilon$,
and $a$. This condition prevents the order-zero approximation from
disappearing (as the cone is contractible).
The tracially large part of
$A$ is therefore genuinely detected by the one-dimensional model,
even though the model is not intended to determine the $K$-theory of
$A$ since the cone has zero K-theory (see Remark \ref{R1} for further discussion).

However, there is a further stronger  non-degenerate condition introduced later in the paper 
(Section 13).
In the unital case, Theorem \ref{TA}  in fact holds with this stronger
condition. This of course comes no surprise as compact spaces usually behave
better than non-compact spaces. This additional control leads to the following result.


\begin{thm}\label{TB}
Let $A$ be a unital separable simple finite amenable 
${\cal Z}$-stable \CA\, with real rank zero.  Then $A$ has tracial rank zero.
\end{thm}

This result was one of the starting points of the present work.
None of the results described here requires the UCT. We also
characterize finite, separable, simple, amenable, $\mathcal Z$-stable
$C^*$-algebras by a weaker tracial-rank-zero property; see
Theorem~\ref{prop12.5a}. This property may be viewed as a form of ``tracial rank
zero'' in the quotient of the sequence algebra by the trace-kernel
ideal. It is closely related to the quasidiagonality question in
\cite[Question 8.8]{BBSTWW}. 

\begin{thm}\label{TT}
Let $A$ be a unital separable simple amenable \CA.
Then every tracial state is quasidiagonal. Moreover,  when $T(A)\not=\emptyset,$
 $A$ is quasidiagonal.
\end{thm}

Under the UCT, quasidiagonality of traces was proved by Tikuisis, White, and Winter in  \cite{TWW}. 
 Theorem \ref{TT} gives an affirmative answer to the Blackadar--Kirchberg quasidiagonality problem
 \cite[Problem~VII]{STW} in the unital simple case. It also answers \cite[Problem~IX(1)]{STW}
in that case. If $A$ is additionally ${\cal Z}$-stable, then $\operatorname{dr}(A)\leq1$  by \cite[Theorem~B]{CETWW}. Thus Theorem \ref{TT}, together with \cite[Theorem B]{CETWW}, confirms 
what is called in \cite[the paragraph following Question 8.8]{BBSTWW} the strongest form of the Toms--Winter conjecture on the stably finite side: 
every finite, simple, separable, unital, nuclear, ${\cal Z}$-stable C*-algebra has finite decomposition rank. We record the following consequence, which generalizes \cite[Corollary 8.6]{CETWW}.


\begin{cor}\label{CMC}
Let \(A\) be a non-elementary unital separable finite amenable 
simple \(C^*\)-algebra. Then the following are equivalent:
\begin{enumerate}
\item \(\operatorname{ETR}(A)\leq 1\);
\item \(A\cong A\otimes\mathcal Z\);
\item \(\dim_{\mathrm{nuc}}(A)<\infty\);
\item \(\operatorname{dr}(A)<\infty\).
\end{enumerate}
When these conditions hold,
\[
   \dim_{\mathrm{nuc}}(A)\leq 1
   \quad\text{and}\quad
   \operatorname{dr}(A)\leq 1.
\]
More precisely,
\[
\operatorname{dr}(A)=\dim_{\mathrm{nuc}}(A)=
\begin{cases}
0, & \text{if \(A\) is AF},\\
1, & \text{otherwise}.
\end{cases}
\]
\end{cor}

The flexibility of property ETR1 comes at a technical cost.  In contrast
to the projection-based decompositions used in earlier tracial
approximation theories, the element $e$ above is only a positive
contraction.  Consequently, $1-e$ is not the complementary projection of
$e$, and the image $\theta(D)\subset\overline{(1-e)A(1-e)}$ need not be
orthogonal to the part controlled by $e$.  In particular, 
the map $a\mapsto (1-e)^{1/2}a(1-e)^{1/2}$    need not be approximately multiplicative. 
These difficulties appear,
for example, in the proofs of Theorem \ref{Ltrace} and Proposition \ref{Phered}, as well as 
that of Proposition \ref{prop:continuous-scale-improvement}. At the same
time, this lack of orthogonality is precisely what allows property
ETR1 to hold in such generality: it avoids the need for sufficiently
many projections or for a prescribed subhomogeneous model and
therefore also applies in projectionless settings.


The paper is organized as follows. Section~2 records notation and
preliminary results concerning traces, the Cuntz semigroup, continuous
scale, and sequence algebras. Section~3 introduces property ETR1.
Section~4 shows that, in the definition of ETR1,  the matrix sizes of the finite-dimensional
approximating algebras may be chosen arbitrarily large. Section~5 proves
the existence of faithful traces for \CA s with property ETR1. Sections~6 and~7 establish the
hereditary and Morita-invariance properties needed to define essentially
tracial rank at most one. Section~8 studies affine functions and the
rank map for \CA s with property ETR1.  Section~9 proves that simple \CA s with property ETR1 have tracial approximate divisibility and deduces
$\mathcal Z$-stability. Section~10 proves that every separable 
finite simple amenable $\mathcal Z$-stable \CA\, has property ETR1 (Theorem \ref{TmainZ}).
Section 11 studies the constant associated with the prescribed strictly positive element, and Section 12 establishes several approximation and permanence results. In Section 13, we show that unital, separable, simple, finite, amenable, ${\cal Z}$-stable \CA s have the stronger tracial approximation property sETR1. We then prove Theorems \ref{TB} and \ref{TT} and obtain Corollary \ref{CMC}. The final section gives a direct verification of property ETR1 for the Jiang--Su algebra.

\bigskip

{\bf Acknowledgements}  

This research has been carried out on and off since the 2020 pandemic. A preliminary version was reported at the Canadian Operator Symposium (COSy), June 1--5, 2026, at Carleton University, Ottawa.

The original version of one direction of Theorem 1.1 used the existence of selective ultrafilters, which may require the Continuum Hypothesis or ${\mathfrak{p}}={\mathfrak{c}}.$ The author would like to thank Narutaka Ozawa, who convinced the author during a conversation at COSy 2026 that none of these extra axioms are needed.   Indeed we subsequently remove this condition. 

The author began experimenting with the use of OpenAI ChatGPT 5.6 Sol in August 2026. The example in Section 14 was an AI experiment; it was subsequently edited and revised.
The AI assistance in this research is still at an early stage, but the author would like to acknowledge that it significantly improved the current version.

This work is supported by grants from SIMIS. 

  
\section{Preliminaries}

\begin{df}\label{Da1}
Let $A$ be a \CA. 

$\bullet$  Denote by $A^{\bf 1}$ the unit ball of $A.$ 

$\bullet$ $A_{s.a.}$ is the set of self-adjoint elements of $A,$ and $A_+$ is the  set of 
positive elements of $A.$    

$ \bullet$ Suppose that ${\cal G}\subset A.$ Then ${\cal G}_+={\cal G}\cap A_+.$

$\bullet$ $A_{s.a.}^{\bf 1}=A^{\bf 1}\cap A_{s.a.}$ and $A_+^{\bf 1}=A^{\bf 1}\cap A_+.$

$\bullet$  ${\rm Ped}(A)$ is the Pedersen ideal of $A.$ 

$\bullet$ $T(A)$ is the tracial state space of $A$ and $QT(A)$ the set of normalized 2-quasitraces.

\end{df}

\begin{df}\label{DtildeQA}
Let $A$ be a \CA. A densely defined 2-quasitrace is a 2-quasitrace defined on ${\rm Ped}(A\otimes {\cal K})$ 
(see \cite[Definition II.1.1]{BH}).  Denote by $\widetilde{QT}(A)$ the set of densely defined 2-quasitraces on 
$A\otimes {\cal K}.$  In what follows  we will always identify $A$ with $A\otimes e_{1,1}$ whenever it is convenient. 

We endow $\widetilde{QT}(A)$ with the topology in which a net $\{\tau_i\}$ converges to $\tau$ if 
$\{\tau_i(a)\}$ converges to $\tau(a)$ for all $a\in {\rm Ped}(A\otimes {\cal K})$ (see also (4.1) on page 985 of \cite{ERS}).

Note that, for each $a\in (A\otimes {\cal K})_+$ and $\ep>0,$ $f_\ep(a)\in {\rm Ped}(A\otimes {\cal K})_+.$
Define, for each $\tau\in \widetilde{QT}(A),$
\beq
&&\widehat{a}(\tau):=\tau(a):=\lim_{\ep\to 0}\tau(af_\ep(a))\andeqn\\
&&\widehat{[a]}(\tau):=d_\tau(a):=\lim_{\ep\to 0}\tau(f_\ep(a)).
\eneq
Denote by $\widetilde{T}(A)$ the subset of $\widetilde{QT}(A)$ consisting of traces.
\end{df}

\begin{df}\label{DQTGamma}
Let $A$ be a \CA\, with $\widetilde{QT}(A)\setminus \{0\}\not=\emptyset.$
Denote by $L(\widetilde{QT}(A))$ the family of continuous real-valued functions $f$ on 
$\widetilde{QT}(A)$ such that $f(\af \tau)=\af f(\tau)$ and 
$f(\tau+t)=f(\tau)+f(t)$ for all  $\af\in \R_+$ and for all $\tau, t\in \widetilde{QT}(A).$ 

Let $S\subset \widetilde{QT}(A)$ be a convex subset. Set 
\beq
\Aff_+(S)&=&\{f|_S: f\in L(\widetilde{QT}(A)), \,f(\tau)>0\,\, {\rm if}\,\, \tau\in S\setminus \{0\}\}\cup \{0\},\\
{\rm LAff}_+(S)&=& \{f: S\to [0,\infty]: \exists \{f_n\}\subset \Aff_+(S),\, f_n\nearrow f\}.
\eneq

Let $A$ be a simple \CA. Define $\Gamma: {\rm Cu}(A)\to {\rm LAff}_+(\widetilde{QT}(A))$ by
$\Gamma([a])(\tau)=\widehat{[a]}(\tau)=d_\tau(a)$ for all $\tau\in \widetilde{QT}(A).$ 

In case that $A$ is simple and $A={\rm Ped}(A),$ $\Gamma$ also induces
a canonical map $\Gamma: {\rm Cu}(A)\to {\rm LAff}_+(\overline{QT(A)}^w),$ where 
$\overline{QT(A)}^w$ is the weak *-closure of $QT(A).$ 

Let $a\in {\rm Ped}(A)$ and $\tau\in \widetilde{QT}(A).$  Put $\|a\|_{_{2, \tau}}=\tau(a^*a)^{1/2}.$
If $S\subset  \widetilde{QT}(A),$ put
\beq
\|a\|_{_{2, S}}=\sup_{\tau\in S} \|a\|_{_{2, \tau}}=\sup_{\tau\in S}\tau(a^*a)^{1/2}.
\eneq

See \cite[Definition 2.13]{FLosc}.
\end{df}

\begin{df}
Let $A$ be a \CA.  Recall that an element $a\in A_+\setminus \{0\}$ is {\it strongly uniformly full}  in $A,$ 
if there are $M(a)>0$ and an integer $N(a)\in \N$ such that, for any 
$b\in A_+$ with $\|b\|\le 1,$ there are $x_1, x_2,...,x_{N(a)}\in A$ with $\|x_i\|\le M(a),$ $1\le i\le N(a)$
such that
\beq
\sum_{i=1}^{N(a)} x_i^* ax_i=b
\eneq
(see \cite[Definition 5.5]{EGLNs}). 

$\bullet$ If this happens, we say $a$ is $M(a)$-$N(a)$-strongly uniformly full. 

$\bullet$ By \cite[Proposition 5.6]{EGLNs},  every non-zero positive element in  a $\sigma$-unital algebraically simple \CA\, 
$A$ is strongly uniformly full in $A.$ 

 Let $A$  be a $\sigma$-unital algebraically simple \CA\, 
and $a\in A_+\setminus \{0\}.$ 

$\bullet$ Set 
\beq
{\rm SUF}(a)=\inf\{ M(a)N(a): a \,\,\, {\rm is}\,\, {M(a){-}N(a)-}{\rm strongly\,\, uniformly\,\, full}\,\,\}.
\eneq

\end{df}
From the definition, one immediately obtains the following:

\begin{prop}\label{Pfullrank}
Let $A$ be a $\sigma$-unital algebraically simple \CA\, and $a\in A_+\setminus\{0\}.$
Then, for any $\ep>0,$ $a$ is $M(a)$-$N(a)$-strongly uniformly full in $A$ for some 
$M(a)>0$ and $N(a)\in \N$ such that
\beq
M(a)N(a)\le {\rm SUF}(a)+\ep.
\eneq

\end{prop}

$\bullet$ In what follows, when we say $a$ is $M(a)$-$N(a)$-strongly uniformly full in $A,$
we {\it always} assume that $M(a)N(a)\le {\rm SUF}(a)+1.$

\begin{df}\label{Dgenerating}
Let $D$ be a finite dimensional \CA. We may write 
$D=M_{r_1}\oplus M_{r_2}\oplus \cdots \oplus M_{r_k},$ where each $r_s$ is a positive integer 
and $M_{r_s}$ is the \CA\, of $r_s\times r_s$ matrices over $\C,$ $s=1,2,...,k.$ 
Let $\{e_{i,j}^{s}\}$ be a system of matrix units  for $r_s\times r_s$ matrices.
In what follows, by {\it standard} generating set of $D,$  we mean elements 
of the form
\beq
(\overbrace{e_{i_1,j_1}^{1},0,...,0}^k), (\overbrace{0, e_{i_2,j_2}^{2}, 0,...,0}^k), ..., 
(\overbrace{0,...,0, e_{i_k, j_k}^{k}}^k),
\eneq
together with $1_{D},$ the identity of $D.$
\end{df}

\begin{df}\label{dfg}
Throughout this paper, denote by $\iota$ the identity function on $[0,1].$
In particular $\iota\in C_0((0,1]).$

Let $0<\ep<1.$ Define  $f_\ep, {\bar f_\ep}, g_\ep \in C^b((0,\infty))$ as follows 
\beq
&&f_\ep(t)=\begin{cases}   0\,\, \,  & t\in [0,\ep/2];\\
                                      {t-\ep/2\over{\ep/2}} \,\, & t\in (\ep/2, \ep];\\
                                      1 & t\in (\ep, \infty)\end{cases}, \hspace{0.2in}  \,\, 
                                      {\bar f}_\ep(t)=\begin{cases}  {t\over{\ep}} & t\in [0, \ep];\\
                                                                             1  & t\in (\ep, \infty).\end{cases}\andeqn\\
                                                                             &&g_\ep=\begin{cases} 0 & t\in [0, \ep/2];\\
                                                                                                               {(t-\ep/2)\over{1-\ep/2}} &t\in (\ep/2, 1];\\
                                                                                                              1&t\in (1, \infty).\end{cases}.
\eneq
 Note ${f_\ep}|_{(0,1]}, { {\bar f}_\ep}|_{(0,1]}, {g_\ep}|_{(0,1]} \in C_0((0,1]).$
 Sometimes we also write $f_\ep, {\bar f}_\ep$ and $g_\ep$ for their restriction to $(0,1].$
\end{df}

\begin{rem}\label{Lrlambda}
Note that, for any $t\in [0,1],$
\beq\label{Lrlambda-1}
t\ge \ep f_\ep(t).
\eneq

(i) If $t\in [0, \ep/2],$ then $f_\ep(t)=0,$ so $\ep f_\ep(t)=0\le t.$

(ii) If $t\in (\ep/2,\ep),$ then  
$$
\ep f_\ep(t)=2t-\ep.
$$ 
Since $t<\ep,$
$$
\ep f_\ep(t)=2t-\ep<t.
$$

(iii) If $t\in [\ep,1],$ then $f_\ep(t)=1,$ so $\ep f_\ep(t)=\ep\le t.$

Thus, in all cases, 
$$
\ep f_\ep(t)\le t.
$$
\end{rem}

The following result is well known. We record it here for the reader's convenience.

\begin{lem}\label{LMn}
Let $A$ be a non-elementary simple \CA, let $a\in A_+\setminus \{0\}$ and $n\in \N.$ 
Then there is  a non-zero \hm\, $\phi: C_0((0,1])\otimes M_n\to \Her(a)\cap {\rm Ped}(A)$ 
\end{lem}

\begin{proof}
We may assume that $\|a\|=1.$ Choose $a_1=f_{1/4}(a)$ and $B=\Her(a_1).$
Then $B\subset {\rm Ped}(A)$ and  is a non-elementary simple \CA.  By \cite[p.~67]{AS}, there is an element 
$y\in B_+$ with ${\rm sp}(y)=[0,1].$ Therefore, one may choose $n$ non-zero mutually orthogonal 
elements $z_1,z_2,...,z_n\in B_+.$ It follows  from \cite[Lemma 2.3]{Lincs91} that there is $b\in B_+$
with $\|b\|=1$ and $x_i\in B$ such that 
\beq
x_i^*x_i=b\andeqn x_ix_i^*\in \Her(z_i),\,\,\, 1\le i\le n.
\eneq
Define $\phi: C_0((0,1])\otimes M_n\to B$ by 
\beq
\phi(\iota\otimes e_{i,j})=x_ix_j^*,\,\,\, 1\le i,j\le n.
\eneq
One then verifies that $\phi$ indeed defines a \hm. 
\end{proof}

\begin{df}\label{def:Cuntz-null}
Let $A$ be a $C^*$-algebra and let
$a\in A_+\setminus\{0\}$. A bounded  sequence
$\{e_n\}\subset A_+^1$ is said to be \emph{Cuntz small
relative to $a$} if, for every $\lambda>0$, there is
$N\in\mathbb N$ such that
\[
(e_n-\lambda)_+\lesssim a,
\qquad n\geq N.
\]
A  bounded sequence  $\{a_n\}$ of $A$ is said to be \emph{Cuntz null}
if $\{a_n^*a_n\}$ 
is Cuntz
small relative to every element of
$A_+\setminus\{0\}$ (see \cite[Definition 3.1]{FLL}).

A  bounded sequence $\{a_n\}\subset A_+$ is said to be {\it truly Cuntz-null} 
if, for any finite subset ${\cal F}\subset A_+\setminus \{0\},$ there exists $n_0\in \N$ 
such that, for any $n\ge n_0,$ $a_n\lesssim d$ for all $d\in {\cal F}$ (see \cite[Definition 2.14]{FLosc}).
\end{df}

\begin{rem}\label{RCuntznull}

\noindent
(1) It follows from \cite[Lemma 4.3]{FLL}  that, if $A$ is a non-elementary separable simple \CA, then
$A$ admits at least one (non-zero) Cuntz-null sequence.  

(2) Every non-elementary simple \CA\, with strict comparison admits   a non-zero truly Cuntz-null sequence
$\{z_n\}$ with $\|z_n\|=1$ ($n\in \N$).  

To see this,  let 
$A$  be a non-elementary simple \CA.    If $A$ is purely infinite, then any sequence $\{z_n\}$ with 
$\|z_n\|=1$ is a truly Cuntz-null sequence. Therefore we may assume that $\widetilde{QT}(A)\not=\emptyset.$ 

Choose $a_0\in {\rm Ped}(A)_+$ with $\|a_0\|=1.$
 Put $a_1=f_{1/4}(a_0)$ and $a_2=f_{1/2}(a_0).$ 
 Let $S=\{\tau\in \widetilde{QT}(A):  \tau(a_1)=1\}.$ $S$ is compact and a basis for $\widetilde{QT}(A).$
 
 Let $B=\Her(a_2).$   Define $z_1:=a_2.$   So $\|z_1\|=1.$ Since $a_1a_2=a_2,$ we have 
 $d_\tau(z_a)=d_\tau(a_2)\le d_\tau(a_1)=1$ for all $\tau\in S.$
  Suppose we obtain $z_1, z_2, ..., z_n\in B$ such 
 that $\|z_i\|=1$ ($1\le i\le n$) and 
 \beq
 d_\tau(z_{i+1})\le d_\tau(z_i)/(i+1)\rforal \tau\in S,\,\,\, 1\le i\le n-1.
 \eneq 
 By applying Lemma \ref{LMn},  we obtain a non-zero \hm\, $\phi_n: C_0((0,1])\otimes M_{n+1}\to
 \Her(f_{1/2}(z_n)).$ Choose $z_{n+1}=\phi_n(f(\iota)\otimes e_{1,1})$ for some $f\in C_0((0,1])_+$
 with $\|f\|=1$ such that $\|z_{n+1}\|=1.$ 
 Then  the construction gives
 \beq
 d_\tau(z_{n+1})\le d_\tau(z_n)/(n+1)\rforal \tau\in S.
 \eneq
Let $b\in A_+\setminus \{0\}.$ Since $A$ is simple,  we have 
$\inf_{\tau\in S} d_\tau(b)=\sigma>0.$ Choose $n_0\in \N$ such that $1/n_0<\sigma.$ 
Consequently,  for all $n\ge n_0,$ 
\beq\nonumber
d_\tau(z_n)\le 1/n!<\sigma <d_\tau(b)\rforal \tau\in S.
\eneq
Since $A$ has  strict comparison, we have $z_n \lesssim b.$
Thus $\{z_n\}$ forms a non-zero truly Cuntz-null sequence.
\end{rem}

\begin{df}\label{DGamma}
Let $\{A_n\}$ be a sequence of \CA s.
Define 
\beq
&&\prod_n A_n=l^\infty(\{A_n\})=\{\{a_n\}: a_n\in A_n, \sup_n\|a_n\|<\infty\}\andeqn\\
&&\bigoplus_{n=1}^\infty A_n=c_0(\{A_n\})=\{\{a_n\}\in l^\infty: \lim_{n\to\infty}\|a_n\|=0\}.
\eneq
In case that $A_n=A$ for all $n\in \N,$ 
we may write $\prod_n A$ for $\prod_n A_n,$ $l^\infty(A)$ for $l^\infty(\{A_n\}),$ and 
$\bigoplus_{n=1}^\infty A=\bigoplus_{n=1}^\infty A_n$ and $c_0(A)=c_0(\{A_n\}).$ 

Denote $j_A: A\to l^\infty(A)$  by 
$j_A(a)=\{a_n\},$ where $a_n=a$ for all $n\in \N,$ the constant sequence.

Denote by $N_{cu}(A)$ or just $N_{cu}$ the set of all Cuntz-null sequences in $l^\infty(A).$ 
By \cite[Proposition 3.5]{FLL}, if $A$ has no one-dimensional hereditary \CA s, then $N_{cu}(A)$ is 
a (closed  two--sided) ideal of $l^\infty(A).$  Suppose that $A$ is non-elementary, simple and admits  a (non-zero) Cuntz-null sequence $\{z_n\}$ with $\|z_n\|=1$ ($n\in \N$). Then 
$c_0(A)\subsetneqq N_{cu}(A).$  Denote by 
$\Pi_{cu}: l^\infty(A)\to l^\infty(A)/N_{cu}(A)$ the quotient map (when $N_{cu}(A)$ is an ideal).

We often write $Q_A=l^\infty(A)/N_{cu}(A).$

Denote by ${\overline{j_A}}: A\to l^\infty(A)/N_{cu}$ the composition $\Pi\circ j_A,$
 and 
 $$\overline{j_A}(A)^\perp=\{b\in l^\infty(A)/N_{cu}(A): b\overline{j_A}(a)=\overline{j_A}(a)b=0\rforal a\in A\}.$$

\end{df}

\begin{df}\label{DI-S}
Let $A$ be a \CA\, and $S\subset \widetilde{QT}(A)$ be a compact convex subset.
Define 
\beq
I_S=\{\{a_n\}\in l^\infty(A):  \lim_{n\to\infty} \sup_{\tau\in S}\{\tau(a_n^*a_n)\}=0\}.
\eneq

If $A$ is a non-elementary algebraically simple \CA\, with strict comparison,  then
\beq\label{A2.17}
 I_{\overline{QT(A)}^w}=N_{cu}(A)
 \eneq
 (see  the proof of \cite[Proposition 3.8]{FLL}, in particular, separability is not needed for \eqref{A2.17}).
 
 Denote by $q_A: l^\infty(A)\to l^\infty(A)/I_{\overline{QT(A)}^w}$ the quotient map.
\end{df}

\begin{df}\label{Dcontsc}
Let $A$ be a  non-elementary $\sigma$-unital simple \CA. Let $\{e_n\}$ be an approximate identity
for $A$ such that $e_{n+1}e_n=e_ne_{n+1}=e_n,$ $n\in \N.$
We say $A$ has continuous scale, if for any $a\in A_+\setminus \{0\},$ there exists $n_0\in \N$ such that
\beq
e_m-e_n\lesssim a \rforal m>n\ge n_0.
\eneq
The definition does not depend on the choice of $\{e_n\}$  (see \cite[Definition 2.1]{Lincs91} and 
\cite[Definition 2.5]{Lincs04}). 
\end{df}

\begin{rem}\label{Rcs}
Suppose that $A$ also has strict comparison with a strictly positive element $e_A.$  Then $A$ has 
continuous scale if and only if $\widehat{[e_A]}$  is continuous on $\widetilde{QT}(A)$ (see \cite[Theorem 2.19]{FLosc}, and also \cite[Theorem 5.3 and Proposition 5.4]{EGLNs}).  The same references  also state that this is equivalent 
to the condition that $QT(A)$ is compact.

Let $a\in {\rm Ped}(A)_+.$  If $\widehat{[a]}$ is continuous, we may write $\omega(a)=0$
(see \cite[Definition A1]{EGLNkk0} and the last line of \cite[Definition 4.1]{FLosc} on page 15).
\end{rem}

\begin{lem}\label{Lsequnececuntz}

(1) Let $A$ be a \CA\, and  $\{a_n\}, \{b_n\}, \{c_n\}\in l^\infty(A)_+.$ 
Suppose that $a_n\lesssim b_n$ and $c_nb_n=b_n=b_nc_n$ for all $n\in \N.$ 
Then 
\beq
\{a_n\}\lesssim \{c_n\}\in l^\infty(A).
\eneq

(2) Let $A$ be a $\sigma$-unital non-elementary simple \CA\, with continuous scale (see \cite[Definition 2.1]{Lincs04}
and \cite{Lincs91}), and let 
$\{e_n\}$ be an approximate identity 
for $A$ such that $e_{n+1}e_n=e_n$ for all $n\in \N.$
Then, for any $a\in A_+\setminus\{0\},$ there exists $n_0\in \N$ such 
that
\beq
1_{Q_A}-\overline{j_A}(e_n)\lesssim \overline{j_A}(a)\rforal n\ge n_0.
\eneq

\end{lem}

\begin{proof}

(1) 

We may assume that $\|a_n\|\le M_a$ and $\|b_n\|\le M_b$ for all $n\in \N$
for some $M_a, M_b\ge 0.$
Fix $1>\ep>0.$ 
For each $n\in \N,$ there exists $x_n\in A$ such that
\beq
\|a_n-x_n^*x_n\|<\ep\andeqn x_nx_n^*\in\Her(b_n).
\eneq
Then, for all $n\in \N,$
\beq
\|x_n^*x_n\|\le \|a_n\|+\ep \le M_a+1.
\eneq
It follows that $x=\{x_n\}\in l^\infty(A).$
Put $a=\{a_n\}, b=\{b_n\}$ and $c=\{c_n\}\in l^\infty(A).$
Then 
\beq
cxx^*=xx^*=xx^*c\andeqn
\|a-x^*x\|<\ep
\eneq
It follows that $xx^*\in \Her(c)=\Her(\{c_n\}).$  
Hence 
\beq
\{a_n\}\lesssim \{c_n\}\in l^\infty(A).
\eneq

(2) 

Since $A$ has continuous scale, it is standard to check that
$\Pi_{cu}(\{e_n\})=1_{Q_A}.$ 

Choose $r=\|a\|/2$ and $b=(a-r)_+.$ Then $f_{r/2}(a)b=b=bf_{r/2}(a).$
Since $A$ has continuous scale, 
there is $n_0\in \N$ such that
\beq
e_m-e_n\lesssim b\rforal m>n\ge n_0.
\eneq
It follows from (1) that 
\beq
\{e_{n+m}-e_{n}\}\lesssim  \overline{j_A}(f_{r/2}(a))\rforal n\ge n_0.
\eneq
It follows that 
\beq
1_{Q_A}-\overline{j_A}(e_{n})\lesssim  \overline{j_A}(a)\rforal n\ge n_0.
\eneq
\end{proof}

\begin{lem}\label{lem:absorb-tail}
 Let $A$ be a $\sigma$-unital simple $C^*$-algebra with continuous
scale. Let $a\in A_+\setminus\{0\}$, let $q\in(Q_A)_+^1$, and put
$
c=1_{Q_A}-q.
$
Suppose that
$
q\lesssim \overline{j_A}(a_0)
$
for some $a_0\in A_+\setminus\{0\}$.

Choose $a_1\in A_+\setminus\{0\}$ such that
\[
a_0\oplus a_1\lesssim a.
\]
Then there is a positive contraction $g\in A$ such that
\[
1_{Q_A}-\overline{j_A}(g)\lesssim \overline{j_A}(a_1).
\]
Moreover, for every $g\in A_+^{\bf 1}$ satisfying
\[
        1_{Q_A}-j_A(g)\precsim j_A(a_1),
\]
put
\[
b=c^{1/2}\overline{j_A}(g)c^{1/2}\,\,
\andeqn\,\,
q_1=1_{Q_A}-b.
\]
Then
\[
q_1\lesssim \overline{j_A}(a)
\tand
1_{Q_A}-q_1
=c^{1/2}\overline{j_A}(g)c^{1/2}.
\]
\end{lem}

\begin{proof}

It follows from (2) of Lemma \ref{Lsequnececuntz} that there is $g\in A_+^{\bf 1}$ such that 
\[
1_{Q_A}-\overline{j_A}(g)\lesssim \overline{j_A}(a_1).
\]
Now suppose that $g\in \in A_+^{\bf 1}$ satisfies the condition above.
Put
\[
r=c^{1/2}(1_{Q_A}-\overline{j_A}(g))c^{1/2}.
\]
Then
\[
r\lesssim1_{Q_A}-\overline{j_A}(g)
\lesssim \overline{j_A}(a_1).
\]
Moreover,
\[
\begin{aligned}
c^{1/2}\overline{j_A}(g)c^{1/2}
&=c^{1/2}\big(1_{Q_A}-(1_{Q_A}-\overline{j_A}(g))\big)c^{1/2}
=c-r.
\end{aligned}
\]
Consequently,
\[
\begin{aligned}
q_1
&=1_{Q_A}-c^{1/2}j_A(g)c^{1/2}
=1_{Q_A}-(c-r)\\
&=(1_{Q_A}-c)+r
=q+r.
\end{aligned}
\]
It follows that
\[
q_1=q+r\lesssim q\oplus r.
\]
Using
\[
q\lesssim \overline{j_A}(a_0)
\quad\text{and}\quad
r\lesssim \overline{j_A}(a_1),
\]
we obtain
\[
q_1
\lesssim \overline{j_A}(a_0)\oplus \overline{j_A}(a_1)
\lesssim \overline{j_A}(a).
\]

Finally, by the definition of $q_1$,
\[
1_{Q_A}-q_1
=c^{1/2}\overline{j_A}(g)c^{1/2}.
\]
\end{proof}

The following  result is well known.

\begin{lem}\label{Lfinitesp}
Let $A$ be a  \CA\, of real rank zero and let 
$\phi: C:=C_0((0,1])\otimes D\to A$ be a \hm, where $D$ is a finite dimensional \CA.
Let $e=\phi(\iota\otimes 1_D).$ 
Then, for any $\ep>0$ and any  finite subset ${\cal F}\subset C_0((0,1])\otimes D,$ 
there exists a \hm\, $H: C\to \Her(e)$ with finite dimensional range 
such that
\beq\label{Lfinitesp-1}
&&\|H(f)-\phi(f)\|<\ep\tforal f\in {\cal F},\tand\\\label{Lfinitesp-2}
&&\sup_{x\in D^{\bf 1}}\|H(\iota\otimes x)-\phi(\iota\otimes x)\|<\ep.
\eneq
\end{lem}

\begin{proof}
Write
$
D=\bigoplus_{k=1}^{m}M_{r_k},
$
and let
\[
\varphi_k:C_0((0,1])\otimes M_{r_k}\longrightarrow A
\]
be the restriction of $\varphi$ to the $k$-th summand. Put
\[
e_k=\varphi_k(\iota\otimes 1_{M_{r_k}})
\quad\text{and}\quad
B_k={\rm Her}(e_k)
\]
(recall that $\iota: [0,1]\to [0,1]$ is the identity function; see  Definition \ref{dfg}).
We may omit those summands on which $\varphi_k=0$. The algebras
$B_1,\ldots,B_m$ are mutually orthogonal hereditary
$C^*$-subalgebras of ${\rm Her}(e)$. It therefore suffices to prove
the result for each summand separately and then take the direct sum
of the resulting homomorphisms.

Thus, suppose that $D=M_r$, put
\[
B={\rm Her}\big(\varphi(\iota\otimes 1_{M_r})\big),
\]
and fix a system of matrix units $\{e_{i,j}\}_{i,j=1}^{r}$ for
$M_r$. For every $f\in{\cal F}$, write
\[
f=\sum_{i,j=1}^{r}f_{i,j}\otimes e_{i,j},
\qquad f_{i,j}\in C_0((0,1]).
\]
Let ${\cal F}_0$ be the finite subset of $C_0((0,1])$ consisting of
all the functions $f_{i,j}$ occurring in these expressions.

Choose $\delta>0$ such that $\delta<\epsilon$ and
\beq\label{Lfinitesp-9}
\|g(c)-g(d)\|<\frac{\epsilon}{r^2},
\qquad g\in{\cal F}_0,                                
\eneq
whenever $c$ and $d$ are positive contractions in a $C^*$-algebra
satisfying $\|c-d\|<\delta$.

Put
\[
\theta(x)=\varphi(\iota\otimes x)\rforal x\in M_r,
\]
and let $\pi:M_r\to B^{**}$ be the supporting homomorphism of
$\theta$. Thus
\[
\theta(x)=\pi(x)\theta(1_{M_r})
         =\theta(1_{M_r})\pi(x),\qquad x\in M_r.
\]
Set
$
a=\theta(e_{11}).
$

Hence there is a positive contraction of finite spectrum
\[
b=\sum_{l=1}^{N}t_lp_l\in{\rm Her}(a),
\qquad 0<t_l\leq1,
\]
where $p_1,\ldots,p_N$ are mutually orthogonal projections, such
that
\[
\|a-b\|<\delta.
\]

For $l=1,\ldots,N$, define
\[
\rho_l(e_{i,j})=\pi(e_{i,1})p_l\pi(e_{1j}),
\qquad 1\leq i,j\leq r.
\]
Since $p_l\in{\rm Her}(a)$, these elements belong to $B$ and form a
system of matrix units. Thus
\[
\rho_l:M_r\longrightarrow B
\]
is a homomorphism. Moreover, the ranges of
$\rho_1,\ldots,\rho_N$ are mutually orthogonal.

Define
\[
H(f)=\sum_{l=1}^{N}\rho_l\big(f(t_l)\big),
\qquad f\in C_0((0,1])\otimes M_r.
\]
Then $H:C_0((0,1])\otimes M_r\to B$ is a homomorphism with
finite-dimensional range.

For
$
f=\sum_{i,j=1}^{r}f_{i,j}\otimes e_{i,j},
$
we have
\[
\varphi(f)
 =\sum_{i,j=1}^{r}
   \pi(e_{i,1})f_{i,j}(a)\pi(e_{1j}),
\]
whereas
\[
H(f)
 =\sum_{i,j=1}^{r}
   \pi(e_{i,1})f_{i,j}(b)\pi(e_{1j}).
\]
Therefore, by  \eqref{Lfinitesp-9}, for every $f\in{\cal F},$
\[
\|H(f)-\varphi(f)\|
\leq
  \sum_{i,j=1}^{r}\|f_{i,j}(b)-f_{i,j}(a)\|  
<r^2\frac{\epsilon}{r^2}
 =\epsilon.
\]

Finally, for $x\in M_r$, the usual matrix-corner identification
gives
\[
\|H(\iota\otimes x)-\varphi(\iota\otimes x)\|
 \leq \|b-a\|\,\|x\|.
\]
Consequently,
\[
\sup_{x\in M_r^1}
\|H(\iota\otimes x)-\varphi(\iota\otimes x)\|
<\epsilon.
\]

Applying this construction to each summand $M_{r_j}$ gives
homomorphisms
\[
H_j:C_0((0,1])\otimes M_{r_j}\longrightarrow B_j
\]
with finite-dimensional ranges. Since the $B_j$ are mutually
orthogonal, their direct sum
\[
H=\bigoplus_{j=1}^{m}H_j:
C_0((0,1])\otimes D\longrightarrow {\rm Her}(e)
\]
is a homomorphism with finite-dimensional range and satisfies \eqref{Lfinitesp-1} and \eqref{Lfinitesp-2}.
\end{proof}

Next, and for the rest of the paper, we will identify $a\in A$ with $a\otimes e_{1,1}$ in $M_n(A)$
for a fixed system of matrix units $\{e_{i,j}\}_{1\le i,j\le n}.$

\begin{lem}\label{lem:diagonal-contractive-factorization}
Let \(A\) be a \(C^*\)-algebra, let \(e\in A_+^1\), and put
\[
        c=\operatorname{diag}(e,\ldots,e)\in M_n(A).
\]
Suppose that \(a\in A_+^1\) and \(x\in M_n(A)\) satisfy
$
        x^*x=a
        \tand
        cxx^*=xx^*c=xx^*.
$
Then, for every \(0<\varepsilon<1\), there exist
\(z_1,\ldots,z_n\in A^{\bf 1}\) such that  
\[
        ez_i=z_i,\qquad 1\leq i\leq n,
\tand\,\,\,
        f_\varepsilon(a)
        =\sum_{i=1}^n z_i^*z_i
        =\sum_{i=1}^n z_i^*ez_i.
\]
\end{lem}

\begin{proof}
Note that we may assume that $x$ is a column in $M_{n,1}(A).$
The assumptions imply
\[
 \begin{aligned}
 (cx-x)(cx-x)^*
 &=cxx^*c-cxx^*-xx^*c+xx^*
 =0.
 \end{aligned}
\]
Hence \(cx=x\).

Define \(k_\varepsilon\in C([0,1])\) by
\[
 k_\varepsilon(t)=
 \begin{cases}
 \displaystyle
 \left(\frac{f_\varepsilon(t)}{t}\right)^{1/2},
       &t>0,\\[3mm]
 0,     &t=0.
 \end{cases}
\]
Since \(f_\varepsilon\) vanishes on a neighborhood of zero,
\(k_\varepsilon\) is continuous.  Put
\[
        v=xk_\varepsilon(a)\in M_{n,1}(A).
\]
Since \(x^*x=a\),
\[
 \begin{aligned}
 v^*v
 &=k_\varepsilon(a)x^*xk_\varepsilon(a)
 =k_\varepsilon(a)a k_\varepsilon(a)
 =f_\varepsilon(a)\le 1.
 \end{aligned}
\]
Moreover,
\[
        cv=cxk_\varepsilon(a)=v.
\]

Write
\[
        v=
        \begin{pmatrix}
        z_1\\
        \vdots\\
        z_n
        \end{pmatrix}.
\]
The equality \(cv=v\) gives
\[
        ez_i=z_i,\qquad 1\leq i\leq n.
\]

  Finally, since \(ez_i=z_i\),
\[
 \sum_{i=1}^n z_i^*ez_i
 =\sum_{i=1}^n z_i^*z_i =v^*v
 =f_\varepsilon(a)\le 1.
\]
In particular, $\|z_i^*z_i\|\le 1,$ $1\le i\le n.$
\end{proof}


\section{Essentially tracial rank one}

\begin{df}\label{Dtheta}

Let $D$ be a finite dimensional \CA. We may write 
$$D=M_{r_1}\oplus M_{r_2}\oplus \cdots \oplus M_{r_k}.$$
Suppose that $A$ is a \CA\, and $\theta: D\to A$ is an order zero \cpc.
By \cite[Corollary 4.1]{WZ09},  there is a \hm\, $\tilde \theta: C_0((0,1])\otimes D\to A$ 
such that
\beq\label{31}
\tilde \theta(\iota\otimes d)=\theta(d)\rforal d\in D.
\eneq

We say $\theta$ is {\it fully normalized} if $\|\theta|_{M_{r_j}}\|=1,$ $1\le j\le k.$

Put $C=C_0((0,1])\otimes D\cong C_0((0,1], D).$ 
Note that $\tilde \theta(C)\cong \bigoplus_{i=1}^k C_0((X_i\setminus \{0\}), M_{r_i}),$ where $X_i$  is a compact 
subset of $[0,1].$ Note that if $X_i=\{0\},$ then 
$C_0((X_i\setminus \{0\})=\{0\}.$   If $\theta$ is fully normalized, by \eqref{31}, $1\in X_i,$ $1\le i\le k.$ 

If $f\in C_0((X_i\setminus \{0\}), M_{r_i})$ and $1\in X_i,$ then 
$\pi_{1,i}(f)= f(1)$  is a \hm\, from $C_0((X_i\setminus \{0\}), M_{r_i})\to M_{r_i}.$ 
If $1\not\in X_i, $ then we define $\pi_{1,i}=0.$ Let 
$$
D_\theta=\bigoplus_{1\in X_i} M_{r_i}.
$$
Note also  that if $1\in X_i,$ then $\|\theta(x)\|=\|\tilde \theta(\iota\otimes x)\|=\|x\|$ for all $x\in M_{r_i}$
($1\le i\le k$).
Throughout this paper, we define
$\pi_1=\bigoplus_{1\in X_i}\pi_{1,i}: \tilde\theta(C)\to D_\theta$ which  is a \hm.
We call it an evaluation map at $1.$  

Note that if $\theta$ is a  fully normalized order zero \cpc, then $D_\theta=D.$ 
When $\theta$ is fully normalized,  one has $\pi_1: \tilde \theta(C)\to D$
defined by 
\beq
\pi_1(f)=f(1)\rforal  f\in \bigoplus_{i=1}^k C_0((X_i\setminus \{0\}), M_{r_i})
\eneq
is a quotient map. 
In particular,  we may write 
\beq
\pi_1(\widetilde\theta(\iota \otimes d))=d\rforal  d\in D.
\eneq
Therefore a fully normalized order zero \cpc\, is isometric.
%
\end{df}

The following observation will be used later.

\begin{lem}\label{Lproduct}
Let $D=D_1\oplus D_2\oplus \cdots \oplus D_n$ be  a finite dimensional \CA\, 
 and $A$ be a 
\CA.  Suppose that $\theta: D\to A$  is a fully normalized order zero \cpc.
Then
\beq
\|ab-ba\|=\|\theta(a)\theta(b)-\theta(b)\theta(a)\| \rforal a,\, b\in D.
\eneq 
\end{lem}

\begin{proof}
By \cite[Theorem 3.3]{WZ09}, write $\theta(d)=\pi_\theta(d) \theta(1_D)$
for all $d\in D,$ where $\pi: D\to M(C)\cap \{\theta(1_D)\}'$ is a \hm\, and $C:=C^*(\theta(D))\subset A.$ 
Define $\Theta: D\to C$ by $\Theta(d)=\pi_\theta(d)\theta(1_D)^2$ for all $d\in D.$
Then $\|\Theta|_{D_i}\|=\|\theta(1_{D_i})^2\|=\|\theta(1_{D_i})\|^2=1,$ $1\le i\le n.$
Hence $\Theta$ is a fully normalized order-zero \cpc. 
Therefore
\beq
 \|\theta(b)\theta(a)
 -\theta(a)\theta(b)\|
 &=&
\| \pi_\theta
 \bigl(
 ba-ab
 \bigr)\theta(1_D)^2\|\\
 &=&
 \|\Theta
 \bigl(
ba-ab
 \bigr)\|=\|ba-ab\|.
\eneq
\end{proof}

\begin{df}\label{D1}
Let $A$ be a simple \CA\, admitting  a strictly 
positive element $e_A$ with  $\|e_A\|=1.$ 
The \CA\, $A$   is said to have property ETR1,
if  there exists $0<\mathfrak{\kappa}_{e_A}\le 1$ satisfying the following:

For any  $\ep>0,$  any finite subset ${\cal F}\subset A$ 
containing $e_A,$ and
%
%
%
%
 $a\in A_+\setminus \{0\},$  
there is  
a positive element $e\in A^{\bf 1},$ 
a finite dimensional \CA\, $D$ and 
a fully normalized  order-zero \cpc\, 
$\theta: D\to  \overline{(1-e)A(1-e)}$
such that

(1)  $\|ex-xe\|<\ep\tforal x\in {\cal F},$

(2)   
 $(1-e)x\in_{\ep} \theta(D)$ and  $0\le d_e\le 1$ in $D$ such that
\beq\label{D1-(2)}
\|\theta(d_e)-(1-e)^{1/2}e_A(1-e)^{1/2}\|<\ep\andeqn t(d_e)\ge \mathfrak{\kappa}_{e_A}
\eneq
for all $t\in T(D).$

%
%
%

\vspace{0.08in}
(3) $e\lesssim a.$



%
\end{df}

\begin{rem}\label{R1}

(i) Let $\tilde \theta: C:=C_0((0,1])\otimes D\to A$ be the \hm\, 
induced by $\theta.$ 
Let $\pi_1: \tilde \theta(C)\to D$ be the quotient map defined by $\pi_1(\tilde \theta(\iota\otimes d))=d$ for all $d\in D,$ where $\iota$ is defined in Definition \ref{dfg}. Then  the second part of \eqref{D1-(2)} may be replaced by
\beq
t\circ \pi_1(\tilde\theta(\iota\otimes d_e))\ge \mathfrak{\kappa}_{e_A}\rforal t\in T(D).
\eneq

(ii) Note that $\|[(1-e), x]\|<\ep\rforal x\in {\cal F}.$ 
Since one can choose any $\ep>0,$ it
is easy to see  that in Definition \ref{D1} one may also require that

(2') $\|[(1-e)^{1/2},\, x]\|<\ep \rforal x\in {\cal F}.$

Also  in (2),
$(1-e)x\in_\ep \theta(D)$ may be replaced by

($2_a$) $(1-e)^{1/2} x(1-e)^{1/2}\in_\ep \theta(D).$

(iii) Moreover,  if $A$ has property ETR1, and if we assume that  ${\cal F}\subset A^{\bf 1},$ we may choose, for each $x\in {\cal F},$
an element $d_x\in D^{\bf 1}$ such that
\beq
\|(1-e)^{1/2}x(1-e)^{1/2} -\theta(d_x)\|<\ep.
\eneq 

To see this, we may apply ETR1 for $\ep/8$ (instead of $\ep$) first.  For each $x\in {\cal F},$ choose $c_x\in D$ such 
that
$$
\|(1-e)^{1/2}x(1-e)^{1/2}-\theta(c_x)\|<{\ep\over{8}}.
$$
Note $\|(1-e)^{1/2}x(1-e)^{1/2}\|\le 1.$ 
Since $\theta$ is isometric as it is fully normalized, we have
\beq
\|c_x\|\le 1+{\ep\over{8}}.
\eneq
Let $d_x:={c_x\over{\max\{1, |c_x\|\}}}.$ Then  $\|d_x\|\le 1$ and 
\beq
\|(1-e)^{1/2}x(1-e)^{1/2}-\theta(d_x)\|<\ep.
\eneq

(iv) Furthermore, if $y\in {\cal F},$ $0\le y\le 1,$   in the definition of ETR1, with $\ep/3$
(in place of $\ep$),
we may assume that  there exists $d_y\in D$ with 
$0\le d_y\le 1$ such that
\beq
\|(1-e)^{1/2}y(1-e)^{1/2}-\theta(d_y)\|<\ep.
\eneq

To see this,  let us assume that there is  $z_y\in D$ such that
\[
        \|(1-e)^{1/2}y(1-e)^{1/2}-\theta(z_y)\|<\ep/3,
\]
and replace $z_y$ by $(z_y+z_y^*)/2.$   So we may 
assume that $z_y$ is selfadjoint.  Since $y \ge 0$,
order-zero functional calculus and the isometry of $\theta$
give
\[
 \|(z_y)_-\|
 =\|\theta((z_y)_-)\|
 =\|(\theta(z_y))_-\|
 <\ep/3.
\]
Moreover,
\[
        \|z_y\|=\|\theta(z_y)\|<1+\ep/3.
\]
Define
\[
        d_y=\frac{(z_y)_+}{1+\ep/3}\in D_+^{\bf 1}.
\]
Then
\beq\nonumber
 &&\hspace{-0.8in}\|(1-e)^{1/2}y(1-e)^{1/2}-\theta(d_y)\|
 \leq
 \|(1-e)^{1/2}y(1-e)^{1/2}-\theta(z_y)\|\\\label{13.7.2}
 &&\hspace{0.8in}+\|\theta((z_y)_-)\|
 +\left\|
 \theta((z_y)_+)-\frac{\theta((z_y)_+)}{1+\ep/3}
 \right\| 
 <3(\ep/3)=\ep.                                       
\eneq

\end{rem}

\begin{rem}\label{Rea}
As it stands, the definition depends on the choice of the strictly positive element $e_A.$
Later, in  Corollary \ref{CeA}, we will show that the definition does not depend on the choice of $e_A.$

\end{rem}

\begin{rem}\label{D1R3}
Tracial rank,  and in particular  tracial rank zero, was first introduced in \cite{LinTAF} and \cite{Lintr2001} in 1998.
Generalized tracial rank one for unital case was introduced  in \cite{GLNI}. The general case, including the stably projectionless  case, appeared in 
\cite[Definition 8.1]{EGLNs}. The constant $\kappa_{e_A}$ is similar to
 $\mathfrak{f}_a$ in \cite[Definition 8.1]{EGLNs}. Its role  has been recognized  in \cite{EGLNs}. 
 
 Let $C:=C_0((0,1])\otimes D.$
 Consider the $\tilde\theta: C\to \overline{(1-e)A(1-e)}$ in  Remark \ref{R1} (i). 
 Since $\tilde\theta(C)$ is amenable, one obtains 
 a \cpc\, $\Theta: \overline{(1-e)A(1-e)}\to \tilde\theta(C)$ such that
 \beq
 \|\Theta((1-e)^{1/2}x(1-e)^{1/2})-(1-e)^{1/2}x(1-e)^{1/2}\|<\ep\rforal x\in {\cal F}.
 \eneq
 Put $E: A\to \overline{(1-e)A(1-e)}$ by $E(a)=\overline{(1-e)^{1/2}a(1-e)^{1/2}}$ for all $a\in A.$ 
 $E$ is a \cpc. 
 Define $\psi_C: A\to \tilde\theta(C)$ by  $\psi_C=\Theta\circ E.$ 
 
 In \cite[Theorem 8.1]{EGLNs} as well as all previous cases, maps $\phi$ and $\psi$ are orthogonal,
 where $\psi$ in \cite[Theorem 8.1]{EGLNs} plays the same role as $\psi_C$ just defined.
 Note that in Definition \ref{D1} $e$ is not orthogonal to $\psi_C(C).$  In \cite[Theorem 8.1]{EGLNs} 
both $\psi$ and $\phi$ are approximately multiplicative, 
which plays a critical role in all later use.  But $\psi_C$ does not have this important feature. 

For classification purposes, one has also to introduce the notion of ``rationalized tracial rank one". 
But this notion mysteriously disappears when simple \CA s are stably projectionless.  
In this article we will show that for property ETR1, one does not need to tensor $A$ with a UHF-algebra
(see Theorem \ref{TmainZ}).

It should be mentioned that the idea of using a positive contraction instead of a projection 
to obtain a large good \SCA\, was first communicated to the author by  N.C. Phillips 
two decades ago. Publication versions of this idea appeared  later (see \cite{Ph2014}, \cite{ABP}).
See also \cite{FLsigma}, where  ``essentially tracial" approximations were studied. 
\end{rem}

When $A$ is unital, Definition \ref{D1} can be simplified (see Proposition \ref{Asefimprov}).

\begin{rem}\label{Rd1}

(i) For any unital purely infinite simple \CA\, $B,$ let $A=B\otimes {\cal K}.$
Denote by $\{e_{i,j}\}$ the matrix units for ${\cal K}.$
Let ${\cal F}\subset A$ be any finite subset, 
$e_A=\sum_{n=1}^\infty (1_B\otimes e_{nn})/2^n$ and  $0<\ep<1.$  Choose any $a\in A_+\setminus \{0\}.$ 
There is $N\in \N$ such that  $2^{-N}<\ep/2$ and, with $E=\sum_{n=1}^Ne_{n,n},$ 
\beq
(1-E)x\approx_{\ep/2} x(1-E)\andeqn \|(1-E)x(1-E)\|<\ep\rforal x\in {\cal F}\cup \{e_A\}.
\eneq
Choose a positive element $h\in (1-E)A(1-E)$ with ${\rm sp}(h)=[0,1]$ and define  $D=\C,$ 
$\theta: D\to (1-E)A(1-E)$ by $\theta(\lambda)=\lambda h$  for $\lambda\in \C.$
Then $\theta$ is fully normalized. 
Thus

(1) $\|Ex-xE\|<\ep$ for all $x\in {\cal F}\cup\{e_A\},$

(2) For any $x\in {\cal F}\cup \{e_A\}$ choose $y_x=0,$  in particular, $d_{e_A}=0.$ Then 
$$
\|(1-E)^{1/2}x(1-E)^{1/2}-y_x\|<\ep.
$$

Since $A$ is purely infinite, 

(3)  $E\lesssim a.$

This example shows that the second part  of \eqref{D1-(2)}  in Definition \ref{D1} is important.

(ii) The trace condition in (2) of Definition \ref{D1} implies that $\|d_e\|\ge \mathfrak{\kappa}_{eA}.$
It follows that $\|(1-e)^{1/2}e_A(1-e)^{1/2}\|\ge \mathfrak{\kappa}_{e_A}-\ep.$
Please also see  Corollary \ref{R|x|>}.
\end{rem}

\begin{df}\label{Dsfinite}
Let $A$ be a $\sigma$-unital \CA. Then $A$ is said to be stably finite if, for any two non-zero orthogonal elements 
$a, b\in ({\rm Ped}(A\otimes {\cal K}))_+,$ $a+b\not\lesssim a.$ 
\end{df}

\begin{rem}\label{Rstablyfinite}
Suppose that $A$ is a $\sigma$-unital  non-elementary simple \CA\, which is stably finite.
Let  $a_0, a_1,a_2,...,a_n\in A_+$ be finitely many non-zero positive elements.
By \cite[Lemma 2.3]{Lincs91}, there is $b_0\in \Her(a_0)_+\setminus 
\{0\}$ such that $b_0\lesssim a_i$ ($i=1,2,...,n$).
Moreover, $\Her(b_0)$ has at least two non-zero mutually orthogonal elements 
$b_1, b_2.$ Since $A$ is stably finite, $a_i\not\lesssim b_1.$   Indeed, otherwise
$a_i\lesssim b_1.$ We would have 
\[
a_i\otimes b_2\lesssim b_1\oplus b_2\lesssim b_0\lesssim a_i,
\]
contradicting stable finiteness.
\end{rem}

\begin{lem}\label{L8-perturb-1}
Let $A$ be a \CA\,  and let $x, y\in A.$ If $\|x-y\|<r,$ then
\beq
(x^*x-r^2)_+\lesssim y^*y.
\eneq
\end{lem}

\begin{proof}
Let $\dt=\|x-y\|<r$ and $d=(x^*x-r^2)_+.$ 

If $d=0,$ there is nothing to prove. Assume $d\not=0.$ Let $w=d^{1/2}.$ Since $w\in C^*(x^*x),$ it commutes with $|x|=(x^*x)^{1/2}.$  One computes that
\beq
wx^*xw=d^{1/2}x^*xd^{1/2}\ge r^2 d=r^2 w^2.
\eneq
Representing $A$ faithfully on a Hilbert space, one has, for each $\xi\in H,$
\beq
\|yw\xi\|\ge \|xw\xi\|-\|(x-y)w\xi\|\ge r\|w\xi\|-\dt\|w\xi\|=(r-\dt)\|w\xi\|.
\eneq
Hence 
\beq
\la wy^*yw\xi, \xi\ra\ge (r-\dt)^2\la w^2\xi,\xi\ra.
\eneq
Hence
\beq
w^*y^*yw\ge (r-\dt)^2 w^2.
\eneq
It follows that 
\beq
(x^*x-r^2)_+=d=w^2\lesssim  w^*y^*yw=d^{1/2}y^*yd^{1/2}\sim ydy^*\le yy^*\sim y^*y.
\eneq 
\end{proof}

\begin{lem}\label{Pnot=0}
Let $A$ be a non-elementary  algebraically simple \CA\, which is stably finite.  
Suppose that $e_A\in A$ is a strictly positive element with $\|e_A\|=1.$
Then  $A$ has the following property:

For any  $\ep>0,$ any finite subset  ${\cal F}\subset A$ and $s\in A_+\setminus \{0\},$
   there exists $a\in \Her(s)_+\setminus \{0\}$ such that the following holds:
   if $0\le  e\le 1$ and $e\lesssim a,$ then
   \beq
   \|(1-e)x\|\ge \|x\|-\ep\tforal x\in {\cal F}.
   \eneq
   
\end{lem}

\begin{proof}
Fix  $\ep>0,$ a finite subset ${\cal F}\subset A$ and $s\in A_+\setminus \{0\}.$  
We may assume that ${\cal F}\subset A^{\bf 1},$ \wilog.
Choose  ${\cal F}_0=\{x\in {\cal F}: \|x\|>\ep\}.$
If $x\in {\cal F}\setminus {\cal F}_0,$ then 
\beq
\|(1-e)x\|\ge 0\ge \|x\|-\ep.
\eneq

For each $x\in {\cal F}_0,$ put $r_x=\|x\|-{\ep\over{2}}.$
Then $0<r_x<\|x\|.$ 
Define $d_x=(x^*x-r_x^2)_+.$ Then $d_x\not=0.$
%
By Remark \ref{Rstablyfinite}, 
there exists $a\in \Her(s)_+\setminus \{0\}$ such that  
\beq\label{Lnot=0-5}
d_x\not\lesssim a\rforal x\in {\cal F}_0.
\eneq

Now let 
$0\le e\le 1$ such that
%
 $e\lesssim a.$
 
 Suppose, toward a contradiction, that for some $x\in {\cal F},$
 \beq
 \|(1-e)x\|<\|x\|-\ep.
 \eneq
Then $x\in {\cal F}_0.$  Hence
\beq
\|(1-e)x\|<r_x.
\eneq
Set $y=ex.$  Then
\beq
\|x-y\|=\|(1-e)x\|<r_x.
\eneq
By Lemma \ref{L8-perturb-1}
we obtain 
\beq
d_x=(x^*x-r_x^2)\lesssim (ex)^*(ex)=x^*e^2x\sim exx^*e
\lesssim e\lesssim a.
\eneq
This contradicts \eqref{Lnot=0-5}. So
\beq
\|(1-e)x\|\ge \|x\|-\ep\rforal x\in {\cal F}.
\eneq

\end{proof}

\begin{rem}\label{Rs10}
Let $A$ be an algebraically simple \CA. Suppose that $A$ has property ETR1 and 
$A$ is stably finite. Then, by Lemma \ref{Pnot=0}, in  addition to (1)--(3) in Definition   \ref{D1}, we may also assume that

(2b)   $\|(1-e)x\|\ge \|x\|-\ep\rforal x\in {\cal F}$ 

as well as  (with Remark \ref{R1} in mind),

(2c) $\|(1-e)^{1/2}x(1-e)^{1/2}\|\ge \|x\|-\ep\rforal  x\in {\cal F}.$ 
\end{rem}

The following is well known. We record it  here for the convenience.

\begin{lem}\label{Lsqrt}
If
\(b,c\) are positive contractions in a  \(C^*\)-algebra and
$
        \delta=\|b-c\|,
$
then
\[
        \|b^{1/2}-c^{1/2}\|\leq \delta^{1/2}.
\]
\end{lem}

\begin{proof}
Working in the unitization $\widetilde{A},$ if necessary, we may assume that $A$ is unital.
We have 
\[
        b\leq c+\delta 1.
\]
Since the square-root function is operator monotone, we have
\[
        b^{1/2}\leq (c+\delta 1)^{1/2}.
\]
Moreover, by continuous functional calculus and the scalar inequality
\[
        (s+\delta)^{1/2}\leq s^{1/2}+\delta^{1/2}
        \qquad (s\geq 0),
\]
we obtain
\[
        (c+\delta 1)^{1/2}
        \leq c^{1/2}+\delta^{1/2}1.
\]
Consequently,
\[
        b^{1/2}-c^{1/2}\leq\delta^{1/2}1.
\]
Interchanging \(b\) and \(c\) gives
\[
        c^{1/2}-b^{1/2}\leq\delta^{1/2}1.
\]
Therefore,
\[
 -\delta^{1/2}1
 \leq b^{1/2}-c^{1/2}
 \leq\delta^{1/2}1,
\]
and hence
\[
        \|b^{1/2}-c^{1/2}\|\leq\delta^{1/2}.
\]
\end{proof}

\begin{prop}\label{Pertep}
Let $A$ be a $\sigma$-unital algebraically simple \CA\,
with a strictly positive element $e_A$ with $\|e_A\|=1.$ 
Let $1>\kappa_{e_A}>0.$ 

 Then $A$ has property ETR1 associated with $e_A$ and 
 $\kappa_{e_A}$   if and only if the following hold.

For any  $\ep>0,$  any finite subset ${\cal F}\subset A$ 
containing $e_A,$ and  
 $a\in A_+\setminus \{0\},$  

there is  
a positive element $e\in A^{\bf 1},$ 
a finite dimensional \CA\, $D$ and 
a fully normalized  order-zero \cpc\, 
$\theta: D\to  \overline{(1-e)A(1-e)}$
  such that

(1)  $\|ex-xe\|<\ep\tforal x\in {\cal F},$

(2)   
 $(1-e)x\in_{\ep} \theta(D),$ $(1-e)^{1/2}x(1-e)^{1/2}\in_\ep \theta(D)$ and  $0\le d_e\le 1$ in $D$ such that
\beq\label{D1-(2)-a}
\|\theta(d_e)-(1-e)^{1/2}e_A(1-e)^{1/2}\|<\ep\andeqn t(d_e)\ge \mathfrak{\kappa}_{e_A}
\eneq
for all $t\in T(D).$

\vspace{0.08in}
(3) $(e-\ep)_+\lesssim a.$

 \end{prop}
\begin{proof}
The ``only if"  part is clear (using also ($2_a$) of Remark \ref{R1}). We will prove the ``if" part.

To show $A$ has  property ETR1 associated with the given $e_A$ and $\kappa_{e_A},$
let \(\varepsilon>0\), let
\(\mathcal F\subset A\) be a finite subset containing \(e_A\), and let
\(a\in A_+\setminus\{0\}\). Put
\[
        M=\max\bigl(\{1\}\cup\{\|x\|:x\in\mathcal F\}\bigr).
\]
Choose
\[
 0<\eta<
 \min\left\{
 1,\,
 \frac{\varepsilon}{2(1+2M)},\,
 \frac{\varepsilon}{4},\,
 \frac{\varepsilon^2}{64M^2}
 \right\}.
\]
Apply the assumed condition with \(\eta\) in place of
\(\varepsilon\). We obtain \(e_0\in A_+^{\bf 1}\), a
finite-dimensional \(C^*\)-algebra \(D\), a fully normalized
order-zero c.p.c.\ map
\[
        \theta\colon D\longrightarrow
        \overline{(1-e_0)A(1-e_0)},
\]
and \(d_{e_0}\in D_+\), \(0\leq d_{e_0}\leq 1\), such that
\beq\nonumber
&&        \|e_0x-xe_0\|<\eta\,\,\,
\andeqn\\\label{ep-6}
&&        (1-e_0)x,\, (1-e_0)^{1/2} x(1-e_0)^{1/2}\in_\eta\theta(D)
        \rforal x\in\mathcal F\cup \{e_A\},
\\\nonumber
 &&\left\|
 \theta(d_{e_0})-(1-e_0)^{1/2}e_A(1-e_0)^{1/2}
 \right\|<\eta,
\\\nonumber
&&        t(d_{e_0})\geq\kappa_{e_A}
        \quad\quad\hspace{1in}\text{for all }t\in T(D),
\andeqn\\
 &&       (e_0-\eta)_+\lesssim a.
\eneq

Set
\[
        e=(e_0-\eta)_+.
\]
Then
\[
        0\leq e\leq e_0\leq 1,
        \qquad
        \|e-e_0\|\leq\eta,
\]
and, by the last inequality above,
\beq\label{ep-5}
        e\lesssim a.
\eneq
Moreover, \(e\leq e_0\) implies \(1-e_0\leq1-e\), and hence
\[
 \overline{(1-e_0)A(1-e_0)}
 \subseteq
 \overline{(1-e)A(1-e)}.
\]
Therefore, we may regard \(\theta\) as a fully normalized order-zero
c.p.c.\ map
\[
        \theta\colon D\longrightarrow
        \overline{(1-e)A(1-e)}.
\]

For every \(x\in\mathcal F\), we have
\[
\begin{aligned}
 \|ex-xe\|
 &\leq \|e_0x-xe_0\|
       +\|(e-e_0)x-x(e-e_0)\| \\
 &<\eta+2\eta\|x\| 
 \leq (1+2M)\eta
 <\varepsilon.
\end{aligned}
\]
This verifies  (1) in  property ETR1.

For \(x\in\mathcal F\), choose \(y_x\in\theta(D)\) such that
\[
        \|(1-e_0)x-y_x\|<\eta.
\]
Then
\[
\begin{aligned}
 \|(1-e)x-y_x\|
 &\leq \|(1-e_0)x-y_x\|
       +\|(e_0-e)x\| 
 <\eta+\eta\|x\| 
 \leq (1+M)\eta
 <\varepsilon.
\end{aligned}
\]
Consequently,
\[
        (1-e)x\in_\varepsilon\theta(D)
        \quad\text{for all }x\in\mathcal F.
\]
%
By Lemma \ref{Lsqrt}
we obtain
\[
\begin{aligned}
 \bigl\|(1-e)^{1/2}-(1-e_0)^{1/2}\bigr\|
 &\leq
 \bigl\|(1-e)-(1-e_0)\bigr\|^{1/2} \\
 &=\|e-e_0\|^{1/2}
 \leq\eta^{1/2}.
\end{aligned}
\]
Now write
\[
\begin{aligned}
 &(1-e)^{1/2}x(1-e)^{1/2}
 -(1-e_0)^{1/2}x(1-e_0)^{1/2} \\
 &\quad=
 \bigl((1-e)^{1/2}-(1-e_0)^{1/2}\bigr)
 x(1-e)^{1/2} \\
 &\qquad\quad+
 (1-e_0)^{1/2}x
 \bigl((1-e)^{1/2}-(1-e_0)^{1/2}\bigr).
\end{aligned}
\]
Since $0\le e,e_0\le 1,$ 
it
follows that, for all $x\in {\cal F}\cup\{e_A\},$ 
\beq\nonumber
 &&\hspace{-0.6in}\bigl\|
 (1-e)^{1/2}x(1-e)^{1/2}
 -(1-e_0)^{1/2}x(1-e_0)^{1/2}
 \bigr\| \\\label{ep-8}
 &&\quad\quad\hspace{0.4in}\leq
 2M\bigl\|(1-e)^{1/2}-(1-e_0)^{1/2}\bigr\|
 \leq 2M\eta^{1/2}.
\eneq
Since $2M\eta^{1/2}+\eta<\ep,$ by \eqref{ep-6} and  \eqref{ep-8},
\beq
(1-e)^{1/2}x(1-e)^{1/2}\in_\ep \theta(D).
\eneq
In particular,
\[
\begin{aligned}
 &\left\|
 \theta(d_{e_0})-(1-e)^{1/2}e_A(1-e)^{1/2}
 \right\|  
 <\eta+2M\eta^{1/2}
 <\varepsilon.
\end{aligned}
\]
Finally,
\[
        t(d_{e_0})\geq\kappa_{e_A}
        \quad\text{for all }t\in T(D)
\]
is unchanged. Thus \(e,D,\theta\), and \(d_{e_0}\) satisfy all the
conditions in the definition of property \({\rm ETR}1\), including
\(e\lesssim a\). This proves the proposition.
\end{proof}

\section{Rank of $D$}

The purpose of this section  is to present Proposition \ref{Prank}.

\begin{lem}\label{Cclose2}
Let  $D
$ be a finite dimensional \CA.
For any $\ep>0$ and 
any finite subset ${\cal F}\subset D,$  
 there exists $\dt>0$  and a finite subset ${\cal G}\subset D^{\bf 1}$ 
 satisfying the following:
 If $A$ is a \CA, $B\subset A$ is a  \SCA\,  and
 $\phi:D\to A$ is a \cpc\, 
such that, for any $a, b\in {\cal G}_+$ and $ab=0,$ 
  \beq\label{Cclose2-001}
  \|\phi(a)\phi(b)\|<\dt
  \eneq
  and
 \beq\label{Cclose2-002}
 \phi(g)\in_\dt B\rforal g\in {\cal G},
 \eneq
 then there exists an order-zero \cpc\, $\psi: D\to B$ such that
 \beq
 \|\phi(f)-\psi(f)\|<\ep\tforal f\in {\cal F}.
 \eneq
 \end{lem}

\begin{proof}
Suppose that the lemma is false.   Then,  there are 
finite subset ${\cal F}_0\subset D$ and $\ep_0>0,$ for every $n\in \N,$ there would exist a \CA\, $A_n,$ a \SCA\, $B_n\subset A_n$ and a \cpc\, $\phi_n: D\to A_n,$  and an increasing 
sequence ${\cal G}_1\subset {\cal G}_2\subset \cdots \subset D^{\bf 1}$ with $\overline{\cup_n {\cal G}_n}=D^{\bf 1}$ such that  if $a, b\in {\cal G}_n,$ then $ab, ba\in {\cal G}_{n+1},$ 
\beq\label{Cclose-4}
&&{\rm dist}(\phi_n(x),\, B_n)<1/n\rforal x\in {\cal G}_n\andeqn\\\label{Cclose-5}
&&\|\phi_n(x)\phi_n(y)\|<1/n\rforal x, y\in ({\cal G}_n)_+
\eneq
and $xy=0,$ but there is no order-zero \cpc\,
$\psi_n:D\to B_n$ satisfying
\beq
\|\phi_n(x)-\psi_n(x)\|<\ep_0\rforal x\in {\cal F}.
\eneq
Equivalently, for every order-zero \cpc\, $\psi: D\to B_n,$
\beq\label{Cclose-415}
\max_{x\in {\cal F}}\|\phi_n(x)-\psi(x)\|\ge \ep_0.
\eneq
Then put 
\beq
A_\infty=l^\infty(\{A_n\})/c_0(\{A_n\})\andeqn B_\infty= l^\infty(\{B_n\})/c_0(\{B_n\}).
\eneq
Define  $j: l^\infty(\{B_n\})\to l^\infty(\{A_n\})$ to be the natural embedding. 
Let $q_A: l^\infty(\{A_n\})\to l^\infty(\{A_n\})/c_0(\{A_n\})$ be the quotient map. 
If $q_A\circ j(\{b_n\})=0,$ then $\{b_n\}\in c_0(\{A_n\})$ which implies $\{b_n\}\in c_0(\{B_n\}).$
The converse also holds. Therefore $j$ induces 
an embedding ${\bar j}: l^\infty(\{B_n\})/c_0(\{B_n\})\to l^\infty(\{A_n\})/c_0(\{A_n\}).$ 

Define $\bar \Phi:  D\to l^\infty(\{A_n\})/c_0(\{A_n\})$ by $\bar \Phi(x)=q_A(\{\phi_n(x)\})$ for all $x\in D.$
By \eqref{Cclose-4} and \eqref{Cclose-5},  ${\bar \Phi}$ is an order-zero \cpc\, and $\bar \Phi(D)\subset {\bar j}( l^\infty(\{B_n\})/c_0(\{B_n\}).$
So ${\bar j}^{-1}\circ\bar \Phi: D\to l^\infty(\{B_n\})/c_0(\{B_n\})$ is an order - zero \cpc. 
By \cite[Proposition 1.24]{WII}, there is an order zero \cpc\, $\Phi: D\to l^\infty(\{B_n\})$ such that 
\beq
q_A\circ  j\circ \Phi={\bar \Phi}. 
\eneq
Write $\Phi=\{\Phi_n\},$ where $\Phi_n: D\to B_n$ is an order-zero \cpc.
We then obtain that
\beq
\lim_{n\to\infty}\|\phi_n(x)-\Phi_n(x)\|=0\rforal x\in D.
\eneq
This contradicts \eqref{Cclose-415}.
\end{proof}

\begin{lem}\label{Lstable2}
Let $D$ be a finite dimensional \CA\, and ${\cal F}\subset D$ be a finite  subset.
Then, for any $\ep>0,$ there is $\dt>0$  and  a finite subset ${\cal E}\subset D^{\bf 1}$   such that, if $A$ is a \CA,
$\phi: D\to A$ is an order - zero \cpc,
and  $a\in A$ such that $\|a\|\le 1,$ 
\beq\label{Cclose2-1}
&&\|a\phi(f)-\phi(f)a\|<\dt, \,\,\|a^*\phi(f)-\phi(f)a^*\|<\dt \andeqn\\\label{Cclose2-2}
&&a^*\phi(f)a\in_\dt B,
\eneq
where $B\subset A$ 
is a \SCA, 
for all $f\in {\cal E},$ 
there is an order zero \cpc\,  $\psi: D\to   B$
such that
\beq
\|a^*\phi(x) a-\psi(x)\|<\ep\rforal x\in {\cal F}.
\eneq

\end{lem}

\begin{proof}
We apply Lemma \ref{Cclose2}.
Fix $\ep>0$ and a finite subset ${\cal F}.$
Let ${\cal G}\subset D^{\bf 1}$ be the finite subset and $\dt_0>0$ (as $\dt$)  given by Lemma \ref{Cclose2}. 

 Choose $\eta,\delta>0$ so that
\beq
 \qquad 2\delta+4\eta<\dt_0
\label{eq:choices}
\eneq
and let ${\cal E}\subset D^{\bf 1}$ be a finite $\eta$-dense set containing ${\cal G}.$ 
 Define
\[
 \theta(x)=a^*\varphi(x)a\rforal  x\in D.
\]
Then $\theta$ is a \cpc.  If $x\in D^{\bf 1}$ and $f\in {\cal E}$ satisfies
$\lVert x-f\rVert<\eta$, then  \eqref{Cclose2-1} and
\eqref{Cclose2-2} give
\beq\label{eq:extended-comm}
&& \lVert[a,\varphi(x)]\rVert,
 \ \lVert[a^*,\varphi(x)]\rVert<\delta+2\eta,
 \\\label{eq:extended-dist}
 &&\dist(\theta(x),B)<\delta+\eta<\dt_0.
\eneq

Finally, let $x,y\in D$ be positive contractions with $xy=0$.
Since $\varphi$ has order zero, $\varphi(x)\varphi(y)=0$. 
Therefore, by
\eqref{eq:extended-comm}, we obtain
\begin{align*}
 \lVert\theta(x)\theta(y)\rVert
 &=\lVert a^*\varphi(x)aa^*\varphi(y)a\rVert\\
 &<2(\delta+2\eta)<\dt_0.
\end{align*}
Thus \eqref{Cclose2-001} and \eqref{Cclose2-002} hold, and the
stability statement yields an order-zero  \cpc\,
$\psi\colon D\to B$ satisfying
\[
 \lVert a^*\varphi(x)a-\psi(x)\rVert<\varepsilon,
 \qquad x\in{\cal F}.
\]

\end{proof}

\begin{lem}\label{Lfullness}
Let $C$ be a separable \CA. Then there exists a sequence $\{z_j\}_{j=1}^\infty\in C_+^{\bf 1}\setminus \{0\}$
satisfying the following:

For any \CA\, $A$ and any \hm\, $\pi: C\to A,$ and every $c\in C_+^{\bf 1}$ such that
$\pi(c)\not=0,$ there exists $j\in \N$ such that 
\beq
z_j\lesssim c\andeqn \pi(z_j)\not=0.
\eneq
\end{lem}

\begin{proof}
Choose a dense sequence $\{y_m\}$ in $C_+^{\bf 1}.$ Consider the countable family
$$
{\cal Z}=\{(y_m-r)_+: m\in \N,\, r\in \Q\cap (0,1)\}\setminus \{0\}.
$$
Re-enumerate ${\cal Z}$ as $\{z_j\}_{j=1}^\infty.$

Let $\pi: C\to A$ be a \hm\, and $c\in C_+^{\bf 1}$ such that $\pi(c)\not=0.$
Choose $0<\ep<1/2$ such that
\beq\label{Lfullness-3}
(\pi(c)-4\ep)_+\not=0.
\eneq
Choose $m\in \N$ such that
\beq
\|y_m-c\|<\ep.
\eneq
Then, by \cite[Proposition 2.2]{RrII}, 
\beq
(y_m-\ep)_+\lesssim c.
\eneq
We also have 
\beq
\|\pi(y_m)-\pi(c)\|<\ep.
\eneq
It follows that 
\beq
(\pi(c)-3\ep)_+\lesssim (\pi(y_m)-2\ep)_+.
\eneq
By \eqref{Lfullness-3}, 
\beq
(\pi(c)-3\ep)_+\not=0.
\eneq
Hence 
\beq\label{Lfullness-10}
(\pi(y_m)-2\ep)_+\not=0.
\eneq
Then, since $\pi((y_m-2\ep)_+)=(\pi(y_m)-2\ep)_+,$ we conclude that  
\beq
\pi((y_m-2\ep)_+)\not=0.
\eneq
Choose $r\in \Q$ such that $\ep<r<2\ep.$
Then 
\beq
(y_m-r)_+\lesssim (y_m-\ep)_+\lesssim c \andeqn \pi((y_m-r)_+)=(\pi(y_m)-r)_+.
\eneq
Since $r<2\ep,$ \eqref{Lfullness-10} implies that
\beq
(\pi(y_m)-r)_+\not=0.
\eneq
Now $(y_m-r)_+\in {\cal Z}.$ We may write $z_j=(y_m-r)_+$ for some $j\in \N.$
Finally, we have 
\beq
z_j\lesssim c\andeqn \pi(z_j)\not=0.
\eneq
\end{proof}

\begin{prop}\label{Prank}
Let $A$ be an algebraically simple, non-elementary \CA\,  with property ETR1. Then, in Definition \ref{D1}, for every prescribed $n\in\mathbb N,$  the finite-dimensional \CA\, $D$   may additionally be chosen so that every irreducible representation of $D$ has dimension at least $n.$
\end{prop}

\begin{proof}
We assume that $A$ has property ETR1 associated with a strictly positive element $e_A$ of $A$ 
with $\|e_A\|=1$ and a constant $0<\mathfrak{\kappa}_{e_A}\le 1.$

Fix $\varepsilon>0$, a finite subset $\mathcal F\subset A$ containing 
$e_A,$  and fix an element $a\in A_+\setminus\{0\}$, and $n\in\mathbb N$.
We will show that the finite-dimensional algebra
\[
D=\bigoplus_{j=1}^k M_{r_j}
\]
appearing in Definition \ref{D1} may be chosen so that $r_j\geq n$
for every $j$.

\Wlog,  we may assume that $\|a\|=1.$
By Lemma \ref{LMn}, there is a nonzero homomorphism
$
H\colon C_0((0,1])\otimes M_n\longrightarrow A.
$
Define
\[
\varphi(x):=H(\iota\otimes x),
\qquad x\in M_n.
\]
Then $\varphi\colon M_n\to A$ is a nonzero order-zero c.p.c.\ map.

Set $b=\phi(e_{11})$ and $c=b^{1/3}.$  

Let $\pi_\varphi$ be the supporting homomorphism of $\varphi$, so that
$\varphi(x)=\pi_\varphi(x)\varphi(1_{M_n})$. Define
\beq\nonumber
\psi(x)=\pi_\varphi(x)\varphi(1_{M_n})^{1/3},
\qquad x\in M_n.
\eneq
Then $\psi$ is an order-zero c.p.c.\ map and
\beq\nonumber
\psi(e_{11})=\varphi(e_{11})^{1/3}=c.
\eneq

 Since $c\not=0$ and $A$ is 
algebraically simple, there are $v_1,v_2,..., v_m\in A$ such that 
\beq
e_A=\sum_{i=1}^m v_i^*cv_i.
\eneq
Set
\beq
R=\max \{1, \|v_1\|, \|v_2\|, ..., \|v_m\|\}.
\eneq
Before we apply  property ETR1, let us choose certain constants and a finite subset.
Set $\eta={\mathfrak{\kappa}_{e_A}^3\over{16}}$ and $\rho>0$ such that
\beq
0<\rho<{\eta\over{4mR(R+1)}}.
\eneq
To apply Lemma  \ref{Lstable2}, choose a finite subset ${\cal G}_0
\subset M_n^{\bf 1}$ (in place of ${\cal G}$) and 
$\dt_{oz}>0$ (in place of $\dt$) corresponding 
to $\rho$ (in place of $\ep$) and ${\cal F}_0=\{e_{i,j}: 1\le i,j\le n\},$ a system of matrix units for $M_n,$ 
required  by Lemma \ref{Lstable2}. \Wlog, we may assume that ${\cal F}_0\subset {\cal G}_0.$

Next choose $\eta_s>0$ satisfying the following:

whenever $0\le s\le 1$ satisfying 
\beq
\|[s, c]\|<\eta_s, \,\,\|[s, v_i]\|<\eta_s \andeqn \|[s,e_A^{1/3}]\|<\eta_s,
\eneq
one has 
\beq\label{Prank-7}
\|(se_A^{1/3} s)^3-\sum_{i=1}^m (sv_is)^*(scs)(sv_is)\|<\eta/4.
\eneq

Choose $\dt>0$ such that
\beq\label{Prank-10}
0<\dt <\min\{\ep, \eta_s,  {\dt_{oz}\over{(n^2+1)}}, {\mathfrak{\kappa}_{e_A}\over{6}},     {\eta\over{ 4(8+m(2R+1))}}\}.
\eneq

Set
\beq
{\cal F}_1={\cal F}\cup\{ e_A^{1/3}, c, v_l, v_l^*, \psi(g): 1\le l\le m,\, g\in {\cal G}_0\}.
\eneq

Using Remark \ref{R1} (ii)--(iv), 
applying  property ETR1 with $\dt,$ ${\cal F}_1$ and $a.$
We obtain $0\le e\le 1,$ a finite dimensional \CA\, 
\beq
D=M_{r_1}\oplus M_{r_2}\oplus \cdots \oplus M_{r_k},
\eneq
and a fully normalized  order zero \cpc\, $\theta: D\to A$ such that

(i) $\|[e,\, x]\|<\dt \andeqn \|[(1-e)^{1/2},\, x]\|<\dt \rforal x\in {\cal F}_1$ (see Remark \ref{R1});

(ii) $(1-e)^{1/2} x(1-e)^{1/2}\in_{\dt} \theta(D)$ and  $0\le d_e\le 1$ in $D$ such that
\beq
\|\theta(d_e)-(1-e)^{1/2}e_A(1-e)^{1/2}\|<\dt\andeqn t\circ \pi_1(\tilde \theta(\iota \otimes d_e))\ge \mathfrak{\kappa}_{e_A}
\eneq
for all $t\in T(D),$ where $\tilde \theta: C:=C_0((0,1])\otimes D\to A$ is the \hm\, 
induced by $\theta,$ $\pi_1: \tilde \theta(C)\to D$ is the quotient map defined by $\pi_1(\tilde \theta(\iota\otimes d))=d$ for all $d\in D;$

(iii) $ e\lesssim a.$

Put $c_e=\tilde \theta(\iota\otimes d_e)$ and $s=(1-e)^{1/2}.$ 
Then
\beq\label{Prank-12}
\|c_e- se_As\|<\dt\andeqn 
t\circ \pi_1(c_e)\ge \mathfrak{\kappa}_{e_A} \rforal t\in T(D).
\eneq
Since $\psi({\cal G}_0)\subset {\cal F}_1,$ by  (i), 
\beq
\|[\psi(x), s]\|<\dt\le \dt_{oz}\rforal x\in  {\cal G}_0.
\eneq
Moreover, 
\beq
{\rm dist}(s\psi(g)s, \theta(D))<\dt\le \dt_{oz}\rforal g\in {\cal G}_0.
\eneq
By the choice of $\dt_{oz},$ applying Lemma \ref{Lstable2},  one obtains an order zero \cpc\, $\Psi: M_n\to \tilde\theta(C)$ such that
\beq
\|\Psi(e_{i,j})-s\psi(e_{i,j})s\|<\rho,\,\,\, 1\le i, j\le n.
\eneq
In particular,
\beq\label{Prank-17}
\|\Psi(e_{1,1})-scs\|<\rho.
\eneq
For each $i,$ choose $y_i\in \theta(D)\subset\tilde \theta(C)$ satisfying 
\beq\label{Prank-18}
\|y_i-sv_is\|<\dt.
\eneq
Moreover,
\[
\|y_i\|
\leq \|sv_i s\|+\delta
\leq R+\delta.
\]
Since $\delta<\kappa_{e_A}/6\leq 1/6$ and $R\geq1$, we have
\[
(R+\delta)^2\leq R(R+1).
\]
Since $e_A^{1/3}\in {\cal F}_1,$ by Remark \ref{D1} (iv), one may choose $z\in \theta(D)_+^{\bf 1}$ such that
\beq\label{Prank-n56}
\|z-se_A^{1/3}s\|<2\dt
\eneq
Recall that $\dt<\eta_s,$
by (i),
\beq
\|[s,\, c]\|<\eta_s,\,\, \|[s, \,v_i]\|<\eta_s,\andeqn \|[s, \, e_A^{1/3}]\|<\eta_s.
\eneq
Note that 
\beq
{\eta\over{4}}+8\dt+m(2R+1)\dt+mR(R+1)\rho<\eta.
\eneq
By  \eqref{Prank-n56},
\eqref{Prank-7}, \eqref{Prank-18}, \eqref{Prank-17},  and \eqref{Prank-18} again,
 one gets
\beq\label{Prank-20}
z^3&\approx_{6\dt} &(se_A^{1/3}s)^3\approx_{\eta/4} \sum_{i=1}^m (sv_is)^* (scs)(sv_is)\\
&\approx_{m(2R+1)\dt+mR(R+1)\rho}&\sum_{i=1}^m y_i^*\Psi(e_{1,1})y_i.
\eneq
Hence
\beq\label{Prank-n59}
\|z^3-\sum_{i=1}^m y_i^*\Psi(e_{1,1})y_i\|<\eta.
\eneq
Now write 
$\tilde \theta(C)=\bigoplus_{j=1}^k C_0(X_j\setminus \{0\}, M_{r_j}),$
where $ X_j\subset [0,1]$ is compact.  Because $\theta$ is fully normalized, $1\in  X_j$ ($1
\le j\le k$). 
We also write
\beq\label{Prank-19+}
z=z_1\oplus z_2\oplus\cdots \oplus z_k, \,\,\, c_e=c_{e,1}\oplus c_{e,2}\oplus \cdots \oplus c_{e,k},
\eneq
and $\Psi=\Psi_1\oplus \cdots \Psi_k.$   
For each $j$, define
\[
\pi_{1,j}:\widetilde\theta(C)\longrightarrow M_{r_j}
\]
by
\[
\pi_{1,j}(f_1\oplus\cdots\oplus f_k)=f_j(1).
\]
Thus $\pi_{1,j}$ vanishes on every summand except the $j$-th.

Since $0\le e_A\le e_A^{1/3},$ we have 
\beq\label{Prank-19}
se_As\le se_A^{1/3}s.
\eneq
Together with \eqref{Prank-n56}, \eqref{Prank-19} and \eqref{Prank-12}, we have
\beq\label{Prank-21}
z\ge se_A^{1/3}s-2\dt 1_{\tilde A} \ge se_As-2\dt 1_{\tilde A}\ge c_e-3\dt 1_{\tilde A}.
\eneq
Since $z, c_e\in \widetilde\theta(C),$  the resulting inequality $z\ge c_e-3\dt 1$
holds in $\widetilde\theta(C)^\sim.$

Fix $j,$ and let $tr_{r_j}$ be the normalized trace on $M_{r_j}.$ Define 
$\tau_j\in T(D)$ by
\beq
\tau_j(d_1, d_2,...,d_k)={\rm  tr}_{r_j}(d_j).
\eneq
By \eqref{Prank-12},  
\beq\label{Prank-22}
{\rm tr}_{r_j}(\pi_{1,j}(c_{e}))\ge \mathfrak{\kappa}_{e_A}.
\eneq 
By \eqref{Prank-19+}, \eqref{Prank-21} and \eqref{Prank-22} (as well as \eqref{Prank-10}), we compute that
\beq\label{Prank-23}
{\rm tr}_{r_j}(\pi_{1,j}(z))\ge \mathfrak{\kappa}_{e_A}-3\dt>{\mathfrak{\kappa}_{e_A}/2}. 
\eneq
We claim that $\pi_{1,j}\circ \Psi: M_n\to M_{r_j}$ is nonzero for each $1\le j\le k.$
Suppose to the contrary, that $\pi_{1,j}\circ \Psi=0$ for some $j.$  Then 
\beq
\pi_{1,j}\circ \Psi(e_{11})=0.
\eneq
Applying $\pi_{1,j}$ to \eqref{Prank-n59},  
we obtain 
that
\beq\label{Prank-24}
\|\pi_{1,j}(z_j^3)\|<\eta.
\eneq
Since $\pi_{1,j}(z_j)\ge 0,$ 
by \eqref{Prank-23}, we have 
\beq
\|\pi_{1,j}(z)\|\ge {\rm tr}_{r_j}(\pi_{1,j}(z))>{\mathfrak{\kappa}_{e_A}\over{2}}.
\eneq
Hence
\beq
\|\pi_{1,j}(z)^3\|=\|\pi_{1,j}(z)\|^3>{\mathfrak{\kappa}_{e_A}^3\over{8}}.
\eneq
Recall that $\eta={\mathfrak{\kappa}_{e_A}^3\over{16}}.$   This contradicts \eqref{Prank-24} proving the claim.
In other words, we now have that $\pi_{1,j}\circ \Psi: M_n\to M_{r_j}$ is a nonzero order zero \cpc. Hence 
\beq
r_j\ge n.
\eneq
But this holds for each $j.$ In other words
\beq
\min_j r_j\ge n.
\eneq

\end{proof}

\section{Traces}

The following folklore lemma follows from a standard weak*
compactness argument. See, for example, the
proof of \cite[Lemma 9.4]{LinJFA07} for the separable case 
and, for the nonunital argument ensuring
that the limiting trace is nonzero, the proof of
\cite[Proposition 8.6]{EGLNs} for the  argument in the non-unital case, ensuring
that the limiting trace is nonzero.

The following statement does not require
separability; we record it here for convenience.

\begin{lem}\label{LLtrace}
Let $A$ be a $C^*$-algebra, let ${\cal F}\subset A$ be a finite
subset, and let $e_A\in A_+\setminus\{0\}$ with $\|e_A\|\leq1$.
For every $\epsilon>0$ and every $0<\sigma\leq1$, there exist a
finite subset ${\cal G}\subset A$ and $\delta>0$ satisfying the
following condition:

If $\omega\in A^*$ is a positive linear functional such that
\[
\|\omega\|\leq1,
\qquad
\omega(e_A)\geq\sigma,
\tand\,
|\omega(xy)-\omega(yx)|<\delta
\quad\text{for all }x,y\in{\cal G},
\]
then there exists  a trace $\tau$ on $A$  such that
\[
\|\tau\|\leq1,
\qquad
\tau(e_A)\geq\sigma,
\tand
|\tau(x)-\omega(x)|<\epsilon
\quad\text{for all }x\in{\cal F}.
\]
If $A$ is unital and $\omega$ is a state, we may require that $\tau$ is a tracial state.
\end{lem}

\begin{proof}
Suppose, to the contrary, that the conclusion does not hold. Let
$\Lambda$ be the directed set consisting of pairs
\[
\lambda=({\cal G},n),
\]
where ${\cal G}\subset A$ is finite and $n\in{\mathbb N}$. We order
$\Lambda$ by
\[
({\cal G}_1,n_1)\leq({\cal G}_2,n_2)
\]
if
\[
{\cal G}_1\subset{\cal G}_2
\quad\text{and}\quad
n_1\leq n_2.
\]

For every $\lambda=({\cal G},n)\in\Lambda$, there is a positive
linear functional $\omega_\lambda\in A^*$ such that
\beq\label{LLtrace-51}
\|\omega_\lambda\|\leq1,
\qquad
\omega_\lambda(e_A)\geq\sigma,
\andeqn
|\omega_\lambda(xy)-\omega_\lambda(yx)|<\frac1n
\quad\text{for all }x,y\in{\cal G},
\eneq
but there is no positive trace $\tau\in A^*$ satisfying
\beq\label{LLtrace-52}
\|\tau\|\leq1,
\qquad
\tau(e_A)\geq\sigma,
\tand
|\tau(x)-\omega_\lambda(x)|<\epsilon
\quad\text{for all }x\in{\cal F}.
\eneq

The positive part of the closed unit ball of $A^*$ is weak*
compact. 
By weak* compactness, there exist a directed set $\Lambda'$, a
cofinal order-preserving map
\[
\varphi\colon\Lambda'\longrightarrow\Lambda,
\]
and a positive linear functional $\omega\in A^*$ such that
\[
\omega_{\varphi(\mu)}\longrightarrow\omega
\]
in the weak* topology. 
Clearly,
$
|\omega\|\leq1.
$
Moreover, by \eqref{LLtrace-51},
\[
\omega(e_A)\geq\sigma.
\]

Write
\[
\varphi(\mu)=({\cal G}_\mu,n_\mu),
\qquad \mu\in\Lambda'.
\]
We claim that $\omega$ is a trace. Fix arbitrary $x,y\in A$ and
$\eta>0$. Choose $n_0\in{\mathbb N}$ such that
$
\frac1{n_0}<\eta.
$
Put
\[
\lambda_0=(\{x,y\},n_0)\in\Lambda.
\]
Since $\varphi$ is cofinal, there exists $\mu_0\in\Lambda'$ such that
\[
\varphi(\mu)\geq\lambda_0
\quad\text{whenever }\mu\geq\mu_0.
\]
Therefore, for every $\mu\geq\mu_0$,
\[
\{x,y\}\subset{\cal G}_\mu
\quad\text{and}\quad
n_\mu\geq n_0.
\]
It follows that
\[
\left|
\omega_{\varphi(\mu)}(xy)
-\omega_{\varphi(\mu)}(yx)
\right|
<
\frac1{n_\mu}
\leq\frac1{n_0}
<\eta.
\]
Passing to the limit gives
\[
|\omega(xy)-\omega(yx)|\leq\eta.
\]
Since $\eta>0$ was arbitrary,
\[
\omega(xy)=\omega(yx).
\]
Since $x,y\in A$ were arbitrary, $\omega$ is a trace.

By weak* convergence and the finiteness of ${\cal F}$, for all
sufficiently large $\mu\in\Lambda'$,
\beq\label{LLtrace-53}
|\omega(x)-\omega_{\varphi(\mu)}(x)|<\epsilon
\quad\text{for all }x\in{\cal F}.
\eneq
But $\omega$ is a positive trace satisfying
\[
\|\omega\|\leq1
\quad\text{and}\quad
\omega(e_A)\geq\sigma.
\]
Thus, taking $\tau=\omega$, \eqref{LLtrace-53} contradicts the choice of
$\omega_{\varphi(\mu)}$ in  \eqref{LLtrace-51} -- \eqref{LLtrace-52}. This proves the first part of the lemma.
The second part also follows as the state space of a unital \CA\, is compact.
\end{proof}

\begin{thm}\label{Ltrace}
Let $A$ be a separable algebraically simple \CA\, with property ETR1. 
Then $A$ has a faithful trace. Consequently, $A$ is stably finite.
\end{thm}

\begin{proof}
We assume that $A$ satisfies property ETR1 with a strictly positive element $e_A$ such that 
$\|e_A\|=1$ and 
$1\ge \mathfrak{\kappa}_{e_A}>0.$ 

Let 
\beq
\sigma:=\mathfrak{\kappa}_{e_A}.
\eneq

By Lemma \ref{LLtrace}, it suffices to prove the following:

 For  every finite subset ${\cal F}\subset A$ and every $1>\eta>0,$ there is a state $\omega$ on $A$ such that
\beq\label{Trace-L-1}
&&|\omega(xy)-\omega(yx)|<\eta \rforal x, y\in {\cal F},\andeqn\\\label{Trace-L-2}
&&\omega(e_A)\ge {\sigma^2\over{4}}.
\eneq

We may assume that ${\cal F}\subset A^{\bf 1}$ and $x\in {\cal F}$ implies that $x^*\in {\cal F}$ as well as 
$e_A\in {\cal F}.$ 

Choose a number $0<\gamma<{\eta\cdot \sigma^2\over{20}}.$
Put $\dt=\gamma/24.$ 

Set 
$$
{\cal G}={\cal F}\cup\{xy: x, y\in {\cal F}\}.
$$
Now since $A$ has property ETR1, we obtain an element $0\le e\le 1,$ a finite-dimensional \CA\, $D,$ and a
fully normalized 
order zero  \cpc\, $\theta: D\to A.$ With $s:=(1-e)^{1/2}$ and $\tilde \theta: C_0((0,1])\otimes D\to A$ induced by $\theta,$ $C:=\tilde \theta(C_0((0,1])\otimes D),$ $B=\overline{(1-e)A(1-e)},$ 
define a \cpc\, $L: A\to B$ by $L(x)=sxs$ (for $x\in A$). We assume that
\beq
&&\|[e,\, x]\|<\dt\rforal x\in {\cal G};\\
&&\|[s,\, x]\|<\dt \rforal x\in {\cal G};\\\label{Ttrace-9}
&&{\rm dist}(L(x), \theta(D))<\dt\rforal x\in {\cal G},
\eneq
We may also assume that, for all $x, y\in {\cal F},$
\beq\label{Ttrace-4}
\|s^2xys^2-(sxs)(sys)\|<\dt\andeqn 
\|s^2e_As^2-(se_As)^2\|<\dt.
\eneq

Furthermore,  there is $0\le d_e\le 1$ in $D$ such that, with $c_e=\widetilde\theta(\iota\otimes d_e)=\theta(d_e),$ 
\beq\label{Ttrace-58}
\|L(e_A)-c_e\|<\dt\andeqn
(\tau\circ \pi_1)(c_e)\ge \sigma \rforal \tau\in T(D).
\eneq
Fix $\tau\in T(D)$ and  put $\lambda=\tau\circ \pi_1.$ Then $\lambda$ is a tracial state on $C$ and 
\beq\label{Ttrace-13}
\lambda(c_e)\ge \sigma.
\eneq

Let 
\beq
{\cal E}=\{L(x), L(xy), s^2xys^2: x, y\in {\cal F}\}\cup\{L(e_A), L(e_A)^2, s^2e_As^2\}\subset B.
\eneq
Note  that ${\cal G}\subset A^{\bf 1}$ and  ${\cal E}\subset B^{\bf 1}.$ 
By \eqref{Ttrace-9},  for each $x\in {\cal G},$ choose $c_x\in \theta(D),$
such that 
\beq
\|L(x)-c_x\|<\dt.
\eneq
Since $\|c_x\|\le 1+\dt$ for all $x\in {\cal G},$
we have, for all $x,y\in  {\cal G},$ 
\beq
{\rm dist}(L(x)L(y), C) &\le& \|L(x)L(y)-c_xc_y\|\\
     &\le &\|L(x)-c_x\|\|L(y)\|+\|c_x\|\|L(y)-c_y\|\\
     &<&\dt+(1+\dt)\dt<3\dt
\eneq
By  \eqref{Ttrace-4},
\beq
\|s^2xys^2-L(x)L(y)\|<\dt.
\eneq
Hence 
\beq
{\rm dist}(s^2xys^2, C)<\dt+3\dt=4\dt\rforal x, y\in {\cal F}.
\eneq
Similarly,
\beq
{\rm dist}(L(e_A)^2, C)<3\dt,
\eneq
and by (the second part of) \eqref{Ttrace-4},
\beq
{\rm dist}(s^2e_As^2, C)<4\dt.
\eneq
For each $z\in {\cal E},$ choose $c_z\in C$ such that 
\beq
\|z-c_z\|<4\dt<\gamma/2,
\eneq
and put ${\cal E}'=\{c_z: z\in {\cal E}\}\subset C.$

Now since $C$ is amenable, by \cite[Theorem 2.3.13]{LinBook},  we obtain a \cpc\, 
$\Psi_0: B\to C$ such that
\beq
\|\Psi_0(c)-c\|<\gamma/2\rforal c\in {\cal E}'.
\eneq
It follows that, for all $z\in {\cal E},$  
\beq\label{Ttrace-20}
\|\Psi_0(z)-z\|\le \|\Psi_0(z-c_z)\|+\|\Psi_0(c_z)-c_z\|+\|c_z-z\|<8\dt+\gamma/2<\gamma.
\eneq

Now define 
\beq
F(x)=\lambda(\Psi_0(s^2xs^2))\rforal x\in A.
\eneq
So $F$ is a positive linear functional with $\|F\|\le 1.$

Fix $x, y\in {\cal F}.$
Using \eqref{Ttrace-20}, \eqref{Ttrace-4} and \eqref{Ttrace-20} again,
we obtain
\beq\nonumber
|F(xy)-\lambda(\Psi_0(L(x))\Psi_0(L(y)))|&\le &
\|\Psi_0(s^2xys^2)-s^2xys^2\|+\|s^2xys^2-L(x)L(y)\|\\\nonumber
&&+\|L(x)L(y)-\Psi_0(L(x))\Psi_0(L(y))\|\\\nonumber
&<& \gamma+\gamma+\|L(x)\|\|L(y)-\Psi_0(L(y))\|\\\nonumber
&&+\|L(x)-\Psi_0(L(x))\|\|\Psi_0(L(y))\|\\\label{52-26}
&<&4\gamma.
\eneq
Similarly,
\beq
|F(yx)-\lambda(\Psi_0(L(y)\Psi_0(L(x))|<4\gamma.
\eneq
Since $\lambda$ is a trace on $C,$ 
we have, for all $x, y\in {\cal F},$
\beq
\lambda(\Psi_0(L(x))\Psi_0(L(y)))=\lambda(\Psi_0(L(y))\Psi_0(L(x))).
\eneq
Consequently
\beq\label{Ttrace-25}
|F(xy)-F(yx)|<8\gamma\rforal x, y\in {\cal F}.
\eneq

Since $s^2e_As^2\in {\cal E},$  by \eqref{Ttrace-20}, we have 
\beq\label{Ttrace-102}
\|\Psi_0(s^2e_As^2)-s^2e_As^2\|<\gamma.
\eneq
We also have 
\beq\label{Ttrace-103}
\|L(e_A)^2-\Psi_0(L(e_A))^2\|\le 2\|L(e_A)-\Psi_0(L(e_A))\|<2\gamma.
\eneq
Combining   \eqref{Ttrace-102},\eqref{Ttrace-103} and \eqref{Ttrace-4}, we get
\beq
|F(e_A)-\lambda(\Psi_0(L(e_A))^2)|&=&|\lambda(\Psi_0(s^2e_As^2))-\lambda(\Psi_0(L(e_A))^2)|\\
             &\le &\|\Psi_0(s^2e_As^2)-s^2e_As^2\| +\|s^2e_As^2-L(e_A)^2\|\\
             &&+\|L(e_A)^2-\Psi_0(L(e_A))^2\|\\\label{Ttrace-104}
             &<&\gamma+\gamma+2\gamma=4\gamma.
\eneq
By  \eqref{Ttrace-20} and  \eqref{Ttrace-58},  we have 
\beq
\|\Psi_0(L(e_A))-c_e\|\le \|\Psi_0(L(e_A))-L(e_A)\|+\|L(e_A)-c_e\|<\gamma+\dt<7\gamma/6.
\eneq
We then obtain
\beq\label{Ttrace-106}
\|\Psi_0(L(e_A))^2-c_e^2\|<2(7\gamma/6)=7\gamma/3.
\eneq
Thus, from  \eqref{Ttrace-104} and \eqref{Ttrace-106},
\beq\label{Ttrace-29}
F(e_A)>\lambda(\Psi_0(L(e_A))^2)-4\gamma>\lambda(c_e^2)-4\gamma-7\gamma/3.
\eneq
By  the Cauchy-Bunyakovsky-Schwarz inequality (and since   $\|\lambda\|=1$)
\beq
|\lambda(c_e)|^2=|\lambda(c_e \cdot 1)|^2\le \|\lambda\|\lambda(c_e^2)=\lambda(c_e^2).
\eneq
Hence, by \eqref{Ttrace-13}
\beq
\lambda(c_e^2)\ge \lambda(c_e)^2\ge \sigma^2.
\eneq
Therefore, by \eqref{Ttrace-29} (and $\gamma<\eta\cdot \sigma^2/20$)
\beq
F(e_A)>\sigma^2-6\gamma-\gamma/3>\sigma^2/2.
\eneq
Hence 
\beq\label{Ttrace-31}
\|F\|\ge F(e_A)>{\sigma^2\over{2}}.
\eneq
Define $\omega=F/\|F\|.$ Then $\omega$ is a state. From \eqref{Ttrace-25} and \eqref{Ttrace-31},
we have, for all $x,y\in {\cal F},$  
\beq
|\omega(xy)-\omega(yx)|<{16\gamma\over{\sigma^2}}<{{16\eta\cdot \sigma^2\over{20}}\over{\sigma^2}}<\eta.
\eneq
Finally, since $\|F\|\le 1,$
\beq\label{Ttrace-100}
\omega(e_A)={F(e_A)\over{\|F\|}}\ge F(e_A)>{\sigma^2\over{2}}.
\eneq
This proves \eqref{Trace-L-1}  and \eqref{Trace-L-2}. Applying  Lemma \ref{LLtrace} (with $\sigma^2/4$ 
in place of $\sigma$), we obtain a nonzero positive trace on $A.$ Since $A$ is algebraically simple, 
every nonzero trace on $A$ is faithful. 
Since $A$ has a faithful trace, it is stably finite. This completes the proof.
\end{proof}

\begin{cor}\label{R|x|>}
Let $A$ be a non-elementary separable algebraically simple \CA\, with property ETR1. Then 
in Definition \ref{D1},  we may also require that 
\beq
\|(1-e)^{1/2}x(1-e)^{1/2}\|\ge \|x\|-\ep \tforal  x\in {\cal F}.
\eneq 

\end{cor}

\begin{proof}
By Theorem \ref{Ltrace}, 
$A$ is stably finite. Then the conclusion follows from Lemma \ref{Pnot=0}  and Remark \ref{Rs10} (using $\ep/2$).
\end{proof}

\section{Hereditary \SCA s}

\begin{lem}\label{L1n}
Let $0<r<1$ and $n\ge 2.$  Then, for any \CA\, $D$ and $d\in D_+^{\bf 1},$
\beq
f_r(\iota \otimes f_r(d))\in \Her(f_r((\iota\otimes d)^{1/n}))
\eneq
in $C:=C_0((0,1])\otimes D.$
\end{lem}

\begin{proof}
Consider the commutative \SCA\, generated by $\iota\otimes 1$ and $1\otimes d.$
At the scalar point $(t, \lambda)\in [0,1]\times {\rm sp}(d),$ set 
$$
y(t, \lambda)=f_r(tf_r(\lambda))\andeqn z_n(t,\lambda)=f_r((t\lambda)^{1/n}).
$$
Recall that (see Remark \ref{Lrlambda}), for $\lambda\in [0,1],$ 
\beq
\lambda \ge r f_r(\lambda).
\eneq
Suppose that $y(t, \lambda)>0.$ Since $f_r(s)>0$ precisely when $s>r/2,$ this implies 
\beq
tf_r(\lambda)>r/2.
\eneq
Consequently,
\beq
t\lambda\ge rtf(\lambda)>{r^2\over{2}}.
\eneq
For $n\ge 2,$
\beq
(t\lambda)^{1/n}\ge (t\lambda)^{1/2}>{r\over{\sqrt{2}}}>{r\over{2}}.
\eneq
It follows that  $f_r((t\lambda)^{1/n})>0.$ Thus
\beq
\{y>0\}\subset \{z_n>0\}.
\eneq
Since $y$ and $z_n$ commute, this support inclusion implies that
\beq
y\in \overline{z_nCz_n}=\Her(z_n).
\eneq
Hence, for all $n\ge 2,$
\beq
f_r(\iota \otimes f_r(d))\in \Her(f_r((\iota\otimes d)^{1/n})).
\eneq
\end{proof}

\begin{lem}\label{Ldaxiao}
Let $A$ be a \CA, let $0<r<1$ be a given positive number,   and let $C\subset A$ be a \SCA.
Suppose that 
\beq
0\le x\le y\le 1 \,\,{\rm in}\,\, A \andeqn 0\le c_x, c_y\le  1 \,\,{\rm in}\,\, C,\\
\|c_x-x\|<r\andeqn \|c_y-y\|<r.
\eneq
Then there is $c_x', c_y'\in C$ such that
\beq
&&0\le c_x'\le c_y\andeqn \|c_x'-x\|<2\sqrt{2r}+r;\\
&&0\le c_x\le c_y'\andeqn \|c_y'-y\|<3r.
\eneq
\end{lem}

\begin{proof}
Note in $\tilde A,$ we have 
\beq\label{L62-613}
c_x\le x+r 1\le y+r 1\le c_y+2r 1.
\eneq
If $c_x=x$ and $c_y=y,$  then the conclusion is immediate. Otherwise 
set $\dt=\|c_x-x\|+\|c_y-y\|>0.$  Note that  $\dt<2r.$ Then 
\beq\label{Ldashao-5}
0\le c_x\le c_y+\dt 1.
\eneq
Put
$$
s=c_y^{1/2}(c_y+\dt 1)^{-1/2}\in C\andeqn c_x'= sc_xs=c_y^{1/2}(c_y+\dt 1)^{-1/2} c_x(c_y+\dt 1)^{-1/2} c_y^{1/2}.
$$
We will show this $c_x'$ works.
Put 
$$
e=(c_y+\dt 1)^{-1/2} c_x(c_y+\dt 1)^{-1/2}.
$$
Since $c_x\le c_y+\dt 1,$ we also have $0\le e\le 1.$
It follows that
\beq
c_x'=c_y^{1/2}ec_y^{1/2}\le c_y.
\eneq
So now we have $c_x'\in C,$ $0\le c_x'\le c_y\le 1.$

Since 
\beq
c_x-sc_xs=(1-s)c_x+sc_x(1-s),
\eneq
Put $v= (c_y+\delta 1)^{-1/2}c_x^{1/2}.$ Then,
by \eqref{Ldashao-5}, 
$\|(c_y+\delta 1)^{-1/2}c_x^{1/2}\|\le 1.$
By functional calculus,
$$\|(1-s)(c_y+\delta 1)^{1/2}\|=\sup_{t\geq 0}(\sqrt{t+\delta}-\sqrt{t})\leq\sqrt{\delta}.$$
Therefore,  
$$\|(1-s)c_x\|=\|(1-s)(c_y+\delta 1)^{1/2}vc_x^{1/2}\| \le \|(1-s)(c_y+\delta 1)^{1/2}\|\leq\sqrt{\delta}.$$
Moreover (as $\|s\|\le 1$),
$$
\|sc_x(1-s)\|\leq \|c_x(1-s)\|=\|(1-s)c_x\|\leq\sqrt{\delta}.
$$
Consequently, 
$$
\|c_x-c_x'\|\leq\|(1-s)c_x\|+\|sc_x(1-s)\|\leq 2\sqrt{\delta}<2\sqrt{2r}.
$$
Therefore,
\[
\|c_x'-x\|
\leq \|c_x'-c_x\|+\|c_x-x\|
<2\sqrt{2r}+r.
\]

Next,  put
\[
c_y'=c_y+(c_x-c_y)_+\in C_+.
\]
By \eqref{L62-613}
\[
c_x-c_y\leq  2r 1,
\]
and hence
\[
\|(c_x-c_y)_+\|\leq 2r.
\]
Moreover,
\[
c_y'-c_x
=(c_x-c_y)_+-(c_x-c_y)
=(c_x-c_y)_-\geq0.
\]
Thus $0\leq c_x\leq c_y'.$ Finally,
\[
\begin{split}
\|c_y'-y\|
&\leq \|c_y-y\|+\|(c_x-c_y)_+\|\\
&\leq \|c_y-y\|+2r<3r.
\end{split}
\]
\end{proof}

\begin{lem}\label{Lfc=c}
Let $D$ be a \CA\, and $0<r<1.$ 
Suppose that $h\in D,$ $d\in D_+^{\bf 1},$   and  $dh=h=hd.$
Then, with $c=g_{2r}\otimes h\in C_0((0,1], D),$ 
\beq
f_r(\iota\otimes d)c=c=cf_r(\iota\otimes d).
\eneq
\end{lem}

\begin{proof}
Put $C:=C_0((0,1], D)=C_0((0,1])\otimes D.$
For $F\in C_0((0,1], D),$ we will write $F(1)$ for $\pi_1(F),$ 
where $\pi_1: C\to D$ is the evaluation at $1,$ i.e., $\pi_1(F)=F(1).$ 

For each fixed $t\in [0,1],$ define
\beq
\phi_t(s)=f_r(ts)\rforal s\in [0,1].
\eneq
From $dh=h=hd,$ by induction, 
\beq
d^nh=h=hd^n\,\,\, (n\in \N).
\eneq
Therefore, for all polynomials $p,$ 
\beq
p(d)h=p(1)h
\eneq
Approximating $\phi_t$ uniformly on $[0,1]$ by polynomials,  we obtain
\beq
f_r(td)h=\phi_t(d)h=\phi_t(1)h=f_r(t)h.
\eneq
Similarly 
\beq
hf_r(td)=f_r(t)h.
\eneq
Now $c(t)=g_{2r}(t) h.$ Hence 
\beq
(f_r(\iota\otimes d)c)(t)=f_r(td)g_{2r}(t)h=g_{2r}(t)f_r(t)h.
\eneq
Note that $f_r(t)g_{2r}(t)=g_{2r}(t)$ for all $t\in [0,1].$ Thus
\beq
(f_r(\iota\otimes d)c)(t)=g_{2r}(t)f_r(t)h=g_{2r}(t)h=c(t).
\eneq
In other words
\beq
f_r(\iota \otimes d)c=c.
\eneq
The same argument shows that
\beq
cf_r(\iota\otimes d)=c.
\eneq
\end{proof}

\begin{prop}\label{Phered}
Let $A$ be a non-elementary $\sigma$-unital  algebraically simple \CA\, with property ETR1.
Then every nonzero $\sigma$-unital hereditary \SCA\, of $A$ also has the property ETR1.
\end{prop}

\begin{proof}
Let $A$ be an algebraically simple \CA\, with property ETR1 associated with a given 
strictly positive element $e_A$ with $\|e_A\|=1$ and $0<\sigma\le 1.$

{\bf 1. Set up constants}

Let $B\subset A$ be a nonzero hereditary \SCA\, with a strictly positive element 
$e_B$ with $\|e_B\|=1.$
Since $A$ is algebraically simple, 
there are $x_1, x_2,...,x_m\in A$ such that
\beq\label{Phere-68}
\sum_{i=1}^m x_i^* e_Bx_i=e_A.
\eneq
Put $R=\max\{1,\|x_i\|: 1\le i\le m\}.$
Choose 
\beq\label{6sigma}
\af={\sigma\over{4m(R+1)^2}}\andeqn \sigma_B=\af^3=({\sigma\over{4m(R+1)^2}})^3.
\eneq

Now we will show that $B$ has property ETR1 associated with $e_B$ and $\sigma_B.$

Fix $\ep\in (0, 1/4)$ and finite subset ${\cal F}\subset B$ and $s\in B_+$ with $\|s\|=1.$
\Wlog,  we may assume that ${\cal F}\subset B_+^{\bf 1}.$ We may further assume that $e_B\in {\cal F}.$

Choose $\eta_b>0$  such that, whenever $0\le h\le 1$ and $\|y\|\le 1$ satisfying 
$\|[h,\, y]\|<\eta_b$ implies that 
\beq\label{Phere-70}
\|[(1-h)^{1/2}, \, y]\|<{\ep\over{16}}.
\eneq
Choose 
\beq
0<\ep_1<\min\{ {\ep\over{100}}, {\eta_b\over{100}}, {\ep \cdot \sigma_B\over{10^5m(R+1)^2}}, {\af^3\over{10^5}}\}.
\eneq
Choose $e_b\in B$ with $0\le e_b\le 1$ such that 
\beq\label{Phered-72}
\|e_bx-x\|<\ep_1/16 \andeqn \|e_b^{1/2} x-x\|<\ep_1/16\rforal x\in {\cal F}. 
\eneq

Put
\beq
{\cal F}_1={\cal F}\cup  \{e_b, e_b^{1/2}, e_b^{1/3}\}\cup \{e_b^{1/2} xe_b^{1/2}: x\in {\cal F}\}
\cup \{x_i,  x_i^*: 1\le i\le m\}\cup \{e_A, e_A^{1/3}\}.
\eneq
Choose $0<r_0<1$ such that
\beq
r_0<\min\{{\ep_1^2\over{1000}}, {\af^3\over{100}}\}\andeqn \rho={r_0\over{8}}^3.
\eneq
Let $\dt_{p}>0$ (in place of  $\dt$) associated with $\ep_1$ (in place of $\ep$) and $\rho$ (in place 
of $\sigma$) as in \cite[Lemma 3.3]{EGLNs}.

Choose 
\beq
0<r_1<\min\{{r_0\over{4}}, {\dt_{p}\over{4}}\}.
\eneq
In particular,
\beq\label{A635}
2\sqrt{2r_1}+r_1<{\varepsilon_1\over 2}.
\eneq
Indeed, this follows from
$r_1<r_0/4$ and $r_0<\varepsilon_1^2/1000$.

Choose  $0<\dt<{r_1\over{16}}.$

With these choices, we have 
\beq\label{Phere-n84}
m(R+1)^2 \ep_1<{\sigma^3\over{16}}\andeqn\\\label{Phere-n85}
(9+6mR+4m(R+1))\dt<{\sigma^3\over{16}}.
\eneq
Therefore
\beq\label{Phere-n86}
(9+6mR+4m(R+1))\dt+m(R+1)^2\ep_1<{\sigma^3\over{8}}.
\eneq

{\bf 2. Applying  ETR1 property of $A$}

Since $A$ has property ETR1 associated with $e_A$ and $\sigma,$
one obtains  $0\le e\le 1$ in $A, $ a finite dimensional \CA\, 
$D$ and a fully normalized  order zero \cpc\, $\theta: D\to \overline{(1-e)A(1-e)}$ such that
(using  Remark \ref{R1} (ii)--(iv))

(i) $\|[e,\, x]\|<\dt\andeqn \|[(1-e)^{1/2},\, x]\|<\dt \rforal x\in {\cal F}_1,$ 

(ii) $(1-e)x, (1-e)^{1/2} x(1-e)^{1/2}\in_{\dt} \theta(D),$ there is $0\le d_e\le 1$ in $D$ such that
   $\|(1-e)^{1/2}e_A(1-e)^{1/2}-c_e\|<\dt,$ where $c_e=\theta(d_e),$ and 
   \beq
   t\circ \pi_1(\tilde \theta(\iota\otimes d_e))\ge \sigma\rforal t\in T(D),
   \eneq
   where $\tilde \theta: C_0((0,1], D)\to \overline{(1-e)A(1-e)}$ is the \hm\, induced by $\theta$, and 
   $\pi_1: \tilde \theta(C_0((0,1], D))\to D$ defined by 
   $\pi_1(\tilde \theta(f))=f(1)$ for all $f\in C=C_0((0,1], D);$ and
   
   (iii) $e\lesssim s.$ 
  
 Here
  $\pi_1: \tilde \theta(C)\to D$ is the evaluation map at 1 (defined in Definition \ref{Dtheta}) satisfying
  \beq\label{Phere-n87}
  \pi_1(\theta(d))=d \,\,\,{\rm (} d\in D{\rm )}.
  \eneq
  Set 
  $$
  p=(1-e)^{1/2}.
  $$
  Retaining the notation $d_e$ and $c_e=\theta(d_e)$ from
condition~(ii) above, 
we have
  \beq\label{Phere-n91}
  \|pe_Ap-c_e\|<\dt\andeqn t(d_e)\ge \sigma\rforal t\in T(D).
  \eneq
  Also there exists $d_0\in D_+^{\bf 1}$ such that $c_0:=\theta(d_0)$ satisfies 
  \beq\label{Phere-n92}
  \|c_0-pe_A^{1/3}p\|<\dt.
  \eneq
 Put $\lambda=t\circ \pi_1$ as a state of $\tilde \theta(C)\subset A.$
 We may extend it to a state of $A,$ still denote it by $\lambda.$
 Then since $pe_Ap\le pe_A^{1/3}p,$ we obtain, for all $t\in T(D),$ 
 \beq
 t(d_0)=\lambda(c_0)&>&\lambda(pe_A^{1/3}p)-\dt\ge \lambda(pe_Ap)-\dt\\\label{Phere-n93}
 &\ge&  \lambda(c_e)-2\dt=t(d_e)-2\dt\ge \sigma-2\dt>\sigma/2.
 \eneq
 
 {\bf 3. First estimates}
  
  By \eqref{Phere-68}, with $p=(1-e)^{1/2},$  we compute that
  \beq\label{Phere-91}
  \|pe_Ap-\sum_{i=1}^m x_i^*pe_Bpx_i\| &=& \|\sum_{i=1}^m (px_i^*e_Bx_ip-x_i^*pe_Bpx_i)\|
 < 2mR\dt<{\sigma\over{16}}.
  \eneq
  Set
  \beq 
  b=e_b^{1/2}ee_b^{1/2}\in B\andeqn x_b=e_b^{1/2}xe_b^{1/2}\rforal x\in {\cal F}.
  \eneq
  By \eqref{Phered-72} and ${\cal F}\subset B_+^{\bf 1},$ 
  \beq\label{Phere-89}
  \|x_b-x\|<{\ep_1\over{4}}\rforal x\in {\cal F}.
  \eneq
   By (i),  for all $x\in {\cal F},$ 
   \beq\nonumber
   \|bx-xb\|&\le & \|e_b^{1/2}ee_b^{1/2}x-ee_b x\|+
  \|ee_bx-ex\|+\|ex-xe\|\\\nonumber
              &&+\|xe-xe_be\|+\|xe_be-xe_b^{1/2}ee_b^{1/2}\|\\\label{Phere-53nn}
              &<& \dt+{\ep_1\over{16}}+\dt+{\ep_1\over{16}}+\dt<\eta_b.
   \eneq
   Applying \eqref{Phere-70}, we obtain 
   \beq\label{Phere-57nn}
   \|[(1-b)^{1/2},\, x]\|<{\ep\over{16}}\rforal x\in {\cal F}.
   \eneq
   Hence  (by also \eqref{Phere-89}, (i) above, and \eqref{Phered-72})
   \beq\label{Pherd-90}
  \hspace{-0.4IN} \|(1-b)^{1/2}x(1-b)^{1/2}-px_bp\| &\le & \|(1-b)^{1/2}x(1-b)^{1/2}-(1-b)x\|\\
   &&+\|(1-b)x-(1-b)x_b\|+\|(1-b)x_b-(1-e)x_b\|\\
   &&+
   \|(1-e)x_b-px_bp\|\\
   &<&  {\ep\over{16}}+{\ep_1\over{4}}+\|e_b^{1/2}ee_b^{1/2} x-ee_b^{1/2}xe_b^{1/2}\|+\dt\\\label{Phere-99}
   &<&{\ep\over{16}}+
   {\ep_1\over{4}}+\dt+{2\ep_1\over{16}}+\dt<\ep/8.
   \eneq
   
   {\bf 4. A hereditary cut}
   
   Set $y_b=e_b^{1/2}(1-e)e_b^{1/2}=e_b^{1/2} p^2e_b^{1/2}\in B.$ Then 
   $$b+y_b=e_b^{1/2}(e+(1-e))e_b^{1/2}=e_b.$$
   Since $0\le e_b\le 1,$ we have
    \beq
   0\le y_b=e_b-b\le 1-b.
   \eneq
   Hence
   \beq
   y_b\in \overline{(1-b)B(1-b)}.
   \eneq
   Choose $0\le d_b\le 1$ in $D$ such that
   \beq\label{Phere-r65}
   \|\theta(d_b)-pe_bp\|<\dt.
   \eneq
   We compute that, by (i), 
   \beq
  pe_bp\approx_\dt e_b^{1/2}pe_b^{1/2}p\approx_\dt e_b^{1/2} p^2 e_b^{1/2}=y_b.
  \eneq
     Hence
     \beq\label{Phere-102}
     \|\theta(d_b)-y_b\|\le \|\theta(d_b)-pe_bp\|+\|pe_bp-y_b\|<\dt+2\dt=3\dt.
     \eneq
     Note that $3\dt<\dt_p.$ 
     Applying \cite[Lemma 3.3]{EGLNs},  we obtain  a partial isometry $w\in  A^{**}$ such that
     \beq
     ww^*f_{\rho}(\theta(d_b))=f_{\rho}(\theta(d_b))ww^*=f_{\rho}(\theta(d_b)),\\\label{Phere-109}
     w^*cw\in \Her(y_b)\subset \overline{(1-b)B(1-b)},\andeqn\\\label{Phere-109+}
     \|w^*cw-c\|<\ep_1\|c\|\rforal c\in \Her(f_{\rho}(\theta(d_b))).
     \eneq
     
     {\bf 5. Construction of $D_B$ and $\Theta$}
     
     Put $g=g_{2r_0}$ (recall   Definition \ref{dfg} ).  Then $g\in C_0((0,1])_+,\, 0\le g \le 1,$  $g(1)=1,$ and 
     \beq\label{LH-69}
     \|g-\iota\|\le 2r_0.
     \eneq

          Define  $D_B=\overline{f_{2r_0}(d_b)Df_{2r_0}(d_b)}.$
     Since $D$ is finite dimensional, so is $D_B.$   Let $q$ be its unit. Then 
 \beq
     0\le f_{2r_0}(d_b)\le q=\chi_{(r_0,1]}(d_b),
     \eneq
     where $\chi_{(r_0,1]}(d_b)$ is the spectral projection associated with $(r_0,1].$ 
        Therefore 
     \beq\label{Phere-112}
     cf_{r_0}(d_b)=c=f_{r_0}(d_b)c\rforal c\in D_B.
     \eneq
     By Lemma \ref{Lfc=c},  for all $c\in D_B,$ 
     \beq
    f_{r_0}(\iota\otimes f_{r_0}(d_b))
    (g_{2r_0}\otimes c)=g_{2r_0}\otimes c=(g_{2r_0}\otimes c) f_{r_0}(\iota\otimes f_{r_0}(d_b)).
     \eneq
     It follows  that
     \beq
     g_{2r_0}\otimes c\in \Her(f_{r_0}(\iota\otimes f_{r_0}(d_b))).
     \eneq
     By Lemma \ref {L1n}, 
     \beq
     f_{r_0}(\iota\otimes f_{r_0}(d_b))\in \Her(f_{r_0}((\iota\otimes d_b)^{1/3})).
     \eneq
     Recall that ($\rho={r_0^3\over{8}}$)
     \beq
     f_\rho(t)f_{r_0}(t^{1/3})=f_{r_0}(t^{1/3}) \,\,\, (t\in [0,1]).
     \eneq
     It follows that
     \beq
     f_\rho(\iota \otimes d_b)f_{r_0}((\iota\otimes d_b)^{1/3})=f_{r_0}((\iota\otimes d_b)^{1/3}).
     \eneq
     Thus 
     \beq
     f_{r_0}((\iota\otimes d_b)^{1/3})\in \Her(f_\rho(\iota\otimes d_b)).
     \eneq
     Hence
     \beq\label{Phere-123}
     \tilde \theta(g_{2r_0}\otimes c)\in \Her(f_\rho(\theta(d_b)))\rforal c\in D_B.
     \eneq
     Define  $\Theta: D_B\to \overline{(1-b)B(1-b)}$ by
     \beq
     \Theta(c)=w^*\tilde \theta(g_{2r_0}\otimes c)w\rforal c\in D_B.
     \eneq
   Since $w^*(\cdot)w$ is multiplicative on $\Her(f_\rho(\theta(d_b))),$  $\Theta$ is an order zero \cpc\, 
   into $\overline{(1-b)B(1-b)}.$ 
   Let $\tilde \Theta: C_0((0,1])\otimes D_B\to \overline{(1-b)B(1-b)}$ be the induced \hm. 
     
     {\bf 6. Approximate ${\cal F}$ by $\Theta(D_B)$}

  %
     As $0\le x\le 1$ for all $x\in {\cal F},$ we have  (recall that $x_b=e_b^{1/2}xe_b^{1/2}$)
     \beq\label{Phere-109}
     0\le px_bp=pe_b^{1/2}xe_b^{1/2}p\le pe_bp.
     \eneq
     For each $x\in{\cal F}$, choose $d_x\in D_+^1$ and put
$c_x=\theta(d_x)$
    such that 
     \beq\label{Phere-110}
     \|c_x-px_bp\|<\dt<r_1.
     \eneq
     Let $c_y=\theta(d_b)$  and $c_x=\theta(d_x)$ with $0\le d_x\le 1.$
     By  \eqref{Phere-r65},
     \beq\label{Phere-111}
     \|pe_bp-c_y\|<\dt<r_1.
     \eneq
     By applying Lemma \ref{Ldaxiao} and \eqref{Phere-109}, \eqref{Phere-110} and \eqref{Phere-111} we obtain $c_x'\in \widetilde\theta(C_+)$ such that
     
     \beq\label{Phered-130}
     0\le c_x'\le \theta(d_b)\andeqn \|c_x'-px_bp\|<2\sqrt{2r_1} +r_1<{\ep_1\over{2}}.
     \eneq
     In particular,   
     \beq
     \|c_x'-\theta(d_x)\|\le \|c_x'-px_bp\|+\|px_bp-\theta(d_x)\|<\ep_1+\dt<2\ep_1.
     \eneq
     Now apply $\pi_1: \widetilde\theta(C)\to D.$ Set $d_x'=\pi_1(c_x').$
     Since $\pi_1$ is positive and $\pi_1(\theta(d_b))=\pi_1(\tilde \theta(\iota\otimes d_b))=d_b,$ by 
     \eqref{Phered-130},
     \beq
     0\le d_x'\le d_b.
     \eneq
     Since $\pi_1(\theta(d_x))=\pi_1(\tilde \theta(\iota\otimes d_x))=d_x,$   we have
     \beq\label{Phere-135}
     \|d_x'-d_x\|=\|\pi_1(c_x')-\pi_1(\tilde \theta(\iota\otimes d_x))\|\le \|c_x'-\theta(d_x)\|<2\ep_1.
     \eneq
     Define 
     \beq\label{Phere-n90}
     a_x:=qd_x'q\in D_B.
     \eneq
         Since $0\le d_x'\le d_b,$ and $q\le \chi_{(r_0,1]}(d_b),$ we have 
         \beq
         \|(1-q)(d_x')^{1/2}\|^2\le \|(1-q)d_b(1-q)\|<r_0.
         \eneq
         Thus
                  \beq\label{Phere-137}
         \|d_x'-a_x\|\le\|d_x'-qd_x'\|+\|qd_x'-qd_x'q\|<2\sqrt{r_0}.
         \eneq
   By \eqref{Phere-135} and \eqref{Phere-137}, we obtain
   \beq
   \|a_x-d_x\|\le \|a_x-d_x'\|+\|d_x'-d_x\|<2\sqrt{r_0}+2\ep_1<3\ep_1.
   \eneq
     Using $\|g_{2r_0}-\iota\|\le 2r_0<\ep_1$ (see \eqref{LH-69}),  we get
     \beq
     \|\tilde \theta(g_{2r_0}\otimes a_x)-\theta(d_x)\|&=&\|\tilde\theta(g_{2r_0}\otimes a_x)-\tilde \theta(\iota\otimes d_x)\|\\
     &\le & \|g_{2r_0}\otimes a_x-\iota\otimes d_x\|\\
     &\le & \|g_{2r_0}\otimes a_x-\iota\otimes a_x\|+\|\iota\otimes a_x-\iota \otimes d_x\|\\
     &<& 2r_0+3\ep_1<4\ep_1.
     \eneq
    Together with \eqref{Phere-110}, this gives 
    \beq\label{Phere-140}
    \|\tilde \theta(g_{2r_0}\otimes a_x)-px_bp\|<4\ep_1+\dt<5\ep_1.
    \eneq 
     Since $a_x\in D_B,$  \eqref{Phere-123} implies that
     \beq
     \tilde \theta(g_{2r_0}\otimes a_x)\in \Her(f_\rho(\theta(d_b))).
     \eneq
     Hence, by \eqref{Phere-109+}, 
     \beq\label{109-92}
     \|w^*\tilde \theta(g\otimes a_x)w-\tilde \theta(g_{2r_0}\otimes a_x)\|<\ep_1.
     \eneq
     By definition, $\Theta(a_x)=w^*\tilde \theta(g_{2r_0}\otimes a_x)w.$ Therefore, by \eqref{Phere-140}
     and \eqref{109-92},
     \beq
     \|\Theta(a_x)-px_bp\|<6\ep_1.
     \eneq
     Moreover, by \eqref{Phere-99}, we obtain
     \beq\nonumber
    &&\hspace{-0.4in} \|(1-b)^{1/2} x(1-b)^{1/2}-\Theta(a_x)\|\\\nonumber
     &\le &\|(1-b)^{1/2}x(1-b)^{1/2}-px_bp\|+\|px_bp-\Theta(a_x)\|\\\label{Phere-99+nn}
     &<&{\ep\over{8}}+6\ep_1<\ep\,\,\,\rforal x\in {\cal F}.
     \eneq
     
     {\bf 7. $D_B\not=0$}
     
     We now show that $D_B\not=\{0\}.$   From \eqref{Phere-n91},  $\|c_e\|\ge \sigma,$  and hence
     \beq
     \|pe_Ap\|>\sigma-\dt.
     \eneq
     Then, by \eqref{Phere-91}, 
     \beq
     \|\sum_{i=1}^m x_i^*pe_Bpx_i\|>\sigma-\dt-{\sigma\over{16}}>{7\sigma\over{8}}.
     \eneq
      Hence
      \beq
      mR^2\|pe_Bp\|>{7\sigma\over{8}}.
      \eneq
     It follows that
     \beq\label{Phere-n154}
     \|pe_Bp\|>{7\sigma\over{8mR^2}}>\af.
     \eneq
     By \eqref{Phered-72}, we have 
     \beq
    \|pe_Bp-pe_be_Be_bp\|<{\ep_1\over{8}}.
     \eneq
     We have (note that $\ep_1<\af/400$)
     \beq
     \|pe_bp\|\ge \|pe_be_Be_bp\|>\af-{\ep_1\over{8}}>2{\af\over{3}}.
     \eneq
     Therefore, applying \eqref{Phere-r65}, we obtain 
     \beq\label{Phere-nn9}
     \|\theta(d_b)\|>{2\af\over{3}}-\dt>{\af\over{2}}.
     \eneq
     It follows that
     \beq\label{Phere-nn11}
     \|d_b\|>{\af\over{2}}.
     \eneq
     Recall that $2r_0<{\af\over{50}}.$ Hence  $f_{2r_0}(d_b)\not=0$ which implies that $D_B\not=0.$

     {\bf 8.  About $e_B$  and $a_B$}
     
     Since $e_B\in {\cal F},$ let $a_{e_B}:=qd_{e_B}'q$ be the element  defined in \eqref{Phere-n90} and set
     $a_B=a_{e_B}.$
   In particular, we have 
     \beq
     \|(1-b)^{1/2}e_B(1-b)^{1/2}-\Theta(a_{e_B})\|<\ep.
     \eneq 
         Moreover,
\[
d_{e_B}'=\pi_1(c_{e_B}'),
\qquad
0\leq c_{e_B}'\leq\theta(d_b).
\]
By \eqref{Phered-130} and \eqref{Phere-89},
\beq\nonumber
\|c_{e_B}'-pe_Bp\|
&\leq
\|c_{e_B}'-pe_b^{1/2}e_Be_b^{1/2}p\|
+\|p(e_b^{1/2}e_Be_b^{1/2}-e_B)p\|\\\label{A6104}
&<
{\varepsilon_1\over 2}
+{\varepsilon_1\over 4}
<\varepsilon_1.
\eneq
    
    Let $t\in T(D).$ Define a trace on $\widetilde \theta(C)$ by
    \beq
    \lambda=t\circ \pi_1.
    \eneq 
    This is a trace on $\widetilde\theta(C).$ Recall that $\pi_1(c_0)=d_0$ (see the line above \eqref{Phere-n92}). 
    We have 
    \beq
    \lambda(c_0^3)=t(d_0^3). 
    \eneq
    Since $0\le d_0\le 1,$   considering the commutative \SCA\, $C^*(d_0),$ by the Jensen inequality, 
    as $s\to s^3$ on $[0,1]$ is convex, we obtain 
    \beq
    t(d_0^3)\ge t(d_0)^3\rforal t\in T(D).
    \eneq
    The choices above imply $2\dt<\sigma/8.$ Hence 
    by \eqref{Phere-n93},   $t(d_0)>\sigma-2\dt>{7\sigma\over{8}}.$ Therefore
    \beq\label{Phere-n167}
    \lambda(c_0^3)=t(d_0^3)\ge t(d_0)^3>({7\sigma\over{8}})^3>{\sigma\over{2}}^3.
    \eneq
 By \eqref{Phere-n92},
    \beq
    \|c_0^3-(pe_A^{1/3}p)^3\| &\le& \|c_0^3-c_0^2(pe_A^{1/3}p)\|
                       +\|c_0^2(pe_A^{1/3}p)-c_0(pe_A^{1/3}p)^2\|\\
                       &&\hspace{0.3in}+\|c_0(pe_A^{1/3}p)^2-(pe_A^{1/3}p)^3\|<3\dt.
    \eneq
    We also have (by (i))
    \beq
    \|(pe_A^{1/3}p)^3-p^3e_Ap\|<6\dt.
    \eneq
    Hence
    \beq\label{Phere-n170}
    \|c_0^3-p^3e_Ap^3\|<9\dt.
    \eneq
    By \eqref{Phere-68},  we have 
    \beq
    p^3e_Ap^3=\sum_{i=1}^m p^3x_i^*e_Bx_ip^3.
    \eneq
    By (i), we have 
    \beq
    \|[p^3,\, x_i]\|<3\dt\andeqn \|[p^3,\, x_i^*]\|<3\dt,\,\,\, 1\le i\le m.
    \eneq
    Therefore
    \beq
    \|p^3x_i^*e_Bx_ip^3-x_i^*p^3e_Bp^3x_i\| &\le & \|p^3x_i^*e_Bx_ip^3-x_i^*p^3e_Bx_ip^3\|\\
                &&+\|x_i^*p^3e_Bx_ip^3-x_i^*p^3e_Bp^3x_i\|\\
                &\le & 3R\dt+3R\dt=6R\dt.
    \eneq
    Thus
    \beq\label{Phere-n174}
    \|p^3e_Ap^3-\sum_{i=1}^m x_i^*p^3 e_Bp^3x_i\|<6mR\dt.
    \eneq
    Next we observe that
    \beq\label{Phere-n175}
    x_i^*p^3e_Bp^3x_i=(p^2x_i)^*(pe_Bp)(p^2x_i).
    \eneq
    For each $i,$ since $x_i\in {\cal F}_1,$ choose $v_i\in D$ such that
    \beq\label{Phere-n178}
    \|\theta(v_i)-px_ip\|<\dt\andeqn \|\theta(v_i)\|\le R+1.
    \eneq
    Since 
    \beq
    \|px_ip-p^2x_i\|\le \|[x_i,\, p]\|<\dt, 
    \eneq
    we obtain 
    \beq\label{Phere-n177}
    \|\theta(v_i)-p^2x_i\|<2\dt.
    \eneq

       From \eqref{A6104}, 
        \eqref{Phere-n177} and \eqref{Phere-n178}, we get
    \beq
   \hspace{-0.5in} \|\theta(v_i)^*c_{e_B}'\theta(v_i)-(p^2x_i)^*(pe_Bp)(p^2x_i)\|&\le&
    \|\theta(v_i)^*c_{e_B}'\theta(v_i)-\theta(v_i)^*c_{e_B}'(p^2x_i)\|\\
    &&+\|\theta(v_i)^*c_{e_B}'(p^2x_i)-
    \theta(v_i)^*pe_B(p^2x_i)\|\\
    &&+\|\theta(v_i)^*pe_B(p^2x_i)-(p^2x_i)^*pe_B(p^2x_i)\|\\
    &<& 2(R+1)\dt+ (R+1)^2\ep_1+2(R+1)\dt.
    \eneq
    Hence
    \beq\label{Phere-n180}
    &&\|\sum_{i=1}^m \theta(v_i)^* c_{e_B}'\theta(v_i)-\sum_{i=1}^m (p^2x_i)^*(pe_Bp)(p^2x_i)\|\\
    &&<m(4(R+1)\dt+(R+1)^2\ep_1).
    \eneq
       By \eqref{Phere-n170}, \eqref{Phere-n174}, \eqref{Phere-n175} and \eqref{Phere-n180},
  \beq\nonumber
  &&\|c_0^3-\sum_{i=1}^m \theta(v_i)^*c_{e_B}'\theta(v_i)\|\\
  &&\le \|c_0^3-p^3e_Ap^3\|+\|p^3e_Ap^3-\sum_{i=1}^m x_i^*p^3e_Bp^3x_i\|\\
  &&+\|\sum_{i=1}^m x_i^*p^3e_Bp^3x_i-\sum_{i=1}^m\theta(v_i)^*c_{e_B}'\theta(v_i)\|\\
 &&<9\dt+6mR\dt+m(4(R+1)\dt+(R+1)^2\ep_1).
  \eneq

    By \eqref{Phere-n86}, we obtain 
    \beq\label{Phere-n182}
    \|c_0^3-\sum_{i=1}^m \theta(v_i)^*c_{e_B}'\theta(v_i)\|<{\sigma^3\over{8}}.
    \eneq
  Applying $\lambda,$ and using \eqref{Phere-n167},
  \beq
  {\sigma^3\over{2}}\le \lambda(c_0^3)<{\sigma^3\over{8}}+\sum_{i=1}^m \lambda(\theta(v_i)^*c_{e_B}'\theta(v_i)).
  \eneq
We also have
\beq
\theta(v_i)\theta(v_i)^*\le (R+1)^2\cdot 1.
\eneq
  Therefore
  \beq\label{Phere-n185}
  \lambda(c_{e_B}'\theta(v_i)\theta(v_i)^*)=\lambda((c_{e_B}')^{1/2}\theta(v_i)\theta(v_i)^*(c_{e_B}')^{1/2})
  \le (R+1)^2 \lambda(c_{e_B}').
  \eneq
    It follows that
    \beq
    {\sigma^3\over{2}}<{\sigma^3\over{8}}+m(R+1)^2\lambda(c_{e_B}').
    \eneq
    Thus
    \beq
    \lambda(c_{e_B}')>{ 2\sigma^3\over{8m(R+1)^2}}.
    \eneq
    Set 
    \beq
    \bt:={\sigma^3\over{4m(R+1)^2}}.
    \eneq
    Since $d_{e_B}'=\pi_1(c_{e_B}'),$ we have 
    \beq\label{Phere-n190}
    t(d_{e_B}')\ge \bt\rforal t\in T(D).
    \eneq
    Let $q=(q_j)_{j=1}^l$ and $d_{e_B}'=(d_{e_B,j}')_{j=1}^l$  be 
 with respect to the decomposition 
    $D=\bigoplus_{j=1}^l M_{r_j}.$   If $q_j\not=0,$
    applying \eqref{Phere-n190} to the normalized trace on $M_{r_j}$ gives 
    \beq\label{Phere-n190a}
    {\rm tr}_{r_j}(d_{e_B,j}')\ge \bt.
    \eneq
Note that
\beq\label{Phere-n182}
(1-q)d_{e_B}'(1-q)\le (1-q)d_b(1-q)\le r_0(1-q).
\eneq
    One computes, by  \eqref{Phere-n190a} and \eqref{Phere-n182}  that
    \begin{align}
 {\rm tr}_{r_j}(q_jd'_{e_B,j}q_j)
 &=
 {\rm tr}_{r_j}(d'_{e_B,j})
 -
 {\rm tr}_{r_j}\bigl((1-q_j)d'_{e_B,j}(1-q_j)\bigr)
 \notag\\
 &\geq \beta-r_0.
\end{align}
   Since $r_0<\sigma^3/100(4m(R+1)^2)<\bt/4,$ we have 
   \beq
   \bt-r_0> {3\bt\over{4} }.
   \eneq
    Now let $\tau_j$ be the normalized trace on $q_jM_{r_j}q_j.$ Then 
    \beq\label{Phere-n195}
    \tau_j(q_jd_{e_B,j}'q_j)&=&{r_j\over{{\rm rank}(q_j)}}{\rm tr}_{r_j}(q_jd_{e_B,j}'q_j)\\
    &\ge& {\rm tr}_{r_j}(q_jd_{e_B,j}'q_j)\ge {3\bt\over{4}}.
    \eneq
    This implies that
    \beq
    \tau(qd_{e_B}'q)\ge {3\bt\over{4}}\rforal \tau\in T(D_B).
    \eneq
    Since $a_B=qd_{e_B}'q,$
    \beq\label{Phere-n197}
    \tau(a_B)\ge {3\bt\over{4}}={3\sigma^3\over{16m(R+1)^2}}\ge \sigma_B \rforal \tau\in T(D_B).
    \eneq
    
    {\bf 9. $\Theta$ is fully normalized}
    
    Since $\|\theta|_{M_{r_j}}\|=1$ for $1\le j\le l,$ and since $g_{2r_0}(1)=1,$  the associated cone \hm\, 
    must have 
    \beq
    \|\tilde\theta(g_{2r_0}\otimes q_j)\|=1
    \eneq
    for every non-zero $q_j:=q {\rm 1}_{M_{r_j}}.$   Moreover, by \eqref{Phere-123},
    $\tilde \theta(g_{2r_0}\otimes q_j)\in \Her(f_\rho(\theta(d_b)),$ and 
    $w^*(\cdot) w$ is isometric on $\Her(f_\rho(\theta(d_b)).$  Hence
    $$
    \|\Theta(q_j)\|=\|w^*\tilde \theta(g_{2r_0}\otimes q_j)w\|=\|\tilde \theta(g_{2r_0}\otimes q_j)\|=1.
    $$
    Thus $$\|\Theta|_{q_jDq_j}\|=1.$$ 
    
    {\bf 10. Cuntz comparison}
    
    Since $b=e_b^{1/2}ee_b^{1/2},$ we have 
    \beq
    b\lesssim e.
    \eneq
    By  condition (iii),   
    \beq
    b\lesssim e\lesssim s.
    \eneq
    Since $B$ is a hereditary \SCA\, and $b, s\in B_+,$ it follows that 
    \beq
    b\lesssim_B s.
    \eneq
    
    {\bf 11. Conclusion}
    With the given strictly positive element $e_B\in B_+$ and 
    $\sigma_B=\af^3=({\sigma\over{4m(R+1)^2}})^3, $
    for any $\ep>0$ and finite subset  ${\cal F}\subset B_+^{\bf 1},$ 
    we have constructed $0\le b\le 1$ in $B$ and a non-zero finite dimensional \CA\, $D_B=qDq,$ 
    and a fully normalized order zero \cpc\, $\Theta: D_B\to {\overline{(1-b)B(1-b)}}$ such that:
    
    (1) $\|[b,\, x]\|<\ep,$ $\|[(1-b)^{1/2},\, x]\|<\ep$ for all $x\in {\cal F};$ (by \eqref{Phere-53nn} and \eqref{Phere-57nn})
    
    (2) $(1-b)^{1/2}x(1-b)^{1/2}\in_\ep \Theta(D_B),$ (by \eqref{Phere-99+nn}),  there exists $a_B\in D_B$ with 
    $0\le a_B\le 1,$ such that
    \beq\nonumber
    &&\|(1-b)^{1/2} e_B(1-b)^{1/2}-\Theta(a_B)\|<\ep\,\,\hspace{0.4in} \andeqn\\
    &&\tau(a_B)\ge \sigma_B \rforal \tau\in T(D_B)\hspace{0.4in}  ({\rm by}\,\,\, \eqref{Phere-n197}); \andeqn
    \eneq
    
   (3) $b\lesssim_B s.$ 
   
   So $B$ has property ETR1 (see also  ($2_a$) of Remark \ref{R1}), and  the proposition then follows.
 \end{proof}
   
   \begin{rem}\label{Rhere}
   In the above proof, 
   when $e_B$ is given, 
$\mathfrak{\kappa}_{e_B}$ may be chosen $(\mathfrak{\kappa}_{e_A}/ 4m(R+1)^2)^3$
(see \eqref{6sigma}). 
Note that $mR$ can be controlled by $SUF(e_B).$ 
   
 The proof does include the case that $B=A.$ In other words, suppose that $A$ has property ETR1 with 
   strictly positive element  $e_A$ ($\|e_A\|=1$) and associated with $\mathfrak{\kappa}_{e_A}$,
   then, for any strictly positive element $e_B$ with $\|e_B\|=1,$ $A$ has ETR1 with 
   strictly positive element $e_B.$ 
   Thus we have the following corollary.
   \end{rem}

    \begin{cor}\label{CeA} 
    Let $A$ be  a $\sigma$-unital algebraically simple \CA.
    If $A$ has property ETR1 for one 
strictly positive element $e_A$ with 
    $\|e_A\|=1,$ then it has property ETR1 for any other strictly positive element $e_A'$ with $\|e_A'\|=1.$
    In other words, Definition \ref{D1} does not depend on the choice of the strictly positive element
    (but $\mathfrak{\kappa}_{e_A'}$ may be different from $\mathfrak{\kappa}_{e_A}$).
     \end{cor}

   \section{Essentially tracial rank one}
   
   We are ready to give the following definition:
   
   \begin{df}\label{DD1}
   Let $A$ be a $\sigma$-unital simple \CA.  Suppose that there is $a\in {\rm Ped}(A)_+\setminus \{0\}$
   such that $\Her(a)$ has  property ETR1. Then we say that $A$ has essentially tracial rank at most one
   and write $ETR(A)\le 1.$ 
   \end{df}
   
   \begin{rem}\label{Reetr1}
Later Theorem \ref{TTRE1-1}  shows that separable simple  \CA s
with $ETR(A)\le 1$ can be tracially approximated by one-dimensional cones  without assuming it is algebraically simple.
\end{rem}

   The following is folklore. We record it here for convenience.
   
   \begin{lem}\label{Lped}
Let $A$ be a $\sigma$-unital simple $C^*$-algebra, and let
$a,b\in {\rm Ped}(A)_+\setminus\{0\}$. Then there exist
$n\in{\mathbb N}$ and $x\in A\otimes{\cal K}$
such that
\[
x^*x=a
\]
and
\[
xx^*\in M_n({\rm Her}(b)).
\]
\end{lem}

\begin{proof}
Since $A$ is simple and $b\neq0$, the algebraic ideal generated by
$b$ is a dense ideal of $A$. Since ${\rm Ped}(A)$ is the minimal
dense ideal of $A$, the algebraic ideal generated by $b$ contains
${\rm Ped}(A)$.

Put $c=a^{1/2}$. Since $c\in{\rm Ped}(A)$, there exist
$n\in{\mathbb N}$ and $r_i,s_i\in A$, $1\leq i\leq n$, such that
\[
c=\sum_{i=1}^n r_i b s_i.
\]
Define
\[
R=
\begin{pmatrix}
r_1b^{1/2}&r_2b^{1/2}&\cdots&r_nb^{1/2}
\end{pmatrix}
\in M_{1,n}(A)\subset A\otimes {\cal K}
\]
and
\[
S=
\begin{pmatrix}
b^{1/2}s_1\\
b^{1/2}s_2\\
\vdots\\
b^{1/2}s_n
\end{pmatrix}
\in M_{n,1}(A)\subset A\otimes {\cal K}.
\]
Then
$
RS=c.
$
Put
$
h=R^*R.
$
The $(i,j)$-entry of $h$ is
\[
b^{1/2}r_i^*r_jb^{1/2}\in{\rm Her}(b).
\]
Therefore
\[
h\in M_n({\rm Her}(b))_+.
\]

Define
\[
x=h^{1/2}S\in M_{n,1}(A).
\]
Then
\[
x^*x=S^*hS=S^*R^*RS=c^*c=a.
\]
Moreover,
\[
xx^*=h^{1/2}SS^*h^{1/2}\in M_n({\rm Her}(b)),
\]
since $M_n({\rm Her}(b))$ is a hereditary $C^*$-subalgebra of
$M_n(A)\subset A\otimes {\cal K}.$

\end{proof}

   \begin{prop}\label{Pperm}
Let $A$ be a non-elementary $\sigma$-unital simple \CA.

\begin{enumerate}

\item If $A$ is algebraically simple and has property ETR1, then, for every $n\in\N,$ 
$M_n(A)$ has property ETR1.  Moreover, if $A$ has ETR1 associated with a strictly positive element
$e_A$ with $\|e_A\|=1$ and $\mathfrak{\kappa}_{eA},$ then 
$M_n(A)$ has ETR1 associated with ${\bar e}=e_A\otimes 1_n$ with $\mathfrak{\kappa}_{\bar e}=\mathfrak{\kappa}_{e_A}.$ 

\item If $ETR(A)\le 1,$ then, for any $a\in {\rm Ped}(A)_+\setminus \{0\},$
$\Her(a)$ has property ETR1. 

\item $ETR(A)\leq 1$ if and only if $ETR(B)\le 1$ for some non-zero $\sigma$-unital hereditary \SCA\, of $A,$ and this is equivalent to  $ETR(B)\le 1$ for all non-zero $\sigma$-unital hereditary \SCA s $B$ of $A.$

\item For every $n\in\mathbb N,$
\[
ETR(A)\leq 1\quad\Longleftrightarrow\quad ETR(M_n(A))\leq 1.
\]
Moreover,
\[
ETR(A)\leq 1\quad\Longleftrightarrow\quad
ETR(A\otimes\mathcal K)\leq 1.
\]
\end{enumerate}

\end{prop}


\begin{proof}
For (1), 
suppose that $A$ has property ${\rm ETR}1$ associated with a strictly
positive element $e_A\in A_+$, $\|e_A\|=1$, and a constant
$\mathfrak f_{e_A}>0$. Put
\[
e_{M_n(A)}=e_A\otimes 1_n.
\]
Then $e_{M_n(A)}$ is a strictly positive element of $M_n(A)$ and
$\|e_{M_n(A)}\|=1$.

Let $\varepsilon>0$, let ${\cal F}\subset M_n(A)$ be a finite subset,
and let $0\neq a\in M_n(A)_+$. Write
\[
x=(x_{ij})_{i,j=1}^n
\]
for $x\in{\cal F}$, and put
\[
{\cal G}=\{x_{ij}:x\in{\cal F},\ 1\leq i,j\leq n\}\cup\{e_A\}\subset A.
\]
Choose $\delta>0$ sufficiently small that
\[
n\delta<\varepsilon.
\]
Since $A$ is non-elementary and simple,  by, for example, Lemma \ref{LMn}, 
 there exists
$0\neq b\in A_+$ such that
\[
b\otimes 1_n\precsim a.
\]
Apply property ${\rm ETR}1$ to ${\cal G}$, $\delta$, and $b$. Using Remark \ref{R1} (ii). we obtain
a positive element $e\in A_+^{\bf 1}$, a finite-dimensional $C^*$-algebra
$D$, a fully normalized order-zero c.p.c.\ map
\[
\theta:D\longrightarrow (1-e)A(1-e),
\]
and a positive contraction $d_e\in D$ such that
\beq
\|[e,x_{ij}]\|<\delta
\andeqn
{\rm dist}\big((1-e)^{1/2}x_{ij}(1-e)^{1/2},\theta(D)\big)<\delta
\eneq
for all $x\in{\cal F}$ and $1\leq i,j\leq n$. Moreover,
\[
\left\|
(1-e)^{1/2}e_A(1-e)^{1/2}-\theta(d_e)
\right\|<\delta,
\andeqn 
t(d_e)\geq\mathfrak f_{e_A}
\quad\text{for all }t\in T(D),
\]
and
\[
e\precsim b.
\]

Put
\[
E=e\otimes 1_n,\qquad D_n=M_n(D),
\]
and define
\[
\Theta=\theta\otimes{\rm id}_{M_n}:D_n\longrightarrow M_n(A).
\]
Then $\Theta$ is an order-zero c.p.c.\ map with
\[
\Theta(D_n)\subset \overline{(1-E)M_n(A)(1-E)}.
\]
Since $\theta$ is fully normalized, $\Theta$ is also fully normalized.

For $x=(x_{ij})\in{\cal F}$, we have
\[
[E,x]=([e,x_{ij}])_{i,j=1}^n.
\]
Therefore
\[
\|[E,x]\|\leq n\max_{i,j}\|[e,x_{ij}]\|<n\delta<\varepsilon.
\]
Similarly,
\[
{\rm dist}\big((1-E)^{1/2}x(1-E)^{1/2},\Theta(D_n)\big)
<n\delta<\varepsilon.
\]

Put
$
d_E=d_e\otimes 1_n\in D_n.
$
Then
\[
\begin{split}
&\left\|
(1-E)^{1/2}e_{M_n(A)}(1-E)^{1/2}-\Theta(d_E)
\right\|\\
&\qquad =
\left\|
(1-e)^{1/2}e_A(1-e)^{1/2}-\theta(d_e)
\right\|<\delta<\varepsilon.
\end{split}
\]

Write
$
D=\bigoplus_{j=1}^m M_{r_j}.
$
Then
\[
D_n\cong\bigoplus_{j=1}^m M_{nr_j}.
\]
The normalized trace on $M_{nr_j}$ is the tensor product of the
normalized traces on $M_{r_j}$ and $M_n$. Hence
\[
t(d_E)\geq\mathfrak f_{e_A}
\quad\text{for all }t\in T(D_n).
\]

Finally,
\[
E=e\otimes 1_n\precsim b\otimes 1_n\precsim a.
\]
Thus all the requirements of property ${\rm ETR}1$ are satisfied for
$M_n(A)$, with strictly positive element $e_A\otimes 1_n$ and the
same constant $\mathfrak f_{e_A}$.  The "Moreover" part also follows from the proof above.

(2) Suppose that ${\rm ETR}(A)\le 1.$ 
Choose a nonzero element
$a\in {\rm Ped}(A)_+$ such that ${\rm Her}(a)$ has property ${\rm ETR}1$.

Let $b\in {\rm Ped}(A)_+\setminus \{0\}.$   By Lemma \ref{Lped}, 
there is $x\in A\otimes {\cal K}$ such that, for some $n\in \N,$ 
\beq
x^*x=b\andeqn xx^*\in M_n(\Her(a)).
\eneq
Note that since $a\in {\rm Ped}(A)_+,$ $\Her(a)$ is algebraically simple  and non-elementary.
 By (1), $M_n(\Her(a))$ has property ETR1. It follows from Proposition \ref{Phered} that
 $\Her(xx^*)$ has property ETR1. Since $\Her(b)=\Her(x^*x)$ is isomorphic to $\Her(xx^*),$
 $\Her(b)$ has property ETR1.

For (3),  suppose that  ${\rm ETR}(A)\le 1.$

Let $B\subset A$ be a nonzero $\sigma$-unital hereditary $C^*$-subalgebra.
Choose $0\neq b\in {\rm Ped}(B)_+$.  Then $\Her(b)\subset A$ 
is a $\sigma$-unital hereditary \SCA\, and is algebraically simple. Moreover, $b\in {\rm Ped}(A).$ By (2), 
$\Her(b)$ has property ETR1. Hence ${\rm ETR}(B)\le 1.$ 

Conversely, suppose that ${\rm ETR}(B)\leq 1$ for some nonzero
$\sigma$-unital hereditary $C^*$-subalgebra $B\subset A$. Choose
$0\neq b\in {\rm Ped}(B)_+$ such that ${\rm Her}_B(b)$ has property
${\rm ETR}1$. Since $B$ is hereditary in $A$,
\[
{\rm Ped}(B)={\rm Ped}(A)\cap B
\]
and
\[
{\rm Her}_B(b)={\rm Her}_A(b).
\]
Therefore $b\in {\rm Ped}(A)_+$ and ${\rm Her}_A(b)$ has property
${\rm ETR}1$. It follows directly from Definition \ref{DD1} that
${\rm ETR}(A)\leq 1$.

This proves that ${\rm ETR}(A)\leq 1$ if and only if
${\rm ETR}(B)\leq 1$ for one, and hence for every, nonzero
$\sigma$-unital hereditary $C^*$-subalgebra $B$ of $A$.

(4) For $n\in {\mathbb N}$, let
\[
B_n=\left\{
\begin{pmatrix}
a&0&\cdots&0\\
0&0&\cdots&0\\
\vdots&\vdots&&\vdots\\
0&0&\cdots&0
\end{pmatrix}
:a\in A
\right\}\subset M_n(A).
\]
Then $B_n$ is a nonzero $\sigma$-unital hereditary $C^*$-subalgebra
of $M_n(A)$, and the map
\[
a\mapsto
\begin{pmatrix}
a&0&\cdots&0\\
0&0&\cdots&0\\
\vdots&\vdots&&\vdots\\
0&0&\cdots&0
\end{pmatrix}
\]
is an isomorphism from $A$ onto $B_n$. It follows from (3) that
\[
{\rm ETR}(M_n(A))\leq 1
\quad\Longleftrightarrow\quad
{\rm ETR}(B_n)\leq 1
\quad\Longleftrightarrow\quad
{\rm ETR}(A)\leq 1.
\]

Similarly, let
\[
B
=\{a\otimes e_{11}:a\in A\}\subset A\otimes{\cal K},
\]
where $e_{11}$ is a rank-one projection in ${\cal K}$. Then $B$ is
a nonzero $\sigma$-unital hereditary $C^*$-subalgebra of
$A\otimes{\cal K}$ and is isomorphic to $A$. Hence, by (3),
\[
{\rm ETR}(A\otimes{\cal K})\leq 1
\quad\Longleftrightarrow\quad
{\rm ETR}(B)\leq 1
\quad\Longleftrightarrow\quad
{\rm ETR}(A)\leq 1.
\]
\end{proof}

\section{Affine functions}

\begin{prop}\label{Pquas=trace}
Let $A$ be a $\sigma$-unital simple \CA\, with $ETR(A)\le 1.$
Then every 2-quasitrace is a trace.  In addition, if $A$ is separable and  if $a\in {\rm Ped}(A)_+\setminus \{0\},$ 
then $T(\Her(a))=QT(\Her(a))\not=\emptyset.$
\end{prop}

\begin{proof}
We may assume that $A$ is non-elementary.

Let $a\in {\rm Ped}(A)_+\setminus \{0\}$ and $B=\Her(a).$
Since $B$ is a full hereditary \SCA, it suffices to show every 2-quasitraces are traces.
Then $B$ is algebraically simple and $\sigma$-unital.  By
Proposition \ref{Pperm} (2), $B$ has property ETR1.   
 Since every quasitrace on  a quotient 
 of $C_0((0,1])\otimes M_n$  is a trace, 
\cite[Proposition  4.6]{FLsigma} (see also  Corollary 
\ref{R|x|>})  implies that every 2-quasitrace  on $A$ is a trace. 
For the second part, if $A$ is separable,  we apply
 Theorem \ref{Ltrace} to conclude that $T(\Her(a))\not=\emptyset.$
\end{proof}

The following is standard.  

\begin{lem}\label{L82}
Let $A$ be a $C^*$-algebra, let $b\in A_+$, let $0\leq e\leq 1$, 
and let $\nu>0$. Put
\[
c=b^{1/2}(1-e)b^{1/2}
\quad\text{and}\quad
d=b^{1/2}eb^{1/2}.
\]
Then
\[
[b]\leq [c]\oplus[d]\leq [c]\oplus[e]
\]
in $\Cu(A)$. Moreover,
\[
[(b-\nu)_+]
 \leq
[(c-\nu/2)_+]\oplus[e].
\]
\end{lem}

\begin{proof}
Since
$
b=c+d,
$
and $x+y\precsim x\oplus y$ for positive elements $x$ and $y$, we have
\beq
[b]\leq[c]\oplus[d]\le [c]\oplus [e].
\eneq
For the second assertion, put
\[
c_0=(c-\nu/2)_+.
\]
It follows that 
\[
\|b-(c_0+d)\|
=\|c-c_0\|
\leq \nu/2.
\]
By \cite[Proposition 2.2]{RrII}
\beq
(b-\nu)_+\precsim (c-\nu/2)_++d.
\eneq
Thus
\[
[(b-\nu)_+]
\leq
[(b^{1/2}(1-e)b^{1/2}-\nu/2)_+]\oplus[e].
\]
\end{proof}

\begin{prop}\label{LLdiv}
Let $A$ be a non-elementary $\sigma$-unital simple \CA\, with $ETR(A)\le 1.$ Then $A$ has the following property (D):
For every $x,\,  x' \in {\rm Cu}(A)$ with $x'\ll x,$ and every $n\in \N,$ there exists $y\in {\rm Cu}(A)$ such that
\beq
n\widehat{y}\le \widehat{x}\tand \widehat{x'}\le (n+1) \widehat{y}.
\eneq

\end{prop}
\begin{proof}
Let $x,x'\in {\rm Cu}(A).$ Suppose that $x'\ll x$ and $n\in \N.$
If $x'=0,$ choose $y=0.$ 
Thus we may assume that $x'\not=0.$ 

Choose $a\in (A\otimes {\cal K})_+$ such that $x=[a].$ Since $x'\ll [a],$ there exists $\mu>0$ such that
\beq\label{LLdvi-3}
x'\ll [(a-\mu/2)_+].
\eneq
Put   $b=(a-\mu/2)_+$ and $B=\Her(b).$   Note that  $b\not=0$ and
$b\in {\rm Ped}(A\otimes {\cal K})_+.$ By  part (2) of Proposition \ref{Pperm},  $B$ has property ETR1. 
Since $x'\ll [b],$ there is $0<\nu<\mu/2$ such that
\beq\label{LLdvi-3+}
x' \ll [(b-\nu)_+].
\eneq
Put $b_0=(b-\nu)_+.$ 
Then $b_0\not=0.$ Fix $\bt=1/16n.$
Choose  an integer $R_s>1/\bt.$  Since $\Her(b_0)$ is a non-zero and non-elementary simple \CA, 
By Lemma \ref{LMn}, 
there are  $R_s$  many mutually orthogonal nonzero mutually (Cuntz) equivalent  elements.
Thus we may choose  $s\in \Her(b_0)\setminus \{0\}$ such that 
\beq
R_s[s]\le [b_0].
\eneq
 Moreover
 \beq\label{Pdiv-5}
 \widehat{[s]}\le {1\over{R_s}}\widehat{[b_0]}<\bt \widehat{[b_0]}.
 \eneq

 Choose $$R\ge 2n(n+1)\andeqn \dt=\min\{{\mu\over{4}}, {\nu\over{4}}\}.$$
  By Proposition \ref{Prank}, since $B$ has property ETR1 x, we obtain $0\le e\le 1$ in $B,$ a finite-dimensional 
 \CA\, $D=\bigoplus_{j=1}^m M_{r_j}$ with $r_j\ge R,$ and a fully normalized order-zero \cpc\, $\theta: D\to B$ 
 such that, for some $c\in D_+^{\bf 1}$ (see also Remark \ref{R1} (iv)),
 \beq\label{Pdiv-6}
&& \|(1-e)^{1/2}b(1-e)^{1/2}-\theta(c)\|<\dt\andeqn e\lesssim s.
\eneq

Now put $q=(1-e)^{1/2},$  $z=qb^{1/2},$ and $w=qbq.$
Then $zz^*=w$ and $z^*z=b^{1/2}(1-e)b^{1/2}.$ 
Note that
\beq
[b]\le [(b^{1/2}(1-e)b^{1/2}]\oplus [b^{1/2}eb^{1/2}]\le [(b^{1/2}(1-e)b^{1/2}]\oplus [e].
\eneq
By Lemma \ref{L82}, 
\beq
[b_0]=[(b-\nu)_+]\le [(z^*z-{\nu\over{2}})_+]= [(w-{\nu\over{2}})_+]\oplus [e].
\eneq
By \eqref{Pdiv-6} and \cite[Proposition 2.2]{RrII}, 
\beq
(w-{\nu\over{2}})_+\lesssim (\theta(c)-{\nu\over{4}})_+.
\eneq
Since  $0\le c\le 1_D,$
we have
\beq\label{Pdiv-12}
[b_0]\le [(w-{\nu\over{2}})_+]\oplus [e]\le [(\theta(c)-{\nu\over{4}})_+]\oplus [e]\le [\theta(1_D)]\oplus [e].
\eneq
On the other hand,  by \eqref{Pdiv-5}, we have 
\beq
\widehat{[e]}\le \widehat{[s]}<\bt \widehat{[b_0]}\le \bt (\widehat{[\theta(1_D)]}+\widehat{[e]}).
\eneq
It follows that
\beq\label{Pdiv-14}
\widehat{[e]}\le {\bt\over{1-\bt}}\widehat{[\theta(1_D)]}<{1\over{8n}}\widehat{[\theta(1_D)]}.
\eneq
Let 
\beq\label{Pdiv-15}
r_j=q_jn+\eta_j,\,\, \,q_j \in \N\andeqn  \eta_j\in \{0,1,...,n-1\}.
\eneq 
For each $j,$ choose a projection $p_j\in M_{r_j}$
with ${\rm rank}(p_j)=q_j.$ Put
\beq\label{Pdiv-17}
p=\bigoplus_{j=1}^m p_j\in D.
\eneq
Define $y=[\theta(p)]\in {\rm Cu}(A).$  In ${\rm Cu}(D),$  by the choice of $q_j,$ we have 
\beq
n[p]\le [1_D].
\eneq
Hence, by applying $\theta,$ 
\beq
ny\le [\theta(1_D)].
\eneq
Since $b$ is strictly positive for $B$ and $\theta(1_D)\in B,$  $[\theta(1_D)]\le [a]\le x.$ It follows that 
\beq
ny\le x.
\eneq

We compute that (by \eqref{Pdiv-15}), since $r_j\ge R\ge 2n(n+1),$ 
\beq\label{Pdiv-22}
(n+1)q_j>(1+{1\over{4n}})r_j.
\eneq
Let $e_{11}^{(j)}$ be a rank-one projection in $M_{r_j}.$  We have
\beq
\widehat{[\theta(1_{M_{r_j}})]}=r_j\widehat{[\theta(e_{11}^{(j)}]}\andeqn 
\widehat{[\theta(p_j)]}=q_j\widehat{[\theta(e_{11}^{(j)})]}.
\eneq 
Hence, by \eqref{Pdiv-22},
\beq
(n+1)\widehat{[\theta(p_j)]}\ge (1+{1\over{4n}})\widehat{[\theta(1_{M_{r_j}})]}.
\eneq
It follows that (see \eqref{Pdiv-17})
\beq
(n+1)y=(n+1)\sum_{j=1}^m\widehat{[\theta(p_j)]} \ge ((1+{1\over{4n}})\widehat{[\theta(1_{D})]}.
\eneq
By \eqref{LLdvi-3+}, \eqref{Pdiv-12}, \eqref{Pdiv-14}, 
\beq
\widehat{x'} &\le& \widehat{[(b-\nu)_+]}=\widehat{[b_0]}\le \widehat{([\theta(1_D)]\oplus [e])}\\
&\le & (1+{1\over{8n}})\widehat{[\theta(1_D)]}\le (1+{1\over{4n}})\widehat{[\theta(1_D)]}\\
&\le &  (n+1) \widehat{y}.
\eneq
\end{proof}

\begin{prop}\label{Palmunp}
Let $A$ be a $\sigma$-unital simple \CA\, with $ETR(A)\le 1.$ 
If $A$ either is stably finite, or separable, 
then ${\rm Cu}(A)$ is almost unperforated. 
\end{prop}

\begin{proof}
Choose $a\in {\rm Ped}(A)_+\setminus\{0\}$ and put $B=\Her(a)$.
By Proposition \ref{Pperm}(2), $B$ has property ETR1.
Since $B$ is a full hereditary subalgebra of $A$, it suffices
to prove that $\Cu(B)$ is almost unperforated. By
\cite[Remark~2.5]{FLsigma}, it is enough to prove that
$W(B)$ is almost unperforated.

If $A$ is stably finite, then $B$ is stably finite. If $A$ is
separable, then $B$ is separable and algebraically simple, and
Theorem \ref{Ltrace} shows that $B$ is stably finite. Thus, in
either case, $B$ is stably finite.

Every \CA\ isomorphic to a quotient of
$C_0((0,1])\otimes M_n$ has almost unperforated Cuntz semigroup.
Therefore, Lemma~\ref{Pnot=0}and
\cite[Theorem~4.3]{FLsigma}, applied to the ETR1
approximations of $B$, show that $W(B)$ is almost
unperforated. Consequently, $\Cu(B)$, and hence $\Cu(A)$, is
almost unperforated.

\end{proof}

\begin{cor}\label{Cstrictcomp}
Let $A$ be a $\sigma$-unital non-elementary simple \CA\, with $ETR(A)\le 1.$ 
Suppose that $A$ is either stably finite  or separable.
Then 

(1)  ${\rm Cu}(A)=V(A)\cup {\rm LAff}(\widetilde{QT}(A));$ 

(2)  ${\rm Cu}(A)$ is almost unperforated and almost divisible;

(3) $A$ has strict comparison;

(4) $A$ has tracial approximate oscillation zero;

(5) $A$ has stable rank one.

\end{cor}

\begin{proof}
By Proposition \ref{Palmunp}, ${\rm Cu}(A)$ is almost unperforated.
Hence, by  \cite[Proposition 3.2]{Ror04JS} (see also \cite[Theorem 2.13]{Linsr1}),
$A$ has strict comparison, proving
(3).  

Note also, by \cite[Proposition 4.9]{FLsigma} (see \cite[Corollary 5.1]{Ror04JS}), 
$A$ has nonzero 2-quasitraces on $\Her(a)$ for every $a\in {\rm Ped}(A)_+\setminus \{0\}.$
It follows that $\widetilde{QT}(A)\not=\{0\}.$

By  (3) and  Proposition \ref{LLdiv}, ${\rm Cu}(A)$ is almost divisible.  Combining this with Proposition \ref{Palmunp}, we obtain (2).



To prove (1),   we first note that ${\rm Cu}(A)$ is almost unperforated,  and by Proposition \ref{Pquas=trace},
every 2-quasitrace is a trace, therefore by the proof of 
\cite[Theorem 6.2]{TT}\footnote{\cite[Theorem 6.2]{TT} assumes that $A$ is exact. But it is only used for the purpose that every 2-quasitrace is a trace.},
it suffices to show that the map $\Gamma$ defined in Definition \ref{DQTGamma} is surjective. 
Note that,  by (3),
 $A$ has strict comparison. 
 By the argument in the second paragraph of the proof of \cite[Proposition 6.2.1]{Rl},
  strict comparison together with property
(D) implies that $\Gamma$ is surjective.  But property (D) follows from
Proposition \ref{LLdiv}.  This proves (1).

Finally,  (4) and (5) follow from \cite[Theorem 1.1]{Linsr1}.

\end{proof}


\section{Tracially Approximately divisible}

\begin{lem}\label{LTM}
Let $A$ be a $\sigma$-unital, algebraically simple,  non-elementary   \CA\, such that $A$ has strict comparison and 
tracial approximate oscillation zero.   Then for any non-zero projection 
$p\in l^\infty(A)/N_{cu}$ and any $n\in \N,$ there is a unital embedding
$\Phi: M_n\to p(l^\infty(A)/N_{cu})p.$
\end{lem}

\begin{proof}
Let $p\in l^\infty(A)/N_{cu}$ be a non-zero projection. Choose 
\beq\label{LTM-1}
a=\{a_m\}_{m=1}^\infty \in l^\infty(A)_+,\,\,\, 0\le a_m\le 1\andeqn \Pi_{cu}(a)=p.
\eneq
By 
 \cite[Theorem 8.5]{FLosc}, $A$ has property (TM). 
 Since $A$ is
algebraically simple, $A_+=\operatorname{Ped}(A)_+$. Therefore, for
each $m\in\mathbb N$, property (TM), applied to $a_m$, provides an
order-zero c.p.c.\ map
\[
    \varphi_m\colon M_n\longrightarrow \operatorname{Her}(a_m)
\]
such that
\beq\label{LTM-2}
\|a_m-\phi_m(1_n)a_m\|_{_{2, \overline{T(A)}^w}}<{1\over{m}}.
\eneq
If $a_m=0$, we take $\varphi_m=0$.

Define $\Phi: M_n\to l^\infty(A)/N_{cu}$ by
$\Phi(z)=\Pi_{cu}(\{\phi_m(z)\})$ for all $z\in M_n.$
It follows that $\Phi$ is an order-zero \cpc. Let $h_m=\phi_m(1_n).$ 
Since $\phi_m(M_n)\subset \Her(a_m)$ and $p$ is a projection, 
by \cite[Proposition 6.2, (6)]{FLosc}, we have $\Phi(M_n)\subset p(l^\infty(A)/N_{cu})p.$

By \eqref{LTM-2},
\beq
\{a_m-h_ma_m\}\in I_{\overline{T(A)}^w}=N_{cu}(A).
\eneq
Hence
\[
\begin{aligned}
p
 =\Pi_{\rm cu}(\{a_m\})
 =\Pi_{\rm cu}(\{h_ma_m\})
 =\Phi(1_n)p.
\end{aligned}
\]
On the other hand, since
\[
\Phi(M_n)\subset
p\big(\ell^\infty(A)/N_{\rm cu}\big)p,
\]
we have $\Phi(1_n)p=\Phi(1_n)$. Hence $\Phi(1_n)=p$.
Thus $\Phi$ is a unital homomorphism. Since $M_n$ is simple,
$\Phi$ is injective and therefore is a unital embedding.

\end{proof}

\begin{thm}\label{Tdiv}
Let $A$ be a non-elementary  separable   simple \CA\, with $ETR(A)\le 1.$ 
Then $A$ is  tracially approximately  divisible.
\end{thm}

\begin{proof}

Choose $a_0\in\operatorname{Ped}(A)_+\setminus\{0\}$ and put
$
    B=\operatorname{Her}(a_0).
$
By Proposition 7.3(2), $B$ has property ETR1. Moreover, $B$ is
separable, algebraically simple, and non-elementary. 
Suppose that $B$ has property (TAD) (see \cite[Definition 4.1]{FLL}).
Then by (3) of \cite[Remark 4.2]{FLL}, $B\otimes {\cal K}\cong A\otimes {\cal K}$
has property (TAD).  It follows from  (2) of \cite[Remark 4.12]{FLL} (see also \cite[Theorem 5.5]{FLsigma}
and \cite[Theorem 4.11]{FLL}) that $A$ is tracially approximately divisible. 
%
Replacing $A$
by $B$, we may therefore assume, for the remainder of the proof, that
$A$ is algebraically simple.

By  Corollary \ref{Cstrictcomp}, $A$ has strict comparison. 
Then $I_{\overline{T(A)}^w}=N_{cu}(A).$

We will show that $A$ has property (TAD-3) as defined in \cite[Definition 4.8]{FLL}. By \cite[Theorem 4.11]{FLL}, property  (TAD-3) is equivalent to  tracial approximate divisibility.

{\bf Step 1. Claim 1}:     Fix a finite subset ${\cal F}\subset A^{\bf 1},$  $1>\ep>0, \, n\in \N$ and $s\in A_+\setminus \{0\}$ as well as a strictly positive element $e_A$ with $\|e_A\|=1.$ \Wlog, we may assume that $e_A\in {\cal F}.$
There exists $0\le e\le 1,$ $e\lesssim s,$ and  a \hm\, $\Theta: M_n\to Q:=l^\infty(A)/I_{\overline{T(A)}^w}$
with  $p=\Theta(1_n)$ such that, for all $x\in {\cal F}$ and $z\in  M_n^{\bf 1},$
\beq\label{Tdiv-claim1-1}
&&\|[\Theta(z),\, \Pi_{cu}((1-e)x)]\|<\ep,\\\label{Tdiv-claim1-2}
&&\|(1-p)\Pi_{cu}((1-e)x)\|<\ep\andeqn \|\Pi_{cu}((1-e)x)(1-p))\|<\ep.
\eneq
(Norm in $Q$).

To prove Claim 1, let $0<\eta<\ep/20.$ Since $A$ has ETR1, we obtain $0\le e\le 1$ in $A$ with 
$e\lesssim s,$ a finite dimensional \CA\, $D=\bigoplus_{j=1}^L M_{r_j},$ a fully normalized 
order-zero \cpc\, $\theta: D\to A,$ and, for each $x\in {\cal F},$  an element $d_x\in D$ such 
that
\beq
&&(1-e)x\approx_\eta (1-e)^{1/2}x (1-e)^{1/2}\andeqn\\
&&\|(1-e)x-\theta(d_x)\|<\eta.
\eneq

Define $\Theta_0: C_0((0,1])\otimes D\to Q$ by  the constant sequence map
\beq
\Theta_0(c)=\Pi_{cu}(\{\tilde \theta(c), \tilde\theta(c),...\}).
\eneq
Then 
\beq\label{Tdiv-10}
\|\Pi_{cu}((1-e)x)-\Theta_0(\iota\otimes d_x)\|<\eta\rforal x\in {\cal F}.
\eneq

It follows from \cite[Theorem 6.4]{FLosc}
 that $Q$ has real rank zero.
Applying Lemma \ref{Lfinitesp}, 
we obtain a homomorphism
\[
H:C_0((0,1])\otimes D
  \longrightarrow
  {\rm Her}\big(\Theta_0(\iota\otimes1_D)\big)\subset Q
\]
with finite-dimensional range such that
\[
\|H(\iota\otimes d_x)
       -\Theta_0(\iota\otimes d_x)\|<\eta,
\qquad x\in{\cal F}.
\]
Put
\[
B=H\big(C_0((0,1])\otimes D\big)
\quad\text{and}\quad
b_x=H(\iota\otimes d_x),\qquad x\in{\cal F}.
\]
Then $B\subset Q$ is a finite-dimensional $C^*$-subalgebra and
\beq\label{Tdiv-11}
\|\Theta_0(\iota\otimes d_x)-b_x\|<\eta,
\qquad x\in{\cal F}.                                  
\eneq

%

Combining \eqref{Tdiv-10} and \eqref{Tdiv-11}, we get
\beq\label{Tdiv-12}
\|\Pi_{cu}((1-e)x)-b_x\|<2\eta\rforal x\in {\cal F}.
\eneq

Moreover, by  \eqref{Lfinitesp-2},
$H$ may be chosen so that
\[
 \sup_{d\in D^1}
 \left\|
 H(\iota\otimes d)-\Theta_0(\iota\otimes d)
 \right\|<\eta.
\]

Since $\theta$ is fully normalized,
\[
 \|\Theta_0(\iota\otimes 1_D)\|=\|\theta(1_D)\|=1.
\]
It follows that
\[
 \|H(\iota\otimes1_D)\|>1-\eta>0.
\]
Consequently, $B\neq\{0\}$.


Write $B=\bigoplus_{t=1}^m M_{R_t},$ and choose matrix units $\{e_{i,j}^{(t)}: 1\le i,j\le R_t\}$ 
for the $t$-th summand $M_{R_t}$ of $B.$

By  Corollary \ref{Cstrictcomp} and  Lemma \ref{LTM}, for each non-zero minimal projection $e_{1,1}^{(t)}\in B,$ 
there exists a unital embedding $\kappa_t: M_n\to e_{1,1}^{(t)} Qe_{1,1}^{(t)}.$
Define $\Theta: M_n\to Q$ by 
\beq
\Theta(z)=\sum_{t=1}^m \sum_{f=1}^{R_t}e_{f,1}^{(t)} \kappa_t(z)e_{1,f}^{(t)}\rforal z\in M_n.
\eneq
A direct computation shows that $\Theta$ is a \hm.  Moreover
\beq
\Theta(1_n)=\sum_{t=1}^m \sum_{f=1}^{R_t} e_{f,1}^{(t)}e_{1,1}^{(t)}e_{1,f}^{(t)}=1_B.
\eneq
Furthermore, $\Theta(z)$ commutes with every element in $B$ for all $z\in M_n.$ Indeed,
for a matrix unit $e_{u,v}^{(t)}\in B,$
\beq
e_{u,v}^{(t)}\Theta(z)=e_{u,1}^{(t)}\kappa_t(z)e_{1,v}^{(t)}=\Theta(z) e_{u,v}^{(t)}.
\eneq
Hence  $[\Theta(z), b]=0$ for all $z\in M_n$ and $b\in B.$ In particular,
\beq
[\Theta(z), b_x]=0\rforal x\in {\cal F}\andeqn z\in M_n.
\eneq
Therefore, for all $z\in M_n^{\bf 1},$
\beq\label{Tdiv-100}
\|[\Theta(z), \, \Pi_{cu}((1-e)x)]\|\le 2\|\Pi_{cu}((1-e)x)-b_x\|<4\eta<\ep.
\eneq
Since $p=\Theta(1_n)=1_B,$ we have 
\beq
(1-p)b_x=0=b_x(1-p).
\eneq
Thus, by \eqref{Tdiv-12}, 
\beq\label{Tdiv-101}
\|(1-p)\Pi_{cu}((1-e)x)\|\le  \|(1-p)\|\|\Pi_{cu}((1-e)x-b_x)\|<2\eta<\ep.
\eneq
Similarly,
\beq\label{Tdiv-102}
\|\Pi_{cu}((1-e)x)(1-p)\|<\ep.
\eneq
Claim 1 follows.

 {\bf Step 2: Lifting}

Let $\widetilde \Theta: C_0((0,1])\otimes M_n\to Q$ be the associated cone \hm\, determined by
$\widetilde \Theta(\iota \otimes z)=\Theta(z)\rforal z\in M_n.$  Since the cone is projective, there is 
a \hm\, $\widehat{\theta}: C_0((0,1])\otimes M_n\to l^\infty(A)$ such that 
$\Pi_{cu}\circ \widehat{\theta}=\widetilde\Theta.$ We write $\widehat{\theta}=\{\widehat{\theta}_l\}_{l=1}^\infty,$
where $\widehat{\theta}_l: C_0((0,1])\otimes M_n\to A$ is a \hm. Let 
$\theta_l: M_n\to A$ be the order-zero \cpc\, defined by $\theta_l(z)=\widehat{\theta}_l(\iota\otimes z).$
Then 
\beq
\Pi_{cu}(\{\theta_l(z)\}_l)=\Theta(z)\rforal z\in M_n.
\eneq
Let $p_l=\theta_l(1_n).$    By \eqref{Tdiv-claim1-1}, \eqref{Tdiv-claim1-2}, for all $x\in {\cal F}$ and 
$z\in M_n^{\bf 1},$
\beq\label{Tdiv-st2-1}
\|\Pi_{cu}(\{[\theta_l(z), \,(1-e)x]\}_l)\|<\ep,\\
\|\Pi_{cu}(\{(1-p_l)(1-e)x\}_l)\|<\ep\andeqn
\|\Pi_{cu}(\{(1-e)x(1-p_l)\}_l)\|<\ep.
\eneq

{\bf Step 3: Diagonal selection of the order-zero lifts} 

Now choose finite subsets ${\cal F}_1\subset {\cal F}_2\subset\cdots \subset A^{\bf 1}$ such that 
$\cup_{m=1}^\infty {\cal F}_m$ is dense in $A^{\bf 1}.$ Choose a truly Cuntz-null  (bounded) sequence 
$\{s_k\}\in N_{cu}(A).$   
Applying Claim 1 with $\mathcal F=\mathcal F_k$,
$\varepsilon=1/k$, and $s=s_k$, and then applying Step 2,
we obtain $0\le e_k\le 1$ in $A$ with 
$e_k\lesssim s_k$ and a \hm\,  $\Theta_k: M_n\to Q,$  (after lifting) such that 
there are, for each $k,$  order-zero \cpc s 
$\theta_{k,l}: M_n\to A,$ $l=1,2,...,$ 
\beq
\Pi_{cu}(\{\theta_{k,l}(z)\}_l)=\Theta_k(z)\rforal z\in M_n
\eneq
(with $p_{k,l}=\theta_{k,l}(1_n)$), and, for every $x\in\mathcal F_k$ and $1\leq i,j\leq n$, 
\beq\label{99e924}
\|\Pi_{cu}(\{[\theta_{k,l}(e_{i,j}),\, (1-e_k)x]\}_l\|<{1\over{k}},\\\label{99e925}
\|\Pi_{cu}(\{(1-p_{k,l})(1-e_k)x\}_l)\|<{1\over{k}},\andeqn\\\label{99e926}
\|\Pi_{cu}(\{(1-e_k)x(1-p_{k,l})\}_l)\|<{1\over{k}}.
\eneq
Since $e_k\precsim s_k$ for every $k$, and since $\{s_k\}$ is
truly Cuntz-null, the sequence $\{e_k\}$ is Cuntz-null.
Therefore
\begin{equation}
 \{e_k\}\in N_{\mathrm{cu}}(A).
\end{equation}
Put, for each $x\in\mathcal F_k$, $1\leq i,j\leq n$, and
$l\in\mathbb N$,
\beq\label{99e927}
 a_{x,i,j,k,l}
 =
 [\theta_{k,l}(e_{i,j}),(1-e_k)x],
\eneq
and put
\beq\label{99e928}
 b_{x,k,l}=(1-p_{k,l})(1-e_k)x
 \andeqn
 c_{x,k,l}=(1-e_k)x(1-p_{k,l}).
\eneq
Then
\beq\label{99e929}
 \|a_{x,i,j,k,l}\|\leq 2,\qquad
 \|b_{x,k,l}\|\leq 1,\qquad
 \|c_{x,k,l}\|\leq 1.
\eneq

By \eqref{99e924},
for each $x\in\mathcal F_k$ and
$1\leq i,j\leq n$, there exists a sequence
\[
 \{r^a_{x,i,j,k,l}\}_{l=1}^{\infty}\in N_{\mathrm{cu}}(A)
\]
such that
\beq\label{99e30}
 \sup_l
 \|a_{x,i,j,k,l}-r^a_{x,i,j,k,l}\|<\frac1k.
\eneq
Similarly, by \eqref{99e925} and \eqref{99e926},
for each
$x\in\mathcal F_k$, there are sequences
\[
 \{r^b_{x,k,l}\}_{l=1}^{\infty},
 \quad
 \{r^c_{x,k,l}\}_{l=1}^{\infty}
 \in N_{\mathrm{cu}}(A)
\]
such that
\beq\label{99e931}
 \sup_l\|b_{x,k,l}-r^b_{x,k,l}\|<\frac1k
 \quad\text{and}\quad
 \sup_l\|c_{x,k,l}-r^c_{x,k,l}\|<\frac1k.
\eneq

It follows that, for all $x\in\mathcal F_k$, $1\leq i,j\leq n$,
and $k,l\in\mathbb N$,
\beq\label{99e932}
 \|r^a_{x,i,j,k,l}\|\leq3,\qquad
 \|r^b_{x,k,l}\|\leq2,\qquad
 \|r^c_{x,k,l}\|\leq2.
\eneq
Hence
\beq\nonumber
 &&\hspace{-1in}\left\|
 a_{x,i,j,k,l}^*a_{x,i,j,k,l}
 -
 (r^a_{x,i,j,k,l})^*r^a_{x,i,j,k,l}
 \right\|                                               
 \leq
 \|a_{x,i,j,k,l}\|
 \|a_{x,i,j,k,l}-r^a_{x,i,j,k,l}\|                     \\\label{99e933}
 &&\hspace{1.6in}+
 \|r^a_{x,i,j,k,l}\|
 \|a_{x,i,j,k,l}-r^a_{x,i,j,k,l}\|
 <\frac5k .
\eneq
Similarly, for every $x\in\mathcal F_k$,
\beq\label{99e934}
 \left\|
 b_{x,k,l}^*b_{x,k,l}
 -
 (r^b_{x,k,l})^*r^b_{x,k,l}
 \right\|<\frac3k
\,\,\,\andeqn\,
 \left\|
 c_{x,k,l}^*c_{x,k,l}
 -
 (r^c_{x,k,l})^*r^c_{x,k,l}
 \right\|<\frac3k.
\eneq

Since $\mathcal F_k$ and
$\{e_{i,j}:1\leq i,j\leq n\}$ are finite,
\beq\label{99e35}
&& \lim_{l\to\infty}
 \max_{\substack{x\in\mathcal F_k\\1\leq i,j\leq n}}
 \sup_{\tau\in T(A)^w}
 \tau\!\left(
 (r^a_{x,i,j,k,l})^*r^a_{x,i,j,k,l}
 \right)=0.
 \eneq
 Similarly, 
\beq\label{99e936}
 &&\lim_{l\to\infty}
 \max_{x\in\mathcal F_k}
 \sup_{\tau\in T(A)^w}
 \tau\!\left(
 (r^b_{x,k,l})^*r^b_{x,k,l}
 \right)=0,
\andeqn\\\label{99e937}
 &&\lim_{l\to\infty}
 \max_{x\in\mathcal F_k}
 \sup_{\tau\in T(A)^w}
 \tau\!\left(
 (r^c_{x,k,l})^*r^c_{x,k,l}
 \right)=0.
\eneq

Therefore, for each $k$, we may choose $l(k)\in\mathbb N$ such that
\beq\label{99e938}
&& \max_{\substack{x\in\mathcal F_k\\1\leq i,j\leq n}}
\{ \sup_{\tau\in T(A)^w}
 \tau\!\left(
 (r^a_{x,i,j,k,l(k)})^*r^a_{x,i,j,k,l(k)}
 \right)\}<\frac 1k,\\\label{99e938+}                         
&& \max_{\substack{x\in\mathcal F_k\\\alpha=b,c}}
 \{\sup_{\tau\in T(A)^w}
 \tau\!\left(
 (r^\alpha_{x,k,l(k)})^*r^\alpha_{x,k,l(k)}
 \right)
 \}<\frac1k.
\eneq

Thus, for every $x\in\mathcal F_k$ and $1\leq i,j\leq n$,
\beq\label{99e40-}
 \sup_{\tau\in T(A)^w}
 \tau\!\left(
 a_{x,i,j,k,l(k)}^*a_{x,i,j,k,l(k)}
 \right)<\frac6k.
\eneq
Moreover, for every $x\in\mathcal F_k$,
\beq\label{99e939}
 \sup_{\tau\in T(A)^w}
 \tau\!\left(b_{x,k,l(k)}^*b_{x,k,l(k)}\right)<\frac4k
\andeqn
 \sup_{\tau\in T(A)^w}
 \tau\!\left(c_{x,k,l(k)}^*c_{x,k,l(k)}\right)<\frac4k.
\eneq
It follows that, for each fixed $x\in \cup_{m=1}^\infty {\cal F}_m,$ there exists 
$k_0\in \N$ such that $x\in {\cal F}_k$ for all $k\ge k_0.$ By \eqref{99e40-},
\beq\label{Tdiv-201}
\{[\theta_{k,l(k)}(e_{i,j}),\, (1-e_k)x]\}_k=\{a_{x, i,j, k,l(k)}\}_k \in I_{\overline{T(A)}^w}=N_{cu}(A).
\eneq
By linearity, the above also holds with $e_{i,j}$ replaced by any $z\in M_n.$ 
Similarly,
\beq\label{Tdiv-202}
\{(1-p_{k,l(k)})(1-e_k)x\}_k=\{b_{x, k,l(k)}\}_k\in N_{cu}(A),\\\label{Tdiv-203}
\{(1-e_k)x(1-p_{k,l(k)})\}_k=\{c_{x, k,l(k)}\}_k\in N_{cu}(A).
\eneq

Define $\psi_k=\theta_{k,l(k)}: M_n\to  A$  and $p_k=\psi_k(1_n)$ and define
$\Psi: M_n\to Q$ by $\Psi(z)=\Pi_{cu}(\{\psi_k(z)\}_k).$
Since each $\psi_k$ is an order-zero \cpc, so is $\Psi.$  

{\bf Step 4: Conclusion}

Let $x\in \cup_{m=1}^\infty {\cal F}_m.$ 
Then, for each $z\in M_n^{\bf 1},$ we compute 
\beq
\{[\psi_k(z),\,x]\}_k=\{[\psi_k(z),\, (1-e_k)x]\}_k+\{[\psi_k(z),\,e_kx]\}_k
\eneq
By \eqref{Tdiv-201}, the first term is in $N_{cu}(A).$ The second term is in $N_{cu}(A)$ because 
$\{e_k\}\in N_{cu}(A)$ and $N_{cu}(A)$ is a closed ideal. 
Therefore 
\beq
\{[\psi_k(z),\, x]\}_k\in N_{cu}(A).
\eneq
In other words,  for all $x\in \cup_{m=1}^\infty{\cal F}_m,$
\beq
[\Psi(z), \, \overline{j_A}(x)]=0.
\eneq
Since $\cup_{m=1}^\infty {\cal F}_m$ is dense in $A^{\bf 1},$ 
\beq
\Psi(M_n)\subset  {\overline{j_A}}(A)'\cap Q.
\eneq

Next, for $x\in \cup_{m=1}^\infty{\cal F}_m,$  we have 
\beq
\{(1-p_k)x\}_k=\{(1-p_k)(1-e_k)x\}+\{(1-p_k)e_k x\}
\eneq
%
By  \eqref{Tdiv-202}, the first term is in $N_{cu}(A).$  The second term is in $N_{cu}(A)$ because 
$\{e_k\}\in N_{cu}(A).$ 
It follows  that
\beq\label{99e949}
\{(1-p_k)x\}_k\in N_{cu}(A).
\eneq
Similarly, 
\beq\label{99e950}
\{x(1-p_k)\}_k\in N_{cu}(A).
\eneq
By density,  \eqref{99e949} and \eqref{99e950} hold for all $x\in A^{\bf 1}, $    whence for 
all $x\in A.$ 

Equivalently,
\beq
1-\Psi(1_n)\in \overline{j_A}(A)^\perp.
\eneq
Consequently  $\Psi$ induces a unital order-zero \cpc\, 
$$
\bar\Psi: M_n\to (Q\cap \overline{j_A}(A)')/\overline{j_A}(A)^\perp.
$$
A unital order-zero \cpc\, is a \hm. In other words, $\bar \Psi$ is a unital \hm.
So $A$ has (TAD-3). 
\end{proof}

\begin{thm}\label{MMCC}
Let $A$ be a non-elementary separable   amenable simple \CA\, with $ETR(A)\le 1.$ 
Then $A$ is ${\cal Z}$-stable.
\end{thm}

\begin{proof}
By  Theorem \ref{Tdiv}, $A$ is tracially approximately divisible (equivalently $A$ has property (TAD-2) 
in \cite{FLL}).  Hence $A$ satisfies the definition \cite[Definition 3.1]{CLZ}\footnote{
 Let ${\cal F}\subset A^{\bf 1}$ be a finite subset, let $
 \ep>0,$ nonzero elements $a, b\in A_+^{\bf 1}.$ 
 There exist $e\in A_+^{\bf 1}$ with 
 $ex\approx_{\ep/4} x\approx_{\ep/16} xe$ for all $x\in {\cal F}\cup \{b\}. $ 
 Since $A$ has (TAD-2), for any finite subset ${\cal G}\subset M_n,$ there is a \hm\, 
 $\phi: C_0((0,1])\otimes M_n\to A$ such that
 (1) $\|[\phi(\iota\otimes g), x]\|<{\ep/16}$ for all $x\in {\cal F}\cup\{b\},$ $g\in {\cal G},$ and 
 (2) $(e-e^{1/2} \phi(\iota \otimes 1_n)e^{1/2}-(\ep/16))_+\lesssim a.$
 
 But, with $h:=\phi(\iota\otimes 1_n),$ we have  $b(1-h)b\approx_{\ep/8} be(1-h)eb.$
 Hence 
 \beq\nonumber
 (b(1-h)b-\ep)_+&\lesssim &(be(1-h)eb-\ep/8)_+\sim
 ((1-h^{1/2}eb^2e(1-h)-\ep/8)\\\nonumber
 &\lesssim&
 ((1-h^{1/2}e^2(1-h^{1/2}-\ep/16)_+]\sim 
 (e(1-h)e-\ep/16)_+\lesssim a.
 \eneq}.
 Since  $A$ is also  amenable, by \cite[Theorem 4.1]{HO} and also  \cite[Theorem 5.7]{CLZ},
 $A$ is ${\cal Z}$-stable.
\end{proof}
\section{${\cal Z}$-stable \CA s}
 In this section, we prove that every separable amenable simple ${\cal Z}$-stable \CA\, has essential tracial rank at most one.

Let $A$ be a separable \CA\, with compact $T(A).$ 
Denote by
\beq
I_{A}=\{\{a_n\}\in \prod_n A: \lim_{n\to\infty} \sup_{t\in  T(A)}\tau(a_n^*a_n)=0\}.
\eneq
Denote by $Q_A=\prod_n A/I_A$ and  let
$q_A: \prod_n A\to Q_A$ be the quotient map.

Given a sequence $\{\tau_n\}_{n=1}^\infty$ in $T(A)$ and a free
ultrafilter $\omega$ on $\mathbb N$, define
\[
 \tau_{\{\tau_n\},\omega}
   \big(q_A(\{a_n\})\big)
   :=\lim_{n\to\omega}\tau_n(a_n).
\]
We denote by $T_\infty(A)$ the set of all traces on $Q_A$ obtained
in this way, where both $\{\tau_n\}$ and $\omega$ are allowed to vary.

Let $j_A: A\to \prod_n A$ by $j_A(a)=(a,a,...,a,...)$ for all $a\in A.$

There is an affine continuous surjective map $J_T: T(Q_A)\to T(A)$ 
defined by
\beq
J_T(\tau)(a)=\tau(q_A\circ j_A(a))\rforal  a\in A\andeqn \tau\in T(Q_A).
\eneq

Fix a free ultrafilter $\omega$ on $\N.$
Let $\rho_\tau\in T(Q_A)$ be defined by
\beq\label{71}
\rho_\tau(q_A(\{a_n\}))=\lim_{n\to\omega}(J_T(\tau)(a_n))\rforal \{a_n\}\in \prod_n A.
\eneq
Then 
\beq\label{72}
J_T(\tau)=J_T(\rho_\tau). 
\eneq

Suppose that  $B$ is an amenable  separable \CA\, with compact $T(B)$ and $h: 
B\to Q_A$ is a \hm\, such that the induced affine \hm\, $h_\sharp: T(Q_A)\to T(B)$ 
has the property:
\beq
h_\sharp=\Lambda\circ J_T,
\eneq
where $\Lambda: T(A)\to T(B)$ is an affine continuous surjective \hm. 
Therefore, for any $\tau\in T(Q_A),$
\beq
h_{\sharp}(\tau)(b)=\Lambda(J_T(\tau))(b)\rforal b\in B.
\eneq
It follows from \eqref{72} that 
\beq\label{74}
h_\sharp(\rho_\tau)(b)=\Lambda(J_T(\rho_\tau))(b)=\Lambda(J_T(\tau))(b)=h_\sharp(\tau)(b)
\rforal b.
\eneq

Let $L=\{L_n\}: B\to \prod_n A$ be  a \cpc\, such that $h=q_A\circ L.$ 
Hence, by \eqref{74},
\beq
h_\sharp(\tau)(b)=h_\sharp(\rho_\tau)(b)=
\rho_\tau(h(b))=\rho_\tau(q_A(\{L_n(b)\}))=\lim_{n\to\omega}J_T(\tau)(L_n(b))
\eneq
for all  $b\in B.$

In what follows (in this section as well as later sections) we will also use  the following 
fact without warning:
If $A$ is a non-elementary algebraically simple \CA\, with strict comparison,  then
\beq\label{A2.17a}
 I_{\overline{QT(A)}^w}=N_{cu}(A)
 \eneq
(see Definition \ref{DI-S}).  Moreover, when $A$ is amenable,  in addition, we also have 
\beq\label{A2.17b}
I_{\overline{T(A)}^w}=N_{cu}(A)
\eneq
as well as when, further assuming $T(A)$ is compact,
\beq\label{A2.17c}
I_A=N_{cu}(A).
\eneq

\begin{lem}\label{Luniform7}
Let $A$ and $B$ be  separable \CA s  with compact $T(A)$  and $T(B),$ respectively.
Suppose that $\Lambda: T(A)\to T(B)$ is an affine continuous surjective map.  
Suppose that there is a \hm\, $h: B\to Q_A$ such that 
$h_\sharp: T(Q_A)\to T(B)$ induced by $h$ factors through $T(A)$ and 
$h_\sharp=\Lambda\circ J_T.$

Suppose that $L=\{L_n\}: B\to \prod_n A$ is a c.p.c.  lift, i.e, 
$\Pi\circ L=h,$ where $\Pi: \prod_n A\to Q_A$ is the quotient map. 

Then, for any free ultrafilter $\omega,$    if   $\tau_{\{\tau_n\}}(q_A(\{a_n\})=\lim_{n\to\omega}\tau_n(a_n),$
then, for every $b\in B,$
\beq\label{10-1010}
\lim_{n\to\omega}|\tau_n(L_n(b))-h_\sharp(\tau_{\{\tau_n\}})(b)|=0.
\eneq

Moreover, 
\beq\label{Luniform7-1}
\lim_{n\to\infty}\sup_{t\in T(A)} |t(L_n(b))-\Lambda(t)(b)|=0\tforal b\in B.
\eneq
\end{lem}

\begin{proof}
For any sequence $\{\tau_n\}$ in $T(A),$ let 
\beq
\tau_{\{\tau_n\}}(q_A(\{a_n\}))=\lim_{n\to\omega}\tau_n(a_n)\rforal \{a_n\}\in \prod_n A.
\eneq
Fix $b\in B.$  Then 
\beq
h_\sharp(\tau_{\{\tau_n\}})(b)=\tau_{\{\tau_n\}}(h(b))=\tau_{\{\tau_n\}}(q_A(\{L_n(b)\}))=\lim_{n\to\omega}\tau_n(L_n(b)).
\eneq
Hence 
\beq
\lim_{n\to\omega}|\tau_n(L_n(b))-h_\sharp(\tau_{\{\tau_n\}})(b)|=0.
\eneq
This proves the first part of the lemma.

Suppose that \eqref{Luniform7-1}
does not hold. Then there are
$\varepsilon_0>0$ and a strictly increasing sequence $\{n_k\}$
such that
\[
\sup_{t\in T(A)}
 \left|t(L_{n_k}(b))-\Lambda(t)(b)\right|
 \geq \varepsilon_0
\]
for every $k$. Choose a free ultrafilter $\omega$ containing
\[
Y=\{n_k:k\in\mathbb N\}.
\]

For each $n\in Y,$ choose $\tau_n\in T(A)$ such that
\beq
|\tau_n(L_n(b))-\Lambda(\tau_n)(b)|\ge \ep_0.
\eneq
For $n\not\in Y,$ choose $\tau_n\in T(A)$ arbitrarily.

Since $T(A)$ is compact, $\tau_0=\lim_{n\to\omega}\tau_n$ exists in $T(A).$
Let $\tau=\tau_{\{\tau_n\}}\in T(Q_A).$ 
By the definition of $J_T,$ $J_T(\tau)=\tau_0.$  Since $h_\sharp=\Lambda\circ J_T.$ We have 
\beq
h_\sharp(\tau)=\Lambda(\tau_0)
\eneq
On the other hand, by the first part of the lemma, 
\beq
\lim_{n\to\omega}\tau_n(L_n(b))=h_\sharp(\tau)(b)=\Lambda(\tau_0)(b).
\eneq
Since $\Lambda$ is continuous, the function $t\mapsto \Lambda(t)(b)$ is continuous on $T(A).$ 
Therefore ($\tau_0=\lim_{n\to\omega}\tau_n$) 
\beq
\lim_{n\to\omega}\Lambda(\tau_n)(b)=\Lambda(\tau_0)(b).
\eneq
Hence 
\beq
\lim_{n\to\omega}|\tau_n(L_n(b))-\Lambda(\tau_n)(b)|=0
\eneq
This contradicts the choice of $\tau_n.$ The lemma follows.
\end{proof}

\begin{lem}\label{Leasylift}
Let $A$ and $B$ be separable \CA s with compact $T(A)$ and $T(B).$  
Suppose that the weak*-closure of convex hull of $T_\infty(A)$ is $T(Q_A)$
and $\Lambda_B: T(A)\to T(B)$ is an affine continuous map.
Suppose that $h:B\to Q_A$ is a \hm\,  with a c.p.c. lift $L=\{L_n\}: B\to \prod_nA$
such that
\beq\label{Leasylift-1}
\lim_{n\to \infty}\sup_{t\in T(A)} |t(L_n(b))-\Lambda_B(t)(b)|=0.
\eneq
Then $h_\sharp=\Lambda_B\circ J_{T,A}.$
\end{lem}
\begin{proof}
Consider  affine weak*-continuous maps
\beq
h_\sharp, \Lambda_B\circ J_{T,A}: T(Q_A)\to T(B).
\eneq
It suffices to show that $h_\sharp(\tau)=\Lambda_B\circ T_{J,A}(\tau)$  for every ultralimit trace 
$\tau=\tau_{\{t_n\}, \omega},$ where $t_n\in T(A)$ and $\omega$ is an ultrafilter. 

Fix such a trace $\tau.$ 
For each $b\in B,$ since $L=\{L_n\}$ is a lift of $h,$ 
\beq\label{Leasylift-5}
\tau(h(b))=\lim_{n\to\omega}  \tau_n(L_n(b)).
\eneq
By applying \eqref{Leasylift-1},
\beq
\lim_{n\to\omega} |t_n(L_n(b))-\Lambda_B(t_n)(b)|=0,
\eneq
and therefore
\beq
\tau(h(b))=\lim_{n\to\omega}\Lambda_B(t_n)(b).
\eneq
Let $\rho:=J_{T,A}(\tau)\in T(A).$ Then for every $a\in A,$
\beq
\rho(a)=J_{T,A}(\tau)(a)=\lim_{n\to\omega}t_n(a).
\eneq
Thus $t_n\to_{\omega}  \rho$ in $T(A).$  Since $\Lambda_B$ is continuous,
\beq\label{Leasylift-8}
\lim_{n\to\omega} \Lambda_B(t_n)(b)=\Lambda_B(\rho)(b)\rforal b\in B.
\eneq
Combining \eqref{Leasylift-5} and \eqref{Leasylift-8},  we obtain 
\beq
\tau(h(b))=\Lambda_B(\rho)(b)=\Lambda_B(J_{T,A}(\tau))(b)\rforal b\in B.
\eneq
Hence 
\beq
h_\sharp(\tau)=\Lambda_B\circ J_{T,A}(\tau)
\eneq
for all $\tau\in T_\infty(A).$  Consequently,   
\beq
h_\sharp=\Lambda_B\circ J_{T,A}.
\eneq
\end{proof}


\begin{lem}\label{Lidealsep}
Let $A$ be a separable \CA\, with non-empty compact $T(A).$ 
Let $S\subset l^\infty(A)$ be a separable \SCA\, and $J=S\cap I_A.$
For each $n\in \N,$ let 
$$
{\cal X}_n\subset J,\,\,\, {\cal Y}_n\subset S
$$
be finite subsets, let 
 $\af_n>0,$ $\bt_n>0$ be numbers, and let $N_n\in \N.$
 
Then there exists strictly increasing integers
$$
s(1)<s(2)<\cdots\,\,\,\andeqn m(1)<m(2)<\cdots,
$$
and positive contractions 
\beq
e_n=f_{s(n),m(n)}\in A,
\eneq
where $f_s=(f_{s,m})_{k=1}^\infty\in J$ is a quasicentral approximate identity for $J$ relative to $S,$
such that $m(n)\ge N_n$  and
\beq\label{Lidealsep-0}
\|e_n\|_{_{2, T(A)}}<\af_n,\,\, \, \|(1-e_n)^{1/2}x_{m(n)}(1-e_n)^{1/2}\|<\af_n,\,\,\,x=\{x_k\}\in {\cal X}_n,\\\label{Lidealsep-01}
 \|[e_n,\, g_{m(n)}]\|<\bt_n \andeqn \|[(1-e_n)^{1/2},\, g_{m(n)}]\|<\bt_n,\,\,\, g=\{g_n\}\in {\cal Y}_n.
\eneq
In particular, if $\af_n\to 0,$ then 
\beq
\{e_n\}\in I_A.
\eneq
\end{lem}

\begin{proof}
Choose a sequential quasicentral approximate identity 
\beq
f_1, f_2,...,f_n,.... \in J
\eneq
for $J$ relative to $S,$ with $f_s=\{f_{s,k}\}_{k=1}^\infty.$
Thus (in $l^\infty(A)$)
\beq\label{Lidealsep-10}
\|(1-f_s)x\|\to 0 \rforal x\in J\andeqn \|[f_s, g]\|\to 0 \rforal g\in S.
\eneq
We construct $s(n)$ and $m(n)$ inductively.

{\it Stage $n=1.$}

Because ${\cal X}_1$ is finite,  by \eqref{Lidealsep-10}, we obtain an integer $s(1)$ such that
\beq
&&\|(1-f_{s(1)})^{1/2}x(1-f_{s(1)})^{1/2}\|<\af_1\rforal x\in {\cal X}_1\andeqn\\
&&\|[f_{s(1)}, \, g]\|<\bt_1\rforal g\in {\cal Y}_1.
\eneq
Since $f_{s(1)}\in I_A,$
\beq
\lim_{k\to\infty}\|f_{s(1),k}\|_{_{2, T(A)}}=0.
\eneq
Choose $m(1)\ge N_1$ such that
\beq
\|f_{s(1), m(1)}\|_{_{2, T(A)}}<\af_1.
\eneq
Set $e_1=f_{s(1), m(1)}.$  We obtain \eqref{Lidealsep-0} and \eqref{Lidealsep-01} for $n=1.$ 

 {\it Inductive stage}
 
 Assume $s(1), s(2),...,s(n-1)$ and $m(1), m(2), ...,m(n-1)$ have been chosen.
 
 Since ${\cal X}_n$ is finite, choose  $s(n)>s(n-1)$ so large that
 \beq\label{Lidealsep-12}
 &&\|(1-f_{s(n)})^{1/2}x(1-f_{s(n)})^{1/2}\|<\af_n\rforal x\in {\cal X}_n\andeqn\\\label{Lidealsep-13}
 && \|[f_{s(n)}, \, g]\|<\bt_n \rforal g\in {\cal Y}_n.
 \eneq
 Since $f_{s(n)}\in I_A,$ 
 \beq
 \lim_{k\to\infty} \|f_{s(n),k}\|_{_{2, T(A)}}=0.
 \eneq
 Choose $m(n)>\max\{N_n, m(n-1)\}$ such that 
 \beq\label{Lidealsep-14}
 \|f_{s(n),m(n)}\|_{_{2, T(A)}}<\af_n.
 \eneq
 Define 
 \beq
 e_n=f_{s(n),m(n)}.
 \eneq
 By \eqref{Lidealsep-12} and \eqref{Lidealsep-13} (and continuous functional calculus), we obtain  
 \beq
 \|(1-e_n)^{1/2}x_{m(n)}(1-e_n)^{1/2}\|&<&\af_n\rforal x\in {\cal X}_n \andeqn\\
 \|[e_n,\, g_{m(n)}]\|&<&\bt_n\rforal g=\{g_k\}\in {\cal Y}_n.
 \eneq
 Moreover, by \eqref{Lidealsep-14},
 \beq
 \|e_n\|_{_{2, T(A)}}<\af_n.
 \eneq
 This completes the induction.
\end{proof}

\begin{lem}\label{Lher} 
Let $A$ be a \CA\, and $D\subset A$ be a \SCA. Suppose that
$$
0\le a\le b\le 1, \,\, 0\le e, c\le 1, \,\,\, e, c\in D,
$$
with
$$
\|b-e\|<\dt\tand \|a-c\|<\dt.
$$
Then there exists $c'\in \overline{eDe},\, 0\le c'\le 1$ such that
\beq
\|a-c'\|< 2\sqrt{2\dt}+\dt.
\eneq

\end{lem}

\begin{proof}
Choose ${\bar f}_\dt$ as in \eqref{dfg} and put  $s={\bar f}_\dt(e).$ Since ${\bar f}_\dt(0)=0,$
\beq
s\in C^*(e)\cap \overline{eAe}\subset \overline{eDe}.
\eneq
Moreover
\beq\label{Lher-5}
\|(1-s)e(1-s)\|=\sup_{t\in [0,1]} t(1-{\bar f}_\dt(t))^2\le \dt.
\eneq
Since $\|b-e\|<\dt,$
$$
b\le e+\dt \cdot 1_{\tilde A}.
$$
Therefore, using $a\le b,$
\beq
(1-s)a(1-s)\le (1-s)b(1-s)\le (1-s)e(1-s)+\dt(1-s)^2.
\eneq
Combining  \eqref{Lher-5},  we obtain 
\beq
\|(1-s)a(1-s)\|<2\dt.
\eneq
Hence
\beq
\|(1-s)a^{1/2}\|^2=\|(1-s)a(1-s)\|<2\dt, 
\eneq
and consequently
\beq
\|(1-s)a\|<\sqrt{2\dt}\andeqn \|sa(1-s)\|<\sqrt{2\dt}.
\eneq
On the other hand, we have
\beq
a-sas=(1-s)a+sa(1-s).
\eneq
Thus
\beq
\|a-sas\|\le 2\sqrt{2\dt}.
\eneq
Now choose  $c'=scs.$ Then 
\beq
c'\in \overline{eDe} \andeqn 0\le c'\le s^2\le 1.
\eneq
Finally, 
\beq
\|a-c'\|\le \|a-sas\|+\|s(a-c)s\|<2\sqrt{2\dt}+\dt.
\eneq
\end{proof}

\begin{lem}\label{Lher2}
Let $A$ be a \CA \, and let $e_A\in A$ be strictly positive with $\|e_A\|=1.$ Let ${\cal F}\subset A$ be a finite subset of the unit ball of $A,$ and let $\ep>0.$ 
Then there are finite subset ${\cal G}\subset A$  and $\dt>0,$ with $e_A\in {\cal G}$ such that
the following holds.

Suppose that $D\subset B$ are \CA s  and   $H: A\to B$ is a \cpc. Suppose that, for every $x\in {\cal G},$ there is $d(x)\in D$ satisfying 
\beq
\|d(x)-H(x)\|<\dt.
\eneq
Put $d=d(e_A),$ and suppose that $0\le d\le 1.$ Let $D_1=\overline{dDd}.$

Then, for every $f\in {\cal F},$ there exists $c(f)\in D_1$ such that
\beq
\|H(f)-c(f)\|<\ep.
\eneq
\end{lem}

\begin{proof}
For each $f\in {\cal F},$ write
\beq\label{Lher2-4}
f=a_{f,1}-a_{f,2}+i(a_{f,3}-a_{f,4}),
\eneq
where $0\le a_{f,j}\le 1,$ $j=1, 2, 3, 4.$  

Since $e_A$ is a strictly positive element, choose $h\in C_0((0,1])_+$ with $\|h\|=1$ such that
\beq\label{Lher2-5}
\|a_{f,j}-h(e_A)a_{f,j}h(e_A)\|<\ep/8\rforal f\in {\cal F},\,\,1\le j\le 4.
\eneq
We may assume that $h$ vanishes on a neighborhood of $0,$ 
say $h(t)=0$ for all $t\in [0, \sigma]$ for some $0<\sigma<1/2.$ 
Define $g(t)=0$ if $t\in [0, \sigma];$ $g(t)=h^2(t)/t$ if $t\in (\sigma, 1].$ So 
$g\in C_0((0,1]).$
Hence
there is a constant
$L\ge 1$ such that
\beq
h(t)^2\le L t\rforal t\in [0,1].
\eneq
Consequently,
\beq
h(e_A)^2\le L e_A.
\eneq
For $f\in {\cal F}$ and $1\le j\le 4,$ put
\beq\label{Lher2-9}
x_{f,j}=L^{-1} h(e_A)a_{f,j}h(e_A).
\eneq
Then
\beq\label{Lher2-10}
0\le x_{f,j}\le L^{-1} h(e_A)^2\le  e_A\le 1
\eneq
Choose a number $\eta>0$ such that 
\beq
4L\eta<\ep/2.
\eneq
Choose $\dt_0>0$ such that
\beq
2\sqrt{2\dt_0}+\dt_0<\eta.
\eneq

Set 
\beq
{\cal G}=\{e_A\}\cup \{x_{f,j}: f\in {\cal F}, 1\le j\le 4\}.
\eneq
Choose $0<\dt<\dt_0/2.$   

Now suppose that  $D\subset B,$ $H: A\to B$ and $d(x)\in D$ satisfy the hypotheses.
In particular, there are, for each $f\in {\cal F}$ and $1\le j\le 4,$  $d(x_{f,j})\in D$ such that 
\beq\label{Lher2-11}
\|d(x_{f,j})-H(x_{f,j})\|<\dt.
\eneq
Since $H$ is positive,  by \eqref{Lher2-10},
\beq\label{Lher2-11+}
0\le H(x_{f,j})\le H(e_A)\le 1.
\eneq
There are $0\le c_{f,j}\le 1$ in $D$ such that
\beq\label{Lher2-11++}
\|c_{f,j}-H(x_{f,j})\|<2\dt<\dt_0\rforal f\in {\cal F}, 1\le j\le 4 \andeqn \|d-H(e_A)\|<\dt<\dt_0.
\eneq
Now applying Lemma \ref{Lher} to 
$$
a=H(x_{f,j}), \,\, b=H(e_A),\,\, c=c_{f,j}\andeqn e=d,
$$
(with \eqref{Lher2-11+} and  \eqref{Lher2-11++}), we obtain $y_{f,j}\in D_1=\overline{dDd}$ such that
\beq
0\le y_{f,j}\le 1\andeqn \|H(x_{f,j})-y_{f,j}\|<\eta
\eneq
Recall that $h(e_A)a_{f,j}h(e_A)=L x_{f,j}$  (see \eqref{Lher2-9}). Therefore 
$H(h(e_A)a_{f,j}h(e_A))=L H(x_{f,j}),$ whence
\beq\label{Lher2-12}
\|H(h(e_A)a_{f,j}h(e_A))-Ly_{f,j}\|<L\eta.
\eneq
Define 
\beq
c(f)=L(y_{f,1}-y_{f,2})+i L( y_{f,3}-y_{f,4}).
\eneq
Then 
\beq
c(f)\in D_1.
\eneq
Using \eqref{Lher2-4}, \eqref{Lher2-12} and \eqref{Lher2-5}, we have
\beq\nonumber
\|H(f)-c(f)\| &\le &\sum_{j=1}^4 \|H(a_{f,j})-H(Lx_{f,j})\|+\sum_{j=1}^4\|H(h(e_A)a_{f,j}h(e_A))-Ly_{f,j}\|\\
&\le & \sum_{j=1}^4 \|a_{f,j}-h(e_A)a_{f,j}h(e_A)\|+4L \eta\\
&<& 4 {\ep\over{8}}+4 L \eta={\ep\over{2}}+{\ep\over{2}}=\ep.
\eneq

\end{proof}

We omit the proof of the following easy fact.

\begin{lem}\label{LtraceL}
Let $E$ be a unital AF-algebra and let $0\le y\le 1$ be an element  in $E.$ Suppose that
\beq
\inf_{t\in T(E)}t(y)>r.
\eneq
Then, for every $\rho>0,$ there are a finite-dimensional \SCA\, $C\subset E$ and $c\in C_+$ with 
$\|c\|\le 1$ such that
\beq
\|c-y\|<\rho\tand \inf_{s\in T(C)} s(c)>r.
\eneq
\end{lem}

\begin{lem}\label{LAFt}
Let $B$ be a separable simple AF-algebra with non-empty compact $T(B).$ Suppose that 
$0\le b\le 1$ in $B$ and 
\beq
\tau(b)\ge 1-\ep\tforal \tau\in T(B),
\eneq
where $0<\ep<1/2.$

Then, for every $\eta>0$  and  every $\dt>0,$ there are a finite dimensional \SCA\, $C\subset B$ and $c\in C_+$ with 
$\|c\l \le 1$ such that
\beq
\|b-c\|<\eta\tand \inf_{s\in T(C)} s(c)\ge 1-\ep-\dt.
\eneq
\end{lem}

\begin{proof}
Let $\{p_n\}$ be an approximate identity for $B$ consisting of projections. Since 
$T(B)$ is compact, and $\tau(p_n)\ne 1$ for all $\tau\in T(B),$ the Dini Theorem gives 
\beq\label{LAFt-3}
\sup_{\tau\in T(B)} (1-\tau(p_n))\to 0.
\eneq
Choose $0<\sigma<\min\{\varepsilon,\delta\}.$
 Using \eqref{LAFt-3}, one may choose a projection $p$ such that
\beq\label{LAFt-4}
\inf_{\tau\in T(B)} \tau(p)>1-\sigma\andeqn \|b-pbp\|<\eta/2.
\eneq
Since  $B$ is simple and $p\not=0,$ $p$ is full.  Hence every $s\in T(pBp)$ has the form
\beq\label{LAFt-5}
s(x)={\tau(x)\over{\tau(p)}}\rforal x\in pBp,
\eneq
where $\tau\in T(B)$ for some $\tau\in T(B).$ 

Put $y=pbp\in pBp.$  Then for any $\tau\in T(B),$ we estimate that
\beq
\tau(b)-\tau(pbp)=\tau((1-p)b(1-p))\le \tau(1-p)=1-\tau(p).
\eneq
Therefore, by \eqref{LAFt-4}, for all $\tau\in T(B),$ 
\beq
\tau(pbp)\ge \tau(b)-(1-\tau(p))>1-\ep-\sigma.
\eneq
Applying   \eqref{LAFt-5},
\beq
s(y)={\tau(pbp)\over{\tau(p)}}\ge \tau(pbp)>1-\ep-\sigma.
\eneq
Hence
\beq
\inf_{s\in T(pBp)}s(y)>1-\ep-\sigma.
\eneq
Recall that $pBp$ is a unital AF-algebra with unit $p.$  Applying Lemma \ref{LtraceL}
to $E=pBp,$ $y=pbp$ and $r=1-\ep-\sigma,$
we obtain  a unital finite dimensional \SCA\, $C\subset pBp\subset B$ and $c\in C$ with $0\le c\le 1$ such that
\beq\label{LAFt-10}
\|c-pbp\|<\eta/2\andeqn \inf_{s\in T(C)} s(c)\ge 1-\ep-\sigma.
\eneq
Finally, by \eqref{LAFt-4}  and \eqref{LAFt-10},
\beq
\|b-c\|\le \|b-pbp\|+\|pbp-c\|<\eta.
\eneq
\end{proof}

\begin{lem}\label{LmainZ}
Let $A$ and  $C$ be separable amenable ${\cal Z}$-stable \CA s with nonempty compact tracial state spaces, and let $B$ be a separable simple AF-algebra with nonempty compact $T(B).$
Fix any $0<\ep<1/4$ and 
any strictly positive element $e_C$ of $C$ with $\|e_C\|=1$ and $\inf_{t\in T(C)}t(e_C)>1-\ep.$

Suppose 
\beq
\Lambda_B: T(B)\to T(C)\tand \Lambda_A: T(A)\to T(B)
\eneq
are affine homeomorphisms and 
\beq
h: C\to Q_B\tand \phi: B\to Q_A
\eneq
are \hm s satisfying 
\beq\label{LmainZ-03}
h_\sharp=\Lambda_B\circ J_B\andeqn \phi_\sharp=\Lambda_A\circ J_A.
\eneq
Let 
$$
H=\{H_n\}: C\to l^\infty(B)\andeqn \Phi=\{\Phi_n\}: B\to l^\infty(A)
$$
be \cpc s  which lift $h$ and $\phi,$ respectively.
Suppose that 
\beq
{\cal F}_1\subset {\cal F}_2\subset \cdots \subset C
\eneq
is an increasing sequence of finite subsets of the unit ball of $C$ with dense union.

Then there are strictly increasing sequences $l(n)$ and $m(n)$ such that
\beq\label{LmainZ-0-2}
\psi(a)=q_A(\{\Phi_{m(n)}(H_{l(n)}(a))\}_{n=1}^\infty)
\eneq
defines a \hm\, $\psi: C\to Q_A$ with 
\beq\label{LmainZ-0-3}
\psi_\sharp=\Lambda_B\circ \Lambda_A\circ J_A.
\eneq
Moreover, there are finite dimensional \SCA s $D_n\subset B,$ $n\in \N,$ and 
positive contractions $e_n\in A$ satisfying 
\beq\label{LmainZ-0-4}
\{e_n\}\in I_A,
\eneq
and 
a sequence of  order zero \cpc s
\beq
\Theta_n: D_n\to A
\eneq
such that  $\|\Theta|_{D_{n,j}}\|=1$ (for each summand $D_{n,j}$ of $D_n$), for all $a\in {\cal F}_n,$ 
\beq\label{LmainZ-0-5}
&&\hspace{-0.3in}\|[e_n,\, \Phi_{m(n)}(H_{l(n)}(a))]\|<{1\over{n}} \tand \|[(1-e_n)^{1/2},\, \Phi_{m(n)}(H_{l(n)}(a))]\|<{1\over{n}},\tand\\\label{LmainZ-0-6}
&&{\rm dist}((1-e_n)^{1/2}\Phi_{m(n)}(H_{l(n)}(a))(1-e_n)^{1/2},\Theta_n(D_n))<{3\over{n}},
\eneq
 and there is  a strictly positive element $0\le d_{c,n}\le 1$ in $D_n$  such that
\beq\label{LmainZ-0-10}
&&\|(1-e_n)^{1/2} \Phi_{m(n)}(H_{l(n)}(e_C))(1-e_n)^{1/2} -\Theta_n(d_{c,n})\|<{3\over{n}}\tand\\\label{LmainZ-0-11}
&&\tau(d_{c,n})>1-2\ep\,\,\,
%
\tforal \tau\in T(D_n).
\eneq
%
%
Furthermore, suppose that $A$, $B,$  and $C$ are unital, that
$e_C=1_C,$  that the homomorphisms $h$ and $\varphi$ are unital,
and that the lifts $H=\{H_n\}$ and $\Phi=\{\Phi_n\}$ are u.c.p.
Then the algebras $D_n$ may be chosen to be unital subalgebras of
$B$, with $1_{D_n}=1_B$, and $d_{c,n}$ may be chosen to be
$1_{D_n}$ for every $n.$

\end{lem}

\begin{proof}
Let ${\cal F}_1\subset {\cal F}_2\subset \cdots $ be a sequence of self-adjoint finite subsets of the unit ball 
of $C$ with dense union (in the unit ball of $C$). We assume  that 
\beq
e_C\in {\cal F}_1\andeqn \{ab: a, b\in  {\cal F}_n\}\subset {\cal F}_{n+1}.
\eneq
Put $\ep_n=\ep/100n,$ $n\in \N.$

{\bf 1. The choice of $l(n).$}

For each $n,$ define 
\beq
{\cal G}_n={\cal F}_n\cup \{ab: a,b\in {\cal F}_n\}.
\eneq
Since $H$ lifts the \hm\, $h,$ for each pair $a, b\in C,$
$$
\{H_k(ab)-H_k(a)H_k(b)\}_k\in I_B.
$$
Choose strictly increasing integers (applying Lemma \ref{Luniform7} for \eqref{LmainZ-6})
$$
l(1)<l(2)<\cdots 
$$
such that, for all $a, b\in {\cal F}_n,$
\beq\label{LmainZ-5}
&&\|H_{l(n)}(ab)-H_{l(n)}(a)H_{l(n)}(b)\|_{_{2, T(B)}}<\ep_n\andeqn\\
\label{LmainZ-6}
&&\sup_{\tau\in T(B)}|\tau(H_{l(n)}(a))-\Lambda_B(\tau)(a)|<\ep_n.
\eneq
Note that
\beq
\Lambda_B(\tau)(e_C)>1-\ep\rforal \tau\in T(B).
\eneq
Hence 
\beq
\tau(H_{l(n)}(e_C))>1-\ep-\ep/16\rforal \tau\in T(B).
\eneq
{\bf 2. Finite dimensional \SCA s $D_n$ and choice of $d_n(x)$}

For each $n\in \N,$  choose $\dt_n>0$ (as $\dt$) and a finite subset $\tilde{\cal G}_n\subset C$
(as ${\cal G}_n$) associated with $\ep_n$ (as $\ep$) and ${\cal G}_n$ (as ${\cal F}$) given by Lemma 
\ref{Lher2}.

Since $B$ is AF, one may choose inductively  increasing finite dimensional \SCA s
$$
D_1'\subset D_2'\subset \cdots \subset B
$$
such that, for all $a\in \tilde{\cal G}_n,$
$$
{\rm dist}(H_{l(n)}(a),\, D_n')<\dt_n.
$$

Let us first consider the last "Moreover" part of the lemma. 
Suppose first that the additional unital hypotheses in the statement
hold.  Choose each $D_n'$ to be unital, with
$1_{D_n'}=1_B.$  Since $e_C=1_C$ and $H_{l(n)}$ is unital, put
\[
D_n=D_n'
\quad\text{and}\quad
d_n=d_n(e_C)=1_{D_n}.
\]
In particular,
\[
s(d_n)=1
\quad\text{for every }s\in T(D_n).
\]

In general case, choose  $d_n(e_C)\in D_n'$ such that
$$
\|H_{l(n)}(e_C)-d_n(e_C)\|<\dt_n.
$$
\Wlog, we  may assume that there is, for each $n\in\N,$ $d_n=d_n(e_C)\in D_n'$ with $0\le d_n\le 1.$
Put $D_n =\overline{d_nD_n'd_n},$ $n\in \N.$
By applying Lemma \ref{LAFt} with $17\ep/16$ (in place of $\ep$),  we may assume that
\beq\label{LmainZ-n7}
\inf_{s\in T(D_n)} s(d_n)>1-{18\ep\over{16}}=1-{9\ep\over{8}}.
\eneq
%
Thus \eqref{LmainZ-n7} holds in both cases, and in the additional
unital case one has the stronger equality
\[
d_n=1_{D_n}.
\]

By applying Lemma \ref{Lher2}, we may choose 
$d_n(a)\in D_n$ with
\beq\label{LmainZ-9}
\|H_{l(n)}(a)-d_n(a)\|<\ep_n\rforal a\in {\cal G}_n.
\eneq
Note that $\|d_n(a)\|\le 1+1/n,$ $n\in \N.$
For $a, b\in {\cal F}_n,$ 
\beq\nonumber
\|d_n(ab)-d_n(a)d_n(b)\|_{_{2, T(B)}}&\le& \|d_n(ab)-H_{l(n)}(ab)\|+\|H_{l(n)}(ab)-H_{l(n)}(a)H_{l(n)}(b)\|_{_{2, T(B)}}\\\label{LmainZ-10}
      &&+\|H_{l(n)}(a)H_{l(n)}(b)-d_n(a)d_n(b)\|_{_{2, T(B)}} .
      \eneq
      Note that  (by also \eqref{LmainZ-9})
      \beq\nonumber
      &&\hspace{-0.4in}\|H_{l(n)}(a)H_{l(n)}(b)-d_n(a)d_n(b)\|_{_{2, T(B)}}\\
      &&\le \|H_{l(n)}(a)-d_n(a)\|\|H_{l(n)}(b)\|+\|d_n(a)\|\|H_{l(n)}(b)-d_n(b)\|\\
      &&<{1\over{n}}+(1+{1\over{n}}){1\over{n}}={2\over{n}}+{1\over{n^2}}.\label{LmainZ-12}
      \eneq
      Put 
      \beq
      w_n(a,b)=d_n(ab)-d_n(a)d_n(b).
      \eneq
      Then, by \eqref{LmainZ-5}, \eqref{LmainZ-9}, \eqref{LmainZ-10} and \eqref{LmainZ-12},
      \beq\label{LmainZ-s2-20}
      \max_{a,b\in {\cal F}_n}\|w_n(a,b)\|_{_{2, T(B)}}<2\ep_n+{2\over{n}}+{1\over{n^2}}\to 0.
      \eneq

   Write 
   \beq
   D_n=\bigoplus_{i=1}^{r(n)} D_{n,i},
   \eneq
   where each $D_{n,i}$ is a full matrix algebra, and write $p_{n,i}=1_{D_{n,i}}.$ 

{\bf 3. Cone lifts  associated with $D_n$}

For each $n\in \N,$ define  a \hm\,  ${\widehat{\phi_n}}: C_0((0,1])\otimes D_n\to Q_A$ by
\beq
\widehat{\phi_n}(\iota\otimes d)=\phi(d)\rforal d\in D_n.
\eneq
Since the cone $C_0((0,1])\otimes D_n$ is projective, choose a \hm\, 
\beq
\tilde \phi_n: C_0((0,1])\otimes D_n\to l^\infty(A)
\eneq
such that
\beq
q_A\circ \tilde \phi_n=\widehat{\phi_n},\,\,\,\,\, n\in \N.
\eneq
Write 
\beq
\tilde \phi_n(x)=\{\gamma_{n,k}(x)\}_{k=1}^\infty.
\eneq
Here, for each $n, k\in \N,$
\beq
\gamma_{n,k}: C_0((0,1])\otimes D_n\to A
\eneq
is a \hm\, and,  for all $a\in {\cal F}_n,$
\beq\label{LmainZ-16}
q_A(\{\gamma_{n,k}(\iota\otimes d_n(a))\})=q_A\circ \tilde \phi_n(\iota\otimes d_n(a))=\widehat{\phi}(\iota\otimes d_n(a))
=\phi(d_n(a)).
\eneq
For each $a\in {\cal F}_n,$ define
\beq
Y_{n,a}=\{\Phi_k(d_n(a))-\gamma_{n,k}(\iota\otimes d_n(a))\}_{k=1}^\infty.
\eneq
Note that (see \eqref{LmainZ-16}), for $a\in {\cal F}_n,$
\beq
q_A(\{\Phi_k(d_n(a))\}_k)=\phi(d_n(a))=q_A(\{\gamma_{n,k}(\iota\otimes d_n(a))\}_k).
\eneq
Hence, for all $a\in {\cal F}_n,$   $q_A(Y_{n,a})=0.$ In other words
\beq\label{LmainZ-s3-5}
Y_{n,a}\in I_A\rforal a\in {\cal F}_n
\eneq

{\bf 4.   Trace  on $r_{n,i,k}$}

\vspace{0.1in}

Choose, for each $n\in \N,$ a finite subset  
$$S_n=\{d_n(a): a\in {\cal F}_n\}
\cup\{p_{n,1},p_{n,2},...,p_{n,r(n)}\}\subset D_n
$$ and 
$0<\eta_n<\ep_n.$ 

Choose a finite subset ${\cal E}_n\subset D_n$ (as  ${\cal E}$) and $\dt_n$ (as $\dt$)
associated with $\eta_n$ (in place of $\ep$) and $S_n$ (in place of ${\cal F}$) given by Lemma \ref{Lstable2}.



 For each summand unit $p_{n,i},$ put
 $$
 z_{n,i,k}=\gamma_{n,k}(\iota\otimes p_{n,i})\in A_+.
 $$
 Since 
 \beq
 q_A(\{z_{n,i,k}\}_k)=\phi(p_{n,i})
 \eneq
and $\phi(p_{n,i})$ is a projection in $Q_A,$ we have 
\beq\label{LmainZ-n13}
\{z_{n,i,k}-z_{n,i,k}^2\}_k\in I_A.
\eneq
Since $B$ is simple, $p_{n,i}\not=0$  and $T(B)$ is compact, we have
\beq
\mu_n=\min_{1\le i\le r(n)} \inf_{\tau\in T(B)} \tau(p_{n,i})>0.
\eneq
By  \eqref{LmainZ-03}, for every $t\in T(Q_A),$
\beq
t(\phi(p_{n,i}))=\Lambda_A(J_A(t))(p_{n,i})\ge \mu_n.
\eneq
Choose  a number $0<\af_n=\min\{1/2n, \ep_n\}.$  Then $\af_n\to 0.$
Fix $n,$   choose an integer $N_n'$ such that for all $k\ge N_n',$
\beq\label{LmainZ-n16}
\inf_{t\in T(A)}t(z_{n,i,k})\ge {\mu_n\over{2}}, \,\, 1\le i\le r(n).
\eneq

Now, by \eqref{LmainZ-n13},  choose $N_n''\ge N_n'$ such that, for all $k\ge N_n''$ and $1\le i\le r(n);$ 
\beq\label{LmainZ-n17a}
\|z_{n,i,k}-z_{n,i,k}^2\|_{_{2, T(A)}}<{\af_n\mu_n\over{64}}.
\eneq
For $k\ge N_n'',$ define 
the  spectral projection 
$$
r_{n,i,k}=\chi_{[1-\af_n/8, 1]}(z_{n,i,k})\in A^{**}.
$$
Let 
$q_{n,i,k}=1-r_{n,i,k}.$
On the spectrum of $q_{n,i,k}z_{n,i,k}q_{n,i,k},$ we have 
\beq
0\le t\le 1-{\af_n\over{8}}.
\eneq
Therefore, for $t\in {\rm sp} (q_{n,i,k}z_{n,i,k}q_{n,i,k}),$ 
\beq
t-t^2=t(1-t)\ge {\af_n\over{8}}t.
\eneq
By functional calculus, 
\beq
q_{n,i,k}(z_{n,i,k}-z_{n,i,k}^2)q_{n,i,k}\ge {\af_n\over{8}}q_{n,i,k}z_{n,i,k}q_{n,i,k}.
\eneq 
Let $\tau\in T(A),$ extended normally to $A^{**}.$ Then 
\beq
{\af_n\over{8}}\tau(q_{n,i,k}z_{n,i,k}q_{n,i,k})\le \tau(z_{n,i,k}-z_{n,i,k}^2).
\eneq
Since $0\le z_{n,i,k}-z_{n,i,k}^2\le 1,$  by Cauchy-Bounyakowski- Schwartz inequality,
\beq
\tau(z_{n,i,k}-z_{n,i,k}^2)\le \|z_{n,i,k}-z_{n,i,k}^2\|_{_{2, T(A)}}.
\eneq
Applying \eqref{LmainZ-n17a}, we obtain 
\beq\label{LmainZ-n17c}
\tau(q_{n,i,k}z_{n,i,k}q_{n,i,k})<{8\over{\af_n}}\cdot \tau(z_{n,i,k}-z_{n,i,k}^2)\le {8\over{\af_n}}\cdot {\af_n\mu_n\over{64}}={\mu_n\over{8}}.
\eneq
Now, since $r_{n,i,k}$ commutes with $z_{n,i,k},$ we have
$$
\tau(z_{n,i,k})=\tau(r_{n,i,k}z_{n,i,k}r_{n,i,k})+\tau(q_{n,i,k}z_{n,i,k}q_{n,i,k}),
$$
Therefore, by  \eqref{LmainZ-n16} and \eqref{LmainZ-n17c}, for all $k\ge N_n'',$ 
\beq\nonumber
\tau(r_{n,i,k}) &\ge &\tau(r_{n,i,k}z_{n,i,k}r_{n,i,k})=\tau(z_{n,i,k})-\tau(q_{n,i,k}z_{n,i,k}q_{n,i,k})\\\label{LmainZ-n17}
&>&{\mu_n\over{2}}-{\mu_n\over{8}}={3\mu_n\over{8}}.
\eneq

{\bf 5. Integer $N_n$ and application of Lemma \ref{Luniform7}}

Recall now that $l(n), D_n, d_n(x),w_n(a,b), \gamma_{n,k},$ ${\cal E}_n,$ and $\dt_n$ have been fixed (for each $n\in \N$). 

For each fixed $n,$ 
one obtains an integer $N_n$ such that, whenever $k\ge N_n\ge N_n'',$ the following three conditions hold.

First, for all $a, b\in {\cal F}_n,$ 
\beq\label{LmainZ-s5-1}
\|\Phi_k(d_n(a)d_n(b))-\Phi_k(d_n(a))\Phi_k(d_n(b))\|_{_{2, T(A)}}<{1\over{n}}
\eneq
(this because $\{\Phi_k(xy)-\Phi_k(x)\Phi_k(y)\}_k\in I_A$).

Second (applying Lemma \ref{Luniform7}), for all $a, b\in {\cal F}_n,$
\beq\label{LmainZ-s5-10}
\sup_{t\in T(A)}|t(\Phi_k(w_n(a,b)^*w_n(a,b)))-\Lambda_A(t)(w_n(a,b)^*w_n(a,b))|<{1\over{n}}.
\eneq
Third (also by  Lemma \ref{Luniform7}),  for all $a\in {\cal F}_n,$
\beq\label{LmainZ-s5-3}
\sup_{t\in T(A)}|t(\Phi_k(H_{l(n)}(a)))-\Lambda_A(t)(H_{l(n)}(a))|<\min\{{1\over{n}}, \ep/32\}.
\eneq
Note that 
\beq
\Lambda_A(t)(H_{l(n)}(e_C))>1-\ep-\ep/16\rforal t\in T(A).
\eneq
It follows that 
\beq
\inf_{t\in T(A)}t(\Phi_k(H_{l(n)}(e_C)))>1-\ep-3\ep/32.
\eneq
Similarly,
\beq
\inf_{t\in T(A)}t(\Phi_k(d_n))>1-\ep-3\ep/32.
\eneq
{\bf 6.   Selection of $\{m(n)\}$}

For $n\in \N,$ $a\in {\cal F}_n$ and $x\in {\cal E}_n,$ define
\beq
U_{n,a}&=&\{\Phi_k(H_{l(n)}(a))\}_k\in l^\infty(A),\andeqn\\
V_{n,x}&=&\{\gamma_{n,k}(\iota\otimes x)\}_k\in l^\infty(A).
\eneq
Let $S\subset l^\infty(A)$ be the separable \CA\, generated by elements 
$$
\{Y_{n,a}, U_{n,a}, V_{n,x}: a\in {\cal F}_n, x\in {\cal E}_n\},\,\,\, n\in \N.
$$
Put $J=S\cap I_A.$  For each $n\in \N,$ define 
\beq
&&{\cal X}_n=\{Y_{n,a}: a\in {\cal F}_n\}\subset J,\andeqn\\ 
&&{\cal Y}_n=\{U_{n,a}: a\in{\cal F}_n\}\cup \{V_{n,x}: x\in {\cal E}_n\}\subset S.
\eneq
Put 
\[
 \rho_n:=\min\{\ep_n,\, {\af_n\mu_n\over{100}}\}\andeqn \bt_n=\min\{\ep_n, \dt_n\}
\]
Applying Lemma \ref{Lidealsep} with  $\rho_n$ (in place of $\af_n$ in that lemma), with 
$\bt_n,$ and with the integers $N_n$ chosen 
in stage 5.  So we obtain strictly increasing sequences $s(n),$ $m(n),$ and positive contractions
$e_n=f_{s(n), m(n)}$ such that
\beq\label{LmainZ-s6-3}
&&m(n)\ge N_n, \,\,\, \|e_n\|_{_{2, T(A)}}<\rho_n,
\\\label{LmainZ-s6-4}
&&\hspace{-0.7in}\|(1-e_n)^{1/2}[\Phi_{m(n)}(d_n(a))-\gamma_{n,m(n)}(\iota\otimes d_n(a))](1-e_n)^{1/2}\|<
\rho_n\rforal a\in {\cal F}_n,\\
&&  {\rm{( recall}}\,\,\eqref{LmainZ-s3-5}{\rm )}\\\label{LmainZ-s6-5}
&&\hspace{-0.7in}\|[e_n,\,\Phi_{m(n)}(H_{l(n)}(a))]\|<\ep_n\andeqn \|[(1-e_n)^{1/2}, \, \Phi_{m(n)}(H_{l(n)}(a))]\|<\rho_n
\\
&&\rforal a\in {\cal F}_n,\andeqn\\\label{LmanZ-s6-8}
&&\hspace{-0.7in}\|[e_n,\, \gamma_{n, m(n)}(\iota\otimes x)]\|<\dt_n,\,\,\|[(1-e_n)^{1/2},\, \gamma_{n,m(n)}(\iota\otimes x)]\|<\dt_n \rforal x\in {\cal E}_n.
\eneq
By \eqref{LmainZ-s6-3}, we also have 
\beq
\{e_n\}\in I_A.
\eneq
This shows \eqref{LmainZ-0-4}. 

Define $\Gamma_n=\gamma_{n,m(n)}: C_0((0,1])\otimes D_n\to A.$ 
We have
\beq\label{LmainZ-n27}
\|[(1-e_n)^{1/2},\, \Gamma_n(\iota\otimes x)]\|<\dt_n \rforal x\in {\cal E}_n.
\eneq

{\bf 7.  Norm preserving}

Fix $n$ and $i.$ Let
$$z:=z_{n,i,m(n)}=\Gamma_n(\iota\otimes p_{n,i})\andeqn r:=\chi_{[1-\af_n/8,1]}(z)\in A^{**}.$$
Hence (also by \eqref{LmainZ-n17})
\beq\label{LmainZ-n17n}
z\ge (1-{\af_n\over{8}})r\andeqn \tau(r)>{\mu_n\over{4}}\rforal \tau\in T(A).
\eneq
We claim that 
\beq\label{LmainZ-n30}
\|(1-e_n)^{1/2}r(1-e_n)^{1/2}\|>1-{\af_n\over{2}}.
\eneq
If not, since $r$ is a projection,
\beq
\|r(1-e_n)r\|=\|(1-e_n)^{1/2}r(1-e_n)^{1/2}\|\le 1-{\af_n\over{2}}.
\eneq
Then
\beq
0\le r(1-e_n)r\le (1-{\af_n\over{2}})r.
\eneq
Therefore
\beq
re_nr=r-r(1-e_n)r\ge {\af_n\over{2}}r.
\eneq
Then, by \eqref{LmainZ-n17n}, for any $\tau\in T(A),$ 
\beq
{\af_n \mu_n\over{8}}<{\af_n\over{2}}\tau(r)\le \tau(re_nr)=\tau(e_n^{1/2}re_n^{1/2})\le \tau(e_n)\le \|e_n\|_{_{2, T(A)}}<\rho_n<{\af_n \mu_n\over{100}}.
\eneq
This contradiction  proves the claim.

From \eqref{LmainZ-n17n} and   the claim,  we estimate that
\beq\nonumber
\|(1-e_n)^{1/2}z(1-e_n)^{1/2}\| &\ge& (1-{\af_n\over{8}})\|(1-e_n)^{1/2}r(1-e_n)^{1/2}\|\\\label{LmainZ-n31}
&>& (1-{\af_n\over{8}})(1-{\af_n\over{2}})>1-\af_n.
\eneq
Thus
\beq\label{LmainZ-n32}
\|(1-e_n)^{1/2} \Gamma_n(\iota\otimes p_{n,i})(1-e_n)^{1/2}\|\ge 1-\af_n,\,\,1\le i\le r(n).
\eneq

{\bf 8.   Construction of $\Theta$}

Set 
\beq
s_n:=1-e_n,\,\,n\in \N.
\eneq
By \eqref{LmanZ-s6-8} and \eqref{LmainZ-n27}
\beq
\|[s_n,\, \Gamma_n(\iota\otimes x)]\|<\dt_n\andeqn \|[s_n^{1/2},\, \Gamma_n(\iota\otimes x)]\|<\dt_n\rforal x\in {\cal E}_n.
\eneq
By the choice of ${\cal E}_n$ and $\dt_n,$ applying Lemma \ref{Lstable2}, 
we obtain an order zero map $\Theta_n^{(0)}: D_n\to \overline{(1-e_n) A(1-e_n)}$  such that 
\beq\label{LmainZ-s9-3}
\|\Theta_n^{(0)}(d)-s_n^{1/2}\Gamma_n(\iota\otimes d)s_n^{1/2}\|<\eta_n<{1\over{n}}\rforal d\in {\cal E}_n.
\eneq
For each summand $D_{n,i},$
$$
|\Theta_n^{(0)}|_{D_{n,i}}\|=\|\Theta_n^{(0)}(p_{n,i})\|.
$$
By \eqref{LmainZ-n32} and \eqref{LmainZ-s9-3},  we get
\beq\label{LmainZ-n34}
\|\Theta_n^{(0)}(p_{n,i})\|>1-\af_n-\ep_n.
\eneq
Set $\lambda_{n,i}=\|\Theta_n^{(0)}(p_{n,i})\|.$
Then $0\le \lambda_{n,i}\le 1$ and 
\beq
\lim_{n\to\infty}\min_{1\le i\le r(n)}\lambda_{n,i}=1.
\eneq
Define  
\beq\label{LmainZ-n36}
\Theta_n|_{D_{n,i}}=\lambda_{n,i}^{-1}\, \Theta_n^{(0)}|_{D_{n,i}}.
\eneq
Then $\Theta_n: D_n\to A$ is an order zero \cpc, and $\|\Theta_n|_{D_{n,i}}\|=1,$ $1\le i\le r(n).$

For each $d\in D_n^{\bf 1},$
\beq\label{LmainZ-n38}
\|\Theta_n(d)-\Theta_n^{(0)}(d)\|\le \max_{1\le i\le r(n)}(1-\lambda_{n,i}).
\eneq
Therefore, for all $d\in {\cal S}_n,$
\beq\label{LmainZ-n39}
\|\Theta_n(d)-s_n^{1/2}\Gamma_n(\iota\otimes d)s_n^{1/2}\|<{1\over{n}}.
\eneq

{\bf 9. $\psi$  a \hm}

Define $\Psi_n=\Phi_{m(n)}\circ H_{l(n)}: C\to A$ and  $\psi:
C\to Q_A$ by $\psi(y) =q_A(\{\Psi_n(y)\}_n)$
for all $y\in C.$

Let $a, b\in \cup_{k=1}^\infty {\cal F}_k.$ We may assume that,
for sufficiently large $n,$  $a, b\in {\cal F}_n.$ Then, by \eqref{LmainZ-9},
\beq\label{LmainZ-s7-1}
\|\Psi_n(a)-\Phi_{m(n)}(d_n(a))\|<\ep_n \andeqn \|\Psi_n(ab)-\Phi_{m(n)}(d_n(ab))\|<\ep_n.
\eneq
%

Recall $w_n(a, b)=d_n(ab)-d_n(a)d_n(b)$ for $a, b\in {\cal F}_n.$ By Schwarz inequality (see, e.g., \cite[Cor. 4.1.3]{BK1}),
\beq
\Phi_{m(n)}(w_n(a,b))^*\Phi_{m(n)}(w_n(a,b))\le \Phi_{m(n)}(w_n(a, b)^*w_n(a,b)).
\eneq
Since $m(n)\ge N_n,$ \eqref{LmainZ-s5-10} gives
\beq
\|\Phi_{m(n)}(w_n(a,b))\|_{_{2, T(A)}}^2&\le& \sup_{t\in T(A)}t(\Phi_{m(n)}(w_n(a,b)^*w_n(a,b)))\\
   &\le & \sup_{t\in T(A)}\Lambda_A(t)(w_n(a, b)^*w_n(a,b))+{1\over{n}}\\
   &\le &\|w_n(a, b)\|_{_{2, T(B)}}^2+{1\over{n}}.
\eneq
Then, by \eqref{LmainZ-s2-20},
\beq\label{LmainZ-s7-5}
\|\Phi_{m(n)}(w_n(a,b))\|_{_{2, T(A)}}\to 0\,\,\, {\rm(as}\,\, n\to \infty{\rm )}.
\eneq
Also, since $m(n)\ge N_n,$  by \eqref{LmainZ-s5-1}, one has 
\beq\label{LmainZ-s7-6}
\|\Phi_{m(n)}(d_n(a)d_n(b))-\Phi_{m(n)}(d_n(a))\Phi_{m(n)}(d_n(b))\|_{_{2, T(A)}}<{1\over{n}}
\eneq
Since 
\beq
d_n(ab)=w_n(a,b)+d_n(a)d_n(b),
\eneq
combining \eqref{LmainZ-s7-5} and \eqref{LmainZ-s7-6}, one obtains 
\beq
\|\Phi_{m(n)}(d_n(ab))-\Phi_{m(n)}(d_n(a))\Phi_{m(n)}(d_n(b))\|_{_{2, T(A)}}\to 0.
\eneq
Applying \eqref{LmainZ-s7-1}, one obtains, for all $a, b\in \cup_{k=1}^\infty {\cal F}_k,$
\beq\label{LmainZ-s7-10}
\lim_{n\to\infty}\|\Psi_n(ab)-\Psi_n(a)\Psi_n(b)\|_{_{2, T(A)}}=0.
\eneq
Since $\cup_{k=1}^\infty {\cal F}_k$ is dense in the unit ball of $C,$ we conclude that
\eqref{LmainZ-s7-10} holds for all $a, b\in C.$ Since each  $\Psi_n$ is \cpc, it follows that
$\psi=q_A(\{\Psi_n\})$ is a \hm.  This also proves \eqref{LmainZ-0-2}.

{\bf 10. The trace map induced by $\psi$}

For each $a\in {\cal F}_n,$ since $m(n)\ge N_n,$ by \eqref{LmainZ-s5-3}, we have
that
\beq
\sup_{t\in T(A)}|t(\Psi_n(a))-\Lambda_A(t)(H_{l(n)}(a))|<\ep_n.
\eneq
From \eqref{LmainZ-6},  since $\Lambda_A(t)\in T(B),$ we have that
\beq
\sup_{t\in T(A)}|\Lambda_A(t)(H_{l(n)}(a))-\Lambda_B(\Lambda_A(t))(a)|<\ep_n\rforal a\in {\cal F}_n.
\eneq
It follows that
\beq
\sup_{t\in T(A)} |t(\Psi_n(a))-\Lambda_B(\Lambda_A(t))(a)|<2\ep_n\rforal a\in {\cal F}_n.
\eneq 
Hence, by Lemma \ref{Leasylift},  every $\tau\in T_\infty(A)$ satisfies 
\beq
\tau(\psi(a))=\Lambda_B\circ \Lambda_A\circ J_A(\tau)(a)\rforal a\in \cup_{k=1}^\infty {\cal F}_k.
\eneq
It follows that $\psi_\sharp$  and $\Lambda_B\circ \Lambda_A\circ J_A$ agree on $T_\infty(A).$ 
By \cite[Proposition 2.5]{CETW} (together with \cite[Theorem 1.4]{CETW}),
\[
{\overline{co}}^{w^*}(T_\infty(A))=T(Q_A). 
\]
Therefore
\beq
\psi_\sharp=\Lambda_B\circ \Lambda_A\circ J_A.
\eneq
This proves  \eqref{LmainZ-0-3}.

{\bf 11. The last adjustment}
%

By \eqref{LmainZ-9},
\beq\label{LmainZ-s9-4}
\|\Psi_n(a)-\Phi_{m(n)}(d_n(a))\|<\ep_n.
\eneq
Then, by \eqref{LmainZ-s6-4}, we have 
\beq\label{LmainZ-s9-5}
\|s_n^{1/2}(\Phi_{m(n)}(d_n(a))-\Gamma_n(\iota\otimes d_n(a)))s_n^{1/2}\|<\ep_n\rforal a\in {\cal F}_n.
\eneq
Combining  \eqref{LmainZ-s9-4}, \eqref{LmainZ-s9-5}, and \eqref{LmainZ-n39},
 we estimate that, for all $a\in {\cal F}_n,$
\beq\nonumber
\hspace{-0.6in}\|s_n^{1/2}\Psi_n(a)s_n^{1/2}-\Theta_n(d_n(a))\| &\le & \|s_n^{1/2}(\Psi_n(a)-\Phi_{m(n)}(d_n(a)))s_n^{1/2}\|\\\nonumber
       &&+\|s_n^{1/2}(\Phi_{m(n)}(d_n(a))-\Gamma_n(\iota\otimes d_n(a)))s_n^{1/2}\|\\\nonumber
      && \hspace{0.2in}+\|s_n^{1/2}\Gamma_n(\iota\otimes d_n(a))s_n^{1/2}-\Theta_n(d_n(a))\|\\\label{LmainZ-n43}
       &<&\ep_n+\ep_n+{1\over{n}}.
\eneq
Therefore, for all $a\in {\cal F}_n,$
\beq
{\rm dist}((1-e_n)^{1/2}\Psi_n(a)(1-e_n)^{1/2},\, \Theta_n( D_n))<{3\over{n}}.
\eneq
Together with \eqref{LmainZ-s6-5},  this provides \eqref{LmainZ-0-5} and \eqref{LmainZ-0-6} in  the lemma.

Finally, since $\{e_n\}_n\in I_A,$
we also have 
\beq
q_A(\{(1-e_n)^{1/2}\Psi_n(c)(1-e_n)^{1/2}\}_n)=q_A(\{\Psi_n(c)\})=\psi(c)\rforal c\in C.
\eneq

Define 
\beq
d_{c,n}=d_n(e_C)\in D_n.
\eneq
Then $0\le d_{c,n}\le 1.$ 
Hence, by \eqref{LmainZ-n7},
\beq
s(d_{c,n})=s(d_n(e_C))>1-{9\ep\over{8}}\rforal s\in T(D_n).
\eneq
Applying \eqref{LmainZ-n43} with $a=e_C,$ we conclude that
\beq
\|(1-e_n)^{1/2}\Psi_n(e_C)(1-e_n)^{1/2}-\Theta_n(d_{c,n})\|<{3\over{n}}.
\eneq
Thus $d_{c,n}$ has the desired properties (\eqref{LmainZ-0-10} and \eqref{LmainZ-0-11}).
In the additional unital case, the choice made in Stage~2 gives
\[
d_{c,n}=d_n(e_C)=1_{D_n}.
\]
Consequently,
\[
\Theta_n(1_{D_n}-d_{c,n})=0.
\]
This proves the additional assertion.
This completes the proof  of the lemma.

%
\end{proof}

\begin{thm}\label{TmainZ}
Let $A$ be a separable amenable finite simple ${\cal Z}$-stable \CA.
Then $ETR(A)\le 1.$     

 Moreover, if \(A\) is unital, then \(A\) has property ETR1
associated with \(1_A\) and the constant
\(\kappa_{1_A}=1\). In this case, the distinguished element \(d_e\)
in Definition~3.3 may always be chosen to be \(1_D\).
\end{thm}

\begin{proof}
Since $A$ is assumed to  be finite and ${\cal Z}$-stable, the rank map 
from ${\rm Cu}(A)$ to ${\rm LAff}_0(\tilde T(A))$ is surjective.
Choose $a\in A\otimes {\cal K}$ such that $\widehat{[a]}$ is continuous.
Then $A_0=\Her(a)$ has continuous scale. By \cite[Theorem 5.3]{EGLNs}, $T(A_0)$ is compact.
Then $T(A_0)$ is a metrizable separable Choquet simplex. 
Note that $A_0$ is also ${\cal Z}$-stable (see \cite[Corollary 3.1]{TW07}). By Definition \ref{DD1},
it suffices to show that $A_0$ has property ETR1.

Fix \(0<\eta<1/2\). This number will remain fixed throughout the
proof and will determine the constant
\[
        \kappa_{e_A}=1-\eta.
\]

 Choose a unital infinite dimensional simple AF-algebra $B$ with $T(B)=T(A_0)$ (see \cite[Theorem 3.10]{Btrace}).
Let $\Lambda_A: T(A_0)\to T(B)$ be the identification (an affine homeomorphism) and let
$\Lambda_B=\Lambda_A^{-1}: T(B)\to T(A_0).$
Let $e\in A_0$ be a strictly positive element with $\|e\|=1.$  Then $d_\tau(e)=1$ for every $\tau\in T(A_0).$
Since 
\beq
1=d_\tau(e)=\lim_{n\to\infty} \tau(\bar{f}_{1/n}(e)),
\eneq
by the Dini Theorem, there is $m_0\in \N$ such that, for all $m\ge m_0,$ 
\beq
\tau({\bar f}_{1/m}(e))>1-\eta/2\rforal \tau\in T(A_0).
\eneq
Note that $\|{\bar f}_{1/m}(e)\|=1.$  Fix $m\ge m_0,$ and put 
$e_A={\bar f}_{1/m}(e)$ and put $\kappa=1-\eta.$ 

Now let $0<\ep<1/4,$  let ${\cal F}\subset A_0^{\bf 1}$  be a finite subset 
containing $e_A$ and a nonzero 
element $a_0\in (A_0)_+^{\bf 1}.$   

We note that $B$ is also ${\cal Z}$-stable (see, for example,  \cite[Theorem 5(a)]{JS}). 
Then, by \cite[Theorem 2.6]{CETW}\footnote{as corrected in the erratum--M${\rm \ddot{u}}$nster J. Math., {\bf 16} (2023), 489--490.}, there are  \hm s $h: A_0\to Q_B$ 
and $\phi: B\to Q_{A_0}$ 
such that  (recall that $\Lambda_B=\Lambda_A^{-1}$)
\beq
h_\sharp=\Lambda_B\circ J_T^B\andeqn \phi_\sharp=\Lambda_{A}\circ J_T^{A_0}.
\eneq

Let ${\cal F}_1={\cal F}\subset {\cal F}_2\subset \cdots \subset A_0^{\bf 1}$ such that 
$\cup_{n=1}^\infty {\cal F}_n$ is dense in unit ball of $A_0.$  We may also assume that
${\cal F}\cup \{e_A\}\subset  {\cal F}_1.$

Apply Lemma 10.8 with
\[
        A=C=A_0,\qquad e_C=e_A,\qquad
        \Lambda_B=\Lambda_A^{-1},
\]
and with \(\eta/2\) in place of the parameter \(\varepsilon\)
occurring in that lemma. We obtain strictly increasing sequences
\(\{l(n)\}\) and \(\{m(n)\}\), a homomorphism
\(\psi\colon A_0\to Q_{A_0}\), finite-dimensional
\(C^*\)-subalgebras \(D_n\subset B\), positive contractions
\(e_n\in A_0\) with \(\{e_n\}\in I_{A_0}\), fully normalized
order-zero c.p.c.\ maps
\[
        \Theta_n\colon D_n\longrightarrow A_0,
\]
and positive contractions \(d_{c,n}\in D_n\) satisfying the
conclusions of Lemma 10.8. In particular,
\beq\label{T109-0-10}
 \left\|
 (1-e_n)^{1/2}\Psi_n(e_A)(1-e_n)^{1/2}
 -\Theta_n(d_{c,n})
 \right\|<\frac{3}{n},
\eneq
where
$
        \Psi_n=\Phi_{m(n)}\circ H_{l(n)}
        \colon A_0\longrightarrow A_0,
$
and
\beq\label{T109-011}
        t(d_{c,n})>1-\eta
        \qquad\text{for all }t\in T(D_n).
\eneq

By \eqref{LmainZ-0-3}, 
\beq\label{n10201}
 \psi_\#
 =\Lambda_B\circ\Lambda_A\circ J_T^{A_0}
 =J_T^{A_0}
 =(j_{A_0})_\#.
\eneq
It follows from \cite[Theorem 2.2]{CETW} that 
there exists a unitary $u\in Q_{A_0}$ such that
\beq
u^*\psi(a)u=j_{A_0}(a)\rforal a\in A_0.
\eneq
(We note that since $A_0$ has continuous scale, $Q_{A_0}$ is unital (see \cite[Proposition 2.20]{FLosc} and \cite[Corollary 4.7]{FLL})).
By taking $S=\C\cdot 1_{Q_{A_0}},$ and applying \cite[Proposition 2.1]{CETW}, 
we may write $u=\exp(\pi i b)$ for some $b\in (Q_{A_0})_{s.a.}$ with $\|b\|\le 1.$ 
Let $h_n\in (A_0)_{s.a.}$ with $\|h_n\|\le 1$ such that 
$q_A(\{h_n\})=b.$ Put $u_n=\exp(\pi i h_n),\,\,n\in \N.$ Then $u_n\in U(\tilde A_0).$ 

We will keep the notation used in the (proof of) Lemma \ref{LmainZ}.

For $a\in \cup_{n=1}{\cal F}_n,$ we have 
\beq
Y(a)=\{u_n^*\Psi_n(a)u_n-a\}_n\in I_{A_0}.
\eneq
Let
\[
 \bar e_n=u_n^*e_nu_n,\qquad
 q_n=1-\bar e_n=u_n^*(1-e_n)u_n,
\]
and define
\[
        \bar\Theta_n(d)=u_n^*\Theta_n(d)u_n,
        \qquad d\in D_n.
\]
Then the conjugated forms of \eqref{LmainZ-0-5} -- \eqref{LmainZ-0-11}
hold for
\(\bar e_n\), \(\bar\Theta_n\), and
\(a\mapsto u_n^*\Psi_n(a)u_n\). Moreover,
\[
        \{\bar e_n\}\in I_{A_0}.
\]

Write $D_n=\bigoplus_{i=1}^{r(n)} D_{n,i}, \,\,\, p_{n,i}=1_{D_{n,i}}.$ 

Set $a_{n,i}:={\bar \Theta}_n(p_{n,i}).$ Then since $\Theta_n$ are fully normalized, $\|a_{n,i}\|=1.$
Choose $\sigma_n=1/(n+1).$ 
Choose \(h_n\in C_0((0,1])_+\) with \(0\leq h_n\leq1\) such that
\[
h_n(t)=0
\quad\text{for }0\leq t\leq 1-\frac{\sigma_n}{8},
\qquad
h_n(1)=1.
\]

Let 
\beq\label{Tm-13}
R_{n,i}=\chi_{[1-\sigma_n/8, 1]}(a_{n,i})\in A_0^{**}.
\eneq
Since $\|a_{n,i}\|=1,$ $h_n(a_{n,i})\not=0$ which has support in $[1-\sigma_n/8, 1]$  (see Definition \ref{dfg} for $g_{\xi}$). Since $A_0$ is simple, $\tau(h_n(a_{n,i}))>0$ for all $\tau\in T(A_0).$
Since $T(A_0)$ is compact, 
we have
\beq
\nu_n:=\min_{1\le i\le r(n)}\min_{\tau\in T(A_0)} \tau(h_n(a_{n,i}))>0.
\eneq
Hence
\beq\label{Tm-15}
\tau(R_{n,i})\ge \tau(h_n(a_{n,i}))\ge \nu_n,\, \, 1\le i\le r(n).
\eneq

For each $a\in {\cal F}_n,$ choose $d_n(a)\in D_n$ so that (by  \eqref{LmainZ-0-6} with conjugating $u_n$)
\beq\label{Tm-15}
\max_{\{a\in {\cal F}_n}\|q_n^{1/2}u_n^*\Psi_n(a)u_nq_n^{1/2}-{\bar \Theta}_n(d_n(a))\|<{3\over{n}}\to 0.
\eneq
We set $d_n(e_A)=d_{c,n}.$   

For  the additional unital assertion, we take $A_0=A$ and
$e_A=1_A,$  and use the additional conclusion of Lemma \ref{LmainZ}.
Consequently,
\beq\label{e10-2add}
d_n(e_A)=d_{c,n}=1_{D_n}.
\eneq

Put
\beq
S_n=\{d_n(a): a\in {\cal F}_n\}\cup \{p_{n,1},...,p_{n,r(n)}\}.
\eneq
For each $n\in \N,$   let ${\cal E}_n\subset D_n^{\bf 1},$ a generating set (as ${\cal E}$) and $\dt_n>0$ (as $\dt$)
for $1/(10n)$ (in place of $\ep$)  and ${\cal S}_n$  (in place of ${\cal F}$) in Lemma \ref{Lstable2}.

Choose $0<\eta_n<\dt_n$ sufficiently small such that
\beq\label{Tm-17}
0\le g\le 1,\,\,\|z\|\le 1,\,\,\, \|[g,\, z]\|<\eta_n \implies 
\|[(1-g)^{1/2},\, z]\|<\min\{\dt_n,\,{1\over{100n}}\}.
\eneq
Write 
$$
{\cal E}_n=\{x_{n,1}, ..., x_{n,s(n)}\}.
$$

For each $j\in \N,$ define  $Z_j=\{Z_{j,n}\}_n\in l^\infty(A_0)$ by
\beq
Z_{j,n}=\begin{cases} {\bar \Theta}_n(x_{n,j}),\,\,\, j\le s(n),\\
                                       0 , \,\,\, j>s(n).\end{cases}.
\eneq
Also write
$$
X(a)=\{X_n(a):=u_n^*\Psi(a)u_n\}.
$$
Let $S'\subset l^\infty(A_0)$ be the separable \CA\, generated by
\beq
\{X(a), Y(a) : a\in \cup_n {\cal F}_n\}  \andeqn \{Z_j,\,\, j\in \N\}, 
\eneq
and put $J'=S'\cap I_{A_0}.$ Recall that $Y(a)\in J'$  for every $a\in \cup_{m=1}^\infty{\cal F}_m.$


Applying Lemma \ref{Lidealsep},   using a 
 quasi-central approximate identity for the separable ideal
$J',$  after passing to a subsequence and relabelling, 
we  obtain 
\beq\label{Tm-20}
g_n\in A_0, \,\,\, 0\le g_n\le 1,\,\, \{g_n\}\in I_{A_0}, \andeqn \|g_n\|_{_{2, T(A_0)}}<{\sigma_n\cdot \nu_n\over{100}}
\eneq
such that, for each $a\in {\cal F}_n,$
\beq\label{Tm-22}
&&\|(1-g_n)^{1/2}(X_n(a)-a)\|<{1\over{n}},\andeqn \|(X_n(a)-a)(1-g_n)^{1/2}\|<{1\over{n}},\\\label{Tm-22+1}
&&\|[g_n,\, X_n(a)]\|<\eta_n, \andeqn \\\label{Tm-22+2}
&&\|[g_n,\, \bar \Theta_n(z)]<\eta_n\rforal z\in {\cal E}_n.
\eneq
Put $p_n=(1-g_n)^{1/2},$ $n\in \N.$    By \eqref{Tm-17}, \eqref{Tm-22+1} and \eqref{Tm-22+2} we have 
\beq\label{Tm-25}
&&\|[p_n,\, X_n(a)]\|<{1\over{n}} \rforal a\in {\cal F}_n \andeqn\\\label{Tm-26}
&&\|[p_n,\, \bar \Theta_n(z)]\|<{1\over{n}}\rforal z\in {\cal E}_n.
\eneq
Also
\beq
\|[(1-g_n),\, \bar \Theta_n(z)]\|=\|[g_n\, \bar \Theta_n(z)]\|<\dt_n\rforal z\in {\cal E}_n.
\eneq

Define $c_n:=p_nq_np_n=(1-g_n)^{1/2}(1-{\bar e_n})(1-g_n)^{1/2}$ and put 
$e_n'=1-c_n,$ $n\in \N.$  Since
\beq
e_n'=g_n+p_n{\bar e_n}p_n=g_n+(1-g_n)^{1/2} u_n^*e_nu (1-g_n)^{1/2},
\eneq
and both $\{g_n\}$ and $\{{\bar e}_n\}$ are in $I_{A_0},$
\beq\label{Tm-31}
\{e_n'\}\in I_{A_0}.
\eneq
For fixed $a\in \cup_k {\cal F}_k,$ 
\beq
\hspace{-0.6in}\|[c_n,\, X_n(a)]\|&=&\|p_nq_np_nX_n(a)-p_nq_nX_n(a)p_n\|\\\label{Tm-32-0}
&&+\|
(p_nq_nX_n(a)-p_nX_n(a)q_n)p_n\|+\|(p_nX_n(a)-X_n(a)p_n)q_np_n\|.
\eneq
Hence, for fixed $a\in \cup_k {\cal F}_k,$
\beq\label{Tm-32}
\|[c_n, X_n(a)]\|\le 2\|[p_n, \, X_n(a)]\|+\|[q_n,\,X_n(a)]\|<{3\over{n}}.
\eneq
Put  $$y_n(a)=X_n(a)-a=u_n^*\Phi_{m(n)}(H_{l(n)}(a))u_n-a.$$
Then, for all $a\in {\cal F}_n,$ by \eqref{Tm-22}, 
\beq
\|c_ny_n(a)\|\le \|(1-g_n)^{1/2}y_n(a)\|<{1\over{n}}\andeqn \|y_n(a)c_n\|\le \|y_n(a)p_n\|<{1\over{n}}.
\eneq
Hence, for all $a\in  {\cal F}_n,$ 
\beq\label{Tm-35}
\|[c_n,\, a]\|\le \|[c_n,\, X_n(a)]\|+\|c_ny_n(a)\|+\|y_n(a)c_n\|<{5\over{n}}.
\eneq
Since $e_n'=1-c_n,$ 
\beq
\|[e_n',\, a]\|<{5\over{n}}\rforal a\in {\cal F}_n.
\eneq
By \eqref{Tm-35}, for fixed $k,$ 
\beq\label{Tm-36} 
\max_{a\in {\cal F}_k}\|[c_n^{1/2}, \, a]\|\to 0. 
\eneq

For $a\in {\cal F}_n,$ we have
\beq\nonumber
\|c_n^{1/2}ac_n^{1/2}-p_n{\bar \Theta}_n(d_n(a))p_n\|&\le &\|c_n^{1/2} ac_n^{1/2}-c_na\|\\\nonumber
  &&+\|c_n(a-X_n(a))\|+\|c_nX_n(a)-p_nq_nX_n(a)p_n\|\\\nonumber
   &&+\|p_nq_nX_n(a)p_n-p_nq_n^{1/2}X_n(a)q_n^{1/2}p_n\|\\
   &&+\|p_n(q_n^{1/2}X_n(a)q_n^{1/2}-{\bar \Theta}_n(d_n(a)))p_n\|.
\eneq
Therefore, by \eqref{Tm-36}, \eqref{Tm-22}, \eqref{Tm-25}, the conjugated form of  \eqref{LmainZ-0-5}, 
and \eqref{Tm-15},  
 we obtain 
\beq\label{Tm-38}
\max_{a\in {\cal F}_k}\|c_n^{1/2} ac_n^{1/2}-p_n{\bar \Theta}_n(d_n(a))p_n\|\to 0.
\eneq
By \eqref{Tm-25}, \eqref{Tm-26}, and \eqref{Tm-38},  Lemma \ref{Lstable2} gives an order-zero \cpc\, 
$L_n: D_n\to  \overline{c_nA_0c_n}$ such that
\beq\label{Tm-39}
\|L_n(d)-p_n{\bar \Theta}_n(d)p_n\|<{1\over{n}}\,\,\,\rforal d\in S_n.
\eneq
Consequently, for each fixed $k,$ 
\beq\label{Tm-40}
\max_{a\in {\cal F}_k}\|(1-e_n')^{1/2}a(1-e_n')^{1/2}-L_n(d_n(a))\|\to 0.
\eneq
In particular, 
\beq\label{Tm-42}
\|(1-e_n')^{1/2} e_A(1-e_n')^{1/2}-L_n(d_{c,n})\|\to 0.
\eneq

It remains to normalize $L_n.$

We claim that 
\beq\label{Tm-43}
\|p_na_{n,i}p_n\|> 1-{\sigma_n\over{2}} \,\,\, (1\le i\le r(n)).
\eneq
Suppose that, for some $i,$ 
$\|p_na_{n,i}p_n\|\le 1-{\sigma_n\over{2}}.$  Since  (see \eqref{Tm-13})
$$
a_{n,i}\ge (1-{\sigma_n\over{8}})R_{n,i},
$$
we obtain
\beq
\|p_nR_{n,i}p_n\|\le {1-\sigma_n/2\over{1-\sigma_n/8}}.
\eneq
Since $R_{n,i}$ is a projection, 
$$
\|p_n R_{n,i}p_n\|=\|R_{n,i}p_n^2R_{n,i}\|=\|R_{n,i}(1-g_n)R_{n,i}\|.
$$
It follows that
\beq
R_{n,i}(1-g_n)R_{n,i}\le {1-\sigma_n/2\over{1-\sigma_n/8}} R_{n,i},
\eneq
Hence
\beq
R_{n,i}g_nR_{n,i}\ge (1-{1-\sigma_n/2\over{1-\sigma_n/8}} )R_{n,i}={3\sigma_n\over{8-\sigma_n}} R_{n,i}>{3\sigma_n\over{8}}R_{n,i}.
\eneq
By \eqref{Tm-15}, for $\tau\in T(A_0),$
\beq
\tau(g_{n})\ge \tau(R_{n,i}g_nR_{n,i})>{3\sigma_n \nu_n\over{8}}.
\eneq
Then, by Cauchy-Bunyakowski-Schwarz inequality and \eqref{Tm-20},
\beq
{3\sigma_n\nu_n\over{8}}<\tau(g_n)\le \|g_n\|_{_{2, T(A_0)}}<{\sigma_n\nu_n\over{100}},
\eneq
a contradiction. This proves the claim and \eqref{Tm-43}. 

Since $p_{n,i}\in {\cal S}_n,$  $a_{n,i}=\bar \Theta_n(p_{n,i}),$
by \eqref{Tm-39} and \eqref{Tm-43}, we obtain
\beq\label{Tm-47}
\bt_{n,i}:=\|L_n(p_{n,i})\|>1-{\sigma_n\over{2}}-{1\over{n}}.
\eneq
Hence, when $n\to\infty,$ 
\beq\label{Tm-48}
\min_{1\le i\le r(n)} \bt_{n,i}\to 1 
\eneq
For all sufficiently large $n,$ define
\beq
\Theta_n'|_{D_{n,i}}=\bt_{n,i}^{-1}L_n|_{D_{n,i}}.
\eneq
Then $\Theta_n': D_n\to A_0$ is an order zero \cpc\, with 
\beq\label{Tm-50}
\|\Theta'_n|_{D_{n,i}}\|=1.
\eneq
Moreover,
\beq\nonumber
&&\Theta_n'(D_n)=L_n(D_n),\andeqn\\\label{Tm-52}
&& \|\Theta_n'-L_n\|\le \max_{1\leq i\leq r(n)}(1-\beta_{n,i})
 =
1-\min_{1\le i\le r(n)}\bt_{n,i}\to 0.
\eneq
Hence, by \eqref{Tm-40} and \eqref{Tm-52}, for every fixed $k,$
\beq\label{Tm-53}
\lim_{n\to\infty}\max_{a\in {\cal F}_k}{\rm dist}((1-e_n')^{1/2}a(1-e_n')^{1/2},\, \Theta_n'(D_n))=0.
\eneq
In particular, by \eqref{Tm-52} (see  also \eqref{Tm-42}),
\beq\label{Tm-54}
\lim_{n\to\infty}\|(1-e_n')^{1/2}e_A(1-e_n')^{1/2}-\Theta_n'(d_{c,n})\|=0.
\eneq
Recall that  (by \eqref{T109-011})
\beq\label{Tm-55}
t(d_{c,n})>1-\eta=\kappa_{e_A} \rforal t\in T(D_n). 
\eneq

Finally, we have $0\le e_n' \le 1,$  $c_n=(1-e_n'),$ and 
a fully normalized order-zero \cpc\, $\Theta_n': D_n\to  \overline{c_nA_0c_n}$ (see \eqref{Tm-50}) and 
an element $0\le d_{c,n}\le 1$ in $D_n$ satisfying 
 \eqref{Tm-35};  \eqref{Tm-53}, \eqref{Tm-54} and \eqref{Tm-55}; and  $\{e_n'\}\in I_{A_0}$ (by \eqref{Tm-31}).  In particular, $\{e_n'\}$ is a Cuntz-null sequence.  Hence, for any $a_0\not=0,$ and $\ep>0,$ 
 \beq\label{e10248}
 (e_n'-\ep)_+\lesssim a_0.
 \eneq
 So, also by Proposition \ref{Pertep}, $A_0$ has property ETR1  (we also use Remark \ref{R1}) by choosing $n$ such that $8/n<\ep$ and 
 $(e_n'-\ep)_+\lesssim a_0.$ 
 Hence $ETR(A)\le 1.$

Suppose now that \(A\) is unital. We use the unital choices made
above, with
\[
A_0=A,\qquad e_A=1_A,\qquad \kappa_{1_A}=1.
\]
As 
recorded  in \eqref{e10-2add},
 the distinguished
element in \(D_n\) has already been chosen to be \(1_{D_n}\). The
subsequent construction of \(L_n\), and the normalization of \(L_n\)
to \(\Theta_n'\), change only the order-zero map; neither the algebra
\(D_n\) nor its distinguished element is changed. Therefore
\eqref{Tm-53},
with \(e_A=1_A\), gives
\beq\label{e10-2a-2}
\lim_{n\to\infty}
\|(1-e_n')-\Theta_n'(1_{D_n})\|=0.
\eneq
Moreover,
\[
t(1_{D_n})=1
\,\,\, \text{for every }\,\,\, t\in T(D_n).
\]
Thus, for $0<\ep<1/4,$ ${\cal F}\subset A_1$ containing $1_A,$ 
nonzero comparison
element \(a_0\), choose \(n\) sufficiently large that \eqref{Tm-35}, \eqref{Tm-53},
and  \eqref{e10-2a-2},
and that \eqref{e10248}
gives
\[
(e_n'-\varepsilon)_+\lesssim a_0.
\]
Then
$
e_n',\, D_n,\, \Theta_n',\, d_e=1_{D_n}
$
satisfy all the requirements of property ETR1 associated with \(1_A\)
and the constant \(\kappa_{1_A}=1\). This proves the 
``Moreover" 
assertion.

\end{proof}

\section{ The  constant ${\mathfrak{\kappa}_{e_A}}$}

The reader might feel the constant $\kappa_{e_A}$ is somewhat mysterious. 
Theorem \ref{TmainZ} actually shows that, when $A$ is amenable, has continuous scale and satisfies
$ETR(A)\le 1,$ the constant $\kappa_{e_A}$ can be chosen arbitrarily close to 1  provided that   $e_A$ is chosen
accordingly.   We show in this section that this is always true without amenability. 
However, this is certainly not the case when $A$ does not have 
continuous scale. It is the oscillation of $\widehat{[e_A]}$ prevents  $\kappa_{e_A}$ being  close to 1.

\begin{lem}
\label{lem11.1}
Let \(A\subset B\) be \(C^*\)-algebras, let \(x,y\in A_+^{\bf 1}\), and
let \(0<\sigma\leq 1\). For every \(\rho>0\), there exist a finite
subset \(\mathcal G\subset A^{\bf 1}\), a number \(\delta>0\), and a
number
\[
0<\rho_0<
\min\left\{
\frac{\rho}{8},\frac{\sigma}{8},\frac{1}{8}
\right\}
\]
satisfying the following property.

Suppose that \(e\in B_+^1\),
\[
D=\bigoplus_{j=1}^m M_{r_j},
\tand
\theta\colon D\longrightarrow  \overline{(1-e)B(1-e)}
\]
is a fully normalized order zero c.p.c.\ map. Put
$
s=(1-e)^{1/2}.
$
Suppose that there are positive contractions \(d_x,d_y\in D\)
such that
\beq
&&\|[e,z]\|<\delta
\quad\text{and}\quad
\|[s,z]\|<\delta \rforal z\in\mathcal G,
\label{e11.1}
\\
&&\operatorname{dist}(szs,\theta(D))<\delta \tforal z\in\mathcal G,
\\\label{e11.2}
&&\left\|
s^2z_1z_2s^2-(sz_1s)(sz_2s)
\right\|<\delta \tforal 
z_1,z_2\in\mathcal G,\\
\label{e11.3}
&&\|\theta(d_x)-sxs\|<\delta
\quad\text{and}\quad
\|\theta(d_y)-sys\|<\delta.
\label{e11.4}
\eneq

Then, for every \(j\in\{1,\ldots,m\}\) satisfying
\begin{equation}
\operatorname{tr}_{r_j}\bigl((d_x)_j^2\bigr)\geq \sigma
\quad\text{and}\quad
\operatorname{tr}_{r_j}\bigl((d_y)_j^2\bigr)>1-\rho_0,
\label{e11.5}
\end{equation}
there is a tracial state \(\tau_j\in T(A)\) such that
\begin{equation}
\left|
\operatorname{tr}_{r_j}\bigl((d_x)_j^2\bigr)
-\tau_j(x^2)
\right|<\rho.
\label{e11.6}
\end{equation}
\end{lem}


\begin{proof}
Put
\[
 \varepsilon_1=
 \min\left\{\frac{\rho}{8},\frac{\sigma}{8},\frac18\right\}
 \quad\text{and}\quad
 \rho_0=\frac{\varepsilon_1}{2}.
\]
Apply Lemma~5.1 with the finite set
$
 \{x^2,y^2\}
$  
(as ${\cal F}$),
with \(x^2\) (as $e_A$), with
\(\sigma/2\) (as $\sigma$),  and with \(\varepsilon_1\) (as $\ep$).
Let
\(\mathcal G_0\subset A^1\) (as ${\cal G}$) and \(\delta_0>0\) (as $\dt$)  be supplied by
Lemma~\ref{LLtrace}.

Define
\[
\mathcal G
:=
\mathcal G_0\cup\mathcal G_0^*\cup\{x,y\}
\cup
\left\{
z_1z_2:
z_1,z_2\in
\mathcal G_0\cup\mathcal G_0^*\cup\{x,y\}
\right\}.
\]
Choose
\[
  0<\delta<
  \min\left\{\frac{\delta_0}{112},
              \frac{\varepsilon_1}{12}\right\}.
\]
Then \(\mathcal G\subset A^{\bf 1}\) is finite and self-adjoint. In
particular, \(\mathcal G\) contains \(x,y,x^2,y^2\), as well as
\(z_1z_2\) 
for all \(z_1,z_2\in\mathcal G_0\). 
We now assume that the assumptions of the lemma hold for these choice of $\rho_0,$ ${\cal G}$ and $\dt.$

Let
\[
 \widetilde\theta:C_0((0,1])\otimes D\longrightarrow B
\]
be the homomorphism associated with \(\theta\), and put
\[
 C=\widetilde\theta(C_0((0,1])\otimes D).
\]
Since \(\theta\) is fully normalized on every summand, point
evaluation at \(1\) induces a surjective homomorphism
\[
 \pi_{1,j}:C\longrightarrow M_{r_j},
 \qquad j=1,\ldots,m.
\]
Define
\[
 \lambda_j={\rm tr}_{r_j}\circ\pi_{1,j}.
\]
Then \(\lambda_j\) is a tracial state on \(C\).

Put \(L(z)=szs\), and consider the finite subset
\[
 \mathcal W
 =
 \left\{
 L(z),\ s^{2}z_{1}z_{2}s^{2},\ L(z_{1})L(z_{2})
 :z,z_{1},z_{2}\in\mathcal G
 \right\}
 \subseteq \overline{(1-e)B(1-e)}.
\]
Notice that, since  \(x,y\in\mathcal G\), the set
\(\mathcal W\) contains, in particular,
\[
 s^{2}x^{2}s^{2},\quad L(x)^{2},\quad
 s^{2}y^{2}s^{2},\quad L(y)^{2}.
\]
By \((\mathrm{e11.2})\) and \((\mathrm{e11.3})\), every
\(w\in\mathcal W\) is within \(4\delta\) of \(C\). For each
\(w\in\mathcal W\), choose \(c_w\in C\) such that
\[
 \|w-c_w\|<4\delta.
\]
Since \(C\) is amenable, by \cite[Theorem 2.3.13]{LinBook}, there is a completely positive contractive
map
$
  \Psi\colon B\longrightarrow C
$
such that 
\[
  \|\Psi(c_w)-c_w\|<\delta\rforal w\in {\cal W}.
\]
Consequently,
\beq
  \|\Psi(w)-w\|<8\delta+\delta=9\delta\rforal w\in {\cal W}.
  \label{e11.7}
\eneq
For \(j=1,\ldots,m\), define
\[
 F_j(a)=\lambda_j\bigl(\Psi(s^2as^2)\bigr),
 \qquad a\in A.
\]
Then \(F_j\) is a positive linear functional and
\[
 \|F_j\|\leq1.
\]
Fix \(z_1,z_2\in\mathcal G_0.\)
By \((\mathrm{e11.3})\)
\[
  \|s^2z_1z_2s^2-sz_1ssz_2s\|<\delta.
\]
Since \(\Psi\) is contractive, it follows that
\[
 \bigl\|
 \Psi(s^2z_1z_2s^2)-\Psi(sz_1ssz_2s)
 \bigr\|<\delta.
\]
Since $sz_1s, sz_2s, sz_1ssz_2s
\in\mathcal W,$ \((\mathrm{e11.7})\)
gives
\[
 \|\Psi(sz_1ssz_2s)-sz_1ssz_2s\|<9\delta
\andeqn
 \|\Psi(L(z_i))-L(z_i)\|<9\delta,\qquad i=1,2.
\]
As \(sz_is\) and \(\Psi(sz_is)\) are contractions, we have
\[
\begin{aligned}
 \|sz_1ssz_2s-\Psi(sz_1s)\Psi(sz_2s)\|
 &\leq
 \|(sz_1s-\Psi(sz_1s))sz_2s\|  \\
 &\quad+
 \|\Psi(sz_1s)(sz_2s-\Psi(sz_2s))\| 
 <18\delta.
\end{aligned}
\]
Consequently,
\beq\nonumber
 &&\hspace{-0.7in}|
 F_j(z_1z_2)
 -\lambda_j\bigl(\Psi(sz_1s)\Psi(sz_2s)\bigr)
 |                                                   
 \leq
 \bigl\|
 \Psi(s^2z_1z_2s^2)-\Psi(sz_1ssz_2s)
 \bigr\|                                                   \\\nonumber
 &&\hspace{1.2in}+
 \|\Psi(sz_1ssz_2s)-sz_1ssz_2s\|                                  
 +
 \|sz_1ssz_2s-\Psi(sz_1s)\Psi(sz_2s)\|                             \\
 &&\hspace{1.4in}<\delta+9\delta+18\delta
 =28\delta.
\eneq
Similarly,
\[
 \left|
 F_j(z_2z_1)
 -\lambda_j\bigl(\Psi(sz_2s)\Psi(sz_1s)\bigr)
 \right|
 <28\delta.
\]
Since \(\lambda_j\) is a trace on \(C\),
\[
 \lambda_j\bigl(\Psi(sz_1s)\Psi(sz_2s)\bigr)
 =
 \lambda_j\bigl(\Psi(sz_2s)\Psi(sz_1s)\bigr).
\]
It follows that
\beq
 \left|F_j(z_1z_2)-F_j(z_2z_1)\right|
 <56\delta
 <\frac{\delta_0}{2}.
 \label{e11.8}
\eneq

We next compare \(F_j(x^2)\) with
\(\operatorname{tr}_{r_j}((d_x)_j^2)\). By
\((\mathrm{e11.3})\), applied with \(z_1=z_2=x\),
\[
 \|s^2x^2s^2-L(x)^2\|<\delta.
\]
Hence, since \(\Psi\) is contractive,
\[
 \bigl\|
 \Psi(s^2x^2s^2)-\Psi(L(x)^2)
 \bigr\|<\delta.
\]
Since \(L(x)^2\in\mathcal W\), \((\mathrm{e11.7})\) gives
\[
 \|\Psi(L(x)^2)-L(x)^2\|<9\delta.
\]
On the other hand, \((\mathrm{e11.4})\) gives
\[
 \|L(x)-\theta(d_x)\|<\delta.
\]
Since \(L(x)\) and \(\theta(d_x)\) are contractions,
\[
\begin{aligned}
 \|L(x)^2-\theta(d_x)^2\|
 &\leq
 \|L(x)(L(x)-\theta(d_x))\| 
 +
 \|(L(x)-\theta(d_x))\theta(d_x)\| 
 <2\delta.
\end{aligned}
\]
Therefore,
\[
\begin{aligned}
 \|\Psi(s^2x^2s^2)-\theta(d_x)^2\|
 &\leq
 \bigl\|
 \Psi(s^2x^2s^2)-\Psi(L(x)^2)
 \bigr\|                                                
 +
 \|\Psi(L(x)^2)-L(x)^2\|                               \\
 &\quad+
 \|L(x)^2-\theta(d_x)^2\|                              \\
 &<\delta+9\delta+2\delta
 =12\delta.
\end{aligned}
\]
Since
\[
 \lambda_j\bigl(\theta(d_x)^2\bigr)
 =
 \operatorname{tr}_{r_j}\bigl((d_x)_j^2\bigr),
\]
we obtain
\beq\label{e11.9}
 \left|
 F_j(x^2)
 -{\rm tr}_{r_j}\bigl((d_x)_j^2\bigr)
 \right|\le \|\Psi(s^2x^2s^2)-\theta(d_x)^2\|
 <12\delta
 <\varepsilon_1.
\eneq
The same argument, with \(y\) in place of \(x\), gives
\beq
 \left|
 F_j(y^2)
 -{\rm tr}_{r_j}\bigl((d_y)_j^2\bigr)
 \right|
 <12\delta
 <\varepsilon_1.
 \label{e11.10}
\eneq

Put
$
 \alpha_j=\|F_j\|.
$
By \eqref{e11.5} and \eqref{e11.10},
\[
 \alpha_j
 \geq F_j(y^2)
 >1-\rho_0-\varepsilon_1
 =1-\frac32\varepsilon_1
 >1-2\varepsilon_1.
\]
Define
\[
 \omega_j=\frac{F_j}{\alpha_j}.
\]
Then \(\omega_j\) is a state. By \eqref{e11.8},
\[
 |\omega_j(z_1z_2)-\omega_j(z_2z_1)|<{\dt_0\over{2(1-1/4)}}
 <\delta_0\rforal z_1,z_2\in\mathcal G_0.
\]
Moreover, by \eqref{e11.5} and \eqref{e11.9},
\[
 \begin{aligned}
 \omega_j(x^2)
 &\geq F_j(x^2)
 >{\rm tr}_{r_j}((d_x)_j^2)-\varepsilon_1
 \geq\sigma-\varepsilon_1
 >\frac{\sigma}{2}.
 \end{aligned}
\]

Lemma~\ref{LLtrace} therefore gives a positive trace \(\sigma_j\) on \(A\)
such that
\beq\label{eq:tw11}
&& \|\sigma_j\|\leq1,\qquad
 \sigma_j(x^2)\geq\frac{\sigma}{2},
\andeqn\\
\label{eq:tw12}
&& |\sigma_j(x^2)-\omega_j(x^2)|<\varepsilon_1,
 \qquad
 |\sigma_j(y^2)-\omega_j(y^2)|<\varepsilon_1.
\eneq
At this point, we do not assume that \(\sigma_j\) is a tracial
state.

Since \(\alpha_j\leq1\), \eqref{e11.10} gives
\[
 \begin{aligned}
 \omega_j(y^2)
 &=\frac{F_j(y^2)}{\alpha_j}
 \geq F_j(y^2)
 >1-\rho_0-\varepsilon_1
 >1-2\varepsilon_1.
 \end{aligned}
\]
It follows from \eqref{eq:tw12} that
\begin{equation}\label{eq:tw13}
 \|\sigma_j\|
 \geq\sigma_j(y^2)
 >1-3\varepsilon_1.
\end{equation}
In particular, \(\sigma_j\ne0\). Define
\[
 \tau_j=\frac{\sigma_j}{\|\sigma_j\|}\in T(A).
\]

Using \eqref{eq:tw12} and \eqref{eq:tw13}, we obtain
\[
 \begin{aligned}
 |\tau_j(x^2)-\omega_j(x^2)|
 &\leq
 \left|
 \frac{\sigma_j(x^2)}{\|\sigma_j\|}
 -\sigma_j(x^2)
 \right|
 +|\sigma_j(x^2)-\omega_j(x^2)|\\
 &\leq
 1-\|\sigma_j\|+\varepsilon_1
 <4\varepsilon_1.
 \end{aligned}
\]
Finally, using \eqref{e11.9}, the preceding estimate, and
\(1-\alpha_j<2\varepsilon_1\), we have
\[
 \begin{aligned}
 \left|
 {\rm tr}_{r_j}((d_x)_j^2)-\tau_j(x^2)
 \right|
 &\leq
 \left|
 {\rm tr}_{r_j}((d_x)_j^2)-F_j(x^2)
 \right| +
 \left|
 \alpha_j\omega_j(x^2)-\tau_j(x^2)
 \right|\\
 &\leq
 \left|
 {\rm tr}_{r_j}((d_x)_j^2)-F_j(x^2)
 \right|
 +
 \alpha_j|\omega_j(x^2)-\tau_j(x^2)|
 +(1-\alpha_j)\tau_j(x^2)\\
 &<\varepsilon_1+4\varepsilon_1+2\varepsilon_1
 =7\varepsilon_1<\rho.
 \end{aligned}
\]
This proves the lemma.
\end{proof}

%
%

\begin{lem}\label{cor-constant-lift}
Let $A$ be a separable simple non-elementary $C^*$-algebra with
continuous scale and strict comparison. Fix a strictly positive
contraction $e_A\in A$ and a number $0<\kappa\leq1$.

Suppose that, for every $\delta>0$, every finite subset
${\cal F}\subset A^{\bf 1}$ containing $e_A$, and every nonzero
$a\in A_+$, there are a positive contraction $q\in Q_A$, a
finite-dimensional $C^*$-algebra
\[
 D=\bigoplus_{j=1}^mM_{r_j},
\]
a fully normalized order-zero c.p.c.\ map
\[
 \Theta_A:D\longrightarrow A,
\]
and contractions $d_x\in D$, $x\in{\cal F}$, such that the element
$d_{e_A}$ corresponding to $e_A$ belongs to $D_+^{\bf 1}$ and, regarding
$\Theta_A$ as the constant-sequence map
$\overline{j_A}\circ\Theta_A:D\to Q_A$, one has
\beq\label{cor-constant-lift.1}
&& \overline{j_A}(\Theta_A(D))
 \subset
 \overline{(1-q)Q_A(1-q)},\\
\label{L117-2}
 &&\|[q,\overline{j_A}(x)]\|<\delta\tforal  x\in{\cal F},
\\\label{L117-3}
 &&\|\overline{j_A}(\Theta_A(d_x))
 -(1-q)^{1/2}\overline{j_A}(x)(1-q)^{1/2}\|<\delta\tforal x\in{\cal F},
\\\label{L117-4}
&& q\lesssim \overline{j_A}(a),
\tand\\\label{L117-5}
&& t(d_{e_A})\geq\kappa,
 \qquad t\in T(D).
\eneq
Then $A$ has property ${\rm ETR}1$ associated with $e_A$ and the
constant $\kappa$.
\end{lem}

\begin{proof}

\medskip
\noindent\textbf{1. Choice of constants and quotient data.}

Fix $\epsilon>0$, a finite subset ${\cal F}\subset A^{\bf 1}$ containing
$e_A$, and a nonzero element $a\in A_+$. Enlarging ${\cal F}$ if
necessary, we may assume that ${\cal F}$ is self-adjoint and that
$0<\epsilon<1$.

Choose mutually orthogonal nonzero positive elements
$a_0,a_1\in\Her(a)$ such that
\[
 a_0\oplus a_1\lesssim a.
\]
Since $a_1\neq0$, simplicity of $A$ and compactness of $T(A)$ give
\beq\label{11.7.7}
 \sigma
 :=
 \inf_{\tau\in T(A)}d_\tau(a_1)>0.
\eneq

Choose $0<\alpha<1$ and $\mu,\eta>0$ sufficiently small that
\beq\label{11.7.8}
 \frac{\alpha}{2}+20\eta<\frac{\epsilon}{2}
\andeqn
 8\eta+2\mu
 <
 \min\left\{
 \frac{\epsilon}{4},
 \left(\frac{\epsilon}{4}\right)^2
 \right\}.
\eneq
Choose $\delta_{\rm fc}>0$ such that, whenever $0\leq b\leq1$ and
$\|x\|\leq1$,
\beq\label{11.7.10}
 \|[b,x]\|<\delta_{\rm fc}
 \quad\Longrightarrow\quad
 \|[b^{1/2},x]\|<\eta.
\eneq
Put
\[
 \delta_0
 =
 \min\left\{\eta,\frac{\delta_{\rm fc}}{8}\right\}.
\]

Apply the hypothesis with $\delta_0$ (as $\dt$), ${\cal F}$, and $a_0$ (as $a$). We
obtain a positive contraction $q\in Q_A$, a finite-dimensional
$C^*$-algebra
\[
 D=\bigoplus_{j=1}^mM_{r_j},
\]
a fully normalized order-zero c.p.c.\ map
$
 \Theta_A:D\longrightarrow A,
$
and contractions $d_x\in D$, $x\in{\cal F}$, satisfying \eqref{cor-constant-lift.1}-- \eqref{L117-5},
 where $\Theta_A$ is viewed as a
constant-sequence map into $Q_A$.

Write
\[
 p_j=1_{M_{r_j}}
 \quad\text{and}\quad
 b_j=\Theta_A(p_j),
 \qquad j=1,\ldots,m.
\]
Since $\Theta_A$ is fully normalized,
\[
 \|b_j\|=1,
 \qquad j=1,\ldots,m.
\]
Put
\beq\label{117Rj}
 R_j=\chi_{[1-\alpha/8,\,1]}(b_j)\in A^{**}.
\eneq
Let $g_\alpha$ be a nonzero positive continuous function supported
in $[1-\alpha/8,1]$, $g_\af(1)=1$ and  with $0\leq g_\alpha\leq1$. Since $A$ is
simple and $T(A)$ is compact,
\[
 \nu
 =
 \min_{1\leq j\leq m}
 \inf_{\tau\in T(A)}
 \tau(g_\alpha(b_j))
 >0.
\]
Consequently,
\beq\label{11.7.13}
 \tau(R_j)\geq\nu,
 \qquad
 \tau\in T(A),\quad j=1,\ldots,m.
\eneq

Let
\[
 S_D
 =
 \{d_x:x\in{\cal F}\}
 \cup
 \{p_1,\ldots,p_m\}.
\]
Let ${\cal E}\subset D^{\bf 1}$ and $\rho>0$ (as $\dt$) be given by Lemma \ref{Lstable2}
for $S_D$ (in place of ${\cal F}$) and $\eta$ (in place of $\ep$).
%
 Choose
\[
 0<\zeta<\min\left\{\eta,\frac{\rho}{16}\right\}.
\]

\medskip
\noindent\textbf{2. Lifting the quotient data.}

Choose a positive-contraction lift
\[
 {\boldsymbol q}=\{q_n\}_{n=1}^\infty
 \in\prod_{n=1}^\infty A
\]
of $q$, and put $c=1-q$. Since
$
 \{c^{1/N}\}_{N=1}^\infty
$
is an approximate identity for $\overline{cQ_Ac}$ and
\[
 \overline{j_A}(\Theta_A(D))\subset\overline{cQ_Ac},
\]
there is $N\in\mathbb N$ such that
\[
 \|c^{1/N}\overline{j_A}(\Theta_A(d))-\overline{j_A}(\Theta_A(d))\|<\zeta
\andeqn
 \|\overline{j_A}(\Theta_A(d))c^{1/N}-\overline{j_A}(\Theta_A(d))\|<\zeta
\]
for every $d\in {\cal E}\cup S_D$. Put
\[
 s_n=(1-q_n)^{1/2},
 \qquad
 h_n=(1-q_n)^{1/N}.
\]

For $x\in{\cal F}$, write
\beq\label{11.7.17}
 &&[q_n,x]=c_{x,n}^{(1)}+z_{x,n}^{(1)},
 \qquad
 \sup_n\|c_{x,n}^{(1)}\|<\delta_0,   \{z_{x,n}^{(1)}\}\in I_A,\,\,{\rm and},\\\label{11.7.18}
 &&s_nxs_n-\Theta_A(d_x)
 =
 c_{x,n}^{(2)}+z_{x,n}^{(2)}\in A,
 \qquad
 \sup_n\|c_{x,n}^{(2)}\|<\delta_0, \andeqn \{z_{x,n}^{(2)}\}\in I_A.
\eneq
By the choice of $\delta_{\rm fc}$,
\[
 \|[(1-q)^{1/2},\overline{j_A}(x)]\|<\eta.
\]
Hence we may also write
\beq\label{11.7.19}
 [s_n,x]
 =
 c_{x,n}^{(3)}+z_{x,n}^{(3)},
 \qquad
 \sup_n\|c_{x,n}^{(3)}\|<\eta\andeqn \{z_{x,n}^{(3)}\}\in I_A.
\eneq

For $d\in E\cup S_D$, write
\beq\label{11.7.20}
 h_n\Theta_A(d)-\Theta_A(d)
 =
 c_{d,n}^{(4)}+z_{d,n}^{(4)}\in A,
 \qquad
 \sup_n\|c_{d,n}^{(4)}\|<\zeta\andeqn \{z_{d,n}^{(4)}\}\in I_A,
\eneq
and
\beq\label{11.7.21}
 \Theta_A(d)h_n-\Theta_A(d)
 =
 c_{d,n}^{(5)}+z_{d,n}^{(5)}\in A,
 \qquad
 \sup_n\|c_{d,n}^{(5)}\|<\zeta\andeqn \{z_{d,n}^{(5)}\}\in I_A.
\eneq

Since $q\lesssim \overline{j_A}(a_0)$, there is $w\in Q_A$ such that
\[
 \|w^*\overline{j_A}(a_0)w-q\|<\eta.
\]
Choose a bounded lift
\[
 {\boldsymbol w}=\{w_n\}\in\prod_{n=1}^\infty A
\]
of $w$, and write
\beq\label{11.7.23}
 w_n^*a_0w_n-q_n=u_n+r_n,
 \qquad
 \sup_n\|u_n\|<\eta,
 \qquad
 {\boldsymbol r}=\{r_n\}\in I_A.
\eneq

For convenience, write
\[
 {\boldsymbol c}_x^{(i)}=\{c_{x,n}^{(i)}\},
 \qquad
 {\boldsymbol z}_x^{(i)}=\{z_{x,n}^{(i)}\},
\andeqn
 {\boldsymbol c}_d^{(i)}=\{c_{d,n}^{(i)}\},
 \qquad
 {\boldsymbol z}_d^{(i)}=\{z_{d,n}^{(i)}\}.
\]

\medskip
\noindent\textbf{3. Quasi-diagonality.}

Let $S\subset l^\infty(A)$ be the separable $C^*$-algebra generated
by
$
 {\boldsymbol q}, {\boldsymbol w},
$
the constant sequences from ${\cal F}$ and $\Theta_A(D)$, and the
$A$-valued sequences
\[
 {\boldsymbol c}_x^{(i)},\quad
 {\boldsymbol z}_x^{(i)},\quad
 {\boldsymbol c}_d^{(j)},\quad
 {\boldsymbol z}_d^{(j)},\quad
 \{u_n\},\quad {\boldsymbol r},
\]
where $x\in{\cal F}$, $i=1,2,3$, $d\in E\cup S_D$, and $j=4,5$.
Put
\[
 J=S\cap I_A.
\]
(The multiplier-valued sequences $\{s_n\}$ and $\{h_n\}$ are not
included among the generators of $S$.)

Let $X\subset J$ contain
\[
 {\boldsymbol z}_x^{(i)},
 \qquad
 x\in{\cal F},\quad i=1,2,3,\andeqn
{\boldsymbol r}. 
\]For each
\[
 {\boldsymbol z}\in
 \left\{
 {\boldsymbol z}_d^{(4)},
 {\boldsymbol z}_d^{(5)}:
 d\in E\cup S_D
 \right\},
\]
also include in $X$ the three elements
$
 {\boldsymbol z},
 \,
 {\boldsymbol z}^*{\boldsymbol z},
 \,
 {\boldsymbol z}{\boldsymbol z}^*.
$
Let $Y\subset S$ consist of the constant sequences determined by
\[
 {\cal F}\cup\{\Theta_A(d):d\in E\}.
\]

Choose
\[
 0<\beta<
 \min\left\{
 \delta_0,\eta,\zeta^2,
 \frac{\alpha\nu}{100},
 \frac{\mu\sqrt{\sigma}}{4}
 \right\}.
\]
For $k\in\mathbb N$, put
\[
 X_k=X,\qquad Y_k=Y,\qquad N_k=1,
\]
\[
 \beta_k=\min\left\{\beta,\frac1k\right\},
 \qquad
 \epsilon_k
 =
 \min\left\{
 \frac{\delta_0}{2},
 \frac{\rho}{8},
 \frac1k
 \right\}.
\]
Apply Lemma \ref{Lidealsep} with
\[
 \alpha_k=\beta_k
\]
and commutator tolerance $\epsilon_k$. We obtain increasing
coordinates $m(k)$ and positive contractions $g_k\in A$. Fix a
sufficiently large $k$, and put
\[
 m=m(k),\qquad
 g=g_k,\qquad
 v=(1-g)^{1/2},
\]
\[
 f=q_m,\qquad
 s=(1-f)^{1/2},\qquad
 h=(1-f)^{1/N}.
\]

Lemma \ref{Lidealsep} gives
\beq\label{11.7.25}
 \|g\|_{2,T(A)}
 <
 \min\left\{
 \frac{\alpha\nu}{100},
 \frac{\mu\sqrt{\sigma}}{4}
 \right\}.
\eneq
Using \eqref{11.7.17}-- \eqref{11.7.19},
we obtain, for $x\in{\cal F}$,
\beq\label{11.7.26}
 \|v[f,x]v\|<2\delta_0,
 \qquad
 \|v(sxs-\Theta_A(d_x))v\|<2\eta,
\eneq
and
\beq\label{11.7.27}
 \|v[(1-f)^{1/2},x]v\|<2\eta,
 \qquad
 \|[v,x]\|<\delta_0.
\eneq

If ${\boldsymbol z}$ is one of
${\boldsymbol z}_d^{(4)}$ or
${\boldsymbol z}_d^{(5)}$, then Lemma \ref{Lidealsep} gives
\[
 \|z_m v\|^2
 =
 \|v z_m^*z_m v\|<\beta_k
\]
and
\[
 \|v z_m\|^2
 =
 \|v z_mz_m^*v\|<\beta_k.
\]
Since $\beta_k<\zeta^2$,  we also have, for all $d\in E\cup S_D,$
(see \eqref{11.7.20}--%
\eqref{11.7.21}) 
\beq\label{11.7.28}
 &\|v(h\Theta_A(d)-\Theta_A(d))\|<2\zeta, 
 \|(h\Theta_A(d)-\Theta_A(d))v\|<2\zeta,\\\label{11.7.28a}
 &\|v(\Theta_A(d)h-\Theta_A(d))\|<2\zeta,
 \|(\Theta_A(d)h-\Theta_A(d))v\|<2\zeta.
\eneq
 Moreover,
\beq\label{11.7.29}
 \|[v,\Theta_A(d)]\|<\frac{\rho}{4},
 \qquad d\in E,
\eneq
and, by \eqref{11.7.23},
\beq\label{11.7.30}
 \|v(f-w_m^*a_0w_m)v\|<2\eta.
\eneq

\medskip
\noindent\textbf{4. Construction of the complementary corner.}

Define
\beq\label{11.7.31}
 c_0=v(1-f)v\andeqn 
 e_0=g+vfv.
\eneq
Then $e_0\in A_+^{\bf 1}$ and, in $\widetilde A$,
\[
 1-e_0=c_0.
\]

For $x\in{\cal F}$,
\beq\label{11.7.33}
 \|[c_0,x]\|
 \leq
 2\|[v,x]\|+\|v[f,x]v\|
 <4\delta_0
 <\delta_{\rm fc}.
\eneq
Therefore, for all $x\in{\cal F},$ 
\beq\label{11.7.34}
 \|[c_0^{1/2},x]\|<\eta.
\eneq
Since $e_0=1-c_0$,
\beq\label{11.7.35}
 \|[e_0,x]\|<4\delta_0,
 \qquad x\in{\cal F}.
\eneq
For $x\in{\cal F}$,  by \eqref{11.7.34}, 
\eqref{11.7.27},   
\eqref{11.7.26}, 
\eqref{11.7.28} and \eqref{11.7.28a}, 
\beq\nonumber
\hspace{-0.4in}\|c_0^{1/2}xc_0^{1/2}
       -vh\Theta_A(d_x)hv\|
&\leq &
 \|c_0^{1/2}xc_0^{1/2}-c_0x\|\\\nonumber
&&+
 \|v(1-f)[v,x]\|
+
 \|v(1-f)^{1/2}[(1-f)^{1/2},x]v\|\\\nonumber
&&\hspace{0.2in}+
 \|v(sxs-\Theta_A(d_x))v\|
+
 \|v(\Theta_A(d_x)-h\Theta_A(d_x)h)v\|\\\label{11.7.36}
&<&\eta+\delta_0+2\eta+2\eta+4\zeta
 <10\eta.
\eneq

Put
\[
 z=hv=(1-f)^{1/N}(1-g)^{1/2}\in\widetilde A.
\]
Then
\[
 z^*\Theta_A(d)z=vh\Theta_A(d)hv.
\]
Since
\[
 h\Theta_A(d)h\in\overline{(1-f)A(1-f)}
\andeqn
 \overline{
 v\,\overline{(1-f)A(1-f)}\,v
 }
 =
 \overline{c_0Ac_0},
\]
we have
\beq\label{11.7.37}
 z^*\Theta_A(d)z\in\overline{c_0Ac_0}\rforal  d\in D.
\eneq
For $d\in {\cal E}$,  by \eqref{11.7.29} and \eqref{11.7.28}, 
\[
 \|[z,\Theta_A(d)]\|
 \leq
 \|h[v,\Theta_A(d)]\|
 +
 \|[h,\Theta_A(d)]v\|
 <
 \frac{\rho}{4}+4\zeta
 <\rho.
\]
Similarly,
\[
 \|[z^*,\Theta_A(d)]\|<\rho\rforal  d\in E.
\]

\medskip
\noindent\textbf{5. Perturbation to a normalized order-zero map.}

Applying Lemma \ref{Lstable2} in $\widetilde A$, with
$
 B=\overline{c_0Ac_0},
$
we obtain an order-zero c.p.c.\ map
\[
 L^{(0)}:D\longrightarrow\overline{c_0Ac_0}
\]
such that
\beq\label{11.7.40}
 \|L^{(0)}(d)-vh\Theta_A(d)hv\|<\eta\rforal  d\in S_D.
\eneq

We next normalize $L^{(0)}$. We claim that
\beq\label{11.7.41}
 \|vb_jv\|>1-\frac{\alpha}{2},
 \qquad j=1,\ldots,m.
\eneq
Indeed, otherwise, for some $j$,
\[
 \|vb_jv\|\leq1-\frac{\alpha}{2}.
\]
Since (see \eqref{117Rj})
\[
 b_j\geq\left(1-\frac{\alpha}{8}\right)R_j,
\]
we obtain
\[
 \|R_j(1-g)R_j\|
 =
 \|vR_jv\|\leq {\|vb_jv\|\over{1-\af/8}}
 \leq
 \frac{1-\alpha/2}{1-\alpha/8}.
\]
It follows that
\[
 R_jgR_j=R_j-R_j(1-g)R_j
 >
 \frac{3\alpha}{8}R_j.
\]
Therefore,  by \eqref{11.7.13},
\[
 \tau(g)> \tau(R_jgR_j)\ge  \frac{3\alpha}{8}\tau(R_j)>
 \frac{3\alpha\nu}{8}\rforal  \tau\in T(A),
\]
contradicting
\[
 \tau(g)
 \leq
 \|g\|_{2,T(A)}
 <
 \frac{\alpha\nu}{100}.
\]
This proves  the claim  and \eqref{11.7.41}.

By \eqref{11.7.28},
\beq\label{11.7.42}
 \|vhb_jhv-vb_jv\|<4\zeta.
\eneq
Put
\[
 \lambda_j=\|L^{(0)}(p_j)\|.
\]
It follows from \eqref{11.7.40}, \eqref{11.7.42}, and \eqref{11.7.41} that 
\beq
 \lambda_j&=&\|L^{(0)}(p_j)\|\ge \|vh\Theta_A(p_j)hv\|-\eta
 =\|vhb_jhv\|-\eta\\
 &\ge & \|vb_jv\|-4\zeta-\eta>
 1-\frac{\alpha}{2}-5\eta.
\eneq
Define
\[
 L|_{M_{r_j}}
 =
 \lambda_j^{-1}L^{(0)}|_{M_{r_j}},
 \qquad j=1,\ldots,m.
\]
Then
\[
 L:D\longrightarrow\overline{c_0Ac_0}
\]
is a fully normalized order-zero c.p.c.\ map, and
\beq\label{11.7.45}
 \|L-L^{(0)}\|
 <
 \frac{\alpha}{2}+5\eta.
\eneq
Combining \eqref{11.7.45}, 
\eqref{11.7.40}, \eqref{11.7.36}, and \eqref{11.7.8}, we obtain for all $x\in{\cal F},$
\beq\nonumber
 \|L(d_x)-(1-e_0)^{1/2}x(1-e_0)^{1/2}\| &\le& \|L^{(0)}(d_x)-(1-e_0)^{1/2}x(1-e_0)^{1/2}\| + \frac{\alpha}{2}+5\eta\\\nonumber
 &<& \|  vh\Theta_A(d_x)hv-c_0^{1/2}xc_0^{1/2}\|+\frac{\alpha}{2}+5\eta \\
 &<&10\eta+\frac{\alpha}{2}+5\eta<\frac{\epsilon}{2}.
\eneq

\medskip
\noindent\textbf{6. Cuntz comparison and verification of property ETR1.}

Put
\[
 y=w_mv.
\]
By \eqref{11.7.31} and \eqref{11.7.30},
\[
 \|e_0-(g+v^*w_m^*a_0w_mv)\|=\|g+vfv-g-vw_m^*a_0w_mv\|=\|v(f-w_m^*a_0w_m)v\|<2\eta.
\]
Choose
\[
 2\eta+2\mu<\gamma<8\eta+2\mu.
\]
By \cite[Proposition 2.2]{RrII},
\[
 (e_0-\gamma)_+
 \lesssim
 (g+y^*a_0y-2\mu)_+\\
 \lesssim
 (g-\mu)_+\oplus y^*a_0y
 \lesssim
 (g-\mu)_+\oplus a_0.
\]
Let $p=\chi_{(\mu,1]}(g)\in A^{**}.$ Therefore (applying normal extension of $\tau$),
for $\tau\in T(A),$  by \eqref{11.7.25},
\beq
d_\tau((g-\mu)_+)=\tau(p)\le {\tau(g^2)\over{\mu^2}} <{\|g\|_{_{2, T(A)}}^2\over{\mu^2}}<({\mu \sqrt{\sigma}\over{4}})^2/\mu^2=
 \frac{\sigma}{16}.
\eneq

 By  \eqref{11.7.7}, since A has strict comparison, we obtain
$
 (g-\mu)_+\lesssim a_1.
$
Therefore,
\[
 (e_0-\gamma)_+
 \lesssim
 a_0\oplus a_1
 \lesssim a.
\]
Put
$
 e=(e_0-\gamma)_+.
$
Then
\beq\label{11.7.52}
 e\lesssim a\andeqn
 \|e-e_0\|\leq\gamma.
\eneq
Since $e\leq e_0$, we have 
\[
 \overline{(1-e_0)A(1-e_0)}
 \subset
 \overline{(1-e)A(1-e)}.
\]
Thus $L$ may be regarded as a fully normalized order-zero c.p.c.\
map
\[
 L:D\longrightarrow\overline{(1-e)A(1-e)}.
\]
For $x\in{\cal F}$,  by \eqref{11.7.52},
\beq\label{11.7-100f}
 \|[e,x]\|
 \leq
 \|[e_0,x]\|+2\gamma
 <\epsilon.
\eneq
Moreover,
\[
 \|(1-e)^{1/2}-(1-e_0)^{1/2}\|
 \leq
 \sqrt{\gamma}.
\]
Consequently,
\beq\label{11.7.53}
 \|L(d_x)-(1-e)^{1/2}x(1-e)^{1/2}\|
 <
 \frac{\epsilon}{2}+2\sqrt{\gamma}
 <\epsilon.
\eneq

Finally, neither $D$ nor $d_{e_A}$ has been changed. Hence
\[
 t(d_{e_A})\geq\kappa,
 \qquad t\in T(D).
\]
Therefore,  together with \eqref{11.7-100f}, \eqref{11.7.53}, $A$ has property ${\rm ETR}1$ associated with $e_A$ and
the constant $\kappa$.
\end{proof}





\begin{prop}\label{prop:continuous-scale-improvement}
 For every $\eta>0$, there exists
$\eta_0>0$ with the following property. 
Let $A$ be a separable simple $C^*$-algebra with continuous
scale and ${\rm ETR}(A)\leq 1$.
If
$e_A\in A_+^1$ is any strictly positive contraction satisfying
\[
\tau(e_A)>1-\eta_0,\qquad \tau\in T(A),
\]
then $A$ has property ETR1 associated with $e_A$, and, in the
definition of property ETR1, the element $d_{e_A}$ may always be
chosen so that
\[
t(d_{e_A})>1-\eta,\qquad t\in T(D).
\]
\end{prop}

\begin{proof}
\medskip
\noindent\textbf{1. Initial reductions and choice of constants.}
\smallskip

It is enough to consider \(0<\eta<1\).  Put
$
   \eta_0=\frac{\eta}{32}
  $ and $
   \rho=\frac{\eta}{16}.
$

Let \(e_A\in A_+^1\) be strictly positive and suppose that
\beq
   \tau(e_A)>1-\eta_0,
   \qquad \tau\in T(A).
  \label{e11.55}
\eneq
Since \(A\) has continuous scale, it is algebraically simple (\cite[Theorem 3.3]{Lincs91}).
Since $ETR(A)\le 1,$ Proposition \ref {Pperm} (2), applied to a strictly positive element, shows that $A$ has property ETR1. 
By Corollary~6.6, \(A\) has property ETR1 associated with the prescribed strictly positive element \(e_A\)
and a constant
$
   \kappa=\kappa_{e_A}>0.
$

By Proposition~\ref{Pquas=trace}, every normalized \(2\)-quasitrace on \(A\)
is a trace.  Thus
$
   QT(A)=T(A).
$
Moreover, Corollary~\ref{Cstrictcomp} (3) shows that \(A\) has strict comparison.

Fix \(\varepsilon>0\), a finite self-adjoint subset
\(\mathcal F\subset A^{\bf 1}\) containing $e_A,$ and
\(a\in A_+\setminus\{0\}\).  Choose mutually orthogonal nonzero
positive elements
$
   a^{(0)},a_1,a_2\in\operatorname{Her}(a)
$
such that
\[
   a^{(0)}\oplus a_1\oplus a_2\lesssim a.
\]
Put
\beq
   \omega=\inf_{\tau\in T(A)}d_\tau(a_2)>0.
   \label{e11.57}
\eneq
Put
\beq
   \sigma=\kappa^2,\qquad
   \rho_0=
   \frac12\min\left\{
      \frac{\rho}{8},\frac{\sigma}{8},\frac18
   \right\}\andeqn\\
 \label{e11.58}
   0<\lambda<
   \min\left\{\frac18,\frac{\eta}{64},
                    \frac{\kappa}{16}\right\},
   \qquad
   C_\lambda=\frac1{1-2\lambda}.
  \label{e11.59}
\eneq
Define
\[
 f_\lambda(t)=
 \begin{cases}
  0, &0\leq t\leq\lambda,\\[0.3em]
  \dfrac{t-\lambda}{\lambda},
      &\lambda<t<2\lambda,\\[0.6em]
  1, &2\lambda\leq t\leq1,
 \end{cases}
\]
and
\[
   Q_\lambda(0)=0,\qquad
   Q_\lambda(t)=\frac{f_\lambda(t)}{t}\quad(t>0),
   \qquad
   R_\lambda=Q_\lambda^{1/2}.
\]
Then
\beq
 \|Q_\lambda\|\leq\frac1\lambda,\qquad
 Q_\lambda(t)t=f_\lambda(t),\qquad
 1-f_\lambda(t)\leq C_\lambda(1-t).
\label{e11.60}
\eneq

Choose
\beq
   0<\gamma<
   \min\left\{\frac18,\frac{\varepsilon^2}{512}\right\}
\andeqn
 0<\beta<
 \min\left\{
   \frac{\rho_0}{12},
   \frac{\rho_0\gamma^2\omega}{12},
   \frac18
 \right\}.
\label{e11.61}
\eneq
\medskip
\noindent\textbf{2. Auxiliary comparison elements and the large
cut-down \(e_0\).}
\smallskip

By the construction in the proof of Remark~2.11(2), applied inside
\(\operatorname{Her}(a^{(0)})\), there is a sequence 
$
   r_n\in\operatorname{Her}(a^{(0)})_+^{\bf 1}\setminus\{0\}
$
such that
\[
   \lim_{n\to\infty}
   \sup_{\tau\in T(A)}d_\tau(r_n)=0.
\]
Choose \(n\) such that
$ \sup_{\tau\in T(A)}d_\tau(r_n)<\beta,$ 
 and
put \(r_0=r_n\).  Then 
\beq
   \sup_{\tau\in T(A)}d_\tau(r_0)<\beta\andeqn r_0\lesssim a^{(0)}.
   \label{e11.62}
\eneq
In particular,
\[
   r_0\oplus a_1\oplus a_2\lesssim a.
\]

Choose a positive contraction \(h\) from an approximate identity of
\(A\) such that
\beq
   \tau(h^2)>1-\beta\rforal 
    \tau\in T(A).
\label{e11.63}
\eneq
Apply Lemma~\ref{lem11.1} with
$x=e_A$ and $y=h,$
with lower bound \(\sigma=\kappa^2\) and tolerance \(\rho\).  Let
\(\mathcal G\subset A^1\) and \(\delta_L>0\) be supplied by the
lemma. 

Put
\[
 \mathcal G_1={\cal G}\cup{\cal G}^*\cup \{e_A, h\}\cup \{z_1z_2:  z_1, z_2\in  {\cal G}\cup{\cal G}^*\cup\{e_A, h\}\}.
 \]

%

Define
\[
   g_\gamma(t)=(1-t-2\gamma)_+.
\]
Choose \(\chi>0\) such that
\[
   \|[b,z]\|<\chi
   \quad\Longrightarrow\quad
   \|[g_\gamma(b),z]\|<\varepsilon
\]
for contractions \(b,z\).  Choose \(\nu_Q>0\) such that
\[
   \|[d,z]\|<\nu_Q
   \quad\Longrightarrow\quad
   \|[Q_\lambda(d),z]\|<\frac{\chi}{2}
\]
whenever \(0\leq d\leq1\) and \(\|z\|\leq1\).

Choose \(\mu,\zeta>0\) satisfying
\[
 \mu,\zeta<
 \min\left\{
   \frac{\varepsilon\lambda}{256},
   \frac{\chi\lambda}{64},
   \frac{\nu_Q}{8},
   \frac{\delta_L}{8}
 \right\}.
\]
Choose \(\nu_{\mathrm{mix}}>0\) such that
$ \|[d,w]\|<\nu_{\mathrm{mix}}$
implies
\beq\label{vmix}
 \left\|
   \bigl((1-d)^{1/2}w(1-d)^{1/2}\bigr)^{1/2}
   -(1-d)^{1/2}w^{1/2}
 \right\|<\mu
\eneq
for positive contractions \(d,w\).

Let $v_n'=f_{2^{-n}}(e_A)$ and $v_n={\bar f}_{2^{-n}}(e_A)$ (see  Definition \ref{dfg}).
Then 
\beq\nonumber
v_n'v_{n+1}'=v_n',\,\,\, v_n'\le v_n\andeqn e_A\le v_n.
\eneq
Since $A$ has continuous scale, we may also assume that, for $n\ge n_0,$
\beq\nonumber
\tau(v_n)>1-\zeta\rforal \tau\in T(A).
\eneq
%
By Lemma~\ref{Lsequnececuntz},
there is \(n_1\in \N\)  such that
\[
   1_{Q_A}-j_A(v_n)\le j_{Q_A}-j_A(v_n')\lesssim j_A(a_1),
   \qquad n\geq n_1.
\]
Since $\{v_n\}$ is an approximate identity for $A,$ there  is $n_2\ge n_1$ 
such that, whenever \(n\geq n_2\) and
\(x\in\mathcal F\cup{\mathcal G}_1\),
\[
 \max\left\{
   \|(1-v_n)x\|,\ \|x(1-v_n)\|,\
   \|(1-v_n^{1/2})x\|,\
   \|x(1-v_n^{1/2})\|
 \right\}<\zeta.
\]
Put \(e_0=v_{\max\{n_0,n_1,n_2\}}\).  Then
\beq
 &&e_A\leq e_0,\,\,\, \tau(e_0)>1-\zeta\rforal \tau\in T(A),\\\label{e11.64a}
 &&\max_{x\in\mathcal F\cup\mathcal G}
 \left\{
   \|(1-e_0)x\|,\ \|x(1-e_0)\|,\
   \|(1-e_0^{1/2})x\|,\
   \|x(1-e_0^{1/2})\|
 \right\}<\zeta,\\
 &&\andeqn 1_{Q_A}-j_A(e_0)\lesssim j_A(a_1).
\label{e11.64}
\eneq

For \(\xi>0\), put
\beq\label{e11.64+}
   \Delta(\xi)=2\sqrt{2\xi}+2\xi,
   \qquad
   \Xi(\xi)=\Delta(\xi)+\xi.
\eneq
Choose \(\xi>0\) such that
\[
 \xi<
 \min\left\{
   \beta,\frac{\delta_L}{4},
   \frac{\gamma}{2C_\lambda}, \zeta,
   \nu_{\mathrm{mix}}
 \right\},
\]
and
\beq
 \begin{split}
 &\Delta(\xi)<
   \min\left\{\frac{\eta}{16},\frac{\kappa}{4}\right\},\,\,\,
 \Xi(\xi)+2\sqrt{\xi}+2\zeta+2\mu
   <\frac{\varepsilon\lambda}{8},\\
 &\Xi(\xi)+\zeta<\frac{\chi\lambda}{8},\,\,\,\,\,\,\,\hspace{0.3in}
 2\zeta+6\Xi(\xi)<\nu_Q.
 \end{split}
\label{e11.65}
\eneq
\medskip
\noindent\textbf{3. Application of ETR1 and decomposition into matrix
summands.}
\smallskip

Put
\[
   \mathcal H=\mathcal F\cup\mathcal G_1\cup\{e_0\}.
\]
Using Remark \ref{R1} (ii) and (iii), apply property ETR1 with \(\xi/8\) (in place of $\ep$), finite set
\(\mathcal H\) (in place of ${\cal F}$), and comparison element \(r_0\) (in place of $a$).  We obtain 
$e\in A_+^{\bf 1},$  a finite-dimensional \(C^*\)-algebra
$
   D^{(0)}=\bigoplus_{j=1}^mM_{r_j},
$
and a fully normalized order-zero c.p.c.\ map
\[
   \theta_0:D^{(0)}\longrightarrow  \overline{(1-e)A(1-e)}
\]
such that
\beq\label{e117-1/2}
   e\lesssim r_0,\,\,\,\,  \|[e,\, x]\|<\frac{\xi}{8} \andeqn
   \|[(1-e)^{1/2},\,x]\|<\frac{\xi}{8},
  \rforal x\in\mathcal H.
\eneq
Moreover,  for \(x\in\mathcal H\), we obtain a
contraction \(d_x^{(0)}\in {D^{(0)}}\) such that  (see (iii) of Remark \ref{R1})
\beq\label{e11.66}
   \|\theta_0(d_x^{(0)}) - (1-e)^{1/2}x(1-e)^{1/2}\|<\xi.
\eneq
We 
may assume that, by Remark \ref{R1} (iv), if $x\in {\cal H}_+,$   $d_x=d_x^{(0)}\ge 0.$   Furthermore,
\beq\label{e11.67}
   t(d_{e_A}^{(0)})\geq\kappa,
   \rforal  t\in T(D^{(0)}).
\eneq

Put \(s=(1-e)^{1/2}\).

For \(1\leq j\leq m\), let
$
 \theta_{0,j}
 =
 \theta_0|_{M_{r_j}},
$
where \(M_{r_j}\) is identified with the \(j\)-th summand of
\(D^{(0)}\).  Let
\[
 h_j=\theta_{0,j}(1_{M_{r_j}}),
\andeqn
 \pi_j:M_{r_j}\longrightarrow A^{**}
\]
be the supporting homomorphism of \(\theta_{0,j}\).  Thus
\[
 \theta_{0,j}(x)
 =
 \pi_j(x)h_j
 =
 h_j\pi_j(x)\rforal  x\in M_{r_j}.
\]
By \eqref{e11.65} and \eqref{e117-1/2},
\[
   \|[s,z]\|
   <
   \min\left\{\zeta,\frac{\delta_L}{4}\right\},
   \rforal  z\in\mathcal G.
\]
For \(z_1,z_2\in\mathcal G\),
\beq
\begin{aligned}
 \|s^2z_1z_2s^2-(sz_1s)(sz_2s)\|
 &=&
 \left\|
   s\bigl(sz_1z_2s-z_1s^2z_2\bigr)s
 \right\|\\
 &\leq &
 \|[s,z_1]\|+\|[s,z_2]\|
 <\delta_L.
\end{aligned}
\label{e11.68}
\eneq
\medskip
\noindent\textbf{4. Passage to \(Q_A\) and separation of good and bad
summands.}
\smallskip

We regard \(A\) as \(j_A(A)\subset Q_A\).  Put
\[
   q_0=j_A(e),
   \qquad
   s=(1_{Q_A}-q_0)^{1/2},
\]
and regard
\[
   j_A\circ\theta_0:
   D^{(0)}
   \longrightarrow
   \operatorname{Her}_{Q_A}(1_{Q_A}-q_0)
\]
as the order-zero map occurring in Lemma~\ref{lem11.1}.
With these
identifications, the preceding estimates are precisely \eqref{e11.1}--\eqref{e11.4}.

Let
\[
 I_b=
 \left\{
 j:
 \operatorname{tr}_{r_j}
 \bigl((d_h^{(0)})_j^2\bigr)
 \leq1-\rho_0
 \right\},
 \qquad
 I_g=\{1,\ldots,m\}\setminus I_b.
\]
Put
\beq
   z_{\mathrm b}
   =\bigoplus_{j\in I_b}1_{M_{r_j}}\andeqn
   H_{\mathrm b}=\theta_0(z_{\mathrm b}).
 \label{e11.69}
\eneq

Fix \(\tau\in T(A)\).  Since \(e\lesssim r_0\),
\[
   \tau(e)\leq d_\tau(e)\leq d_\tau(r_0)<\beta.
\]
Since \(0\leq h\leq1\),
\[
\begin{aligned}
 \tau(shs)
 &=\tau(h(1-e))
 =\tau(h)-\tau(he)\\
 &\geq\tau(h^2)-\tau(e)
 >1-2\beta.
\end{aligned}
\]
Hence
\beq
   \tau((shs)^2)
   \geq\tau(shs)^2
   >(1-2\beta)^2
   >1-4\beta.
   \label{e11.70}
\eneq
By \eqref{e11.66},
\[
   \|\theta_0(d_h^{(0)})-shs\|<\xi.
\]
Both elements are positive contractions, so
\[
   \|\theta_0(d_h^{(0)})^2-(shs)^2\|<2\xi.
\]
Therefore,
\beq
   \tau(\theta_0(d_h^{(0)})^2)
   >1-4\beta-2\xi
   >1-6\beta.
 \label{e11.71}
\eneq

For \(j=1,\ldots,m\), put
\[
   \mu_{\tau,j}
   =\tau\bigl(\theta_0(1_{M_{r_j}})^2\bigr).
\]
Then
\[
   \sum_{j=1}^m\mu_{\tau,j}
   =\tau\bigl(\theta_0(1_{D^{(0)}})^2\bigr)
   \leq1.
\]
Since  
   $\theta_0(x)=h_j\pi_j(x)$ 
on the \(j\)-th summand, the functional
\[
   x\longmapsto\tau(h_j^2\pi_j(x))
\]
is a trace on \(M_{r_j}\).  Thus
\[
 \tau(\theta_0(d_h^{(0)})^2)
 =
 \sum_{j=1}^m
 \mu_{\tau,j}
 \operatorname{tr}_{r_j}
 \bigl((d_h^{(0)})_j^2\bigr).
\]
Consequently,
\[
\begin{aligned}
 \tau(\theta_0(d_h^{(0)})^2)
 &\leq
 \sum_{j\in I_g}\mu_{\tau,j}
 +(1-\rho_0)\sum_{j\in I_b}\mu_{\tau,j}
 \leq
 1-\rho_0\sum_{j\in I_b}\mu_{\tau,j}.
\end{aligned}
\]
Together with \eqref{e11.71}, this gives
\[
   \sum_{j\in I_b}\mu_{\tau,j}
   <\frac{6\beta}{\rho_0}.
\]
Since the images of distinct summands are orthogonal,
\[
   H_{\mathrm b}^2
   =
   \sum_{j\in I_b}
   \theta_0(1_{M_{r_j}})^2.
\]
Therefore,
\beq
   \tau(H_{\mathrm b}^2)
   <\frac{6\beta}{\rho_0}\rforal \tau\in T(A).
   \label{e11.72}
\eneq

The set \(\mathcal I_g\) is nonempty.  Otherwise,
\[
   \tau(\theta_0(d_h^{(0)})^2)
   \leq1-\rho_0,
\]
contradicting \eqref{e11.71}, since \(6\beta<\rho_0\).

For every \(j\),
\[
 \operatorname{tr}_{r_j}
 \bigl((d_{e_A}^{(0)})_j^2\bigr)
 \geq
 \operatorname{tr}_{r_j}
 \bigl((d_{e_A}^{(0)})_j\bigr)^2
 \geq\kappa^2=\sigma.
\]
If \(j\in I_g\), Lemma~\ref{lem11.1}
gives \(\tau_j\in T(A)\)
such that
\[
 \left|
 \operatorname{tr}_{r_j}
 \bigl((d_{e_A}^{(0)})_j^2\bigr)
 -\tau_j(e_A^2)
 \right|<\rho.
\]
By \eqref{e11.55},
\[
   \tau_j(e_A^2)
   \geq\tau_j(e_A)^2
   >(1-\eta_0)^2
   >1-2\eta_0.
\]
Hence
\beq\label{e11.73}
 \operatorname{tr}_{r_j}
 \bigl((d_{e_A}^{(0)})_j^2\bigr)
 >
 1-2\eta_0-\rho\rforal  j\in I_g.
\eneq

\medskip
\noindent\textbf{5. Spectral cut-down of the good summands.}
\smallskip

Let
\[
 C_{\theta_0}
 =
 \widetilde\theta_0
 \bigl(C_0((0,1])\otimes D^{(0)}\bigr).
\]
Since \(e_A\leq e_0\),
\[
   se_As\leq se_0s.
\]
Hence, together with \eqref{e11.66},
applying  Lemma~\ref{Ldaxiao}
in \(C_{\theta_0}\),  
we obtain \(c\in C_{\theta_0}\) such that 
\[
   0\leq c\leq\theta_0(d_{e_0}^{(0)})
\andeqn
   \|c-se_As\|<2\sqrt{2\xi}+\xi.
\]
Let \(\pi_1:C_{\theta_0}\to D^{(0)}\) be evaluation at \(1\), and
put
$
   \widetilde d_{e_A}^{(0)}=\pi_1(c).
$
Then  (recall \eqref{e11.64+})
\beq\nonumber
&&   0\leq\widetilde d_{e_A}^{(0)}
   \leq d_{e_0}^{(0)},
\,\,\,  \|\widetilde d_{e_A}^{(0)}
   -d_{e_A}^{(0)}\|<\Delta(\xi),\\
   &&  \|({\widetilde d_{e_A}}^{(0)})^2
   -(d_{e_A}^{(0)})^2\|<2\Delta(\xi),
\andeqn
   \|\theta_0(\widetilde d_{e_A}^{(0)})-se_As\|
   <\Xi(\xi).
  \label{e11.74}
\eneq
For every \(1\leq j\leq m\), by \eqref{e11.67}
and
\eqref{e11.74},
\[
 \begin{aligned}
 \operatorname{tr}_{r_j}
 \bigl((\widetilde d_{e_A}^{(0)})_j\bigr)
 &\geq
 \operatorname{tr}_{r_j}
 \bigl((d_{e_A}^{(0)})_j\bigr)
 -\Delta(\xi)
 \geq \kappa-\Delta(\xi).
 \end{aligned}
\]
If \(j\in I_{\mathrm g}\), then, by \eqref{e11.74} and \eqref{e11.73},
\[
 \begin{aligned}
 \operatorname{tr}_{r_j}
 \bigl((\widetilde d_{e_A}^{(0)})_j\bigr)
 &\geq
 \operatorname{tr}_{r_j}
 \bigl((d_{e_A}^{(0)})_j\bigr)
 -\Delta(\xi)\\
 &\geq
 \operatorname{tr}_{r_j}
 \bigl((d_{e_A}^{(0)})_j^2\bigr)
 -\Delta(\xi)
 >
 1-2\eta_0-\rho-2\Delta(\xi).
 \end{aligned}
\]

For \(j=1,\ldots,m\), put
\beq
 &&p_j=
  1_{(\lambda,1]}
  \bigl((d_{e_0}^{(0)})_j\bigr),\,\,\,
 Q_j=
  Q_\lambda\bigl((d_{e_0}^{(0)})_j\bigr),\\
 &&R_j=Q_j^{1/2},\,\,\,
 u_j=
  f_\lambda\bigl((d_{e_0}^{(0)})_j\bigr).
\eneq
Note that
\beq\label{e1174+}
p_jR_j=R_jp_j=R_j.
\eneq
Set
\beq
   Q=\bigoplus_{j=1}^mQ_j,\qquad
   R=\bigoplus_{j=1}^mR_j,\qquad
   u=\bigoplus_{j=1}^mu_j.
   \label{e11.75}
\eneq
For \(x\in\mathcal F\), put
\[
 a_x=
 \begin{cases}
 d_x^{(0)},&x\neq e_A,\\
 \widetilde d_{e_A}^{(0)},&x=e_A.
 \end{cases}
\]
Then $0\le a_{e_A}=\widetilde{d_{e_A}^{(0)}}\le d_{e_0}^{(0)}.$
Let
\[
   P_j=
   1_{[2\lambda,1]}
   \bigl((d_{e_0}^{(0)})_j\bigr).
\]
If \(0\leq d\leq(d_{e_0}^{(0)})_j\), then
\[
\begin{aligned}
 \operatorname{tr}_{r_j}(R_jdR_j)
 &=\operatorname{tr}_{r_j}(Q_jd)
 \geq\operatorname{tr}_{r_j}(P_jd)\\
 &=\operatorname{tr}_{r_j}(d)
   -\operatorname{tr}_{r_j}((1-P_j)d).
\end{aligned}
\]
Moreover,
\[
\begin{aligned}
 \operatorname{tr}_{r_j}((1-P_j)d)
 &=
 \operatorname{tr}_{r_j}((1-P_j)d(1-P_j))
 \leq2\lambda.
\end{aligned}
\]
Hence
\beq
   \operatorname{tr}_{r_j}(R_jdR_j)
   \geq\operatorname{tr}_{r_j}(d)-2\lambda.
  \label{e11.76}
\eneq
Applying \eqref{e11.76} to
\(d=(\widetilde d_{e_A}^{(0)})_j\), we obtain, for every \(j\),
\[
 \begin{aligned}
 \operatorname{tr}_{r_j}
 \bigl(R_j(\widetilde d_{e_A}^{(0)})_jR_j\bigr)
 &\geq
 \operatorname{tr}_{r_j}
 \bigl((\widetilde d_{e_A}^{(0)})_j\bigr)-2\lambda\\
 &\geq
 \kappa-\Delta(\xi)-2\lambda
 >\frac{\kappa}{2}.
 \end{aligned}
\]
Since \(p_jR_j=R_j\), this implies \(p_j\neq0\).
If \(j\in I_g\), then, by \eqref{e11.73}, \eqref{e11.74} and \eqref{e11.76},
\beq
 \operatorname{tr}_{r_j}
 (R_j(\widetilde d_{e_A}^{(0)})_jR_j)
 >
 1-2\eta_0-\rho-2\Delta(\xi)-2\lambda
 >
 1-\eta.
\label{e11.78}
\eneq
Let ${\rm tr}_{p_jM_{r_j}p_j}$ be the normalized trace on 
 \(p_jM_{r_j}p_j\). Then
 \beq\label{e11.78a}
 {\rm tr}_{p_jM_{r_j}p_j}(R_j(\widetilde d_{e_A}^{(0)})_jR_j)=
 \frac{
 \operatorname{tr}_{r_j}
 \bigl(R_j(\widetilde d_{e_A}^{(0)})_jR_j\bigr)
 }{
 \operatorname{tr}_{r_j}(p_j)
 }>1-\eta.
 \eneq

Put (recall \eqref{e11.75})
\beq
 &&  z_{\mathrm g}
   =\bigoplus_{j\in I_g}1_{M_{r_j}},
   \qquad
   z_{\mathrm b}=1-z_{\mathrm g},\\
&& D^{\mathrm g}
   =\bigoplus_{j\in I_g}p_jM_{r_j}p_j,\\
&& Q_{\mathrm g}=z_{\mathrm g}Q,\,\,\,
 R_{\mathrm g}=z_{\mathrm g}R,\,\,\,
 u_{\mathrm g}=z_{\mathrm g}u,\,\,\,
 u_{\mathrm b}=z_{\mathrm b}u=(1-z_{\mathrm g})u.
\label{e11.79}
\eneq
Let
\[
   \Theta_A^{\mathrm g}
   =\theta_0|_{D^{\mathrm g}},
   \qquad
   \Theta^{\mathrm g}
   =j_A\circ\Theta_A^{\mathrm g}.
\]
Since $\theta_0$ is fully normalized  and $p_j\not=0,$ 
\(\Theta_A^{\mathrm g}={\theta_0|_{D_g}}\) is fully normalized.

Put
\beq\label{AAA}
   b_{\mathrm g}=\Theta^{\mathrm g}(u_{\mathrm g}),
   \qquad
   b_{\mathrm b}=j_A(\theta_0(u_{\mathrm b})),
   \qquad
   b=b_{\mathrm g}+b_{\mathrm b}.
\eneq
Note that
\beq\label{BBB}
b=b_{\mathrm g}+b_{\mathrm b}=j_A(\theta_0(u_{\mathrm g}))+j_A(\theta_0(u_{\mathrm b}))=j_A(\theta_0(f_\lambda(d_{e_0}^{(0)}))). 
\eneq

\medskip
\noindent\textbf{6. Approximation and commutator estimates on the good
part.}
\smallskip

For \(x\in\mathcal F\), put
\[
   d'_{x,\mathrm g}=R_{\mathrm g}a_xR_{\mathrm g}.
\]
Let \(\pi_\theta\) be the supporting homomorphism.  Put
\beq
   E_{\mathrm{app}}
   =
   \frac{\Xi(\xi)+2\sqrt{\xi}+2\zeta+2\mu}{\lambda}
   <\frac{\varepsilon}{8}.
   \label{e11.80}
\eneq
By  \eqref{vmix} and \eqref{e11.64a} and the choices of \(\nu_{\mathrm{mix}}\), \(e_0\), and \(\xi\),  we obtain
\beq\nonumber
 &&  \|(sj_A(e_0)s)^{1/2}-sj_A(e_0)^{1/2}\|<\mu,\\
 &&  \|sj_A(x)s-(sj_A(e_0)s)^{1/2}j_A(x)(sj_A(e_0)s)^{1/2}\|<2\zeta+2\mu,\andeqn\\
&&\| j_A(\theta_0(d_{e_0}^{(0)}))^{1/2}-(sj_A(e_0)s)^{1/2}
   \|<\sqrt{\xi}.
\eneq
By the definition of \(a_x\), \eqref{e11.66}, and \eqref{e11.74},
\[
 \bigl\|
 j_A(\theta_0(a_x))-sj_A(x)s
 \bigr\|<\Xi(\xi),
 \qquad x\in\mathcal F.
\]
Combining these estimates, we obtain
\[
 \begin{aligned}
 &\bigl\|
 j_A(\theta_0(a_x))
 -j_A(\theta_0(d_{e_0}^{(0)}))^{1/2}
  j_A(x)
  j_A(\theta_0(d_{e_0}^{(0)}))^{1/2}
 \bigr\|\\
 &\hspace{25mm}
 <\Xi(\xi)+2\sqrt{\xi}+2\zeta+2\mu.
 \end{aligned}
\]

Moreover, 
\beq
 \begin{split}
 \Theta^{\mathrm g}(d'_{x,\mathrm g})
 =
 \pi_\theta(R_{\mathrm g})
 j_A(\theta_0(a_x))
 \pi_\theta(R_{\mathrm g})\andeqn 
 b_{\mathrm g}^{1/2}
 =
 \pi_\theta(R_{\mathrm g})
 j_A(\theta_0(d_{e_0}^{(0)}))^{1/2}.
 \end{split}
\label{e11.81}
\eneq

Since
\[
   \|\pi_\theta(R_{\mathrm g})\|^2\leq\frac1\lambda,
\]
we also obtain

\beq
 \left\|
 \Theta^{\mathrm g}(d'_{x,\mathrm g})
 -
 b_{\mathrm g}^{1/2}j_A(x)b_{\mathrm g}^{1/2}
 \right\|=\| \pi_\theta(R_{\mathrm g})
 j_A(\theta_0(a_x))
 \pi_\theta(R_{\mathrm g})-b_{\mathrm g}^{1/2}j_A(x)b_{\mathrm g}^{1/2}\|\\
 \le ({1\over{\lambda}})\|\bigl\|
 j_A(\theta_0(a_x))
 -j_A(\theta_0(d_{e_0}^{(0)}))^{1/2}
  j_A(x)
  j_A(\theta_0(d_{e_0}^{(0)}))^{1/2}
 \bigr\|
 <E_{\mathrm{app}}.
 \label{e11.82}
\eneq
By \eqref{e11.64a}
and \eqref{e117-1/2},
\[
 \begin{aligned}
 \|e_0(1-e)x-x(1-e)\|
 &\leq
 \|[e_0,\,(1-e)]\|+\|e_0x-x\|+\|[(1-e),\,x]\|
 <
 \frac{\xi}{8}+\zeta+\frac{\xi}{8},
 \end{aligned}
\]
and
\[
 \|x(1-e)e_0-x(1-e)\|
 \leq
 \|[(1-e),e_0]\|+\|xe_0-x\|
 <
 \zeta+\frac{\xi}{8}.
\]
Therefore,
\beq
 \begin{aligned}
 \|[se_0s,\,sxs]\|
 &=
 \|s(e_0(1-e)x-x(1-e)e_0)s\|
 <
 2\zeta+\frac{3\xi}{8}.
 \end{aligned}\label{e11.94a}
\eneq
By\eqref{e11.66},\eqref{e11.94a},
\eqref{e11.74}, and the definition of \(a_x\), we estimate 

\beq\nonumber
 &&\hspace{-0.4in}\|\theta_0(d_{e_0}^{(0)})\theta_0(a_x)
 -\theta_0(a_x)\theta_0(d_{e_0}^{(0)})\|
 =
\| \bigl(\theta_0(d_{e_0}^{(0)})-se_0s\bigr)\theta_0(a_x)\|
 +
\| se_0s\bigl(\theta_0(a_x)-sxs\bigr)\|\\\nonumber
&&\hspace{1.8in}+
 \|\bigl(se_0s\,sxs-sxs\,se_0s\bigr)\|
 +
\| \bigl(sxs-\theta_0(a_x)\bigr)se_0s\|\\\nonumber
 &&\hspace{2.2in}+
 \|\theta_0(a_x)
 \bigl(se_0s-\theta_0(d_{e_0}^{(0)})\bigr)\|\\\nonumber
 &&\hspace{1.8in}
  \leq 2\bigl\|\theta_0(d_{e_0}^{(0)})-se_0s\bigr\|
 \\\nonumber
 &&\hspace{2in}+ \|se_0s\,sxs-sxs\,se_0s\|
 +
 2\bigl\|\theta_0(a_x)-sxs\bigr\|\\\nonumber
 &&\hspace{1.8in}< 2\xi+
 (2\zeta+\frac{3\xi}{8})+2\Xi(\xi)<
 2\zeta+6\Xi(\xi)
 <\nu_Q.
\eneq
Since \(\theta_0\) is fully normalized, by Lemma \ref{Lproduct},
 we have
\beq \label{e11.83}
 \begin{aligned}
 \|[d_{e_0}^{(0)},a_x]\|
 &=
 \|[\theta_0(d_{e_0}^{(0)}),\theta_0(a_x)]\|
 <\nu_Q.
 \end{aligned}                                            
\eneq

Since \(z_{\mathrm g}\) is central in \(D^{(0)}\),
\[
   \|[Q_{\mathrm g},a_x]\|<\frac{\chi}{2}.
\]
Also,
\[
 \|
 j_A(\theta_0(d_{e_0}^{(0)}))j_A(x)
 -j_A(\theta_0(a_x))
 \|
 <2\Xi(\xi)+2\zeta,
\]
and, analogously,
\[
 \bigl\|
 j_A(x)j_A(\theta_0(d_{e_0}^{(0)}))
 -j_A(\theta_0(a_x))
 \bigr\|
 <2\Xi(\xi)+2\zeta.
\]

Since
\[
 b_{\mathrm g}
 =
 \pi_\theta(Q_{\mathrm g})
 j_A(\theta_0(d_{e_0}^{(0)}))
 =
 j_A(\theta_0(d_{e_0}^{(0)}))
 \pi_\theta(Q_{\mathrm g}),
\]
we have
\beq
 \begin{aligned}
{ [b_{\mathrm g},j_A(x)]}
 ={}&
 \pi_\theta(Q_{\mathrm g})
 \bigl(
 j_A(\theta_0(d_{e_0}^{(0)}))j_A(x)
 -j_A(\theta_0(a_x))
 \bigr)\\
 &+
 [\pi_\theta(Q_{\mathrm g}),j_A(\theta_0(a_x))]\\
 &+
 \bigl(
 j_A(\theta_0(a_x))
 -j_A(x)j_A(\theta_0(d_{e_0}^{(0)}))
 \bigr)
 \pi_\theta(Q_{\mathrm g}).
\end{aligned}
\eneq
Therefore,  by \eqref{e11.65},
\beq
   \|[b_{\mathrm g},j_A(x)]\|
   <
   \frac{4(\Xi(\xi)+\zeta)}{\lambda}
   +\frac{\chi}{2}
   <\chi.
 \label{e11.84}
\eneq

\medskip
\noindent\textbf{7. Construction of \(q_{\mathrm g}\) and Cuntz
comparison.}
\smallskip

%
Define
\beq\nonumber
   q_{\mathrm g}
   =(1_{Q_A}-b_{\mathrm g}-2\gamma)_+,
   \qquad
   c_{\mathrm g}=1_{Q_A}-q_{\mathrm g}.
\eneq
Then
\beq\label{e11.84a}
   \|[q_{\mathrm g},j_A(x)]\|<\varepsilon.
\eneq
Moreover,
\[
   0\leq c_{\mathrm g}-b_{\mathrm g}\leq2\gamma,
   \qquad
   \|c_{\mathrm g}^{1/2}-b_{\mathrm g}^{1/2}\|
   \leq\sqrt{2\gamma}.
\]
Thus
\beq
 \left\|
 \Theta^{\mathrm g}(d'_{x,\mathrm g})
 -
 c_{\mathrm g}^{1/2}j_A(x)c_{\mathrm g}^{1/2}
 \right\|
 <
 E_{\mathrm{app}}+2\sqrt{2\gamma}.
\label{e11.85}
\eneq

Let
\[
   q_0=j_A(e),\qquad c=1_{Q_A}-q_0,
\andeqn
   q_1=1_{Q_A}-c^{1/2}j_A(e_0)c^{1/2}.
\]
Then
\[
   q_1
   =
   q_0+c^{1/2}(1_{Q_A}-j_A(e_0))c^{1/2},
\]
and hence
\beq
   q_1\lesssim j_A(r_0)\oplus j_A(a_1).
   \label{e11.86}
\eneq

Put
\[
   q:=(1_{Q_A}-b-\gamma)_+.
\]

Recall \eqref{AAA}.  By \eqref{BBB}, 
 \eqref{e11.60}, the 
positivity of \(j_A\circ\theta_0\), together with \(C_\lambda\geq1\),
we have
\[
\begin{aligned}
1_{Q_A}-b
&=1_{Q_A}-j_A(\theta_0(1_{D^{(0)}}))
  +j_A\!\left(
    \theta_0\!\left(
      1_{D^{(0)}}-f_\lambda(d_{e_0}^{(0)})
    \right)
  \right)\\
&\leq
  1_{Q_A}-j_A(\theta_0(1_{D^{(0)}}))
  +C_\lambda j_A\!\left(
    \theta_0\!\left(
      1_{D^{(0)}}-d_{e_0}^{(0)}
    \right)
  \right)\\
&\leq
  C_\lambda\left(
    1_{Q_A}-j_A(\theta_0(1_{D^{(0)}}))
  \right)
  +C_\lambda j_A\!\left(
    \theta_0\!\left(
      1_{D^{(0)}}-d_{e_0}^{(0)}
    \right)
  \right)\\
&=
C_\lambda\left(
1_{Q_A}-j_A\bigl(\theta_0(d_{e_0}^{(0)})\bigr)
\right).
\end{aligned}
\]
Therefore,
\[
q=(1_{Q_A}-b-\gamma)_+
 \lesssim (C_\lambda \left(
1_{Q_A}-j_A\bigl(\theta_0(d_{e_0}^{(0)})\bigr)
\right)-\gamma)_+
 \sim
 \left(\left(
1_{Q_A}-j_A\bigl(\theta_0(d_{e_0}^{(0)})\bigr)
\right)-\frac{\gamma}{C_\lambda}\right)_+.
\]
On the other hand,  by \eqref{e11.66},
\[
\begin{aligned}
\|(1_{Q_A}-j_A\bigl(\theta_0(d_{e_0}^{(0)})\bigr))-q_1\|
&=
\left\|
j_A\bigl(\theta_0(d_{e_0}^{(0)})\bigr)
-c^{1/2}j_A(e_0)c^{1/2}
\right\|  
<\xi
 <\frac{\gamma}{2C_\lambda}.
\end{aligned}
\]
Hence, by \cite[Proposition 2.2]{RrII},
\[
\left((1_{Q_A}-j_A\bigl(\theta_0(d_{e_0}^{(0)})\bigr))-\frac{\gamma}{C_\lambda}\right)_+
\lesssim
\left(q_1-\frac{\gamma}{2C_\lambda}\right)_+
\lesssim q_1.
\]
Consequently,
\beq\label{e11.87}
q\lesssim q_1.
\eneq
By Proposition~\ref{Pquas=trace}, \(QT(A)=T(A)\).  Hence, for every normalized
\(2\)-quasitrace \(\tau\),
\[
\begin{aligned}
 d_\tau((H_{\mathrm b}-\gamma)_+)
 &\leq
 \frac{\tau(H_{\mathrm b}^2)}{\gamma^2}
 <
 \frac{6\beta}{\rho_0\gamma^2}
 <\omega
 \leq d_\tau(a_2).
\end{aligned}
\]
Strict comparison gives
\[
   (H_{\mathrm b}-\gamma)_+\lesssim a_2.
\]
Since \(u_{\mathrm b}\leq z_{\mathrm b}\) (see \eqref{e11.79}),
\[
   b_{\mathrm b}\leq j_A(H_{\mathrm b}),
\]
and therefore
\beq
   (b_{\mathrm b}-\gamma)_+
   \lesssim
   j_A((H_{\mathrm b}-\gamma)_+)
   \lesssim j_A(a_2).
   \label{e11.88}
\eneq

Finally,
\[
   1_{Q_A}-b_{\mathrm g}-2\gamma
   =
   (1_{Q_A}-b-\gamma)+(b_{\mathrm b}-\gamma).
\]
Using
\[
   (x+y)_+\lesssim x_+\oplus y_+
\]
for self-adjoint \(x,y\), we obtain
\beq
\begin{aligned}
 q_{\mathrm g}=(1_{Q_A}-b_{\mathrm g}-2\gamma)_+
 &\lesssim
 q\oplus(b_{\mathrm b}-\gamma)_+
 \lesssim
 j_A(r_0)\oplus j_A(a_1)\oplus j_A(a_2)
 \lesssim j_A(a).
\end{aligned}
\label{e11.89}
\eneq

\medskip
\noindent\textbf{8. Normalization and verification of Lemma~11.2.}
\smallskip

Since \(u_{\mathrm g}\) is strictly positive in \(D^{\mathrm g}\),
\[
   \Theta^{\mathrm g}(D^{\mathrm g})
   \subset
   \operatorname{Her}_{Q_A}(b_{\mathrm g})
   \subset
   \operatorname{Her}_{Q_A}(c_{\mathrm g}).
\]
For \(x\neq e_A\), put
\[
   \widehat d_x
   =
   \frac{d'_{x,\mathrm g}}
        {\max\{1,\|d'_{x,\mathrm g}\|\}}\andeqn
   \widehat d_{e_A}=d'_{e_A,\mathrm g}.
\]
In particular, $\widehat{d_{e_A}}=d_{e_A,g}'.$

Recall that $Rd_{e_0}^{(0)}R=u=f_\lambda(d_{e_0}^{(0)})\le 1.$
Since \(a_{e_A}=\widetilde d_{e_A}^{(0)}\), \eqref{e11.79} and \eqref{e11.75}
give
\[
 \begin{aligned}
 0
 \leq d'_{e_A,\mathrm g}
 &=
 R_{\mathrm g}
 \widetilde d_{e_A}^{(0)}
 R_{\mathrm g}
 \leq
 R_{\mathrm g}d_{e_0}^{(0)}R_{\mathrm g}
 =u_{\mathrm g}
 \leq1.
 \end{aligned}
\]

Since \(\Theta_A^{\mathrm g}\) is fully normalized, it is isometric.
By \eqref{e11.82}
\[
 \begin{aligned}
 \|d'_{x,\mathrm g}\|
 &=
 \|\Theta^{\mathrm g}(d'_{x,\mathrm g})\|
 \leq
 \|b_{\mathrm g}^{1/2}j_A(x)b_{\mathrm g}^{1/2}\|
 +E_{\mathrm{app}}
 \leq1+E_{\mathrm{app}}.
 \end{aligned}
\]
Suppose first that \(x\neq e_A\). If
\(\|d'_{x,\mathrm g}\|\leq1\), then
$
 \widehat d_x=d'_{x,\mathrm g}.
$
If \(\|d'_{x,\mathrm g}\|>1\), then
\[
 \begin{aligned}
 \|\widehat d_x-d'_{x,\mathrm g}\|
 &=
 \left\|
 \frac{d'_{x,\mathrm g}}{\|d'_{x,\mathrm g}\|}
 -d'_{x,\mathrm g}
 \right\|
 =
 \|d'_{x,\mathrm g}\|-1
 \leq E_{\mathrm{app}}.
 \end{aligned}
\]
Thus, in all  cases,  for $x\in {\cal F},$ 
\[
 \|\widehat d_x-d'_{x,\mathrm g}\|
 \leq E_{\mathrm{app}}.
\]
%
Consequently, by \eqref{e11.85}, for all $x\in {\cal F},$ 
\beq
 \begin{aligned}
 &\hspace{-0.2in}\bigl\|
 \Theta^{\mathrm g}(\widehat d_x)
 -c_{\mathrm g}^{1/2}j_A(x)c_{\mathrm g}^{1/2}
 \bigr\|\\
 &\leq
 \bigl\|
 \Theta^{\mathrm g}(\widehat d_x-d'_{x,\mathrm g})
 \bigr\|
 +
 \bigl\|
 \Theta^{\mathrm g}(d'_{x,\mathrm g})
 -c_{\mathrm g}^{1/2}j_A(x)c_{\mathrm g}^{1/2}
 \bigr\|\\
 &<
 E_{\mathrm{app}}
 +
 \bigl(E_{\mathrm{app}}+2\sqrt{2\gamma}\bigr)
 =
 2E_{\mathrm{app}}+2\sqrt{2\gamma}
 <\varepsilon.
 \end{aligned}   \label{e11.103}                                             
\eneq

For every \(j\in I_g\), \eqref{e11.78} gives
\[
   \operatorname{tr}_{p_jM_{r_j}p_j}
   ((\widehat d_{e_A})_j)>1-\eta.
\]
Every tracial state on \(D^{\mathrm g}\) is a convex combination of
these normalized matrix traces.  Therefore,
\beq
   t(\widehat d_{e_A})>1-\eta,
   \qquad t\in T(D^{\mathrm g}).
  \label{e11.91}
\eneq
We apply Lemma~\ref{cor-constant-lift}
with
\[
 q=q_{\mathrm g},\qquad
 D=D_{\mathrm g},\qquad
 \Theta_A=\Theta_A^{\mathrm g},\qquad
 d_x=\widehat d_x\quad (x\in\mathcal F),
 \qquad
 \kappa=1-\eta,
\]
and with
\[
 \Theta^{\mathrm g}
 =
 j_A\circ\Theta_A^{\mathrm g}
 \colon D_{\mathrm g}\longrightarrow Q_A.
\]
The algebra \(D_{\mathrm g}\) is finite-dimensional, and
\(\Theta_A^{\mathrm g}\) is a fully normalized order-zero c.p.c.\ map.
By their definitions, the elements \(\widehat d_x\),
\(x\in\mathcal F\), are contractions, and, by \eqref{e11.74},
\[
 \widehat d_{e_A}
 =
 d'_{e_A,\mathrm g}
 =
 R_{\mathrm g}\widetilde d_{e_A}^{(0)}R_{\mathrm g}
\]
is a positive contraction.

Since
$
 q_{\mathrm g}
 =
 (1_{Q_A}-b_{\mathrm g}-2\gamma)_+
 \leq
 (1-2\gamma)1_{Q_A},
$
we have
\[
 c_{\mathrm g}
 =
 1_{Q_A}-q_{\mathrm g}
 \geq
 2\gamma1_{Q_A}.
\]
Thus \(c_{\mathrm g}\) is invertible in \(Q_A\), and consequently
\[
 \operatorname{Her}_{Q_A}(c_{\mathrm g})
 =
 Q_A
 =
 c_{\mathrm g}Q_Ac_{\mathrm g}.
\]

Therefore
\[
 \Theta^{\mathrm g}(D_{\mathrm g})
 \subset
 \operatorname{Her}_{Q_A}(b_{\mathrm g})
 \subset
 \operatorname{Her}_{Q_A}(1_{Q_A}-q_g)
 =
 c_{\mathrm g}Q_Ac_{\mathrm g}
\]
verifies \eqref{cor-constant-lift.1}.
By \eqref{e11.84a} and the choice of
\(\chi\),
\[
 \|[q_{\mathrm g},j_A(x)]\|<\varepsilon\rforal x\in\mathcal F,
\]
which verifies \eqref{L117-2}.
Inequality \eqref{e11.103}
gives
\[
 \bigl\|
 \Theta^{\mathrm g}(\widehat d_x)
 -(1_{Q_A}-q_{\mathrm g})^{1/2}
  j_A(x)
  (1_{Q_A}-q_{\mathrm g})^{1/2}
 \bigr\|
 <\varepsilon\rforal x\in\mathcal F,
\]
and therefore verifies \eqref{L117-3}.
Relation~\eqref{e11.89} gives
\[
 q_{\mathrm g}\lesssim j_A(a),
\]
which verifies \eqref{L117-4}.
Finally, \eqref{e11.91} gives
\[
 t(\widehat d_{e_A})>1-\eta=\kappa,
 \qquad t\in T(D_{\mathrm g}),
\]
and hence verifies \eqref{L117-5}.

Therefore, Lemma~\ref{cor-constant-lift}
shows that \(A\) has property ETR1
associated with \(e_A\) and the constant \(1-\eta\).
This proves the proposition.
\end{proof}
%
%

\begin{rem}\label{Rkappa}
Let $A$ be a separable  simple \CA\, with a strictly positive element 
$e_0\in {\rm Ped}(A)$ with $0\le e_0\le 1.$     
 Suppose that $T(A)\neq\emptyset$, and let $S$ be
a compact convex subset of $\widetilde T(A)$ such that $0\not\in S$ and 
\[
    \mathbb R_+\cdot S=\widetilde T(A).
\]
 Suppose that $A$ has property ETR1.   We claim that $A$ has property ETR1 associated with a strictly positive element $e_A$  with $\|e_A\|=1$  and 
 with a constant $\kappa_{e_A}$ depending only on   the discontinuity of $d_\tau(e_A).$ 
 Recall that $\widehat{[e_0]}(s)=d_s(e_0)$ ($s\in S$).

Fix $0<\eta<1/2$. By \cite[Lemma~7.2]{EGLNs}, there exist
$0<\delta _0<\eta/2$ and $e_a\in A_+^1$ such that
$\tau\mapsto d_\tau(e_a)$ is continuous on $S$ and
\beq
 f_{\eta/2}(e_0)\leq e_a\leq f_{\delta _0}(e_0)
 \leq f_{\delta _0/2}(e_0)
 \leq \bar f_{\delta _0/2}(e_0).
 \label{e11.109}
\eneq
Put
\[
    e_A:=\bar f_{\delta _0/2}(e_0).
\]
Then $e_A$ is a strictly positive element of $A$ and $\|e_A\|=1$.
Since $e_0\in {\rm Ped}(A)_+,$ $\widehat{[e_0]}$ is bounded. 

In what follows write $f\gg g$ (and $g\ll f$) when $\inf_{s\in S}\{(f-g)(s)\}>0.$ 
Choose $f\in\operatorname{Aff}(S)$ such that
\beq
 d_\tau(e_0)=d_\tau(e_A)<f(\tau)\rforal \tau\in S.
 \label{e11.110}
\eneq
Replacing $f$ by $f+c$ for some $c>0,$ we may 
assume 
that 
\beq\label{e11.110a}
\inf_{\tau\in S} \{f(\tau)-d_\tau(e_0)\}\ge c>0.
\eneq
So we may choose $f\in\Aff(S)$ such that $f\gg \widehat{[e_0]}.$
Since \(e_a\neq 0\) and \(A\) is simple,
and $\widehat{[e_a]}$ is continuous,
\[
m_a:=\min_{\tau\in S}d_\tau(e_a)>0.
\]
Since $s\to (f(s)-s(f_\eta(e_0))$  is continuous, it is bounded on $S.$ 
Therefore one may choose \(n\in\mathbb N\) such that  
\beq
 n d_\tau(e_a)>f(\tau)-\tau(f_\eta(e_0))\rforal \tau\in S.
 \label{e11.111}
\eneq
Since $f_\eta(e_0)\leq f_{\eta/2}(e_0)\leq e_a$, we have
$\tau(f_\eta(e_0))\leq d_\tau(e_a)$. Therefore,
\beq
 (n+1)d_\tau(e_a)>f(\tau)\rforal \tau\in S.
 \label{e11.112}
\eneq

By Corollary \ref{Cstrictcomp}, there exists
$e\in\operatorname{Ped}(A\otimes\mathcal K)_+^1$ such that
\[
    d_\tau(e)=f(\tau)\rforal  \tau\in S.
\]
Since $f\gg \widehat{[e_0]}$ and $d_\tau(e)=\lim_{k\to\infty} \tau(f_{1/k}(e)),$
by Dini's Theorem, 
we may choose $0<\delta<1$ such that
\beq
 d_\tau(e_A)<\tau(f_\delta(e))
 \leq d_\tau(e)=f(\tau)\rforal  \tau\in S.
 \label{e11.113}
\eneq
Strict comparison gives
\beq
    e_A\precsim f_\delta(e).
    \label{e11.114}
\eneq
Using stable rank one and identifying $A$ with a Cuntz-equivalent
hereditary subalgebra of $A\otimes\mathcal K$, we may assume that
\[
 e_A\in\operatorname{Her}(f_\delta(e))
 \subset\operatorname{Her}(f_{\delta/2}(e)).
\]

By \cite[Lemma~7.2]{EGLNs}, there exist
$e_1\in\operatorname{Her}(e)_+^1$ and $0<\eta _1<\delta/2$ such that
$\tau\mapsto d_\tau(e_1)$ is continuous and
\beq
 f_{\delta/2}(e)\leq e_1
 \leq f_{\eta _1}(e)
 \leq f_{\eta _1/2}(e).
 \label{e11.115}
\eneq
Put
\[
    B:=\operatorname{Her}_{A\otimes\mathcal K}(e_1).
\]
Since $e_A\in\operatorname{Her}(f_{\delta/2}(e))\subset B$, and
$A=\operatorname{Her}_{A\otimes\mathcal K}(e_A)$, we have $A\subset B$.
Moreover, $B$ has continuous scale because
$\tau\mapsto d_\tau(e_1)$ is continuous. Thus, for any $0<\sigma<1$,
Proposition~11.3 gives a strictly positive element $e_B\in B_+$, with
$\|e_B\|=1$, such that $B$ has property ETR1 associated with $e_B$ and a
constant
\[
    \kappa_{e_B}>1-\sigma.
\]

Put
\[
 c_1=\operatorname{diag}(e_a,\ldots,e_a)\andeqn
 c_2=\operatorname{diag}(e_A,\ldots,e_A),
\]
where each diagonal has $n+1$ entries. By \eqref{e11.112} and strict
comparison, there exists a column
\[
    x\in M_{n+1,1}(A\otimes\mathcal K)
\]
such that
\beq
 x^*x=f_{\eta _1/4}(e)
 \andeqn
 xx^*\in\operatorname{Her}(c_1).
 \label{e11.116}
\eneq
The choice of $e_A$ and \eqref{e11.109} imply
\[
    e_Ae_a=e_a=e_ae_A.
\]
Consequently, $c_1c_2=c_2c_1=c_1$, and hence
\[
    c_2xx^*=xx^*c_2=xx^*.
    \label{e11.117}
\]

Apply Lemma \ref{lem:diagonal-contractive-factorization}, with cutoff $\eta _1/16$. There are 
$
    y_1,\ldots,y_{n+1}\in A\otimes\mathcal K
$ with $\|y_i\|\le 1,$ $1\le i\le n+1,$
such that
\[
    e_Ay_i=y_i,\qquad 1\leq i\leq n+1,
\]
and
\beq
 f_{\eta _1/16}(x^*x)
 =\sum_{i=1}^{n+1}y_i^*y_i
 =\sum_{i=1}^{n+1}y_i^*e_Ay_i.
 \label{e11.118}
\eneq
We also have
\beq
 f_{\eta _1/16}(x^*x)
 =f_{\eta _1/16}\bigl(f_{\eta _1/4}(e)\bigr)
 \geq f_{\eta _1/2}(e).
 \label{e11.119}
\eneq
It follows that
\beq
 f_{\eta _1/2}(e)
 \leq\sum_{i=1}^{n+1}y_i^*e_Ay_i
 \leq 1.
 \label{e11.120}
\eneq

Because $e_1\leq f_{\eta _1}(e)$,  we obtain
\[
 e_Bf_{\eta _1/2}(e)=e_B
 =f_{\eta _1/2}(e)e_B.
\]
Put
\[
    w_i:=y_i e_B^{1/2},\qquad 1\leq i\leq n+1.
\]
Since $y_i=e_Ay_i$ and $e_A,e_B\in B$, we have $w_i\in B$, and
$\|w_i\|\leq1$. Equations \eqref{e11.118}--\eqref{e11.120} now give
\[
 \sum_{i=1}^{n+1}w_i^*e_Aw_i=e_B,
 \qquad
 \|w_i\|\leq1,\quad 1\leq i\leq n+1.
\]
The first assertion of Remark \ref{Rhere} therefore shows that $A$ has ETR1
associated with $e_A$, with
\[
 \kappa_{e_A}
 \geq
 \left(\frac{1-\sigma}{4(n+1)^2}\right)^3\rforal 
  0<\sigma<1/2.
\]
It is important to note that  the estimate of $\kappa_{e_A}$  above only depends 
on $d_\tau(e_0).$

Finally, take $S=\overline{T(A)}^{\,w}.$ Since $A$ is $\sigma$-unital and algebraically simple,
$A={\rm Ped}(A).$ Therefore, by \cite[Lemma 4.5]{EGLNs}, $0\not\in\overline{T(A)}^w.$ 
Define
\beq\nonumber
 \bar\omega:=
 \inf_{\widehat{[e_0]}\ll f\in\operatorname{Aff}(S)}
 \left\{
 \lim_{\eta\downarrow0}
 \sup_{\tau\in S}
 \frac{f(\tau)-\tau(f_\eta(e_0))}
      {\tau(f_\eta(e_0))}
 \right\}=
 \inf_{\widehat{[e_0]}\ll f\in\operatorname{Aff}(S)}
 \left\{
 \lim_{\eta\downarrow0}
 \sup_{\tau\in S}
 \frac{f(\tau)}
      {\tau(f_\eta(e_0))}-1
 \right\}.
\eneq
Since $A$ is simple and $0\not\in S,$ 
$\tau\mapsto \tau(f_\eta(e_0))$  is  positive continuous on the compact subset $S.$ 
 One then checks that 
$\bar\omega<\infty.$ 
Given any integer $n>\bar\omega$, first choose $f\gg \widehat{[e_0]}$ for which the
expression inside the braces is smaller than $n$, and then choose
$\eta>0$ sufficiently small. Since
\[
    d_\tau(e_a)\geq\tau(f_\eta(e_0)),
\]
condition \eqref{e11.111} follows. Thus any integer $n>\bar\omega$ may be
used in the preceding construction; in particular,
\[
    n=\lfloor\bar\omega\rfloor+1
\]
is admissible. If $\tau\mapsto d_\tau(e_0)$ is continuous on $S$, then
$\bar\omega=0$, and one may take $n=1$.

We expect that the argument can be refined to give the sharper
estimate
\beq\nonumber
   \kappa_{e_A}\geq
   \frac{1-\sigma}{\max\{1,\lceil\overline\omega\rceil\}}
\eneq
for any $0<\sigma<1.$
Since this sharper estimate is not needed below, we do not pursue it
here.
%
Please also see Theorem \ref{TTRE1-1}.
\end{rem}
\section{Approximations}

This section collects approximation and permanence consequences of
the preceding results.  Lemma \ref{Lappid}  produces approximate identities consisting of elements
with zero oscillation.  Theorem \ref{Tinductive} establishes permanence of the
condition \(\operatorname{ETR}\leq 1\) under simple inductive limits.
  Proposition \ref{Asefimprov} and
Theorem \ref{TTRE1-1} give approximations in which the distinguished element
has uniformly large trace.  Proposition \ref{L112} transfers these approximations
to the quotient \(Q_A\).

\begin{lem}\label{Lappid}
Let $A$ be a $\sigma$-unital simple \CA\, with strict comparison and the canonical map
$\Gamma$ is surjective. 
Then $A$ contains an approximate identity  $\{e_n\}$ with $e_n\in {\rm Ped}(A)$ and 
$\omega(e_n)=0.$ 
\end{lem}
\begin{proof}
It follows from \cite[Theorem 1.1]{Linsr1} that $A$  has tracial approximate oscillation zero. 

Fix a strictly positive element $a$ with $\|a\|=1.$ 

Claim:  For any $\ep>0,$ there is $e\in A_+^{\bf 1}\cap {\rm Ped}(A)$ such that 
\beq
f_\ep(a)\le e\le f_{\ep/16}(a).
\eneq

Let $b=f_{\ep/8}(a).$  If $0$ is an isolated point in ${\rm sp}(b),$ then 
for some $0<\dt<\ep/16,$  $e=f_{\dt}(b)$ is a projection. Hence $\omega(e)=0,$
and $0\le f_\ep(a)\le e \le f_{\ep/16}(a).$ We are done. So we assume that 
$0$ is a limit point of ${\rm sp}(b).$

Put $B=\Her(b).$ Then $B$ also has tracial approximate  oscillation zero by \cite[Proposition 5.4]{FLosc}.  It follows from \cite[Theorem 7.11]{FLosc} that  
$\Gamma$ is surjective.  Applying \cite[Lemma 7.2]{EGLNs}, 
there is $\dt>0$ and an element $e\in B_+$ 
such that
\beq
0\le f_\ep(a)\le f_{\ep/4}(b)\le e\le f_\dt(b)\le f_{\ep/16}(a)
\eneq
and $d_\tau(e)$ is continuous on $\overline{T(A)}^w.$ Hence 
$\omega(e)=0.$  Since $f_{\ep/16}(a)\in {\rm Ped}(A),$  $e\in {\rm Ped}(A).$

This proves the claim.

By the claim, for each $n\in \N,$ we obtain 
$e_n'\in A_+^{\bf 1}\cap {\rm Ped}(A)$ such that 
\beq
&&\omega(e_n')=0\andeqn\\
&& 0\le f_{1/2^n}(a)\le e_n'\le f_{1/2^{n+4}}(a).
\eneq

Let $x\in A.$ Then 
\beq
\|(x(1-e_n')\|=\|x(1-e_n')^2x^*\|^{1/2}\le \|x(1-f_{1/2^n}(a))x^*\|^{1/2}\to 0.
\eneq
It follows that 
$\{e_n'\}$   forms an approximate identity for $A.$
%
\end{proof}

\begin{thm}\label{Tinductive}
Let $A$ be a non-elementary separable  simple \CA\, which is an inductive limit 
of separable  simple \CA s $A_n$ satisfying $ETR(A_n)\le 1.$ 
Then $ETR(A)\le 1.$
\end{thm}

\begin{proof}
Write $A=\lim_{n\to\infty}(A_n, \phi_n),$ where each $A_n$ is simple and $ETR(A_n)\le 1$ 
and $\phi_n: A_n\to A_{n+1}$ is a \hm.
\Wlog, we may assume that each $\phi_n$ is non-zero, and hence 
$\varphi_{n,\infty}:A_n\to A$ is nonzero for every $n$. Since
$A_n$ is simple, each $\phi_n$  and $\varphi_{n,\infty}$ is injective.

Choose
$
e_1\in {\rm Ped}(A_1)_+^1\setminus\{0\}
$
such that $\omega(e_1)=0.$  If $A_1$ is non-elementary, this follows from 
Lemma \ref{Lappid} and  Corollary \ref{Cstrictcomp}.  If $A_1$ is elementary, choose $e_1$ to be 
a non-zero finite-rank projection.
Put
\[
e_n=\varphi_{1,n}(e_1)\andeqn
e=\varphi_{1,\infty}(e_1),
\andeqn\\
B_n={\rm Her}(e_n),\,\,\, B={\rm Her}(e).
\]
Define $\psi_n={\phi_n}|_{B_n}: B_n\to B_{n+1}.$ Note $\psi_n$ is injective.
Put
\[
B=\varinjlim(B_n,\psi_n).
\]
By Proposition \ref{Pperm},
\[
{\rm ETR}(B_n)\leq 1 \rforal n\geq 1.
\]
Moreover, $\omega(e_n)=0$. Indeed, for every lower semicontinuous
quasitrace $\tau$ on $B_n$,
\[
d_\tau(e_n)=d_{\tau\circ\psi_{1,n}}(e_1).
\]

Consequently, by Remark \ref{Rcs}, each $B_n$ has continuous scale.

Fix $\eta=1/2.$ Let $0<\eta_0<1$  be given by 
Proposition \ref{prop:continuous-scale-improvement} associated with $\eta=1/2.$  We may assume 
that $\eta_0<1/2.$ 

Since $B_1$ has continuous scale, $T(B_1)$ is compact (see
Remark \ref{Rcs}).
Therefore, there is a strictly positive contraction
$b_1\in B_1$ such that
\[
 \tau(b_1)>1-\eta_0,\qquad \tau\in T(B_1).
\]
If $B_1$ is elementary, take $b_1=e_1$. Otherwise, for example,
one may take $b_1=\overline f_{1/m}(e_1)$ (recall Definition \ref{dfg})  for a sufficiently large
$m$. Put
\[
 b_n=\psi_{1,n}(b_1)\quad\text{and}\quad
 b=\psi_{1,\infty}(b_1).
\]
Since $\psi_{1,n}$ is injective, $b_n$ is a strictly positive
contraction of $B_n$. If $\tau\in T(B_n)$, then
$\tau\circ\psi_{1,n}\in T(B_1)$, and hence
\[
 \tau(b_n)>1-\eta_0.
\]
Suppose first that $B_n$ is elementary. Then $B_n$ is unital with
unit $e_n$, and hence is finite-dimensional. Since $B_1$ embeds
into $B_n$, $B_1$ is also elementary; consequently, $b_1=e_1$ and
$b_n=e_n$. In this case, $B_n$ has property ETR1 associated with
$b_n$ and the constant $1/2:$  in the defining approximation, take
\[
 q_n=0,\qquad D=B_n,\qquad
 \theta_n=\operatorname{id}_{B_n},\qquad d_{b_n}=e_n.
\]
If $B_n$ is non-elementary, Proposition \ref{prop:continuous-scale-improvement},  applied with the
prescribed strictly positive element $b_n$, shows that $B_n$ has
property ETR1 associated with $b_n$ and the constant $1/2$. Thus,
in either case, $B_n$ has property ETR1 associated with $b_n$, with $\mathfrak \kappa_{b_n}=\frac12,$
a constant independent of $n$.

We now prove that $B$ has property ETR1 associated with $b$, with
constant $1/2$. Let $\epsilon>0$, let ${\cal F}\subset B$ be finite,
and let
\[
a\in B_+\setminus\{0\}.
\]
We may assume that $b\in{\cal F}$ and $0\not\in {\cal F}.$
Set $\af:=\inf\{\|x\|: x\in {\cal F}\}>0.$  Choose
\[
0<\delta<\min\{\epsilon/4,\|a\|/4, \af/4\}.
\]
By the density of
$
\bigcup_{n=1}^{\infty}\psi_{n,\infty}(B_n)
$
in $B$, after increasing $n$ we can find $x_n\in B_n$, for
$x\in{\cal F}$, such that
\[
\|x-\psi_{n,\infty}(x_n)\|<\delta.
\]
For $x=b$, take $x_n=b_n$. We may also choose $a_n\in(B_n)_+$ such
that
\[
\|a-\psi_{n,\infty}(a_n)\|<\delta.
\]
Put
$
c_n=(a_n-\delta)_+.
$
Then 
 $c_n\neq0$.   By \cite[Proposition 2.2]{RrII},
\[
\psi_{n,\infty}(c_n)\lesssim a.
\]

Apply property ETR1 of $B_n$, associated with $b_n$  and constant
$1/2,$ to the finite set
\[
{\cal F}_n=\{x_n:x\in{\cal F}\}
\]
(as ${\cal F}$), 
$\delta$ (as $\ep$)  and the element $c_n$ (as $a$). Using Remark~3.3, we
obtain a positive contraction $q_n\in B_n$, a finite-dimensional
$C^*$-algebra $D$, and a fully normalized order-zero c.p.c.\ map
\[
\theta_n:D\longrightarrow
\overline{(1-q_n)B_n(1-q_n)}
\]
such that
\[
\|[q_n,x_n]\|<\delta\andeqn 
{\rm dist}\big(
(1-q_n)^{1/2}x_n(1-q_n)^{1/2},\theta_n(D)
\big)<\delta\rforal  x\in{\cal F}.
\]
Furthermore, there is $d_b\in D_+^1$ such that
\[
\left\|
(1-q_n)^{1/2}b_n(1-q_n)^{1/2}-\theta_n(d_b)
\right\|<\delta
\]
and
\[
t(d_b)>\frac12\rforal  t\in T(D),
\]
while
\[
q_n\lesssim c_n.
\]

Set
\[
q=\psi_{n,\infty}(q_n),\qquad
\theta=\psi_{n,\infty}\circ\theta_n.
\]
Hence
\[
\theta:D\longrightarrow
\overline{(1-q)B(1-q)}.
\]
Since $\psi_{n,\infty}$ is injective, $\theta$ is again a fully
normalized order-zero c.p.c.\ map.

For $x\in{\cal F}$, we have
\[
\begin{split}
\|[q,x]\|
&\leq
 \|[q,\psi_{n,\infty}(x_n)]\|
 +2\|x-\psi_{n,\infty}(x_n)\|  
<3\delta<\epsilon.
\end{split}
\]
Also,
\[
\begin{split}
&{\rm dist}\big(
(1-q)^{1/2}x(1-q)^{1/2},\theta(D)
\big) \\
&\quad\leq
\|x-\psi_{n,\infty}(x_n)\|
+
{\rm dist}\big(
(1-q)^{1/2}\psi_{n,\infty}(x_n)(1-q)^{1/2},
\theta(D)
\big) 
\,<2\delta<\epsilon.
\end{split}
\]
Since $b=\psi_{n,\infty}(b_n)$,
\[
\left\|
(1-q)^{1/2}b(1-q)^{1/2}-\theta(d_b)
\right\|<\delta<\epsilon,
\]
and
\[
t(d_b)>\frac12,\qquad t\in T(D).
\]
Finally,
\[
q\lesssim\psi_{n,\infty}(c_n)\lesssim a.
\]
Thus $B$ has property ETR1 associated with $b$, with constant
$\mathfrak \kappa_b=1/2$.  Since $B$ is a hereditary \SCA\, of $A,$  by Proposition \ref{Pperm} (3),
$ETR(A)\le 1.$ 
\end{proof}

The following is a simplification of property ETR1 when $A$ is unital.

\begin{prop}\label{Asefimprov}
Let $A$ be a unital separable simple \CA. 
Then 
%
$A$ has property ETR1 if and only if 
 the following holds:

$\bullet$ For any  $\ep>0$ and $1>\eta>0,$  any finite set
$\mathcal F\subset A$, any nonzero $a_0\in A_+$,

$\bullet$ there exist
$0\le e\le 1,$  

$\bullet$ a finite dimensional \CA\, 
$D=\bigoplus_{j=1}^m D_j,\quad D_j\cong M_{r_j},$

$\bullet$ a fully normalized  order-zero  \cpc\,
$
\theta:D\to \overline{(1-e)A(1-e)},
$

$\bullet$ and 
$
d_e
\in D_+^{\bf 1}
$
such that

\beq\nonumber
\hspace{-0.8in}&&(1) \hspace{0.8in}
\|[e,x]\|<\ep 
\tforal x\in {\cal F},\\
%
\hspace{-1.4in} &&(2) \hspace{0.8in}
{\rm dist}((1-e)^{1/2}x(1-e)^{1/2}, \theta(D))<\ep\tforal x\in {\cal F},\\
&&\hspace{0.8in}\|(1-e)-\theta(d_e)\|<\ep, \,\,\,
t(d_e)\ge 1-\eta\tforal t\in T(D),\tand\\
\hspace{-0.8in}&& (3)\hspace{0.8in}
e\precsim a_0.
\eneq	
\end{prop}

\begin{proof}
Since $A$ is unital, by taking $e_A=1_A,$ the ``if" part follows from the definition immediately. 

It remains to show the ``only if" part.

\Wlog, we may assume that $A$ is non-elementary.
Fix $\ep>0,$ $1>\eta>0,$ a finite subset ${\cal F}\subset A$ and 
$a_0\in A_+\setminus \{0\}.$ 

Let ${\cal F}_1={\cal F}\cup \{1_A\}.$ 

Let \(0<\eta_0<1\) be supplied by Proposition~11.3 for
\(\eta\).  Since \(A\) is unital, the strictly positive
contraction
$
   e_A=1_A
$
satisfies
$
   \tau(e_A)=1>1-\eta_0\rforal
   \tau\in T(A).
$
Apply Proposition \ref{prop:continuous-scale-improvement} with this prescribed element \(e_A\), the
finite set \(\mathcal F_1\), the tolerance \(\varepsilon_1\), and the
comparison element \(a_0\).

 The application of  Proposition \ref{prop:continuous-scale-improvement} described above gives
 the following objects and estimates.
 
 $\bullet$ there exist
$0\le e\le 1,$  

$\bullet$ a finite dimensional \CA\, 
$D=\bigoplus_{j=1}^m D_j,\quad D_j\cong M_{r_j},$

$\bullet$ a fully normalized  order-zero  \cpc\,
$
\theta:D\to \overline{(1-e)A(1-e)},
$

$\bullet$ and 
$
d_e
\in D_+^{\bf 1}
$
such that

\beq\nonumber
\hspace{-0.8in}&&(1) \hspace{0.8in}
\|[e,x]\|<\ep
\tforal x\in {\cal F}_1,\\\nonumber
%
\hspace{-1.4in} &&(2') \hspace{0.8in}
{\rm dist}((1-e)^{1/2}x(1-e)^{1/2}, \theta(D))<\ep\tforal x\in {\cal F}_1,\\\nonumber
&&\hspace{0.8in}\|(1-e)^{1/2}e_A(1-e)^{1/2}-\theta(d_e)\|<\ep, \,\,\,
t(d_e)\ge 1-\eta\tforal t\in T(D),\tand\\
\hspace{-0.8in}&& (3)\hspace{0.8in}
e\precsim a_0.
\eneq	
Since $e_A=1_A$, the second estimate in (2') gives
\[
 \|(1-e)-\theta(d_e)\|<\varepsilon.
\]
Moreover,
\[
 t(d_e)>1-\eta
 \qquad\text{for all }t\in T(D).
\]
Thus $d_e$ itself satisfies (e12.7), and the proposition follows.
%
%
%
\end{proof}

Note that the following statement shows that one can always make $\kappa$ close to 
one if we make a slight modification of Definition \ref{D1}. Moreover, 
the following result also shows that a separable non-elementary simple \CA\, $A$ with 
$ETR(A)\le 1$ can be tracially approximated by one-dimensional cones without assuming 
it is algebraically simple. 

\begin{thm}\label{TTRE1-1}
Let $A$ be a separable non-elementary simple \CA\, with $ETR(A)\le 1.$
Then, for any $\ep>0,$  any $0<\eta<1,$ any finite subset ${\cal F}\subset A,$ and any
$a\in A_+\setminus \{0\},$ there exists a strictly positive element
$e_A\in A$ with $\|e_A\|=1,$  an element $e\in A_+^{\bf 1},$  a finite dimensional \CA\, $D,$ 
and a fully normalized order zero \cpc\, $\theta:   D\to \overline{(1-e)A(1-e)}$ such that

(1) $\|[x,\, e]\|<\ep,$ $\|[x,\, (1-e)^{1/2}]\|<\ep\tforal x\in {\cal F};$

(2)  \hspace{-0.05in} 
$\bullet$\hspace{0.1in}  ${\rm dist}((1-e)^{1/2}x(1-e)^{1/2}, \theta(D))<\ep\tforal  x\in {\cal F},$ and  there is 
$d_{e_A}\in D_+^{\bf 1}$ such that 
\beq
&&\hspace{-2.3in}\bullet\hspace{0.3in}\|(1-e)^{1/2} e_A(1-e)^{1/2}-\theta(d_{e_A})\|<\ep,\\
&&\hspace{-2.3in}\bullet \hspace{0.3in} \tau(d_{e_A})>1-\eta\rforal \tau\in T(D);\,\,\,\,\,\,\,\, \,\,\,\,\tand
 \eneq
 
 (3) $e\lesssim a.$
\end{thm}

\begin{proof}
By Corollary~\ref{Cstrictcomp}, $A$ has strict comparison and
the canonical map $\Gamma$ is surjective.
By Lemma \ref{Lappid}, we obtain an 
approximate identity $\{h_n\}$ consisting of positive contractions
such that
\[
\omega(h_n)=0\andeqn h_n \in {\rm Ped}(A)\rforal 
 n\in\mathbb N.
\]

Fix $\epsilon>0$, $1>\eta>0$, a finite subset
${\cal F}\subset A$, and $a\in A_+\setminus\{0\}$. Choose
$
0<\delta<\frac{\epsilon}{4}.
$
Since $\{h_n\}$ is an approximate identity, for some sufficiently
large $n$,
\beq\label{12-12.1}
\|x-h_nxh_n\|<\delta\rforal x\in{\cal F}\andeqn 
a_n:=h_nah_n\ne 0.
\eneq
Put
\[
B={\rm Her}(h_n).
\]
Then $B$ is a separable simple $\sigma$-unital hereditary
$C^*$-subalgebra of $A$. Since $\omega(h_n)=0$, equivalently  $d_\tau(h_n)$
is continuous on $\widetilde{QT}(A)=\widetilde{T}(A),$ by Remark \ref{Rcs},
 $B$ has continuous scale.

By (2) of Proposition~\ref{Pperm},  $B$ has property ETR1.
For $x\in{\cal F}$, let
\[
x_n=h_nxh_n\in B,
\andeqn
{\cal G}=\{x_n:x\in{\cal F}\}.
\]
For the fixed $0<\eta<1,$ let $\eta_0$ be given by Proposition ~\ref{prop:continuous-scale-improvement}. Since $B$ has continuous scale, $T(B)$ is compact.  Hence one may choose strictly 
positive element $e_B$  of $B$  with $\|e_B\|=1$ such that
$\tau(e_B)>1-\eta_0$ for all $\tau\in T(B).$  
Apply Proposition~\ref{prop:continuous-scale-improvement} to $B$ (using also Remark~\ref{R1}),  there exists 
 a positive contraction $e\in B$, a
finite-dimensional $C^*$-algebra $D$, a fully normalized order-zero
c.p.c.\ map
\[
\theta:D\longrightarrow
\overline{(1-e)B(1-e)},
\]
and a positive contraction $d_B\in D$ such that, with
$
s=(1-e)^{1/2},
$
we have
\beq\label{12-12.5}
&&\|[e,y]\|<\delta \andeqn
\|[s,y]\|<\delta,
\rforal y\in{\cal G}\cup \{e_B\}
\\\label{12-12.6}
&&{\rm dist}(sys,\theta(D))<\delta,
\qquad y\in{\cal G}\cup \{e_B\},
\\\label{12-12.7}
&&\|se_Bs-\theta(d_B)\|<\delta,\\
\label{12-12.8}
&&t(d_B)>1-\eta,
\qquad t\in T(D),\,\,\andeqn\\\label{12-12.9}
&&e\precsim_B a_n.
\eneq
Choose a strictly positive element $e_A'\in A$ with $\|e_A'\|\le 1.$
Put 
$
e_A={e_B+\dt e_A'\over{\|e_B+\dt e_A'\|}}.
$
Then $e_A$ is a strictly positive element with $\|e_A\|=1.$  We compute that
\beq\label{12-12.13-}
\|e_A-e_B\|<2\dt.
\eneq

We regard $\theta$ as a map into
$\overline{(1-e)A(1-e)}$ and put
$
d_{e_A}=d_B.
$
For $x\in{\cal F}$, \eqref{12-12.1} and \eqref{12-12.5}
give
\beq\label{12-12.13}
\|[e,x]\|
\leq
\|[e,h_nxh_n]\|+2\|x-h_nx h_n\|
<3\delta<\epsilon
\eneq
and, similarly,
\beq\label{12-12.14}
\|[s,x]\|<3\delta<\epsilon.
\eneq
Therefore (1) follows.

Also, by \eqref{12-12.1} and \eqref{12-12.6},
\beq\label{12-12.15}
{\rm dist}(sxs,\theta(D))
\leq
\|s(x-x_n)s\|
+
{\rm dist}(sh_nxh_ns,\theta(D))
<2\delta<\epsilon.
\eneq
By \eqref{12-12.13-} and  \eqref{12-12.7},
\beq\label{12-12.17}
\|se_As-\theta(d_{e_A})\|
&\leq
\|s(e_A-e_B)s\|+\|se_Bs-\theta(d_B)\| 
<2\delta+\delta
<\epsilon.
\eneq
Equation \eqref{12-12.8} gives
\beq\label{12-12.19-}
t(d_{e_A})>1-\eta\rforal 
 t\in T(D).
\eneq
Thus (2) follows from \eqref{12-12.15}, \eqref{12-12.17}   and \eqref{12-12.19-}.

Finally, put $z=a^{1/2}h_n$. Then
\[
z^*z=h_nah_n=a_n
\andeqn
zz^*=a^{1/2}h_n^2a^{1/2}\in{\rm Her}(a).
\]
Consequently,
\[
a_n\precsim_A a.
\]
Together with \eqref{12-12.9}, this gives
\beq\label{12-12.19}
e\precsim_B a_n\precsim_A a.
\eneq
Thus  (3) also follows.
\end{proof}

\begin{prop}\label{L112}
Let $A$ be a  non-elementary separable algebraically simple $C^*$-algebra with
${\rm ETR}(A)\leq1.$  Fix a strictly positive element $e_A\in A$ with $\|e_A\|=1$ and  a constant $\kappa_{e_A}>0$  such that $A$ has property ETR1 associated with $e_A$ and $\kappa_{e_A}.$  Let
\[
\overline{j_A}:A\longrightarrow Q_A
\]
denote the constant-sequence embedding.

For every $\epsilon>0$, every finite subset ${\cal F}\subset A$
containing $e_A$, and every $a\in A_+\setminus\{0\}$, there exist a
positive contraction $e\in Q_A$, a finite-dimensional $C^*$-algebra
\[
D=\bigoplus_{j=1}^m M_{r_j},
\]
and a homomorphism
\[
H:C_0((0,1])\otimes D
\longrightarrow
\overline{(1-e)Q_A(1-e)}
\]
with finite dimensional range 
satisfying the following:
\beq\label{12f-11.1}
&&\|[e,\overline{j_A}(x)]\|<\epsilon
\quad\text{for all }x\in{\cal F};\\\label{12f-11.2}
%
&&\hspace{-0.3in}{\rm dist}\big(
(1-e)^{1/2}\overline{j_A}(x)(1-e)^{1/2},
H(C_0((0,1])\otimes D)
\big)<\epsilon
\quad\text{for all }x\in{\cal F};
\tand\\\label{12f-11.3}
&&e\precsim \overline{j_A}(a).
\eneq
Moreover, there is a positive contraction $d_e\in D$ such that
\beq\label{12f-11.4}
\left\|
H(\iota\otimes d_e)
-(1-e)^{1/2}\overline{j_A}(e_A)(1-e)^{1/2}
\right\|<\epsilon
\tand\\\label{12f-11.5}
t(d_e)\geq\mathfrak f_{e_A}
\quad\text{for all }t\in T(D),
\eneq
where $\iota(t)=t$ for $t\in(0,1]$.
\end{prop}

\begin{proof}
Fix $\epsilon>0$, a finite subset ${\cal F}\subset A$ containing
$e_A$, and $a\in A_+\setminus\{0\}$. Choose
$
0<\delta<\frac{\epsilon}{3}.
$
Apply property ETR1, associated with $e_A$ and the constant
$\mathfrak f_{e_A}$, to ${\cal F}$, $\delta$, and $a$. Using also
Remark~\ref{R1}, we obtain a positive contraction $e_0\in A$, a
finite-dimensional $C^*$-algebra
\[
D=\bigoplus_{j=1}^{m}M_{r_j},
\]
and a fully normalized order-zero c.p.c.\ map
\[
\theta:D\longrightarrow
\overline{(1-e_0)A(1-e_0)}
\]
such that
\beq\nonumber
&&\|[e_0,x]\|<\delta,\qquad x\in{\cal F},\\\nonumber
&&{\rm dist}\big(
(1-e_0)^{1/2}x(1-e_0)^{1/2},\theta(D)
\big)<\delta\rforal x\in{\cal F},
\andeqn\\
&&e_0\lesssim a.
\eneq
Moreover, there is a positive contraction $d_e\in D$ such that
\beq\nonumber
&&\left\|
\theta(d_e)-(1-e_0)^{1/2}e_A(1-e_0)^{1/2}
\right\|<\delta
\andeqn\\
&&t(d_e)\geq\mathfrak f_{e_A}\rforal  t\in T(D).
\eneq

For each $x\in{\cal F}$, choose $d_x\in D$ such that
\[
\left\|
\theta(d_x)-(1-e_0)^{1/2}x(1-e_0)^{1/2}
\right\|<\delta.
\]
For $x=e_A$, we use the distinguished element $d_e$.

Let
\[
\widetilde\theta:C_0((0,1])\otimes D\longrightarrow A
\]
be the homomorphism associated with $\theta$, so that
\[
\widetilde\theta(\iota\otimes d)=\theta(d),
\qquad d\in D.
\]
Put
$
e=\overline{j_A}(e_0)
$
and define
\[
\Phi=\overline{j_A}\circ\widetilde\theta:
C_0((0,1])\otimes D\longrightarrow Q_A.
\]
Then
\[
\Phi\big(C_0((0,1])\otimes D\big)
 \subset\overline{(1-e)Q_A(1-e)}.
\]

By Corollary~\ref{Cstrictcomp} and  \eqref {A2.17},
the trace-kernel ideal defining
$Q_A$ agrees with $N_{\rm cu}(A).$ It follows from
\cite[Theorem~6.4]{FLosc} that $Q_A$ has real rank zero. Applying 
Lemma \ref{Lfinitesp},
we obtain a homomorphism
\[
H:C_0((0,1])\otimes D
 \longrightarrow
 {\rm Her}\big(\Phi(\iota\otimes1_D)\big)
 \subset\overline{(1-e)Q_A(1-e)}
\]
with finite-dimensional range such that
\[
\|H(\iota\otimes d_x)-\Phi(\iota\otimes d_x)\|
 <\delta,\qquad x\in{\cal F}.
\]

Since $e=\overline{j_A}(e_0)$, we have
\[
\|[e,\overline{j_A}(x)]\|<\delta<\epsilon\rforal 
x\in{\cal F},
\andeqn
e\lesssim \overline{j_A}(a).
\]
For $x\in{\cal F}$,
\[
\begin{split}
&\left\|
H(\iota\otimes d_x)
 -(1-e)^{1/2}\overline{j_A}(x)(1-e)^{1/2}
\right\|                                      \\
&\quad\leq
 \|H(\iota\otimes d_x)-\Phi(\iota\otimes d_x)\|\\
&\qquad+
 \left\|
 \overline{j_A}(\theta(d_x))
 -(1-e)^{1/2}\overline{j_A}(x)(1-e)^{1/2}
 \right\|                                      
<2\delta<\epsilon.
\end{split}
\]
This proves \eqref{12f-11.1}-- \eqref{12f-11.3}.
Taking $x=e_A$ gives
\beq\nonumber
&&\left\|
H(\iota\otimes d_e)
 -(1-e)^{1/2}\overline{j_A}(e_A)(1-e)^{1/2}
\right\|<2\delta<\epsilon\\\nonumber
&&\andeqn t(d_e)\geq\mathfrak f_{e_A}\rforal  t\in T(D).
\eneq
\end{proof}

The following gives a weaker tracial approximation for a separable simple \CA\, with $ETR(A)\le 1.$

\begin{cor}\label{Papproxl2}
Let $A$ be a separable non-elementary simple \CA\, with $ETR(A)\le 1$ and let $c\in A_+^{\bf 1}\cap {\rm Ped}(A)\setminus \{0\}.$
Let $S=\{ \tau\in \widetilde{T}(A): \tau(c)=1\}.$ 
Then
there exist finite-dimensional $C^*$-algebras $D_n$ and order-zero
c.p.c.\ maps $\theta_n\colon D_n\to A$ such that, for every
$x\in A$, there are elements $d_{x,n}\in D_n$, $n\in\mathbb N$,
satisfying
\[
 \|d_{x,n}\|\leq\|x\|
\]
and, for every finite subset $\mathcal F\subset A$,
\[
 \lim_{n\to\infty}
 \sup_{x\in\mathcal F}
 \|x-\theta_n(d_{x,n})\|_{2,S}=0.
\] 

\end{cor}

\begin{proof}
Since $c\in {\rm Ped}(A)_+\setminus\{0\}$, the set
$
 S=\{\tau\in\widetilde T(A):\tau(c)=1\}
$
is compact.

Note that $\Her(c)$  is non-elementary and simple. By  Remark  \ref{RCuntznull} (2),
$\Her(c)$ contains a truly Cuntz-null sequence $\{a_n\}$ with $\|a_n\|=1,$ 
$n\in \N.$
By passing to a subsequence, if necessary, we may assume 
that
\beq
\lim_{n\to\infty}\sup_{\tau\in S}d_\tau(a_n)=0.
\eneq
%
Choose an increasing sequence of finite subsets
\[
 \mathcal F_1\subset\mathcal F_2\subset\cdots\subset A
\]
whose union is dense in $A$, with $0\in\mathcal F_n$, and such that,
for every $x\in A$, there are $y_n\in\mathcal F_n$ satisfying
\[
 \|y_n\|\le\|x\|
 \quad\text{and}\quad
 \|x-y_n\|\longrightarrow0.
\]
Put $\varepsilon_n=1/n$. Apply Theorem \ref{TTRE1-1} to
$\mathcal F_n$, $\varepsilon_n$, any fixed $0<\eta<1$, and the
comparison element $a_n$. We obtain a positive contraction $e_n\in A$,
a finite-dimensional \CA\ $D_n$, and a fully normalized order-zero
c.p.c.\ map
$
 \theta_n:D_n\longrightarrow(1-e_n)A(1-e_n)
$
such that
\[
 \dist\bigl((1-e_n)^{1/2}y(1-e_n)^{1/2},
                 \theta_n(D_n)\bigr)<\varepsilon_n\rforal 
 y\in\mathcal F_n,\\
\andeqn
 e_n\lesssim a_n.
\]
In particular,
\[
 \sup_{\tau\in S}\tau(e_n)
 \le\sup_{\tau\in S}d_\tau(e_n)
 \le\sup_{\tau\in S}d_\tau(a_n)
 \to 0.
\]

Put $s_n=(1-e_n)^{1/2}$. For each $y\in\mathcal F_n$, choose
$d'_{y,n}\in D_n$ such that
\[
 \|s_nys_n-\theta_n(d'_{y,n})\|<\varepsilon_n.
\]
Since $\theta_n$ is fully normalized, it is isometric. Therefore
\[
 \|d'_{y,n}\|<\|y\|+\varepsilon_n.
\]
Multiplying $d'_{y,n}$ by a suitable scalar in $[0,1]$, we obtain
$d_{y,n}\in D_n$ such that
\[
 \|d_{y,n}\|\le\|y\|
 \quad\text{and}\quad
 \|s_nys_n-\theta_n(d_{y,n})\|<2\varepsilon_n.
\]

We claim that
\[
 \|y-s_nys_n\|_{2,\tau}
 \le 2\|y\|\tau(e_n)^{1/2},
 \qquad y\in A,\quad \tau\in S.
\]
Indeed, put $r_n=1-s_n$. Since
\[
 0\le r_n=1-(1-e_n)^{1/2}\le e_n
\andeqn\\
 y-s_nys_n=r_ny+s_nyr_n,
\]
we have
\[
 \|r_ny\|_{2,\tau}^2
 =\tau(y^*r_n^2y)
 =\tau(r_nyy^*r_n)
 \le\|y\|^2\tau(r_n^2)
 \le\|y\|^2\tau(e_n),
\]
and, similarly,
\[
 \|s_nyr_n\|_{2,\tau}^2
 =\tau(r_ny^*s_n^2yr_n)
 \le\|y\|^2\tau(r_n^2)
 \le\|y\|^2\tau(e_n).
\]
This proves the claim.

Now fix $x\in A$, and choose $y_n\in\mathcal F_n$ such that
\[
 \|y_n\|\le\|x\|
 \quad\text{and}\quad
 \|x-y_n\|\longrightarrow0.
\]
Define
$
 d_{x,n}:=d_{y_n,n}.
$
Then $\|d_{x,n}\|\le\|x\|$, and
\[
\begin{aligned}
 \|x-\theta_n(d_{x,n})\|_{2,S}
 &\le \|x-y_n\|
     +\|y_n-s_ny_ns_n\|_{2,S}
     +\|s_ny_ns_n-\theta_n(d_{x,n})\|\\
 &\le \|x-y_n\|
     +2\|x\| (\sup_{\tau\in S}\tau(e_n))^{1/2}
     +2\varepsilon_n.
\end{aligned}
\]
The right-hand side converges to zero. Since a finite subset of $A$
contains only finitely many elements, the convergence is uniform on
each finite subset $\mathcal F\subset A$. Hence
\[
 \lim_{n\to\infty}
 \sup_{x\in\mathcal F}
 \|x-\theta_n(d_{x,n})\|_{2,S}=0.
\]
This proves the corollary.
\end{proof}


\section{Tracial rank zero and Applications}

By the last assertion of Theorem \ref{TmainZ}, we may introduce the following definition.

\begin{df}\label{Ds}
Let $A$ be a $\sigma$-unital simple \CA,  and let
 $e_A\in A$ be a strictly positive element with $\|e_A\|=1$ and $0<s<1.$ 
We say $A$ has property sETR1 if the following hold.

Let $\ep>0,$ ${\cal F}\subset A$ be a finite subset containing $e_A$ and 
$a\in A_+\setminus \{0\}.$ 
There is  
a positive element $e\in A^{\bf 1},$ 
a finite-dimensional \CA\, $D$ and 
a fully normalized  order-zero \cpc\, 
$\theta: D\to  \overline{(1-e)A(1-e)}$
such that

(1)  $\|ex-xe\|<\ep\tforal x\in {\cal F},$

(2)   
 $(1-e)^{1/2}x(1-e)^{1/2}\in_{\ep} \theta(D)$ and  $0\le d_e\le 1$ in $D$ such that
\beq\label{D1-(2)13a}
&& \|\theta(d_e)-(1-e)^{1/2}e_A(1-e)^{1/2}\|<\ep,
 \\
&&\|\theta(d_e)-\theta(1_D)\|<s,
\eneq

(3) $e\lesssim a.$

Let $A$ be a $\sigma$-unital simple \CA. We write $sETR(A)\le 1,$ 
if there is nonzero $a\in {\rm Ped}(A)_+$ such that $\Her(a)$ has property $sETR1.$

\end{df}

Let $\tilde\theta: C:=C_0((0,1])\otimes D\to \overline{(1-e)A(1-e)}$ be the \hm\, associated with 
$\theta.$ The second condition may be written as 
\beq
\|\tilde\theta(\iota\otimes d_e)-\tilde\theta(\iota\otimes 1_D)\|<s.
\eneq
So we require that $\tilde\theta(\iota\otimes d_e)$ is sufficiently strictly positive in $\widetilde\theta(C).$ 
It is clear  that property sETR1 is stronger than ETR1. We record here as follows.

\begin{prop}\label{Csetr}
Let $A$ be a separable algebraically simple \CA\, which has property sETR1 
for some $0<s<1$ associated with a strictly positive element $e_A$ with $\|e_A\|=1.$
Then $A$ has property ETR1 associated with $e_A$ and $\kappa_{e_A}=1-s.$
In particular $T(A)\not=\emptyset.$
\end{prop}

\begin{proof}
By  comparing the definition of  sETR1 and  that of ETR1, it suffices to show 
that
\beq
\tau(d_e)\ge 1-s\rforal \tau\in T(D).
\eneq
Since $\theta $ is fully normalized, it is an isometry. Hence
\beq
\tau(d_e)\ge \tau(1_D)-(\tau(1_D)-\tau(d_e))\ge 1-\|\theta(1_D)-\theta(d_e)\|\ge 1-s.
\eneq
By Theorem \ref{Ltrace}, $T(A)\not=\emptyset.$
\end{proof}

The following is an immediate consequence of the  "Moreover" part of Theorem \ref{TmainZ}.

\begin{cor}\label{CorZ}
Let $A$ be a unital  separable  simple finite amenable  ${\cal Z}$-stable \CA. 
Then for every $0<s<1,$ $A$ has property sETR1 with respect to  $e_A=1_A$ and $s.$
\end{cor}

\begin{thm}\label{TZsetr}
Let $A$ be a non-elementary unital separable simple amenable  \CA.
Then the following   are equivalent.

(1) $A$ is ${\cal Z}$-stable and finite.

(2) $ETR(A)\le 1;$

(3) $sETR(A)\le 1.$

\end{thm}

\begin{proof}

(1) $\Longrightarrow$ (3)  follows from  Corollary \ref{CorZ}.

(3) $\Longrightarrow$ (2):  Suppose that (3) holds. By definition, there is a nonzero element
$a\in\operatorname{Ped}(A)_+$ such that
\[
B=\operatorname{Her}(a)
\]
has property $s\mathrm{ETR}1$. Proposition \ref{Csetr} shows that $B$ has
property ETR1. Hence $\mathrm{ETR}(B)\leq 1$, and
Proposition \ref{Pperm} (3) gives $\mathrm{ETR}(A)\leq 1$. Thus (3) implies (2). 

(2) $\Longrightarrow$ (1) follows from Theorem \ref{MMCC}.

\end{proof}

\begin{lem}
\label{lem12.5a-local-unit}
Let $A$ be a $\sigma$-unital simple $C^*$-algebra with
continuous scale. Let $e_A\in A_+^1$, let
$\mathcal F\subset A$ be a finite subset, let
$a\in A_+\setminus\{0\}$, and let $\rho>0$. Then there exists
$g\in A_+^1$ such that
\[
        e_A\leq g,
\]
\[
        \|gx-x\|+\|xg-x\|<\rho,
        \qquad x\in\mathcal F,
\]
and
\[
        1_{Q_A}-\overline{j_A}(g)\lesssim \overline{j_A}(a).
\]
\end{lem}

\begin{proof}
Choose an approximate identity $\{u_n\}$ for $A$ satisfying
$u_nu_{n+1}=u_n$. Since $A$ has continuous scale, by 
Lemma \ref{Lsequnececuntz}, for all sufficiently large $n,$
\[
        1_{Q_A}-\overline{j_A}(u_n)\lesssim \overline{j_A}(a)
\]
For $n\in\mathbb N$, put
\[
 g_n=e_A+(1-e_A)^{1/2}u_n(1-e_A)^{1/2}.
\]
Then $g_n\in A_+^1$, $e_A\leq g_n$, and
\[
 1-g_n=(1-e_A)^{1/2}(1-u_n)(1-e_A)^{1/2}.
\]
Let
\[
 v_n=(1-\overline{j_A}(u_n))^{1/2}(1-\overline{j_A}(e_A))^{1/2}.
\]
Then
\[
 v_n^*v_n=1_{Q_A}-\overline{j_A}(g_n)
\andeqn
 v_nv_n^*
 \leq1_{Q_A}-\overline{j_A}(u_n).
\]
Consequently,
\[
 1_{Q_A}-\overline{j_A}(g_n)
 \lesssim1_{Q_A}-\overline{j_A}(u_n)
 \lesssim \overline{j_A}(a).
\]

Moreover, $\{g_n\}$ is an approximate identity for $A$. Indeed,
for $x\in A$,
\[
 (1-g_n)x
 =
 (1-e_A)^{1/2}(1-u_n)(1-e_A)^{1/2}x
 \longrightarrow0,
\]
and similarly $x(1-g_n)\to0$. A sufficiently large $n$ has
all the required properties.
\end{proof}

\begin{lem}
\label{lem12.5a-cut}
Let $B$ be a unital $C^*$-algebra, let $C\subset B$ be a
finite-dimensional $C^*$-subalgebra, and let
$h\in C_+$  be in the center of $C$ and $c\in C_+$ satisfy
\[
        0\leq h\leq1\tand
         \gamma h\leq c\leq2h
\]
for some $\gamma>0$. For every $\varepsilon>0$, there exists
$\delta>0$, depending only on $\varepsilon$ and $\gamma$, with
the following property.

Suppose that $b\in B_+^{\bf 1}$ and that
$\mathcal G\subset B^{\bf 1}$ is a finite self-adjoint subset such
that
\beq\label{136-1}
       && \|b-c\|<\delta,\\\label{136-2}
      &&  \|bx-xb\|<\delta     
\tand \operatorname{dist}(b^{1/2}xb^{1/2},C)<\delta\tforal x\in\mathcal G.
\eneq
Put
$
        p=\chi_{[1/4,1]}(h)\in C.
$
Then
\beq\nonumber
      &&  \|px-xp\|<\varepsilon\,\,\, \tforal x\in\mathcal G,\\
       && \operatorname{dist}(pxp,pC)<\varepsilon,
        \tforal x\in\mathcal G\tand\\
      &&  1_B-p\lesssim1_B-b.
\eneq
\end{lem}

\begin{proof}
Denote by $Z(C)$ the center of $C.$ 
Since $C$ is finite-dimensional and $h\in Z(C)$,
$
        p\in Z(C).
$
Put $q=1_B-p$. Since $ph\geq p/4$,
\[
        pcp\geq  \gamma ph \ge \frac{\gamma}{4}p.
\]
Thus
\beq
        a=(pcp)^{-1/2}\in pC
        \quad\text{and}\quad
        \|a\|\leq\frac2{\sqrt\gamma}.                \label{12.5A.1}
\eneq
Put $K=\frac2{\sqrt\gamma}. $ Then $\|a\|\le K.$
Let $\eta:[0,1]\to[0,\infty)$ be a function satisfying
$\eta(t)\to0$ as $t\to0$ and such that
\[
 \|d^{1/2}x-xd^{1/2}\|
 \leq\eta(\|dx-xd\|)
\]
whenever $d$ is a positive contraction and $\|x\|\leq1$.
Choose $\delta>0$ sufficiently small that
\beq
\begin{aligned}
 2\bigl(
   \delta^{1/2}K
   +\eta(\delta)K
   +\delta^{1/2}K^2
   +\delta K^2
 \bigr)
 &<\varepsilon\andeqn                                   \label{12.5A.2}
 K^2(2\delta^{1/2}+\delta)
 &<\varepsilon,                                   
\end{aligned}
\eneq
and $\delta<1/4$.
Note that $\dt$ 
depends only on $\varepsilon$ and
$\gamma$.

For $x\in\mathcal G$, choose $y_x\in C$ such that
\beq
        \|b^{1/2}xb^{1/2}-y_x\|<\delta.              \label{12.5A.4}
\eneq
Since $p\in Z(C)$,
\beq
        qy_xp=0.                                    \label{12.5A.5}
\eneq
Also
\beq
        p=c^{1/2}pa=c^{1/2}p (pcp)^{-1/2}=p.                              \label{12.5A.6}
\eneq
We have 
\[
\begin{aligned}
 \|qxb^{1/2}pa\|
 &\leq
 \|qb^{1/2}xpa\|
 +\|q(xb^{1/2}-b^{1/2}x)pa\|                           \\
 &\leq
 \|qb^{1/2}xpa\|+\eta(\delta)\|a\|.                       \label{12.5A.8b}
\end{aligned}
\]
By \eqref{12.5A.6} again, we compute 
\[
        qb^{1/2}xpa
        =qb^{1/2}xc^{1/2}pa^2.
\]
By \eqref{136-1} and Lemma \ref{Lsqrt},
\beq
        \|b^{1/2}-c^{1/2}\|
        \leq\|b-c\|^{1/2}<\delta^{1/2}.           \label{12.5A.7}
\eneq

Consequently (using \eqref{12.5A.4} and  \eqref{12.5A.5})
\beq
\begin{aligned}
 \|qb^{1/2}xpa\|
 &\leq
 \|qb^{1/2}x(c^{1/2}-b^{1/2})pa^2\|                    
 +
 \|qb^{1/2}xb^{1/2}pa^2\|                              \\
 &\leq
 \delta^{1/2}\|a\|^2
 +\|q(b^{1/2}xb^{1/2}-y_x)pa^2\|
 +\|qy_xpa^2\|                                         \\
 &<
 \delta^{1/2}\|a\|^2+\delta \|a\|^2.                         \label{12.5A.8c}
\end{aligned}
\eneq
Thus,   by \eqref{12.5A.6}, \eqref{136-1} and \eqref{12.5A.8c}, for all $x\in {\cal G},$ 
\[
\begin{aligned}
 \|qxp\|
 &=\|qxc^{1/2}pa\|                                      \\
 &\leq
 \|qx(c^{1/2}-b^{1/2})pa\|
 +\|qxb^{1/2}pa\|                                       
 \leq
 \delta^{1/2}\|a\|+\|qxb^{1/2}pa\|                      \label{12.5A.8a}
\\
 &\leq
 \delta^{1/2}\|a\|
 +\eta(\delta)\|a\|                                  
 +
 \delta^{1/2}\|a\|^2
 +\delta\|a\|^2.
\end{aligned}
\]
Recall that ${\cal G}$ is self-adjoint.  Therefore  (see also \eqref{12.5A.2}) 
\[
        \|px-xp\|
        \leq\|pxq\|+\|qxp\|
        <\varepsilon.                               \label{12.5A.9}
\]
Furthermore,
\[
        pxp=apc^{1/2}xc^{1/2}pa.
\]
Therefore, by \eqref{12.5A.4} and \eqref{12.5A.7},
\beq
 \left\|
 pxp-apy_xpa
 \right\|
\le  \|a\|^2\|c^{1/2}xc^{1/2}-y_x\|\\
 \leq
 K^2(2\delta^{1/2}+\delta)
 <\varepsilon.                                     \label{12.5A.10}
\eneq
Since $apy_xpa\in pC$,
\[
        \operatorname{dist}(pxp,pC)<\varepsilon.
\]

Finally, since $qhq\leq q/4$ and $c\leq2h$,
\[
        qcq\leq\frac12q.
\]
Hence (using \eqref{136-1} again)
\[
 q(1_B-b)q
 \geq
 \left(\frac12-\delta\right)q.                     \label{12.5A.11}
\]
Since $\delta<1/4$, the element $q(1_B-b)q$ is
invertible in $qBq$. It follows that
\[
        q\lesssim1_B-b.
\]
\end{proof}

\begin{thm}\label{TMRR0}
Let $A$ be a unital separable simple \CA\, with property sETR1. Suppose that $A$ has real rank zero.
Then $A$ has tracial rank zero.
\end{thm}

\begin{proof}
If $A$ is finite-dimensional, there is nothing to prove.
We may therefore assume that $A$ is non-elementary.

Fix $0<s<1$ and a strictly positive element $e_A\in A$, with
$\|e_A\|=1$, associated with property $s\mathrm{ETR}1$. Let $\ep>0$, let $\mathcal F\subset A$ be a finite
subset, and let $a\in A_+\setminus\{0\}$.  
 Without
loss of generality, assume that $\mathcal F\subset A^{\bf 1}$ is
self-adjoint and contains $1_A$ and $e_A$.


Put
\[
        \gamma=1-s.
\]
Let $\delta>0$ be supplied by Lemma \ref{lem12.5a-cut} for
$\varepsilon$ and $\gamma$. 

 Choose $\eta=\min\{1/3, \dt/5\}.$
Apply property $s\mathrm{ETR}1$, with tolerance $\eta/3$,
to $\mathcal F$ and $a$.  We obtain $e_0\in A_+^1$, a
finite-dimensional $C^*$-algebra $D$, and a fully normalized
order-zero c.p.c.\ map
$
        \theta\colon D\longrightarrow(1-e_0)A(1-e_0)
$
such that
\beq
&&        e_0\precsim a,\\
&&        \|e_0x-xe_0\|<\frac{\eta}{3}\rforal x\in\mathcal F,
\andeqn\\
&& \operatorname{dist}\bigl(
 (1-e_0)^{1/2}x(1-e_0)^{1/2},\theta(D)
 \bigr)<\frac{\eta}{3}\rforal  x\in\mathcal F; 
\eneq
and a  distinguished element
$d_A\in D_+^{\bf 1}$  such that 
\[
 \left\|
 \theta(d_A)-(1-e_0)^{1/2}e_A(1-e_0)^{1/2}
 \right\|<\frac{\eta}{3}\\
 \andeqn
        \|\theta(d_A)-\theta(1_D)\|<s.
\]
Since $\theta$ is fully normalized, it is isometric.  Hence
\[
        \|d_A-1_D\|<s
\]
and therefore
\[
        d_A\geq(1-s)1_D=\gamma1_D.
\]

Put
$
        b=1_A-e_0.
$
Since $1_A\in\mathcal F$,
\[
        \operatorname{dist}(b,\theta(D))<\frac{\eta}{3}.
\]
Since $1_A\in\mathcal F$, Remark \ref{R1} (iv), applied to the positive
contraction $1_A$, gives an element $d_1\in D_+^1$ such that
\beq
\|b-\theta(d_1)\|<\eta.
\eneq
Put
\[
        x_0=(1-e_0)^{1/2}e_A(1-e_0)^{1/2}.
\]
Since $0\leq e_A\leq1_A$, we have $0\leq x_0\leq b$.
Consequently,
\[
 \begin{aligned}
 \theta(d_A)
 &\leq x_0+\frac{\eta}{3}1_A
 \leq b+\frac{\eta}{3}1_A
 \leq\theta(d_1)+\frac{4\eta}{3}1_A.
 \end{aligned}
\]
Thus
\[
        \theta(d_A)-\theta(d_1)<\frac{4\eta}{3}1_A<2\eta1_A.
\]
Since $\theta$ is an isometry, using the  associated \hm\, 
\[
 \widetilde\theta:C_0((0,1])\otimes D\longrightarrow A,
 \qquad
 \widetilde\theta(\iota\otimes d)=\theta(d),
\]
we have 
\[
        \|(d_A-d_1)_+\|<2\eta.
\]

Define
\[
        d_0=d_1+(d_A-d_1)_+.
\]
Then
\[
 d_0-d_A=d_1-d_A+(d_A-d_1)_+=(d_A-d_1)_-\geq0.
\]
It follows that
\beq\label{RR0-10}
        \gamma1_D\leq d_A\leq d_0.
\eneq
Moreover,
\beq\label{RR0-10a}
 &&       0\leq d_0\leq(1+2\eta)1_D\leq2\cdot1_D,
\andeqn\\\label{RR0-11}
&& \|\theta(d_0)-b\|
 \leq
 \|\theta(d_1)-b\|
 +\|\theta((d_A-d_1)_+)\|<
 3\eta.
\eneq

For every $x\in\mathcal F$, choose $d_x\in D$ such that
\beq\label{RR0-5}
 \left\|
 b^{1/2}xb^{1/2}-\theta(d_x)
 \right\|<\frac{\eta}{3}.
\eneq
Since $A$ has real rank zero, Lemma \ref{Lfinitesp} gives a homomorphism
$
        H\colon C_0((0,1])\otimes D\longrightarrow A
$
with finite-dimensional range such that
\beq\label{RR0-12}
   &&\hspace{0.2in}     \|H(z)-\widetilde\theta(z)\|<\eta\,\,\,\,
\rforal 
 z\in
 \{\iota\otimes d_0\}
 \cup
 \{\iota\otimes d_x:x\in\mathcal F\}.
\eneq

Set
\[
 C=H(C_0((0,1])\otimes D),\qquad
 h=H(\iota\otimes1_D)\andeqn
 c=H(\iota\otimes d_0).
\]
Then $C\subset A$ is finite-dimensional, 
\[
        0\leq h\leq1_A
        \quad\text{and}\quad
        h\in Z(C),
\]
where $Z(C)$ is the center of $C.$
By \eqref{RR0-10} and \eqref{RR0-10a},
we have
\[
        \gamma h=\gamma H(\iota\otimes 1_D)\le H(\iota\otimes d_A)\le H(\iota\otimes d_0)=  c\leq H(\iota\otimes 21_D)=2h.
\]
Furthermore, by \eqref{RR0-11} and \eqref{RR0-12},
\[
 \begin{aligned}
 \|b-c\|
 &\leq
 \|b-\theta(d_0)\|
 +\|\widetilde\theta(\iota\otimes d_0)
       -H(\iota\otimes d_0)\|\\
 &<3\eta+\eta
 =4\eta<\delta.
 \end{aligned}
\]
For every $x\in\mathcal F$,
\[
 \|bx-xb\|
 =\|e_0x-xe_0\|
 <\frac{\eta}{3}<\delta.
\]
Finally, by \eqref{RR0-5} and \eqref{RR0-12},
\[
 \begin{aligned}
 \operatorname{dist}(b^{1/2}xb^{1/2},C)
 &\leq
 \|b^{1/2}xb^{1/2}-\theta(d_x)\|
 +
 \|\widetilde\theta(\iota\otimes d_x)
       -H(\iota\otimes d_x)\|\\
 &<\frac{\eta}{3}+\eta
 <\delta.
 \end{aligned}
\]

Thus all the hypotheses of Lemma \ref{lem12.5a-cut} are satisfied in $A$,
with $b,C,h,c$, and $\mathcal F$.  Put
\[
        p=\chi_{[1/4,1]}(h).
\]
Note that $p\in Z(C).$ 
Applying Lemma \ref{lem12.5a-cut},  for all $x\in {\cal F},$ 
\beq
     &&   \|px-xp\|<\varepsilon,\\
        &&  \operatorname{dist}(pxp,pC)<\varepsilon,
      \andeqn\\       
      && 1_A-p\precsim1_A-b=e_0\precsim a.
\eneq
Set
\[
        E=pC.
\]
Then $E\subset A$ is finite-dimensional and $1_E=p$.
Since $1_A-p$ is a projection and $1_A-p\precsim a$, it is
Murray--von Neumann equivalent to a projection in
$\operatorname{Her}(a)$.  Thus $p$ and $E$ satisfy all the
requirements in the definition of tracial rank zero.
Therefore
\[
        \operatorname{TR}(A)=0.
\]

\end{proof}

We now prove  Theorem \ref{TB}.

\begin{thm}\label{TmB}
Let $A$ be a unital finite separable simple amenable  ${\cal Z}$-stable \CA\, with real rank zero.
Then $TR(A)=0,$ i.e., $A$ has tracial rank zero.
\end{thm}

\begin{proof}
It follows from  Corollary \ref{CorZ}  that $A$ has property sETR1. Then Theorem \ref{TMRR0}
implies that $TR(A)=0.$
\end{proof}

\begin{rem}\label{Rrr0}

(1) A longstanding goal was to determine which regularity assumptions,
together with real rank zero, imply tracial rank zero (see early work \cite{Lincrell}). 
Winter proved that a separable simple unital $\mathcal Z$-stable
$C^*$-algebra with real rank zero and finite decomposition rank has
tracial rank zero; see \cite[Theorem~4.1]{WinterRR0FDR}. He subsequently
obtained the same conclusion under the weaker assumption of locally
finite decomposition rank; see
\cite[Theorem~2.1]{WinterLFDR}. At the time,
$\mathcal Z$-stability was a separate hypothesis. 
Finite decomposition rank implies finite nuclear
dimension; see \cite[Remark~2.2(ii)]{WZ10nuc}. It is now known that finite nuclear
dimension  implies $\mathcal Z$-stability in the
non-elementary separable simple unital setting (see
\cite[Corollary~6.3]{WinterPure}). In fact, finite nuclear dimension and $\mathcal Z$-stability are
equivalent for non-elementary separable simple nuclear $C^*$-algebras (\cite[Theorem~B]{CETWW}).
Thus Theorem \ref{TmB} removes the decomposition-rank assumptions and gives
the corresponding result under the natural hypothesis of finite
nuclear dimension.  It is also  worth noting that N. Brown constructed a unital separable simple \CA\,  with
real rank zero, stable rank one, 
strict comparison, and a unique tracial state which nevertheless does not have
tracial rank zero; see \cite[Theorem~1.1]{Brown}. This example shows 
the importance of amenability in Theorem \ref{TmB}.

(2) Since every tracial state on a unital simple $C^*$-algebra with
tracial rank zero is quasidiagonal, Theorem~13.8 implies that every
tracial state on a unital finite separable simple amenable
$\mathcal Z$-stable $C^*$-algebra with real rank zero is
quasidiagonal.

There is also a useful consequence in the unique-trace case without
assuming real rank zero. Let $A$ be a unital separable simple amenable
$\mathcal Z$-stable $C^*$-algebra with unique tracial state $\tau$,
and let $\mathcal Q$ be the universal UHF-algebra, with unique tracial
state $\tau_{\mathcal Q}$. Put
$
        B=A\otimes\mathcal Q.
$
Then $B$ is unital, finite, separable, simple, amenable, and
$\mathcal Z$-stable    with a unique tracial state.  By \cite[Theorem 7.2]{RrII},
$B$ has real rank zero. 
 Theorem \ref{TmB} shows that
$\tau\otimes\tau_{\mathcal Q}$ is quasidiagonal. Its restriction to
$A\otimes 1_{\mathcal Q}$ is $\tau$, so $\tau$ is quasidiagonal.
A much more general statement is given by Theorem \ref{TT} which will be discussed later in this section.
\end{rem}

For simple \CA s which do not have real rank zero, we still  have the following
weaker version of ``tracial rank zero".

\begin{thm}
\label{prop12.5a}
Let $A$ be a non-elementary separable simple $C^*$-algebra with
continuous scale. Fix a strictly positive element $e_A\in A_+$
with $\|e_A\|=1$, and fix $0<s<1$. Suppose that $A$ has
property ${\rm sETR1}$, associated with $e_A.$

Then, for every $\varepsilon>0$, every finite subset
$\mathcal F\subset A$, and every $a\in A_+\setminus\{0\}$,
there exist a projection $e\in Q_A$ and a finite-dimensional
$C^*$-subalgebra $E\subset Q_A$ such that
\beq
 &&       1_E=1_{Q_A}-e,\\
 &&       \|e\overline{j_A}(x)-\overline{j_A}(x)e\|<\varepsilon \tforal x\in\mathcal F,\\
&& \operatorname{dist}\bigl(
   (1_{Q_A}-e)\overline{j_A}(x)(1_{Q_A}-e),E
 \bigr)<\varepsilon\tforal x\in\mathcal F,
\tand\\
&&        e\lesssim \overline{j_A}(a).
\eneq
\end{thm}

\begin{proof}
After replacing $\mathcal F$ by
$\mathcal F\cup\mathcal F^*\cup\{e_A\}$ and rescaling its
elements, we may assume that $\mathcal F\subset A^1$ is
self-adjoint and contains $e_A$.

\smallskip
\noindent
\emph{Step 1: Choice of constants and comparison elements.}
Choose nonzero mutually orthogonal positive elements
$a_0,a_1\in\operatorname{Her}(a)$ such that
\[
        a_0\oplus a_1\precsim a.
\]
Put $\gamma=1-s$, and let $\delta>0$ be supplied by
Lemma~\ref{lem12.5a-cut} for $\varepsilon$ and $\gamma$.

For $q\in\{1/2,1/4\}$, fix a nondecreasing   function 
$\Omega_q$ such that
$\Omega_q(t)\to0$ as $t\to0$ and
\beq\label{e13q-root}
 \|u^qy-yu^q\|\leq\Omega_q(\|uy-yu\|),
\eneq
whenever $0\leq u\leq1$ and $\|y\|\leq1$.  Put
\[
 \nu=\frac{\delta}{8}\andeqn
 \beta(t)=\Omega_{1/4}\bigl(\Omega_{1/2}(t)\bigr).
\]
Choose $0<\rho<1/2$ such that
\begin{gather}
 3\rho+\nu<\delta,\\
 2\Omega_{1/2}(\rho)+\rho<\delta,\\
 2\bigl(\sqrt{2\beta(\rho)}+\beta(\rho)\bigr)
 +2\rho^{1/2}+\rho+\nu<\delta.
 \label{e13.7.rho}
\end{gather}
Such a choice is possible since all the terms involving $\rho$
converge to zero as $\rho\to0$.  Put
\[
        \lambda=\frac{\rho}{3}.
\]

Apply Lemma~\ref{lem12.5a-local-unit}
to $\mathcal F$, $a_1$, and $\rho$.
We obtain $g\in A_+^1$ such that
\begin{gather}
 e_A\leq g,\\\label{e137.g-1}
 \|gx-x\|+\|xg-x\|<\rho\rforal x\in\mathcal F,\\
 1_{Q_A}-\overline{j_A}(g)\precsim \overline{j_A}(a_1).
 \label{e13.7.g}
\end{gather}

\smallskip
\noindent
\emph{Step 2: The order-zero approximation.}
Apply property $\mathrm{sETR}1$, with tolerance $\rho/3$, to
a finite subset containing $\mathcal F\cup\{g\}$ and with
comparison element $a_0$.  We obtain $e_0\in A_+^1$, a
finite-dimensional $C^*$-algebra $D$, and a fully normalized
order-zero c.p.c.\ map
$
 \theta\colon D\longrightarrow \overline{(1-e_0)A(1-e_0)}
$
such that
\beq\label{e13.7.etr-2}
 &&e_0\precsim a_0,\\\label{e13.7.etr-1}
 &&\|e_0x-xe_0\|<\rho/3\qquad x\in\mathcal F\cup\{g\},\\
&& \operatorname{dist}\bigl(
 (1-e_0)^{1/2}x(1-e_0)^{1/2},\theta(D)
 \bigr)<\rho/3,\qquad x\in\mathcal F\cup\{g\}.
 \label{e13.7.etr}
\eneq
Moreover, we obtain 
$d_A\in D_+^{\bf 1}$  such that 
\beq\label{13.7.1-1}
&& \left\|
 \theta(d_A)-(1-e_0)^{1/2}e_A(1-e_0)^{1/2}
 \right\|<\rho/3\\
 &&\andeqn
 \|\theta(d_A)-\theta(1_D)\|<s.
\eneq
Since $\theta$ is isometric,
\beq
 \|d_A-1_D\|<s\andeqn
 d_A\geq(1-s)1_D=\gamma1_D.                              \label{13.7.1}
\eneq

Put
\[
        y_0=(1-e_0)^{1/2}g(1-e_0)^{1/2}.
\]
Then, as in Remark \ref{R1} (iv),  we obtain an element $d_g\in D_+^{\bf 1}$
such that
%
\beq\label{13.7.2}
 \|y_0-\theta(d_g)\|<\rho.
 \eneq
%
Put
\[
        x_0=(1-e_0)^{1/2}e_A(1-e_0)^{1/2}.
\]
Since $e_A\leq g$, we have $0\leq x_0\leq y_0$.  Hence, by \eqref{13.7.1-1} and \eqref{13.7.2},
\[
 \theta(d_A)-\theta(d_g)\le 
 x_0-y_0+
 \left(\frac{\rho}{3}+\rho\right)1_{\widetilde A}
 \leq(\rho/3+\rho)1_{\widetilde A}
 <2\rho\,1_{\widetilde A}.
\]
Order-zero functional calculus and the isometry of $\theta$
therefore give
\beq
        \|(d_A-d_g)_+\|<2\rho.                            \label{13.7.3}
\eneq
Define
\[
        d_0=d_g+(d_A-d_g)_+.
\]
Then
\[
 d_0-d_A=d_g-d_A+(d_A-d_g)_+=(d_A-d_g)_-\geq0.
\]
It follows that  (see \eqref{13.7.1})
\beq
 \gamma1_D\leq  d_A\le d_0\leq(1+2\rho)1_D\leq2\,1_D,             \label{13.7.4}
\eneq
and, by \eqref{13.7.2} and \eqref{13.7.3},
\beq
 \|\theta(d_0)-y_0\|<3\rho.                              \label{13.7.5}
\eneq

\smallskip
\noindent
\emph{Step 3: A finite-dimensional subalgebra of $Q_A$.}
Put
\[
 r=(1_{Q_A}-\overline{j_A}(e_0))^{1/2}\andeqn
 b=r\overline{j_A}(g)r=(1_{Q_A}-\overline{j_A}(e_0))^{1/2}\overline{j_A}(g)(1_{Q_A}-\overline{j_A}(e_0))^{1/2}.
\]
Let
\[
 \widetilde\theta\colon C_0((0,1])\otimes D\longrightarrow Q_A
\]
be the homomorphism associated with $\overline{j_A}\circ\theta$, so that
\[
 \widetilde\theta(\iota\otimes d)=\overline{j_A}(\theta(d))\rforal 
  d\in D.
\]
Since $A$ has continuous scale, it is algebraically simple by  \cite[Theorem 3.3]{Lincs91}.
By  Proposition \ref{Csetr},   $A$ has property ETR1.   So by Theorem  \ref{Ltrace}, $T(A)\not=\emptyset.$ 
 It follows from  Corollary \ref{Cstrictcomp} that $A$ has tracial approximate oscillation zero.
 Then, by  \cite[Theorem 6.4]{FLosc}, $Q_A$ has real rank zero.
For each $x\in\mathcal F$, choose $d_x\in D$ such that
\beq
        \|r\overline{j_A}(x)r-\overline{j_A}(\theta(d_x))\|<\rho/3.               \label{13.7.6}
\eneq
By Lemma~\ref{Lfinitesp},
there is a homomorphism
\[
 H\colon C_0((0,1])\otimes D\longrightarrow Q_A
\]
with finite-dimensional range such that
\beq
 \|H(z)-\widetilde\theta(z)\|<\dt/8                        \label{13.7.7}
\eneq
for every
$
 z\in
 \{\iota\otimes d_0\}
 \cup\{\iota\otimes d_x:x\in\mathcal F\}.
$

Set
\[
 C=H(C_0((0,1])\otimes D),\,\,
 h=H(\iota\otimes1_D)\andeqn
 c=H(\iota\otimes d_0).
\]
Since $\iota\otimes1_D$ is central,
$
        h\in Z(C)_+.
$
By \eqref{13.7.4},
\beq
        \gamma h\leq c\leq2h.                             \label{13.7.8}
\eneq
By \eqref{13.7.7}, \eqref{13.7.5}  and the choice of $\rho$,
\beq
 \|b-c\|&=&\|r\overline{j_A}(g)r-H(\iota\otimes d_0)\|\le \|r\overline{j_A}(g)r-\overline{j_A}(\theta(d_0))\| +\dt/8\\
 &&=\|(1-e_0)^{1/2}g(1-e_0)^{1/2}-\theta(d_0)\|+\dt/8\\
 &&<3\rho+\dt/8<\delta.                                \label{13.7.9}
\eneq

For $x\in\mathcal F$, by \eqref{e13.7.etr-1}, functional calculus and by \eqref{e13q-root},
\[
 \|r\overline{j_A}(x)-\overline{j_A}(x)r\|
 \leq\Omega_{1/2}(\rho/3)
 \leq\Omega_{1/2}(\rho).
\]
Also, by \eqref{e137.g-1},
\[
 \|\overline{j_A}(g)\overline{j_A}(x)-\overline{j_A}(x)\overline{j_A}(g)\|<\rho.
\]
Consequently,
\beq
 \|b\overline{j_A}(x)-\overline{j_A}(x)b\|
 \leq2\Omega_{1/2}(\rho)+\rho<\delta.                     \label{13.7.10}
\eneq
Put
\[
        \zeta=\|r\overline{j_A}(g)^{1/4}-\overline{j_A}(g)^{1/4}r\|.
\]
Since $\|e_0g-ge_0\|<\rho/3$,  by \eqref{e13q-root},
we have 
\beq
 \zeta &\le &\Omega_{1/4}(\|(1-e_0)^{1/2}g-g(1-e_0)^{1/2}\|)\\
 &\leq&\Omega_{1/4}\bigl(\Omega_{1/2}(\|(1-e_0)g-g(1-e_0)\|)\bigr)\\
 &=&\Omega_{1/4}\bigl(\Omega_{1/2}(\|e_0g-ge_0\|)\bigr)\le \Omega_{1/4}\bigl(\Omega_{1/2}(\rho/3)\bigr)\\
 &&\leq\beta(\rho).                                         \label{13.7.11}
\eneq
Let
\[
 z=r\overline{j_A}(g)^{1/2},
 \qquad
 z_0=\overline{j_A}(g)^{1/4}r\overline{j_A}(g)^{1/4}.
\]
Then $zz^*=b$, $z_0\geq0$, and $\|z-z_0\|\leq\zeta$.
Therefore,
\beq
\|b-z_0^2\|&\leq &2\zeta
\andeqn\\
 \|b^{1/2}-r\overline{j_A}(g)^{1/2}\|&\le& \|b^{1/2}-z_0\|+\|z_0-r\overline{j_A}(g)^{1/2}\|\\
 &\leq&\sqrt{2\zeta}+\zeta.                                 \label{13.7.12}
\eneq

By \eqref{e137.g-1} and $(1-g^{1/2})^2\leq1-g$,
applied also to $x^*$,
\beq
 \|(1-\overline{j_A}(g)^{1/2})\overline{j_A}(x)\|
 +\|\overline{j_A}(x)(1-\overline{j_A}(g)^{1/2})\|
 <2\rho^{1/2}.                                            \label{13.7.13}
\eneq
It follows from \eqref{13.7.12} and \eqref{13.7.13} that
\beq
 \left\|
 b^{1/2}\overline{j_A}(x)b^{1/2}-r\overline{j_A}(x)r
 \right\|
 \leq
 2\bigl(\sqrt{2\zeta}+\zeta\bigr)+2\rho^{1/2}.             \label{13.7.14}
\eneq
On the other hand, \eqref{13.7.6} and \eqref{13.7.7} give
\[
 \operatorname{dist}(r\overline{j_A}(x)r,C)<\rho/3+\dt/8.
\]
Thus, by  \eqref{13.7.14}, \eqref{13.7.11},  and
\eqref{e13.7.rho},
\begin{align}
 \operatorname{dist}
 \bigl(b^{1/2}\overline{j_A}(x)b^{1/2},C\bigr)
 &<
 2\bigl(\sqrt{2\zeta}+\zeta\bigr)
 +2\rho^{1/2}+\rho/3+\dt/8\notag\\
 &\leq
 2\bigl(\sqrt{2\bt(\rho)}+\beta(\rho)\bigr)
 +2\rho^{1/2}+\rho/3+\dt/8
 <\delta.                                                 \label{13.7.15}
\end{align}

\smallskip
\noindent
\emph{Step 4: Cutting out the projection.}
By \eqref{13.7.8}--\eqref{13.7.10} and \eqref{13.7.15}, all the hypotheses
of Lemma \ref{lem12.5a-cut}
are satisfied in $Q_A$, with
$C,h,c,b$, and $\overline{j_A}(\mathcal F)$.  Put (recall that $C$ has finite dimension)
\[
        p=\chi_{[1/4,1]}(h)\in C.
\]
Lemma \ref{lem12.5a-cut}
gives
\beq
 &&\|p\overline{j_A}(x)-\overline{j_A}(x)p\|<\varepsilon\rforal x\in\mathcal F,\\
&& \operatorname{dist}(p\overline{j_A}(x)p,pC)<\varepsilon\rforal  x\in\mathcal F,\\
&&\andeqn
       1_{Q_A}-p\precsim1_{Q_A}-b.                       \label{13.7.16}
\eneq

Set
\[
        e=1_{Q_A}-p,\qquad E=pC.
\]
Then $e$ is a projection, $E$ is finite-dimensional, and
\[
        1_E=p=1_{Q_A}-e.
\]
Moreover, for every $x\in\mathcal F$,
\[
 \|e\overline{j_A}(x)-\overline{j_A}(x)e\|<\varepsilon
\]
and
\[
 \operatorname{dist}\bigl(
 (1_{Q_A}-e)\overline{j_A}(x)(1_{Q_A}-e),E
 \bigr)<\varepsilon.
\]

Finally,  by applying Lemma \ref{lem:absorb-tail}
(to 
$
 q=\overline{j_A}(e_0)$ and $1_{Q_A}-q=r^2$), we obtain 
\[
 1_{Q_A}-b
 \precsim \overline{j_A}(e_0)\oplus(1_{Q_A}-\overline{j_A}(g)).
\]
Using \eqref{13.7.16}, \eqref{e13.7.g},  and \eqref{e13.7.etr-2},
we obtain
\[
 \begin{aligned}
 e
 &\precsim1_{Q_A}-b
 \precsim \overline{j_A}(e_0)\oplus(1_{Q_A}-\overline{j_A}(g))\\
 &\precsim \overline{j_A}(a_1)\oplus \overline{j_A}(a_0)
\precsim \overline{j_A}(a).
 \end{aligned}
\]
This completes the proof.
\end{proof}

\begin{cor}\label{cor13.12}
Let $A$ be a non-elementary separable simple $C^*$-algebra with
continuous scale and with property sETR1.
Then every tracial state on $A$ is quasidiagonal. Consequently,
$A$  is quasidiagonal.
\end{cor}

\begin{proof}
By Proposition \ref{Csetr},
$T(A)\neq\emptyset.$  Fix $\tau\in T(A).$

By Corollary \ref{Cstrictcomp}, $A$ has strict comparison. Hence, by  \eqref{A2.17c},
\[
   N_{\mathrm{cu}}(A)=I_A,
\]
so the present quotient \(Q_A\) agrees with the quotient introduced
at the beginning of Section~10. Fix a free ultrafilter \(\omega\) on
\(\mathbb N\). Taking the constant sequence \(\tau_n=\tau\) in the
definition of \(T_\infty(A)\), we obtain a trace
\(\overline{\tau}_\omega\in T_\infty(A)\subset T(Q_A)\) given by
\beq
 \overline{\tau}_\omega
 \bigl(\Pi_{\mathrm{cu}}((x_n))\bigr)
 :=
 \lim_{n\to\omega}\tau(x_n).
 \label{e13.72}
\eneq
In particular,
\beq
 \overline{\tau}_\omega(\overline{j_A}(x))=\tau(x),
 \qquad x\in A.
 \label{e13.73}
 \eneq

Let \(\varepsilon>0\), and let \(\mathcal F\subset A_1\) be a
finite subset. Replacing \(\mathcal F\) by
\(\mathcal F\cup\mathcal F^*\), we may assume that
\(\mathcal F\) is self-adjoint. Choose \(0<\delta,\eta<1/2\)
such that
\beq
 \frac{2\delta}{1-\delta}<\frac{\varepsilon}{3}
 \quad\text{and}\quad
 \eta<\frac{\varepsilon}{18}.
 \label{e13.76}
\eneq
Choose $0\neq a\in A_+$ such that (see Remark \ref{RCuntznull})
\beq
 d_\tau(a)<\delta.
 \label{e13.77}
\eneq

Apply Theorem  \ref{prop12.5a}
to \(\mathcal F\), \(a\), and \(\eta\).
We obtain a projection \(e\in Q_A\) and a finite-dimensional
\SCA\, $E\subset Q_A$ such that, with
$
 p:=1_{Q_A}-e=1_E,
$
we have
\beq \label{e13.78}
&& \|p\overline{j_A}(x)-\overline{j_A}(x)p\|<\eta,
 \qquad x\in\mathcal F,
\\\label{e13.79}
&& \operatorname{dist}(p\overline{j_A}(x)p,E)<\eta,
 \qquad x\in\mathcal F,\andeqn\\\label{e13.80}
 &&e\lesssim \overline{j_A}(a).
\eneq
Since \(e\) is a projection,  \eqref{e13.73},
\eqref{e13.77}, and \eqref{e13.80} give
\beq
 \overline{\tau}_\omega(e)
 =
 d_{\overline{\tau}_\omega}(e)
 \leq
 d_{\overline{\tau}_\omega}(\overline{j_A}(a))
 =
 d_\tau(a)
 <\delta.
 \label{e13.81}
\eneq
In particular,
\[
 \alpha:=\overline{\tau}_\omega(p)>1-\delta.
\]

Define a tracial state on \(pQ_Ap\) by
\[
 \tau_p(z):=
 \frac{\overline{\tau}_\omega(z)}{\alpha}\rforal z\in pQ_Ap.
\]
Let
$
 \pi_{\tau_p}:pQ_Ap\longrightarrow B(H_{\tau_p})
$
be the GNS representation associated with the state $\tau_p,$ and put
\[
 M_\tau:=\pi_{\tau_p}(pQ_Ap)'',
 \qquad
 N:=\pi_{\tau_p}(E).
\]
Then \(N\) is a finite-dimensional von Neumann subalgebra of
\(M_\tau\), with the same identity. Hence there is a
trace-preserving conditional expectation
\[
 \mathbb E_N:M_\tau\longrightarrow N.
\]
Define
$
 \psi:A\longrightarrow N$ by 
 $\psi(x):=
 \mathbb E_N\bigl(\pi_{\tau_p}(p\overline{j_A}(x)p)\bigr)$ for all $x\in A.$
This is a \cpc.

For \(x\in\mathcal F\), put
\[
 u_x:=\pi_{\tau_p}(p\overline{j_A}(x)p).
\]
By \eqref{e13.79},
\[
 \operatorname{dist}(u_x,N)<\eta.
\]
For \(x,y\in\mathcal F\), we have
\[
 u_{xy}-u_xu_y
 =
 \pi_{\tau_p}\bigl(p\overline{j_A}(x)(1-p)\overline{j_A}(y)p\bigr).
\]
Since
\[
 (1-p)\overline{j_A}(y)p
 =
 (1-p)(\overline{j_A}(y)p-p\overline{j_A}(y)),
\]
applying \eqref{e13.78}, we obtain
\beq
 \|u_{xy}-u_xu_y\|<\eta.
 \label{e13.84}
\eneq
Choose \(z_x,z_y\in N\) such that
\[
 \|u_x-z_x\|<\eta
 \andeqn
 \|u_y-z_y\|<\eta.
\]
Since \(\mathbb E_N\) is contractive and fixes \(N\),
\beq
&& \|\psi(x)\psi(y)-z_xz_y\|
 <2\eta+\eta^2,\\
&& \|\mathbb E_N(u_xu_y)-z_xz_y\|=\|\mathbb E_N(u_xu_y-z_xz_y)\|
 <2\eta+\eta^2.
\eneq
Together with \eqref{e13.84}, these give
\beq
 \|\psi(xy)-\psi(x)\psi(y)\|&=&\|\mathbb E_N(u_{xy})-\psi(x)\psi(y)\|\\
 &<& \eta+\|\mathbb E_N(u_xu_y)-\psi(x)\psi(y)\|\\
 &<&\eta+\|\mathbb E_N(u_xu_y)-z_xz_y\|+\|z_xz_y-\psi(x)\psi(y)\|\\
 &<&\eta+4\eta+2\eta^2
 <6\eta
 <\frac{\varepsilon}{3}.
 \label{e13.85}
\eneq

Let \(t_N\) denote the restriction to \(N\) of the normal
tracial state on \(M_\tau\). Since \(\mathbb E_N\) is
trace-preserving,  
\[
 t_N(\psi(x))
 =
 \frac{\overline{\tau}_\omega(p\overline{j_A}(x)p)}{\alpha}\rforal x\in A.
 \label{e13.86}
\]
Put
$
 \beta:=\overline{\tau}_\omega(e),
$
so that $\alpha=1-\beta.$  Since $\overline{\tau}_\omega$ is a trace, we have 
\[
 \overline{\tau}_\omega(p\overline{j_A}(x)e)
 =
 \overline{\tau}_\omega(e\overline{j_A}(x)p)=0.
\]
Consequently,
\[
 \tau(x)
 =
 \overline{\tau}_\omega(p\overline{j_A}(x)p)
 +
 \overline{\tau}_\omega(e\overline{j_A}(x)e).
\]
Since \(\|x\|\leq1\),
\[
 \left|
 \overline{\tau}_\omega(e\overline{j_A}(x)e)
 \right|
 \leq\beta.
\]
It follows that
\beq \label{e13.87}
 \begin{aligned}
 \left|t_N(\psi(x))-\tau(x)\right|
 &=
 \left|
 \frac{\overline{\tau}_\omega(p\overline{j_A}(x)p)}{1-\beta}
 -\tau(x)
 \right|\\\label{e13.87}
 &=
 \frac{\left|
 \beta\tau(x)-\overline{\tau}_\omega(e\overline{j_A}(x)e)
 \right|}{1-\beta}\\
 &\leq
 \frac{2\beta}{1-\beta}
 <
 \frac{2\delta}{1-\delta}
 <\frac{\varepsilon}{3}.
 \end{aligned}
\eneq

Since \(N\) is finite-dimensional, the tracial state \(t_N\)
can be approximated on \(\psi(\mathcal F)\) by normalized
matrix traces. Thus there are \(k\in\mathbb N\) and a unital
homomorphism
$
 \rho:N\longrightarrow M_k
$
such that
\beq
 \left|
 \operatorname{tr}_k(\rho(z))-t_N(z)
 \right|
 <\frac{\varepsilon}{3}\rforal z\in\psi(\mathcal F).
 \label{e13.88}
\eneq
Set
\[
 \Phi:=\rho\circ\psi:A\longrightarrow M_k.
\]
By \eqref{e13.85},
\[
 \|\Phi(xy)-\Phi(x)\Phi(y)\|<\varepsilon\rforal x,y\in\mathcal F,
\]
and, by \eqref{e13.87} and \eqref{e13.88},
\[
 \left|
 \operatorname{tr}_k(\Phi(x))-\tau(x)
 \right|
 <\ep\rforal x\in\mathcal F.
\]
Therefore \(\tau\) is quasidiagonal.

Since \(\tau\in T(A)\) was arbitrary, every tracial state on
\(A\) is quasidiagonal. Finally, $T(A)\not=\emptyset,$ and
every tracial state on the simple algebra \(A\) is faithful.
Therefore \(A\) is quasidiagonal; see
\cite[Proposition~1.4, footnote~3, and
Remark~1.5(iii)]{TWW}.
\end{proof}

{\bf Proof of Theorem \ref{TT}}:

\begin{proof}
If $T(A)=\emptyset,$ the first assertion is vacuous. Suppose that
$T(A)\ne\emptyset.$ Since $A$ is simple, every tracial state on
$A$ is faithful; in particular, $A$ is finite.

Put
$
 B=A\otimes\mathcal Z.
$
Then $B$ is a unital, separable, simple, finite, amenable,
$\mathcal Z$-stable \CA. By Corollary \ref{CorZ}, $B$ has property
$s\mathrm{ETR}1$. Hence Corollary \ref{cor13.12} shows that every tracial
state on $B$ is quasidiagonal.

Let $\tau\in T(A)$, and let $\tau_{\mathcal Z}$ be the unique
tracial state on $\mathcal Z$. Then
$\tau\otimes\tau_{\mathcal Z}$ is a quasidiagonal tracial state on
$B$. Its restriction along the embedding
\[
 j:A\longrightarrow B,\qquad j(a)=a\otimes1_{\mathcal Z},
\]
is $\tau$. Since the restriction of a quasidiagonal trace to a
\SCA\, is quasidiagonal, $\tau$ is quasidiagonal. Thus
every tracial state on $A$ is quasidiagonal.

Finally, since $T(A)\ne\emptyset,$  $A$ has a faithful
quasidiagonal tracial state. It follows that $A$ is quasidiagonal.

\end{proof}

\begin{rem}\label{Rprojection}
Let $A$ be a separable  finite simple amenable \CA\, which is not stably projectionless.
Let $p\in A\otimes {\cal K}$ be a non-zero projection and $B=p(A\otimes {\cal K})p.$ 
Then $B$ is a unital finite separable simple amenable \CA.  
By Theorem \ref{TT},  every trace on $B$ is quasidiagonal.  This actually implies that every 
trace of $A$ is quasidiagonal and $A$ is a quasidiagonal \CA. 
\end{rem}

{\bf Proof of Corollary \ref{CMC}}:

\begin{proof}
The equivalence of (1) and (2) follows from Theorem \ref{TZsetr}.
Clearly, finite decomposition rank implies finite nuclear dimension (\cite[Remark 2.2 (ii)]{WZ10nuc}),
so (4) implies (3). Conversely,  by \cite[Corollary 6.3]{WinterPure}, finite nuclear dimension implies
$\mathcal Z$-stability for a non-elementary, separable, simple,
nuclear \CA; hence (3) implies (2).

Suppose that (2) holds. By Theorem \ref{TT}, every tracial state on $A$
is quasidiagonal. Therefore \cite[Theorem~B]{CETWW} gives
\[
 \operatorname{dr}(A)\leq1.
\]
The same theorem also gives
\[
 \dim_{\mathrm{nuc}}(A)\leq1.
\]
Thus (2) implies both (3) and (4), and all four conditions are
equivalent.

Finally, nuclear dimension zero and decomposition rank zero are
each equivalent to $A$ being AF in the present setting. Therefore,
if $A$ is AF, then
\[
 \dim_{\mathrm{nuc}}(A)=\operatorname{dr}(A)=0.
\]
If $A$ is not AF, neither invariant can be zero; since both are at
most one, they are both equal to one.
\end{proof}

\section{Appendix}

The following is a very special case of Theorem \ref{TmainZ}.
We present  a direct proof that the Jiang--Su algebra
$\mathcal Z$ has property ETR1 associated with $1_{\mathcal Z}$ and
constant $\kappa_1=1$.  This serves as an example. The argument was initially suggested by
ChatGPT-5.6-sol and was subsequently revised  and  checked  by the
authors.   In fact, it  proves the slightly stronger statement that, for every
prescribed integer \(N\geq 1\), the finite-dimensional algebra \(D\) in  Definition \ref{D1}
may be chosen so that every irreducible representation of \(D\) has
dimension at least \(N\).  Of course once  one shows that ${\cal Z}$ has property ETR1, 
this follows from Proposition \ref{Prank}.  
We include it for reference and convenience.

\begin{prop}\label{Apz}
The Jiang--Su algebra $\Zjs$ has property  ETR1 (with $e_A=1$ and $\kappa=1$).
\end{prop}


\begin{proof}
Choose $e_{{\cal Z}}=1_{\cal Z}.$
Fix $\varepsilon>0$, a finite set $\F\subset \Zjs$, a nonzero positive element
$a\in \Zjs$, and an integer $N\geq 1$.

Let $\tau$ be the unique tracial state of $\Zjs$. Since $\Zjs$ is simple and
$a\neq 0$, we have $d_\tau(a)>0$. Choose $\eta>0$ so small that
$
\eta<\varepsilon/3
\andeqn
\eta<d_\tau(a).
$

Write $\Zjs$ as an inductive limit of prime dimension-drop algebras. Passing
sufficiently far out in the system, we may regard
\[
B=Z_{p,q}
\]
as a unital subalgebra of $\Zjs$, with $p,q\geq N$, $(p,q)=1,$ such that for every
$x\in\F$ there is $y_x\in B$ satisfying $\|x-y_x\|<\eta$.

Recall that
\[
Z_{p,q}
=
\left\{
f\in C([0,1],M_p\otimes M_q):
f(0)\in M_p\otimes 1_q,\ 
f(1)\in 1_p\otimes M_q
\right\}.
\]

The restriction of $\tau$ to $B$ has the form
\[
\tau(f)=\int_0^1 {\rm tr}_{pq}(f(t))\,d\mu(t)
\]
for a Borel probability measure $\mu$ on $[0,1]$.

Since  the finite family $\{y_x:x\in\F\}$ is equi-continuous, choose
$\delta>0$ such that
\[
|s-t|<\delta
\quad\Longrightarrow\quad
\|y_x(s)-y_x(t)\|<\eta
\]
for every $x\in\F$.
Choose partition points
\[
0=r_0<r_1<\cdots<r_m<r_{m+1}=1
\]
such that $\max_{1\le i\le m+1}\{|r_i-r_{i-1}|\}<\dt.$
 Since $\mu$ has at most countably many atoms,
we may also arrange that $\mu(\{r_j\})=0$ for $j=1,2,...,m.$.
Let 
\[
d:=\min_{0\le i\le m} |r_{i+1}-r_i|>0.
\]
Since $\mu(\{r_j\})=0$ for $1\leq j\leq m$, choose open intervals
$V_1,\ldots,V_m$ such that
\[
   r_j\in V_j,
   \qquad
   \overline{V_j}
   \subset
   \left(r_j-\frac d8,r_j+\frac d8\right),
   \qquad 1\leq j\leq m,
\]
and
\[
   \mu\left(\bigcup_{j=1}^m V_j\right)<\eta.
\]
The sets $\overline{V_j}$ are then pairwise disjoint.

For $0\leq i\leq m$, put
\[
   J_i:=
   \left(r_i+\frac d8,r_{i+1}-\frac d8\right).
\]
Since $r_{i+1}-r_i\geq d$, we have
\[
   |J_i|
   =r_{i+1}-r_i-\frac d4
   \geq\frac d4>0.
\]
Moreover,
\[
   \overline{J_i}\cap\bigcup_{j=1}^m\overline{V_j}
   =\emptyset,
   \qquad 0\leq i\leq m.
\]
For each $j$, choose a continuous function $e_j:[0,1]\to[0,1]$ such that
\[
e_j(r_j)=1\andeqn
\supp(e_j)\subset V_j.
\]

Set
\[
e=\sum_{j=1}^m e_j,
\]
viewed as the scalar element $t\mapsto e(t)1_{pq}$ of $B$.
Since the supports are pairwise disjoint, $0\leq e\leq 1$. Moreover, $e$ is
central in $B$, and $e(r_j)=1$ for every $j$.

For $j=0,\ldots,m$, define
$h_j\in C([0,1])_+$ by
\[
h_j(t)=
\begin{cases}
1-e(t),& t\in[r_j,r_{j+1}],\\
0,& t\notin[r_j,r_{j+1}].
\end{cases}
\]
Because $(1-e)(r_j)=0$ for $1\le j\le m,$
 each $h_j$ is
continuous. Also,
\[
h_i h_j=0\quad(i\neq j)\andeqn
\sum_{j=0}^m h_j=1-e.
\]
Note that $h_k(t)=1$ for $t\in J_k,$ $0\le k\le m.$
Hence $\|h_j\|=1$ for every $j$.

Now define
\[
D=M_p\oplus M_{pq}^{\oplus(m-1)}\oplus M_q.
\]
Every matrix summand of $D$ has size at least $N$.

Define maps on the summands of $D$ to $B=Z_{p,q}$  as  follows:
\beq\nonumber
&&\theta_0(b)(t)=h_0(t)(b\otimes 1_q),\qquad\,\,\, b\in M_p,\\\nonumber
&&\theta_j(b)(t)=h_j(t)b,\qquad \hspace{0.5in}b\in M_{pq},\ 1\leq j\leq m-1,
\andeqn\\
&&\theta_m(b)(t)=h_m(t)(1_p\otimes b),\qquad b\in M_q.
\eneq

Each $\theta_j$ is c.p.c.\ and order zero, 
because it is a
$*$-homomorphism multiplied by a commuting positive contraction. 
Since the
functions $h_j$ are pairwise orthogonal, the ranges of the $\theta_j$ are
mutually orthogonal. Therefore
\[
\theta=\bigoplus_{j=0}^m\theta_j:D\longrightarrow B\subset\Zjs
\]
is c.p.c.\ and order zero.
Furthermore,
\[
\theta(1_D)=\sum_{j=0}^m h_j=1-e.
\]
Hence $\theta(D)\subset \Her(1-e)$. Moreover, on every summand $D_j$,
\[
\|\theta|_{D_j}\|=\|h_j\|=1.
\]
Thus $\theta$ is fully normalized.

Let
\[
\widetilde\theta:
C_0((0,1])\otimes D\longrightarrow\Zjs
\]
be the cone homomorphism associated with $\theta$.

Choose one point $s_j\in[r_j,r_{j+1}]$ for $1\le j\le m-1,$ and set $s_m=1.$ Define
$
L_0:B\longrightarrow D
$
by
\[
L_0(f)=\bigl(f(0),f(s_1),\ldots,f(s_{m-1}),f(1)\bigr),
\]
using the natural identifications of the endpoint fibers with $M_p$ and
$M_q$.
This is a unital $*$-homomorphism. Since $D$ is finite dimensional, it is an
injective operator system. Therefore $L_0$ extends to a u.c.p.\ map
\[
L:\Zjs\longrightarrow D.
\]

Set $\Psi=\theta\circ L$. Then $\Psi$ is c.p.c., and its range is contained in
$
\theta(D)
$

We now verify the approximation. Fix $x\in\F$. On the first interval
$[0,r_1]$,
\beq
\theta \circ L_0(y_x)(t)=h_0(t)y_x(0)
\andeqn
(1-e(t))y_x(t)=h_0(t)y_x(t).
\eneq
Therefore, since $|t|<\dt,$ 
\beq
\|\theta \circ L_0(y_x)(t)-(1-e(t))y_x(t)\|=
\|h_0(t)y_x(0)-h_0(t)y_x(t)\|<\eta.
\eneq
Similarly, for $t\in [r_j, r_{j+1}],$ $j=1,2,...,m,$
\[
\|\theta\circ L_0(y_x)(t)-(1-e(t))y_x(t)\|=\|h_j(t)y_x(s_j)-h_j(t)y_x(t)\|<\eta.
\]
%
Hence
\[
\|\theta\circ L_0(y_x)-(1-e)y_x\|<\eta.
\]

Since $e$ is central in $B$,
\[
(1-e)y_x=(1-e)^{1/2}y_x(1-e)^{1/2}.
\]
Therefore (using $\|x-y_x\|<\eta$)
\beq
\begin{aligned}
\|\Psi(x)-(1-e)^{1/2}x(1-e)^{1/2}\|
&\leq \|\theta \circ L(x-y_x)\| 
+\|\theta  \circ L_0(y_x)-(1-e)y_x\| \\\label{Z-2a}
&\quad\hspace{0.3in}+\|(1-e)^{1/2}(y_x-x)(1-e)^{1/2}\| \\
&<3\eta<\varepsilon.
\end{aligned}
\eneq

Next, since $e$ and $(1-e)^{1/2}$ are central in $B$, they commute with
$y_x$. Hence
\beq\label{Z-1}
\|[e,x]\|\leq 2\|x-y_x\|<2\eta<\varepsilon,
\eneq
and similarly
\[
\|[(1-e)^{1/2},x]\|<\varepsilon.
\]

It remains to show that $e\precsim a$. The support of $e$ is contained in
$\bigcup_jV_j$. Hence
\[
d_\tau(e)=\mu(\{t:e(t)>0\})
\leq \mu\left(\bigcup_jV_j\right)<\eta.
\]
Since $\eta<d_\tau(a)$, we have $d_\tau(e)<d_\tau(a)$. The Jiang--Su algebra
has strict comparison, so 
\beq\label{Z-3}
e\precsim a.
\eneq

Finally, 
define
$
\widehat c_e=\iota\otimes 1_D
\in C_0((0,1])\otimes D.
$
Then
\[
\widetilde\theta(\widehat c_e)=\theta(1_D)=1-e.
\]
Therefore
\[
(1-e)^{1/2}1_\Zjs(1-e)^{1/2}
=
1-e
=
\widetilde\theta(\widehat c_e).
\]
Also,
$
\pi_1(\widehat c_e)=1_D.
$
Hence every tracial state $t\in T(D)$ satisfies
\beq\label{Z-2b}
t\bigl(\pi_1(\widehat c_e)\bigr)=1.
\eneq
%
Thus,  recall that $\theta: D\to \overline{(1-e){\cal Z}(1-e)}$ is fully normalized, by \eqref{Z-1}, \eqref{Z-2a}, \eqref{Z-2b} and \eqref{Z-3}, all the requirements of
ETR1 are satisfied, and $\Zjs$ has property ETR1.
\end{proof}

\noindent
hlin@uoregon.edu and hxlin@simis.cn

\end{document}